\documentclass[a4paper]{amsart}
\usepackage[a4paper, margin=1in]{geometry}

\usepackage{amsthm,amsmath}
\usepackage{amssymb, amsfonts}
\usepackage{caption}
\usepackage{chngcntr}
\usepackage{color}
\usepackage{dsfont}
\usepackage{enumitem}
\usepackage{fancyhdr}
\usepackage{float}
\usepackage{graphicx}
\usepackage{hyperref}
\usepackage{letterspace}
\usepackage{mathdots}
\usepackage{mathtools}
\usepackage{marvosym}
\usepackage[none]{hyphenat}
\usepackage{tikz}
\usepackage{tikz-cd}
\usepackage{ulem}
\usepackage{upgreek}
\usepackage{wasysym}
\usepackage{lipsum}
\usepackage{setspace}
\usepackage{svrsymbols}
\newcommand{\EightStarConvex}{%
  \raisebox{-0ex}{\scalebox{0.55}{%
    {\usefont{U}{ding}{m}{n}\symbol{'106}}}}}

\usepackage{cleveref}

\usetikzlibrary{trees}
\usetikzlibrary{arrows.meta}

\newtheorem{thm}{Theorem}[section]
\newtheorem{theorem}[thm]{Theorem}

\newtheorem{lemma}[thm]{Lemma}

\theoremstyle{definition}

\newtheorem{notation}[thm]{Notation}

\newtheorem{fact}[thm]{Fact}
\newtheorem{definition}[thm]{Definition}

\newtheorem{question}[thm]{Question}

\newtheorem{proposition}[thm]{Proposition}
\newtheorem{corollary}[thm]{Corollary}

\newtheorem{remark}[thm]{Remark}

\newtheorem{rem/def}[thm]{Remark/Definition}
\newtheorem{rem/not}[thm]{Remark/Notation}

\makeatletter
\let\c@table\c@thm
\makeatother

\makeatletter
\let\c@figure\c@thm
\makeatother

\newtheorem*{thm-one-var}{One-variable theorem for a dividing line D}
\newtheorem*{thm-n-var-NCTP}{$n$-variable theorem for NCTP}

\def \lex {<_{lex}}

\newcommand{\trn}{%
  \mathrel{\ooalign{$\lneq$\cr\raise.21ex\hbox{$\lhd$}\cr}}}
\newcommand{\trrn}{%
  \mathrel{\ooalign{$\gneq$\cr\raise.21ex\hbox{$\rhd$}\cr}}}
\def \coc {{^{\frown}}}

\def \cl {{\rm cl}}

\def \la {\langle}
\def \ra {\rangle}
\def \Tau {\mathcal{T}}

\def \tp {\textrm{tp}}
\def \dom {\textrm{dom}}

\def \age {\textrm{age}}

\def \CC {\mathcal{C}}
\def \CI {\mathcal{I}}
\def \CJ {\mathcal{J}}
\def \CL {\mathcal{L}}
\def \CK {\mathcal{K}}
\def \CP {\mathcal{P}}

\def \cf {\text{cf}}
\def \tp {\operatorname{tp}}
\def \Th {\operatorname{Th}}
\def \acl {\operatorname{acl}}
\def \dcl {\operatorname{dcl}}
\def \span {\operatorname{span}}
\def \qftp {\operatorname{qftp}}
\def \Fraisse {Fra\"{\i}ss\'e\;\,}

\DeclareFontFamily{U}{txoflocal}{}
\DeclareFontShape{U}{txoflocal}{m}{n}{<-> txr-of}{}
\DeclareMathAlphabet{\mathbbof}{U}{txoflocal}{m}{n}

\title{$n$-variable theorems for dividing lines characterized by positive consistency-inconsistency configurations, and their applications to preservation problems}
\author[Joonhee Kim]{Joonhee Kim}
\address{Korea Institute for Advanced Study, School of Mathematics\\ 85 Hoegiro Dongdaemun-gu\\ 02455, Seoul, South Korea}
\email{kimjoonhee@kias.re.kr}
\date{\today}

\begin{document}


\begin{abstract}
We define a class of classes of complete first-order theories, denoted by $\mathfrak{D}_{n\text{-var}}$, in terms of positive consistency-inconsistency configurations and generalized indiscernibles. $\mathfrak{D}_{n\text{-var}}$ contains the class of stable theories, the class of simple theories, the class of NIP theories, the class of NTP$_1$ theories, the class of NTP$_2$ theories, the class of NATP theories, the class of NCTP theories, the class of NBTP theories, the class of theories not having a $(k,1,1)$-weave of depth $\omega$ for any $k<\omega$, the class of theories not having an infinite $k$-grid for any $k<\omega$, and the class of NPM$^{(k)}$ theories, for each $1<k<\omega$.

We prove that for any dividing line $D\in\mathfrak{D}_{n\text{-var}}$ and any complete first-order theory $T\notin D$, we can always find a formula $\varphi(x,y)$ witnessing $T\notin D$  with some indexed set of parameters such that the set of instances of $\varphi$ required to be consistent has a realization whose algebraic dimension over the whole set of parameters is the same as $|x|$. We call statements of this form $n$-variable theorems, as they may be regarded as weak versions of one-variable theorems. We show that for any $D\in\mathfrak{D}_{n\text{-var}}$ and any complete theory $T$ satisfying the appropriate hypotheses,
\vspace{1.5pt}
\begin{itemize}
\item[(i)] $T\in D$ if and only if $T^{gt}\in D$,
\item[(ii)] $T\in D$ if and only if $T_{P}\in D$,
\item[(iii)] $T\in D$ if and only if $T^{ind}\in D$,
\item[(iv)] $T\in D$ if and only if $T^\text{G}_K\in D$,
\item[(v)] $T\in D$ if and only if $\text{ACF}_T\in D$,
\item[(vi)] $T\in D$ if and only if $T^{\bar\delta,nc}_g\in D$,
\end{itemize}
\vspace{1pt}
where $T^{gt}$, $T_P$, $T^{ind}$, $T^\mathrm{G}_K$, $\mathrm{ACF}_T$, and $T^{\bar\delta,nc}_g$ are the theory of the generic trivialization of $T$, the theory of the lovely pair expansion of $T$, the theory of the $H$-structure expansion of $T$, the theory of a vector space with a dense-codense generic $K$-subspace, the theory of an algebraically closed field with a distinguished subfield, and an arbitrary completion of the theory of generic non-commutative derivations of an algebraically bounded field, studied in \cite{BV14}, \cite{BV10}, \cite{BV16}, \cite{BdV22}, \cite{dKN21}, and \cite{FT26}, respectively. In proving (i), (ii), (iii), (iv), and (v), the $n$-variable theorem plays an essential role. We also prove the construction of a special case of Mordell-Lang theories \cite{BHK20} preserves every dividing line in $\mathfrak{D}_{n\text{-var}}$.

We extend this framework by introducing a larger class $\mathfrak{D}^h_{n\text{-var}}\supseteq\mathfrak{D}_{n\text{-var}}$, which captures some higher-arity dividing lines such as NOP$_k$, NFOP$_k$, and NIP$_k$ for  $1<k<\omega$. We show that the $n$-variable theorem and the preservation results (ii), (iii), (iv), and (vi) listed above still hold for $\mathfrak{D}^h_{n\text{-var}}$. (i) also holds for $\mathfrak{D}^h_{n\text{-var}}$ assuming $\acl=\dcl$. 
\end{abstract}
\maketitle
\vspace{-37pt}
\begin{figure}[ht]
\[
\scalebox{1}
{
\begin{tikzpicture}[>={stealth[length=1pt,width=6pt]}, x={(1.6cm,0cm)}, y={(0cm,0.82cm)}, z={(-0.5cm,-0.71cm)}]

\footnotesize
\hspace{5pt}
\node (stable) at (0,0,0) {stable};
\node (NIP) at (0,0,-4) {NIP};
\node (NTP)   at (1,0,0) {simple};
\node (NTP2)   at (1,0,-4) {NTP$_2$};
\node (NTP1)   at (3,0,0) {NTP$_1$};
\node (NSOP4)   at (4,0,0) {NSOP$_4$};
\node (NSOP5)   at (5,0,0) {NSOP$_5$};
\node (cdots)   at (5.81,0,0) {$\cdots$};
\node (NBTP)   at (1.65,0,-1.5) {NBTP};
\node (NCTP)   at (2.,0,-2.4) {NCTP};
\node (NATP)   at (3.,0,-2.4) {NATP};
\node (NWP)   at (2.,0,-4) {NWP};
\node (NGP)   at (3.,0,-4) {NGP};
\node (NPM2)   at (3.75,0,-4.2) {NPM$^{(2)}$\!};
\node (NPM3)   at (4.575,0,-4.4) {NPM$^{(3)}$\!\!\!\!};
\node (cdots2) at (5.3,0,-4.5) {\rotatebox{5}{$\cdots$}};

\node (o-minimal) at (-2.8,0,-4) {o-minimal};
\node (distal) at (-1.8,0,-4) {distal};
\draw[->] (o-minimal) -- (distal);
\node (NIPfake) at (-0.12,0,-4) {};
\draw[->] (distal) -- (NIPfake);

\node[fill=white,inner sep=0pt]
      (distal-nip0) at (0,3.465,0) {\color{white}o};
\node[fill=white,inner sep=0pt]
      (distal-nip1) at (0.313,3.5,0) {\color{white}o};
\node[fill=white,inner sep=0pt]
      (distal-nip2) at (0.622,3.47,0) {\tiny\color{white}o};
\node[fill=white,inner sep=0pt]
      (distal-nip3) at (0.786,3.47,0) {\tiny\color{white}o};
      
\node (NOP2) at (0,1.3,0) {NOP$_2$};
\node (NIP2fake) at (0,1.3,-4) {};
\node (NIP3fake) at (0,2.6,-4) {};

\node[fill=white,inner sep=2pt]
      (NOP3) at (0,2.6,0) {\!\!\!\!NOP$_3$\!\!\!};
\node (NFOP3) at (0,2.6,-2) {\!\!\!\!NFOP$_3$\!\!\!\!\!};
\node[fill=white,inner sep=2pt]
      (NOP4) at (0,3.9,0) {NOP$_4$};
\node (NFOP4) at (0,3.9,-2) {\!\!\!\!NFOP$_4$\!\!\!\!\!};
\node (NIP4) at (0,3.9,-4) {NIP$_4$};

\node (vdots0) at (0,5.2,0) {$\vdots$};
\node (vdots1) at (0,5.2,-2) {$\vdots$};
\node (vdots2) at (0,5.2,-4) {$\vdots$};

\draw[->] (stable) -- (NIP);
\draw[->] (stable) -- (NTP);
\draw[->] (NTP) -- (NTP2);
\draw[->] (NTP) -- (NTP1);
\draw[->] (NIP) -- (NTP2);
\draw[->] (NTP1) -- (NSOP4);
\draw[->] (NSOP4) -- (NSOP5);
\draw[->] (NSOP5) -- (cdots);
\draw[->] (NTP) -- (NBTP);
\draw[->] (NTP1) -- (NBTP);
\draw[->] (NTP2) -- (NBTP);
\draw[->] (NBTP) -- (NCTP);
\draw[->] (NTP1) -- (NCTP);
\draw[->] (NTP2) -- (NCTP);
\draw[->] (NTP1) -- (NATP);
\draw[->] (NTP2) -- (NWP);
\draw[->] (NCTP) -- (NATP);
\draw[->] (NCTP) -- (NWP);
\draw[->] (NATP) -- (NGP);
\draw[->] (NWP) -- (NGP);
\draw[->] (NGP) -- (NPM2);
\draw[->] (NPM2) -- (NPM3);

\node (NPM3fake2)   at (4.9,0,-4.39) {};
\node (cdots2fake) at (5.2,0,-4.46) {};
\node (cdots2fake2) at (5.3725,0,-4.525) {};
\node (NPMfake) at (5.66,0,-4.6) {};
\draw[->] (NPM3fake2) -- (cdots2fake);

\node (NPM3fake)   at (4.305,0,-4.42) {};
\node (NIP2fake2) at (0.05,1.34,-3.8) {};
\draw[->] (NIP2fake2) to[out=-20, in=180, out distance=12mm, in distance=20mm] (NPM3fake);

\node (NTP22)   at (1,1.3,-4) {NTP$_2^{*2}$};
\node (NTP23)   at (1,2.6,-4) {NTP$_2^{*3}$};
\node (NTP24)   at (1,3.9,-4) {NTP$_2^{*4}$};
\node (vdots3) at (1,5.2,-4) {$\vdots$};

\draw[->] (stable) -- (NOP2);
\draw[->] (NIP) -- (NOP2);

\node[fill=white,inner sep=0.7pt]
      (NFOP2) at (0,1.3,-2) {NFOP$_{2\hspace{1.8pt}}$\!\!};

\draw[->] (NIP) -- (NFOP2);
\draw[->] (NOP2) -- (NFOP2);
\draw[->] (NOP2) -- (NOP3);
\draw[->] (NFOP2) -- (NOP3);
\draw[->] (NFOP2) -- (NFOP3);
\draw[->] (NIP2fake) -- (NFOP3);
\node[fill=white,inner sep=0.5pt]
      (NIP2) at (0,1.3,-4) {NIP$_2$};
\draw[->] (NIP) -- (NIP2);
\draw[->] (NFOP2) -- (NIP2);
\draw[->] (NOP3) -- (NFOP3);
\node (NOP4fake) at (0,3.87,0) {};
\draw[->] (NOP3) -- (NOP4fake);
\draw[->] (NFOP3) -- (NFOP4);
\draw[->] (NIP3fake) -- (NFOP4);
\node[fill=white,inner sep=0.5pt]
      (NIP3) at (0,2.6,-4) {NIP$_3$};
\draw[->] (NIP2) -- (NIP3);
\draw[->] (NFOP3) -- (NIP3);
\draw[->] (NIP3) -- (NIP4);
\draw[->] (NOP4) -- (NFOP4);
\draw[->] (NFOP4) -- (NIP4);

\draw[->] (NOP4) -- (vdots0);
\draw[->] (NFOP4) -- (vdots1);
\draw[->] (NIP4) -- (vdots2);

\node (NTP22fake)   at (1,1.3,-4) {};
\node[fill=white,inner sep=0pt]
      (NTP2-NTP22) at (1.31,1.55,-3) {\tiny\color{white}0};
\draw[->] (NTP2) -- (NTP22fake);

\draw[->] (NTP22) -- (NTP23);
\draw[->] (NTP23) -- (NTP24);
\draw[->] (NTP24) -- (vdots3);

\draw[->] (NIP2) -- (NTP22);
\draw[->] (NIP3) -- (NTP23);
\draw[->] (NIP4) -- (NTP24);

\node[fill=white,inner sep=2pt] 
      (simple2)   at (1,1.3,0) {?};
\draw[->] (NTP) -- (simple2);      
      
\node
      (NTP12)   at (3,1.3,0) {?};
\draw[fill=white,inner sep=2pt,->] 
      (NTP1) -- (NTP12); 

\end{tikzpicture}
}
\]
\vspace{-21pt}
\caption{:
\footnotesize 
\Cref{fig: 1} $\oplus$ \Cref{fig: m/nPn} $\oplus$ \Cref{fig: 2}
}
\label{fig: 0}
\end{figure}

\section{Introduction}

We introduce a uniform way of defining dividing lines that is closely connected to the model-theoretic notion of algebraicity. One of the main ways to define dividing lines in model theory is to check whether definable sets can form certain patterns. Many well-known and widely studied dividing lines, such as stability, simplicity, NIP, NTP$_1$, NTP$_2$, and NSOP$_n$, are defined in this way.

Such definitions usually involve two ingredients, namely a partitioned formula $\varphi(x,y)$ and an indexed set of parameters $\{a_\eta\}_{\eta\in I}$. The dividing line is determined by which subsets of $\{\varphi(x,a_\eta):\eta\in I\}\cup\{\neg\varphi(x,a_\eta):\eta\in I\}$ are consistent and which are inconsistent. Thus the second ingredient $\{a_\eta\}_{\eta\in I}$  carries two kinds of information, namely subsets $\mathfrak{X}$ and $\mathfrak{Y}$ of $\mathcal{P}(I)\times\mathcal{P}(I)$ such that
\begin{itemize}
\item[$\ast$] $\{\varphi(x,a_\eta):\eta\in X^+\}\cup\{\neg\varphi(x,a_\eta):\eta\in X^-\}$ is consistent for all $(X^+\!,X^-)\in\mathfrak{X}$,
\item[$\ast$] $\{\varphi(x,a_\nu):\nu\in Y^+\}\cup\{\neg\varphi(x,a_\nu):\nu\in Y^-\}$ is inconsistent for all $(Y^+\!,Y^-)\in\mathfrak{Y}$.
\end{itemize}

Shelah \cite{She99} introduced a general way to describe dividing lines from this point of view. Through the notion of straight definability\footnote{Shelah describes these properties in terms of the weak $\Gamma$-property, with a suitable choice of $\Gamma$ for each one.}, he observed that several familiar properties, including OP, IP, SOP, and TP, fit into this framework. 
More recently, Bailetti \cite{Bai24} and Bodirsky, Bodor, and Marimon \cite{BBM25} further developed and refined this approach, introducing positive maximality and the PM$^{(k)}$ hierarchy, and proving preservation results for core companions, respectively.
Using a slightly different approach, Garc\'ia and Mennuni \cite{GM22} characterized OP, IP, TP$_1$, TP$_2$, ATP, and SOP$_3$ in terms of $\Sigma$-SOP for suitable posets $\Sigma$. Day and Mutchnik \cite{DM26} also showed that SOP$_n$  is both straightly and poset definable and introduced higher-arity straight definability.

Our work continues this line of research. If we restrict our attention to instances of $\varphi$, without considering their negations, then the second ingredient $\{a_\eta\}_{\eta\in I}$ mentioned above can be described in terms of three components, namely 
\begin{itemize}
\item[\EightStarConvex] the index set $I$, 
\item[\EightStarConvex] a collection of subsets $X\subseteq I$ such that $\{\varphi(x,a_\eta):\eta\in X\}$ is required to be consistent, 
\item[\EightStarConvex] a collection of subsets $Y\subseteq I$ such that $\{\varphi(x,a_\nu):\nu\in Y\}$ is required to be inconsistent. 
\end{itemize}
In this paper, we add one more ingredient to this picture, namely 
\begin{itemize}
\item[\EightStarConvex] a language that allows us to view $I$ not merely as a set, but as a structure where generalized indiscernibility behaves well.
\end{itemize}
In other words, we consider a language $\CL^*$ and an $\CL^*$-structure $\CI^*$ on $I$. If we choose $\CL^*$ and its interpretation in $\CI^*$ carefully, then we can obtain generalized indiscernibility for $\CI^*$. 

\smallskip

In \Cref{def: n-var quadruple}, we will introduce a few natural conditions on these four ingredients, namely a language $\CL^*\!$, an $\CL^*$-structure $\CI^*\!$ on a set $I$, a subset $X$ of $I$, and a finite sequence $\bar Y$ of subsets of $I$. 
Each quadruple $(\CL^*\!,\CI^*\!,X,\bar Y)$ satisfying conditions in \Cref{def: n-var quadruple} canonically determines a dividing line $D^{\CL^*\!\!,\,\CI^*}_{X,\bar Y}\!\!$, such that for any complete first-order theory $T$, $T\notin D^{\CL^*\!\!,\,\CI^*}_{X,\bar Y}\!$ if and only if there exist $\mathbb{M}\models T$, $\varphi(x,y)$, and $\CI^*$-indiscernible $(a_\eta)_{\eta\in\CI^*}\!\subseteq \mathbb{M}$ such that
\begin{itemize}
\item[$\ast$] $\{\varphi(x,a_\eta):\eta\in X\}$ is consistent,
\item[$\ast$] $\{\varphi(x,a_\nu):\nu\in Y\}$ is inconsistent for each $Y\in \bar Y$.
\end{itemize}
We will denote the class of all such dividing lines by $\mathfrak{D}_{n\text{-var}}$ (\Cref{def: D_n-var}). In \Cref{sec: D-n-var}, we will show that many known dividing lines can be defined in this way, including stability, simplicity, NIP, NTP$_1$, NTP$_2$, NATP, NCTP, NBTP, NWP, NGP, and NPM$^{(k)}$ for $1<k<\omega$.

\smallskip

A key advantage of defining dividing lines in this way is that the consistent part of a witness can be made as large as possible in terms of model-theoretic algebraic dimension. More precisely, if a complete theory $T$ does not belong to $D^{\CL^*\!\!,\,\CI^*}_{X,\bar Y}\!\!$, then there exist a witnessing formula $\varphi(x,y)$ and $(a_\eta)_{\eta\in\CI^*}$ such that there exists $b$ satisfying
\begin{itemize}
\item[$\ast$] $b\models\{\varphi(x,a_\eta):\eta\in X\}$,
\item[$\ast$] $\dim(b/(a_\eta)_{\eta\in\CI^*\!})=|x|$.
\end{itemize}
 We will call the statement describing this phenomenon the {\it $n$-variable theorem}\footnote{Here, the letter `$n$' is just a part of the name, not a parameter in the statement. The name is intended to emphasize that the theorem is a weaker version of the 1-variable theorem. Instead of reducing a witness to a single free variable, it allows an arbitrary finite tuple $x$, and ensures that the consistent part has full algebraic dimension $|x|$.} (\Cref{thm: n-var theorem}).

The $n$-variable theorem can be applied to preservation problems in several different ways.
Here, by a preservation problem, we mean a question asking whether a dividing line is closed under taking a model-theoretic construction.
When we expand a first-order structure or theory by adding new symbols to its language, or use it to build a new one, its model-theoretic features may change: some model-theoretic properties may appear, while others may be lost. 
Thus, when studying a new construction, it is natural to ask whether a first-order theory remains in a given dividing line after the construction.
Such preservation problems have been studied in many earlier works, particularly those introducing new model-theoretic constructions.

From \Cref{sec: T gt} to \ref{sec: generic derivation fields}, we study generic trivializations $T^{gt}$, lovely pair expansions $T_P$, $H$-structure expansions $T^{ind}$, generic subspace expansions $T^G_R$, algebraically closed fields with a distinguished subfield ACF$_T$, and generic non-commutative derivations $T^{\bar\delta,nc}_g$, introduced or studied in \cite{BV14}, \cite{BV10}, \cite{BV16}, \cite{BdV22}, \cite{dKN21}, and \cite{FT26}, respectively.
As mentioned above, these papers also discuss which dividing lines are preserved by the corresponding constructions.
In this paper, we extend these preservation results to a wider range of dividing lines within a single framework. Under the appropriate hypotheses for each construction, we prove that the original theory belongs to a dividing line in $\mathfrak{D}_{n\text{-var}}$ if and only if the resulting theory belongs to the same dividing line. More precisely, for any complete first-order theory $T$ and any dividing line $D\in\mathfrak{D}_{n\text{-var}}$, \vspace{-2pt}
\begin{itemize}
\item[(i)] $T\in D$ if and only if $T^{gt}\in D$  when $T$ is geometric,
\item[(ii)] $T\in D$ if and only if $T_{P}\in D$ when $T$ is geometric,
\item[(iii)] $T\in D$ if and only if $T^{ind}\in D$ when $T$ is geometric,
\item[(iv)] $T\in D$ if and only if $T^\text{G}_K\in D$ when $T$ is a geometric theory of $\mathbb{F}$-vector spaces with $\text{char}(\mathbb{F})=0$, $\acl=\dcl=\span_\mathbb{F}$, and quantifier elimination, possibly equipped with additional structure, for a subfield $K\subseteq\mathbb{F}$,
\item[(v)] $T\in D$ if and only if $\text{ACF}_T\in D$ when $T$ is a theory of fields, possibly equipped with additional structure,
\item[(vi)] $T\in D$ if and only if every (some) completion of $T^{\bar\delta,nc}_g$ belongs to $D$ when $T$ is a theory of algebraically bounded fields of characteristic $0$, possibly equipped with additional structure.\vspace{-2pt}
\end{itemize}
We also give a sufficient condition for Mordell-Lang theories introduced in \cite{BHK20} that makes the construction preserve every dividing line in $\mathfrak{D}_{n\text{-var}}$.

We will prove these results by transferring witnessing formulas from one theory to the other. And while doing this, the $n$-variable theorem plays an essential role in the first five constructions. We will see that the theorem can be useful in settings where model-theoretic algebraic independence or dimension plays an important role. The constructions $T^{gt}$, $T_P$, $T^{ind}$, and $T^G_R$ assume that the original theory $T$ is geometric, where algebraic closure satisfies the exchange property and algebraic independence and dimension behave well.
In this setting, we will see that the $n$-variable theorem interacts well with the density property, $H$-independence, and the relative quantifier elimination results in \cite{BV10, BV14, BK16} to provide useful techniques.
 In ACF$_T$, we do not assume that the original theory is geometric, but we still have good properties related to algebraic closure, such as the fact that any element outside the algebraic closure of the distinguished subfield can be sent to any other such element by an automorphism fixing this closure. This property and the stable embeddedness of the distinguished subfield allow us to use the $n$-variable theorem to prove the preservation result as well. Thus, when looking for applications of the $n$-variable theorem, we do not need to restrict ourselves to geometric theories.

As mentioned above, we define dividing lines from quadruples satisfying the conditions in \Cref{def: n-var quadruple} and prove preservation results for these dividing lines. However, we do not use all these conditions in every proof. For example, the preservation result for ACF$_T$ uses the modeling property and conditions (ii), (iii), and (vii) in \Cref{def: n-var quadruple}, but does not require the finiteness of $\bar Y$ in (i). For generic non-commutative derivations $T^{\bar\delta,nc}_g$, we only use the modeling property, the finiteness of $\bar Y$, and the cofinality condition (iii) in \Cref{def: n-var quadruple}. This means that our framework can be customized for different preservation problems. If we want to prove that a particular model-theoretic construction preserves a particular dividing line, then we can try to modify the conditions in \Cref{def: n-var quadruple} to fit that problem. We may weaken the conditions so that the resulting class $\mathfrak{D}_{n\text{-var}}$ includes the dividing lines we want to study, or strengthen them so that we can use stronger techniques to prove preservation. We explain this further in Subsection~4.2 and summarize the conditions used in each proof in \Cref{tab: conditions used}.

In recent years, higher-arity dividing lines have also been actively studied in model theory. In \Cref{sec: higher arity}, we extend our framework to $\mathfrak{D}^h_{n\text{-var}}\supseteq\mathfrak{D}_{n\text{-var}}$, to cover higher-arity dividing lines such as NOP$_k$, NFOP$_k$, and NIP$_k$. We introduce the family ${l\over k}$NP$_k$, which includes these three dividing lines as special cases. Our framework offers a new way to study these dividing lines through their connections with model-theoretic algebraic dimension. The $n$-variable theorem and several of our preservation results hold in this setting as well. More precisely, for any complete theory $T$ and any dividing line $D\in\mathfrak{D}^h_{n\text{-var}}$,
\begin{itemize}
\item[(vii)] $T\in D$ if and only if $T^{gt}\in D$  when $T$ is geometric such that $\acl=\dcl$,
\item[(viii)] $T\in D$ if and only if $T_{P}\in D$ when $T$ is geometric,
\item[(ix)] $T\in D$ if and only if $T^{ind}\in D$ when $T$ is geometric,
\item[(x)] $T\in D$ if and only if $T^\text{G}_K\in D$ when $T$ is a geometric theory of $\mathbb{F}$-vector spaces with $\text{char}(\mathbb{F})=0$, $\acl=\dcl=\span_\mathbb{F}$, and quantifier elimination, possibly equipped with additional structure, for a subfield $K\subseteq\mathbb{F}$,
\item[(xi)] $T\in D$ if and only if every (some) completion of $T^{\bar\delta,nc}_g$ belongs to $D$ when $T$ is a theory of algebraically bounded fields of characteristic $0$, possibly equipped with additional structure.
\end{itemize}
 In the process, we obtain a natural generalization of Takeuchi's OP$_2$ \cite{Tak17}, which we will call OP$_k$.

\section*{Acknowledgment}
The author thanks Alexander Berenstein, Yvon Bossut, Gabriel Day, James Hanson, Mark Kamsma, Elliot Kaplan, Itay Kaplan, Junguk Lee, Anand Pillay, and Nicholas Ramsey for the inspiring discussions and helpful comments.

The author was supported by a KIAS individual grant
(project no.~6G091801) and by an NRF of Korea grant
(No.~2021R1A2C1009639). 

The author developed the main ideas, results, and proof strategies of this paper. After completing the first draft, the author used GPT-6 Astra to check the proofs, and some minor corrections were made during this process. During this process, GPT pointed out minor errors caused by the author's mistake, and the author revised the arguments. No new results were added using AI. GPT was also used to correct typos and refine the overall presentation of the paper.

Subsection \ref{subsection: ML-theories} and Subsection \ref{subsection: non-commutative} are newly added in the third version of the paper. The author thanks Elliot Kaplan again for his helpful advice.

\pagebreak[0]
\section{Preliminaries}\label{sec: preliminaries}

\subsection{Dividing lines given by positive consistency-inconsistency configurations}

We first recall some terminology needed to review the main dividing lines considered in this paper.

\begin{notation}\label{notation: language of omega^<omega}
Let $\kappa$ and $\lambda$ be cardinals. Let $i, i_0,...,i_{n-1}$ be arbitrary elements of $\lambda$ with $n<\omega$. Let $\alpha$ be an arbitrary element of $\kappa$.
\begin{enumerate}
\item[(i)] $\lambda^0:=\{ \emptyset\}$.
\item[(ii)] $\lambda^\kappa$ is the set of all functions from $\kappa$ to $\lambda$, when $\kappa\neq0$.
\item[(iii)] $\lambda^{<\kappa}:=\bigcup_{\beta<\kappa}{\lambda^\beta}$.
\item[(iv)] $\la i_0, ...,i_{n-1}\ra$ is the function $\eta$ from $n$ to $\lambda$ such that $\eta(m)=i_m$ for each $m<n$.
\item[(v)]$\la i\ra^\alpha$ is the function $\eta$ from $\alpha$ to $\lambda$ such that $\eta(m)=i$ for all $m<\alpha$. Note that $\la i\ra^0=\emptyset$.
\end{enumerate}
Let $\eta,\nu\in \lambda^{<\kappa}$, $i<\lambda$, $\alpha<\kappa$.
\begin{itemize}
\item[(vi)] $\eta\unlhd\nu$ if $\eta \subseteq \nu$. $\lambda^{<\kappa}$ is partially ordered by $\unlhd$. 
\item[(vii)] $\eta$ and $\nu$ are {\it comparable} if $\eta\unlhd\nu$ or $\nu\unlhd\eta$.
\item[(viii)] $\eta\perp\nu$ if $\eta\not\!\!\unlhd\,\nu$ and $\nu\not\!\!\unlhd\,\eta$.
\item[(ix)] $\eta$ and $\nu$ are {\it incomparable} if $\eta\perp\nu$.
\item[(x)] $\eta\wedge\nu$ is the $\unlhd$-maximal element $\xi\in\lambda^{<\kappa}$ such that $\xi\unlhd\eta$ and $\xi\unlhd\nu$.
\item[(xi)] $l(\eta)$ is the domain of $\eta$.
\item[(xii)] $\eta<_l \nu$ if $l(\eta)< l(\nu)$. We write $\eta=_l\nu$ if $l(\eta)=l(\nu)$.
\item[(xiii)] $\eta\lex\nu$ if either $\eta\lhd\nu$, or $\eta\perp\nu$ and $\eta(l(\eta\wedge\nu))<\nu(l(\eta\wedge\nu))$. 
\item[(xiv)] $\eta\coc\nu:=\eta\cup\{(l(\eta)+i,\nu(i)):i< l(\nu)\}$. Note that $\emptyset\coc\nu$ is just $\nu$.
\item[(xv)] $\eta\lhd_i\nu$ if $\eta^\frown\la i\ra\unlhd\nu$.
\item[(xvi)] $t(\eta):=\eta(l(\eta)-1)$ if $l(\eta)$ is a successor ordinal.
\item[(xvii)] $\eta^-:=\eta|_{l(\eta)-1}$ if $l(\eta)$ is a successor ordinal.
\item[(xviii)] $P_\alpha:=\{\eta\in\lambda^{<\kappa}: l(\eta)=\alpha\}$.
\end{itemize}
Let $\eta_0,...,\eta_n\in\lambda^{<\kappa}$.
\begin{enumerate}
\item[(xix)] $\cl(\eta_0,...,\eta_n):=(\eta_0\wedge\eta_0,...,\eta_0\wedge\eta_n,...,\eta_n\wedge\eta_0,...,\eta_n\wedge\eta_n)$. 
\end{enumerate}
\end{notation}

The following notion of an ill-founded tree is a slight modification
of the one in \cite[Definition 5.1]{KR20}.

\begin{notation}\cite{KR20}\label{notation: language of Tau_omega}
Let $\kappa$ be an infinite cardinal.
\begin{itemize}
\item[(i)] $\Tau_\kappa$ is the set of all functions $\eta$ such that
\begin{itemize}
\item[$\ast$] the domain of $\eta$ is of the form $[\alpha,\kappa)$ for some ordinal $\alpha<\kappa$,
\item[$\ast$] the range of $\eta$ is a subset of $\omega$,
\item[$\ast$] (finite support) $|\{i\in\dom(\eta):\eta(i)\neq0\}|<\omega$. 
\end{itemize}
We call $\Tau_\kappa$ the {\it ill-founded tree on $\kappa$}.
\end{itemize}
Let $\eta,\nu\in \Tau_\kappa$, $i<\omega$, $\alpha<\kappa$.
\begin{itemize}
\item[(ii)] $\eta\unlhd\nu$ if $\eta \subseteq \nu$. As in the case of $\lambda^{<\kappa}$, $\unlhd$ is a partial order on $\Tau_\kappa$. 
\item[(iii)] $\eta$ and $\nu$ are {\it comparable} if $\eta\unlhd\nu$ or $\nu\unlhd\eta$. 
\item[(iv)] $\eta\perp\nu$ if $\eta\not\!\!\unlhd\,\nu$ and $\nu\not\!\!\unlhd\,\eta$. 
\item[(v)] $\eta$ and $\nu$ are {\it incomparable} if $\eta\perp\nu$.
\item[(vi)] $\eta\wedge\nu$ is the $\unlhd$-maximal element $\xi\in\Tau_\kappa$ such that $\xi\unlhd\eta$ and $\xi\unlhd\nu$.
\item[(vii)] $l(\eta)$ is the least element of the domain of $\eta$.
\item[(viii)] $\eta<_l \nu$ if $l(\eta)> l(\nu)$. We write $\eta=_l\nu$ if $l(\eta)=l(\nu)$.
\item[(ix)] $\eta\lex\nu$ if either $\eta\lhd\nu$, or $\eta\perp\nu$ and $\eta(l(\eta\wedge\nu)-1)<\nu(l(\eta\wedge\nu)-1)$. 
\item[(x)] $\eta^\frown\la i\ra^0:=\eta$.
\item[(xi)]
$\eta^\frown\la i\ra^1:=\eta^\frown\la i\ra:=\eta\cup\{(l(\eta)-1,i)\}$.
\item[(xii)] $\eta^\frown\la i\ra^{n+1}:=(\eta^\frown\la i\ra^n{})^\frown\la i\ra$ for $n<\omega$.
\item[(xiii)] $\eta\lhd_i\nu$ if $\eta^\frown\la i\ra\unlhd\nu$.
\item[(xiv)] $t(\eta):=\eta(l(\eta))$.
\item[(xv)] $\eta^-:=\eta|_{[l(\eta)+1,\kappa)}$.
\item[(xvi)] $P_\alpha:=\{\eta\in\Tau_\kappa: l(\eta)=\alpha\}$.
\item[(xvii)] $\la 0\ra^\alpha$ is the function $\eta\in\Tau_{\kappa}$ such that $\dom(\eta)=[\alpha,\kappa)$ and $\eta(\beta)=0$ for all $\alpha\le \beta<\kappa$.
\end{itemize}
Let $\eta_0,...,\eta_n\in\Tau_\kappa$.
\begin{enumerate}
\item[(xviii)] $\cl(\eta_0,...,\eta_n):=(\eta_0\wedge\eta_0,...,\eta_0\wedge\eta_n,...,\eta_n\wedge\eta_0,...,\eta_n\wedge\eta_n)$. 
\end{enumerate}
\end{notation}

\begin{definition}\label{def: path antichain descending comb}
Let $\kappa$ be an infinite cardinal, $\lambda$ a cardinal, and $X$ be a non-empty subset of $\lambda^{<\kappa}$ or $\Tau_\kappa$.
\begin{itemize}
\item[(i)] $X$ is  a {\it path} if $X$ is linearly ordered by $\lhd$.
\item[(ii)] $X$ is an {\it antichain} if $\eta\perp\nu$ for all distinct $\eta,\nu\in X$.
\item[(iii)]
$X$ is a {\it descending comb} if it is an antichain and,
for all $\eta_0,\eta_1,\eta_2\in X$ with
$\eta_0\lex\eta_1\lex\eta_2$, we have
$
\eta_0\wedge\eta_2
=
\eta_1\wedge\eta_2
\lhd
\eta_0\wedge\eta_1.
$
\item[(iv)]
$X$ is an {\it increasing comb} if it is an antichain and,
for all $\eta_0,\eta_1,\eta_2\in X$ with
$\eta_0\lex\eta_1\lex\eta_2$, we have
$
\eta_0\wedge\eta_2
=
\eta_0\wedge\eta_1
\lhd
\eta_1\wedge\eta_2.
$
\item[(v)] $X$ is a {\it left-leaning path} if there is an enumeration $\{\eta_i\}_{i<\lambda}$ of $X$ for some cardinal $\lambda$ such that $\eta_i^- {^\frown} \la j\ra\lhd \eta_{i+1}$ for some $j\le t(\eta_i)$ for all $i<\lambda$ with $i+1<\lambda$. 
\item[(vi)] $X$ is a {\it right-veering path} if there is an enumeration $\{\eta_i\}_{i<\lambda}$ of $X$ for some cardinal $\lambda$ such that $\eta_i^- {^\frown} \la j\ra\unlhd \eta_{i+1}$ for some $j> t(\eta_i)$ for all $i<\lambda$ with $i+1<\lambda$. 
\item[(vii)] $X$ is a {\it strict left-leaning path} if there is an enumeration $\{\eta_i\}_{i<\lambda}$ of $X$ for some cardinal $\lambda$ such that $\eta_i^- {^\frown} \la j\ra\lhd \eta_{i+1}$ for some $j< t(\eta_i)$ for all $i<\lambda$ with $i+1<\lambda$. 
\item[(viii)] $X$ is a {\it strict right-veering path} if there is an enumeration $\{\eta_i\}_{i<\lambda}$ of $X$ for some cardinal $\lambda$ such that $\eta_i^- {^\frown} \la j\ra\lhd \eta_{i+1}$ for some $j> t(\eta_i)$ for all $i<\lambda$ with $i+1<\lambda$. 
\item[(ix)] $X$ is a {\it direct sibling set} if there exists $\eta$ such that $X\subseteq \{\eta^\frown\la i\ra:i<\lambda\}$.
\end{itemize}
Let $X\subseteq (2^2)^\omega$. Note that $(2^2)^\omega$ is isomorphic to $4^\omega$ by giving an order on $2^2$ such that $(0,0)<(0,1)<(1,0)<(1,1)$. Since $(2^2)^\omega\subseteq (2^2)^{<\omega_1}$, we can use the notation we defined above on $X$.
\begin{itemize}
\item[(x)] $X$ is a {\it right-$1$-comb} if $X$ is an increasing comb, and for each $\eta_i,\eta_j\in X$ with $\eta_i\lex\eta_j$, there exists $k<2$ such that $\eta_i\wedge\eta_j\lhd_{(0,k)}\eta_i$ and $\eta_i\wedge\eta_j\lhd_{(1,k)}\eta_j$.
\item[(xi)] $X$ is an {\it up-$1$-comb} if $X$ is an increasing comb, and for each $\eta_i,\eta_j\in X$ with $\eta_i\lex\eta_j$, there exists $k<2$ such that $\eta_i\wedge\eta_j\lhd_{(k,0)}\eta_i$ and $\eta_i\wedge\eta_j\lhd_{(k,1)}\eta_j$.

\end{itemize}
Let $1<k<\omega$. 
\begin{itemize}
\item[(xii)] 
By a {\it $k$-hypergraph}, we mean a structure
$(V,E)$ where $E\subseteq V^k$ is symmetric and irreflexive.
 The {\it generic ordered $k$-hypergraph} is
the Fra\"{\i}ss\'e limit of the class of finite structures $(V,<,E)$
such that $<$ is a linear order on $V$ and $(V,E)$ is a $k$-hypergraph. We say $X\subseteq V$ is a {\it clique} if $E(x_0,...,x_{k-1})$ for all distinct $x_0,...,x_{k-1}\in X$.
\end{itemize}
\end{definition}

\begin{definition}\label{def: stable NIP} 
Let $T$ be a complete theory.
\begin{itemize}
\item[(i)] \cite{She90} We say $T$ is {\it unstable} if there exist $\varphi(x,y)$, $(a_i)_{i<\omega}$, and $(b_i)_{i<\omega}$ such that\vspace{-3pt}
\[\models\varphi(a_i,b_j)\text{ if and only if }i<j.\vspace{-5pt}\]
If $T$ is not unstable, then we say it is {\it stable}. 

\smallskip

\item[(ii)] \cite{She90} We say $T$ has {\it IP} if there exist $\varphi(x,y)$, $(a_i)_{i<\omega}$, and $(b_I)_{I\subseteq\omega}$ such that\vspace{-3pt}
\[\models\varphi(a_i,b_I)\text{ if and only if }i\in I.\vspace{-5pt}\]
If $T$ does not have IP, then we say it is {\it NIP}. 
\end{itemize}
\end{definition}

\begin{definition}\label{def: ATP CTP BTP} 
Let $T$ be a complete theory.
\begin{itemize}

\item[(iii)] \cite{She90} We say $T$ has {\it TP} if there exist $\varphi(x,y)$ and $(a_\eta)_{\eta\in\omega^{<\omega}}$ such that
\begin{itemize}
\item[$\ast$] $\{\varphi(x,a_\eta)\}_{\eta\in X}$ is consistent for any path $X$,
\item[$\ast$] $\{\varphi(x,a_\eta),\varphi(x,a_\nu)\}$ is inconsistent for any direct sibling set $\{\eta,\nu\}$ with $\eta\neq\nu$.
\end{itemize} 
If $T$ does not have TP, then we say it is {\it NTP}. We say $T$ is {\it simple} if it is NTP.

\smallskip

\item[(iv)] \cite{She90,DS04,KK11} We say $T$ has {\it TP$_1$} if there exist $\varphi(x,y)$ and $(a_\eta)_{\eta\in\omega^{<\omega}}$ such that
\begin{itemize}
\item[$\ast$] $\{\varphi(x,a_\eta)\}_{\eta\in X}$ is consistent for any path $X$,
\item[$\ast$] $\{\varphi(x,a_\eta),\varphi(x,a_\nu)\}$ is inconsistent for any incomparable $\eta,\nu\in\omega^{<\omega}$.
\end{itemize}
If $T$ does not have TP$_1$, then we say it is {\it NTP$_1$}.

\smallskip

\item[(v)] \cite{She90,Che13} We say $T$ has {\it TP$_2$} if there exist $\varphi(x,y)$ and $(a_{i,j})_{i,j<\omega}$ such that
\begin{itemize}
\item[$\ast$] $\{\varphi(x,a_{i,f(i)})\}_{i<\omega}$ is consistent for any function $f:\omega\to\omega$,
\item[$\ast$] $\{\varphi(x,a_{i,j}),\varphi(x,a_{i,j'})\}$ is inconsistent for any $i<\omega$ and $j,j'<\omega$ with $j\neq j'$.
\end{itemize}
If $T$ does not have TP$_2$, then we say it is {\it NTP$_2$}.

\smallskip

\item[(vi)] \cite{AK20,AKL21} We say $T$ has {\it ATP} if there exist $\varphi(x,y)$ and $(a_\eta)_{\eta\in\omega^{<\omega}}$ such that
\begin{itemize}
\item[$\ast$] $\{\varphi(x,a_\eta)\}_{\eta\in X}$ is consistent for any antichain $X$,
\item[$\ast$] $\{\varphi(x,a_\eta),\varphi(x,a_\nu)\}$ is inconsistent for any comparable $\eta,\nu\in\omega^{<\omega}$ with $\eta\neq\nu$.
\end{itemize}
If $T$ does not have ATP, then we say it is {\it NATP}.

\smallskip

\item[(vii)] \cite{Mut22}\footnote{In the original definition by Mutchnik \cite{Mut22}, this property is formulated in terms of $k$-DCTP$_2$, with the underlying index set $2^{<\omega}$. It is easy to check that having $k$-DCTP$_2$ for some $k<\omega$ is equivalent to having CTP as defined here.} We say $T$ has {\it CTP} if there exist $\varphi(x,y)$, $(a_\eta)_{\eta\in\omega^{<\omega}}$, and $k<\omega$ such that
\begin{itemize}
\item[$\ast$] $\{\varphi(x,a_\eta)\}_{\eta\in X}$ is consistent for any descending comb $X$,
\item[$\ast$] $\{\varphi(x,a_{\eta|_n})\}_{n<\omega}$ is $k$-inconsistent for any $\eta\in\omega^\omega$.
\end{itemize}
If $T$ does not have CTP, then we say it is {\it NCTP}.

\smallskip

\item[(viii)] \cite{KR23}\footnote{In the original definition by Kruckman and Ramsey \cite{KR23}, the underlying index set is $\omega^{<\omega}\setminus\{\emptyset\}$. We may instead define an equivalent notion of BTP using $\Tau_\omega$, as in this paper. See \cite[Lemma 4.2]{AK26}.} We say $T$ is {\it BTP} if there exist $\varphi(x,y)$, $(a_\eta)_{\eta\in\Tau_\omega}$, and $k<\omega$ such that
\begin{itemize}
\item[$\ast$] $\{\varphi(x,a_\eta)\}_{\eta\in X}$ is consistent for any strict left-leaning path $X$,
\item[$\ast$] $\{\varphi(x,a_\eta)\}_{\eta\in X}$ is $k$-inconsistent for any strict right-veering path $X$.
\item[$\ast$] $\{\varphi(x,a_\eta)\}_{\eta\in X}$ is $k$-inconsistent for any direct sibling set $X$.
\end{itemize}
If $T$ does not have BTP, then we say it is {\it NBTP}.

\smallskip 

\item[(ix)] \cite{She96}
For an integer\footnote{The SOP$_n$ hierarchy for integers $n\geq 3$ can be extended to SOP$_r$ for real numbers $r\geq 3$. See \cite{Mut25,Mut26}.} $n\geq 3$, we say that $T$ has {\it SOP$_n$} if there exist a formula $\varphi(x,y)$ with $|x|=|y|$ and a sequence $(a_i)_{i<\omega}$ such that
\begin{itemize}
\item[$\ast$] $\models\varphi(a_i,a_j)$ for all $i<j<\omega$,
\item[$\ast$] $\{\varphi(x_i,x_{i+1})\}_{i<n-1}\cup\{\varphi(x_{n-1},x_0)\}$ is inconsistent.
\end{itemize}
If $T$ does not have SOP$_n$, then we say it is {\it NSOP$_n$}.

\smallskip 

\item[(x)] \cite{Han25} For $k<\omega$, we say $T$ has a {\it$(k,1,1)$-weave of depth $\omega$} if there exist $\varphi(x,y)$ and $(a_\eta)_{\eta\in (2^2)^\omega}$ such that
\begin{itemize}
\item[$\ast$] $\{\varphi(x,a_\eta)\}_{\eta\in X}$ is consistent for each finite right-1-comb $X$,
\item[$\ast$] $\{\varphi(x,a_\eta)\}_{\eta\in X}$ is $k$-inconsistent for each finite up-1-comb $X$.
\end{itemize}
If $T$ does not have a $(k,1,1)$-weave of depth $\omega$ for any $k<\omega$, then we say it is {\it NWP}.

\smallskip

\item[(xi)] \cite{Han25} For $k<\omega$, we say $T$ has an {\it infinite $k$-grid} if there exist $\varphi(x,y)$ and $(a_{i,j})_{i,j<\omega}$ such that
\begin{itemize}
\item[$\ast$] $\{\varphi(x,a_{i,j})\}_{(i,j)\in X}$ is consistent if $X$ satisfies $(i,j),(i',j')\in X\Rightarrow i\le i'\wedge j\le j'$ or $i'\le i\wedge j'\le j$,
\item[$\ast$] $\{\varphi(x,a_{i,j})\}_{(i,j)\in X}$ is $k$-inconsistent if for all distinct $(i,j),(i',j')\in X$, $i\not\le i'\vee j\not\le j'$ and $i'\not\le i\vee j'\not\le j$.
\end{itemize}
If $T$ does not have an infinite $k$-grid for any $k<\omega$, then we say it is {\it NGP}.

\smallskip

\item[(xii)] \cite{Bai24}\footnote{In the original definition \cite[Definition 6.1]{Bai24}, $T$ is $\mathrm{PM}^{(k)}$ if
some formula exhibits every positive $n$-pattern, for every
$n<\omega$, where each inconsistency condition has size $k$. The
equivalent hypergraph formulation used here follows from
\cite[Propositions 6.3 and 6.4]{Bai24}.} For $1<k<\omega$, we say $T$ is {\it PM$^{(k)}$} if there exists $\varphi(x,y)$ such that for a generic ordered $k$-hypergraph $(\omega,<,E)$ on $\omega$, there exists $(b_i)_{i<\omega}$ such that
\begin{itemize}
\item[$\ast$] $\{\varphi(x,b_i)\}_{i\in X}$ is consistent if $X\subseteq \omega$ is a clique,
\item[$\ast$] $\{\varphi(x,b_i)\}_{i\in X}$ is inconsistent if $X\subseteq \omega$ contains distinct $i_0,...,i_{k-1}$ such that $\neg E(i_0,...,i_{k-1})$.
\end{itemize}
We say a formula is {\it PM$^{(k)}$} if it satisfies the above conditions. If $T$ is not PM$^{(k)}$, then we say it is {\it NPM$^{(k)}$}.
\end{itemize}
In the above definitions, we call $x$ and $y$ the {\it free variable part} and the {\it parameter variable part} of $\varphi(x,y)$, respectively.
\end{definition}

\begin{definition}
Let $T$ be a complete theory and $\mathcal{F}$ be the set of all nondecreasing functions from $\omega$ to $\omega$. By nondecreasing, we mean that $i<j$ implies $f(i)\leq f(j)$.
\begin{itemize}
\item[(xiii)] \cite{Tak17} We say $T$ has {\it OP$_2$} if there exist $\varphi(x,y_0,y_1)$, $(a^0_i)_{i<\omega}$, $(a^1_i)_{i<\omega}$, and $(b_f)_{f\in\mathcal{F}}$ such that\vspace{-3pt}
\[\models\varphi(b_f,a^0_i,a^1_j)
\text{ if and only if }
j\leq f(i).\vspace{-3pt}
\]
If $T$ does not have OP$_2$, then we say it is {\it NOP$_2$}.

\smallskip

\item[(xiv)] \cite{TW21,ACT23} For $0<k<\omega$, we say $T$ has {\it FOP$_k$} if there exist $\varphi(x,y_0,...,y_{k-1})$, $(a^0_i)_{i<\omega}$, ..., $(a^{k-1}_i)_{i<\omega}$, and $(b_f)_{f:\omega^{k-1}\to\omega}$ such that\vspace{-3pt}
\[\models\varphi(b_f,a^0_{i_0},...,a^{k-1}_{i_{k-1}})
\text{ if and only if }
i_{k-1}\leq f(i_0,...,i_{k-2}).\vspace{-3pt}
\]
If $T$ does not have FOP$_k$, then we say it is {\it NFOP$_k$}.

\smallskip

\item[(xv)] \cite{She14,CPT19} For $0<k<\omega$, we say $T$ has {\it IP$_k$} if there exist $\varphi(x,y_0,...,y_{k-1})$, $(a^0_i)_{i<\omega},...,(a^{k-1}_i)_{i<\omega}$, and $(b_I)_{I\subseteq\omega^k}$ such that\vspace{-3pt}
\[\models\varphi(b_I,a^0_{i_0},...,a^{k-1}_{i_{k-1}})
\text{ if and only if }
(i_0,...,i_{k-1})\in I.\vspace{-3pt}
\]
If $T$ does not have IP$_k$, then we say it is {\it NIP$_k$}.

\end{itemize}
\end{definition}

\begin{fact}\label{fact: relation between tree properties}
Let $T$ be a complete theory.
\begin{itemize}
\item[(i)] \cite{She90} If $T$ is stable, then it is NIP and simple.
\item[(ii)] \cite{She90} If $T$ is simple, then it is NTP$_1$ and NTP$_2$.
\item[(iii)] \cite{She80} If $T$ is NIP, then it is NTP$_2$.
\item[(iv)] \cite{KK11,Che26} $T$ is NTP$_1$ if and only if it is NSOP$_3$.
\item[(v)] \cite{She96} For every $3\leq n<\omega$, if $T$ is NSOP$_n$, then it is NSOP$_{n+1}$.
\item[(vi)] \cite{KR23} If $T$ is NTP$_1$ or NTP$_2$, then it is NBTP.
\item[(vii)] \cite{KR23,Han23} If $T$ is NBTP, then it is NCTP.
\item[(viii)] \cite{Mut22} If $T$ is NCTP, then it is NATP.
\item[(ix)] \cite{Han25} If $T$ is NCTP, then it is NWP.
\item[(x)] \cite{Han25} If $T$ is NATP or NWP, then it is NGP.
\item[(xi)] \cite{Han25,Bai24} If $T$ is NGP, then it is NPM$^{(2)}$.
\item[(xii)] \cite{Bai24} For every $1<k<\omega$, if $T$ is NPM$^{(k)}$, then it is NPM$^{(k+1)}$.
\item[(xiii)] \cite{ACT23} For every $0<k<\omega$, if $T$ is NFOP$_k$, then it is NIP$_k$, and if $T$ is NIP$_k$, then it is NFOP$_{k+1}$.
\item[(xiv)] \cite{Bai24} For every $0<k<\omega$, if $T$ is NIP$_k$, then it is NPM$^{(k+1)}$.
\end{itemize}
\end{fact}
Thus we have the following diagram.
\vspace{0pt}

\begin{figure}[H]
\[
\scalebox{1}
{
\begin{tikzpicture}[>=stealth, x=1.6cm, y=1.2cm]
\node (stable) at (0.15,-1) {stable};
\node (NIP) at (0.15,0.7) {NIP};
\node (NIP2) at (0.15,1.35) {NIP$_2$};
\node (NIP3) at (0.15,2.05) {NIP$_3$};
\node (vdots) at (0.15,2.85) {$\vdots$};
\node (vdots2) at (-0.8,2.85) {$\vdots$};
\node (NFOP2) at (-0.8,1.35) {NFOP$_2$};
\node (NFOP3) at (-0.8,2.05) {NFOP$_3$};
\node (NTP)   at (1.1,-1) {simple};
\node (NTP2)   at (1.1,0.7) {NTP$_2$};
\node (NTP1)   at (4.3,-1) {NTP$_1$};
\node (NSOP3)   at (5.3,-1) {NSOP$_4$};
\node (NSOP4)   at (6.35,-1) {NSOP$_5$};
\node (cdots)   at (7.20,-1) {$\cdots$};
\node (NBTP)   at (2.32,-0.35) {\!\!NBTP\!\!};
\node (NCTP)   at (3.2,0.1) {\!NCTP\!};
\node (NATP)   at (4.3,0.1) {NATP};
\node (NWP)   at (3.2,0.7) {NWP};
\node (NGP)   at (4.3,0.7) {\!\!NGP\!\!};
\node (NPM2)   at (5.1,1.) {\!\!NPM$^{(2)}$\!\!\!};
\node (NPM3)   at (6,1.45) {\!\!NPM$^{(3)}$\!\!};
\node (NPM4)   at (6.6,2.05) {\!\!NPM$^{(4)}$\!\!};
\node (diagdotsfake)   at (6.95,2.56) {};
\node (diagdotsfake2)   at (7.1,2.64) {};
\node (diagdots)   at (6.9,2.75) {\rotatebox{15}{\;\;\;$\iddots$}};
\node (NPM2fake) at (5.38,1.1) {};
\node (NPM3fake) at (5.65,1.35) {};
\node (NPM3fake2) at (6.1,1.55) {};
\node (NPM3fake3) at (5.725,1.325) {};
\node (NPM4fake) at (6.65,2.1) {};

\draw[->] (stable) -- (NIP);
\draw[->] (stable) -- (NTP);
\draw[->] (NTP) -- (NTP2);
\draw[->] (NIP) -- (NIP2);
\draw[->] (NIP2) -- (NIP3);
\draw[->] (NIP3) -- (vdots);
\draw[->] (NFOP3) -- (vdots2);

\draw[->] (NIP) -- (NFOP2);
\draw[->] (NFOP2) -- (NIP2);
\draw[->] (NFOP2) -- (NFOP3);
\draw[->] (NIP2) -- (NFOP3);
\draw[->] (NFOP3) -- (NIP3);
\draw[->] (NTP) -- (NTP1);
\draw[->] (NIP) -- (NTP2);
\draw[->] (NTP1) -- (NSOP3);
\draw[->] (NSOP3) -- (NSOP4);
\draw[->] (NSOP4) -- (cdots);
\draw[->] (NTP) -- (NBTP);
\draw[->] (NTP1) -- (NBTP);
\draw[->] (NTP2) -- (NBTP);
\draw[->] (NBTP) -- (NCTP);
\draw[->] (NTP1) -- (NCTP);
\draw[->] (NTP2) -- (NCTP);
\draw[->] (NTP1) -- (NATP);
\draw[->] (NTP2) -- (NWP);
\draw[->] (NCTP) -- (NATP);
\draw[->] (NCTP) -- (NWP);
\draw[->] (NATP) -- (NGP);
\draw[->] (NWP) -- (NGP);
\draw[->] (NGP) -- (NPM2);
\draw[->] (NPM2fake) -- (NPM3fake3);
\draw[->] (NPM4fake) -- (diagdotsfake);
\draw[->] (NIP2) -- (NPM3fake);
\draw[->] (NPM3fake2) -- (NPM4);
\draw[->] (NIP3) -- (NPM4);
\end{tikzpicture}
}
\]
\vspace{-22.5pt}
\caption{
\footnotesize  
}
\label{fig: 1}
\end{figure}

Unlike the other dividing lines introduced above, stability, NIP, OP$_2$, FOP$_k$, and NIP$_k$ were not defined in terms of positive consistency-inconsistency configurations. In other words, the consistency pattern appearing in their definitions involves negations of the witnessing formula. Recently, Day \cite{Day25} found positive consistency-inconsistency characterizations of stable and NIP theories. We recall them below.

\begin{notation}\cite{Day25}
Let $\CL_c:=\lbrace <,C_0,C_1\rbrace$, where $<$ is a binary relation symbol and $C_0,C_1$ are unary relation symbols. Let $\CK_c$ be the class of all finite $\CL_c$-structures $A$ such that
\begin{itemize}
\item[(i)] $<$ is a linear order on $A$,
\item[(ii)] $C_0(A)$ and $C_1(A)$ form a partition of $A$.
\end{itemize}
We denote by $\mathbf{c}_2$ the \Fraisse limit of $\CK_c$. Thus $\mathbf{c}_2$ is a countable homogeneous linear order equipped with two unary predicates $C_0$ and $C_1$, which we regard as two colors, and every finite linearly ordered set colored by two colors embeds into $\mathbf{c}_2$.
\end{notation}

\begin{fact}\cite{Day25}
Let $T$ be a complete theory.
\begin{itemize}
\item[(i)] $T$ is unstable if and only if there exist $\varphi(x,y)$, $(a_\eta)_{\eta\in\mathbf{c}_2^{<\omega}}$, $n<\omega$, and a complete quantifier-free $n$-type $q$ in $\CL_c$ such that
\begin{itemize}
\item[$\ast$] $\lbrace\varphi(x,a_{\eta|_i})\rbrace_{i<\omega}$ is consistent for every $\eta\in\mathbf{c}_2^\omega$,
\item[$\ast$] $\lbrace \varphi(x,a_{\nu^\frown\la j_0\ra}), ...,\varphi(x,a_{\nu^\frown\la j_{n-1}\ra})\rbrace$ is inconsistent whenever $\nu\in\mathbf{c}_2^{<\omega}$ and $(j_0,...,j_{n-1})\models q$ in $\mathbf{c}_2$.
\end{itemize}
\item[(ii)] $T$ has IP if and only if there exist $\varphi(x,y)$, $(a_{i,j})_{i<\omega,j\in\mathbf{c}_2}$, $n<\omega$, and a complete quantifier-free $n$-type $q$ in $\CL_c$ such that
\begin{itemize}
\item[$\ast$] $\{\varphi(x,a_{i,f(i)})\}_{i<\omega}$ is consistent for every function $f:\omega\to\mathbf{c}_2$,
\item[$\ast$] $\{\varphi(x,a_{i,j_0}), ...,\varphi(x,a_{i,j_{n-1}})\}$ is inconsistent whenever $i<\omega$ and $(j_0, ...,j_{n-1})\models q$ in $\mathbf{c}_2$.
\end{itemize}
\end{itemize}
\end{fact}

In \Cref{subsection: l over kNPk}, we will make use of Day's argument to obtain positive consistency-inconsistency characterizations of OP$_2$, FOP$_k$, and IP$_k$, together with a natural generalization of OP$_2$, which we denote by OP$_k$.

\subsection{Generalized indiscernibles}

\begin{definition}
Let $\CL^*$ be a language, $\CI^*$ an $\CL^*$-structure. Let $\CL$ be another language, $T$ an $\CL$-theory, and $\mathbb{M}$ a sufficiently saturated model of $T$. Fix a cardinal $\kappa$, a small subset $A$ of $\mathbb{M}$, and an $\CI^*$-indexed set $(a_i)_{i\in\CI^*}$ such that $a_i\in \mathbb{M}^\kappa$ for each $i\in\CI^*$. We say $(a_i)_{i\in\CI^*}$ is {\it $\CI^*_{\CL^*}$-indiscernible over $A$} if 
\[ \qftp_{\CL^*}^{\CI^*}(I)=\qftp_{\CL^*}^{\CI^*}(J)\;\Rightarrow\; \tp_\CL^\mathbb{M}( (a_i)_{i\in I}/A )=\tp_\CL^\mathbb{M}( (a_i)_{i\in J}/A ) \]
for all $I,J\subseteq\CI^*$. If $\CL^*$ is clear from the context, we omit $\CL^*$ and just say that the indexed set is {\it $\CI^*$-indiscernible}. Depending on the context, we may instead emphasize the language and simply say that the indexed set is {\it $\CL^*$-indiscernible}.
\end{definition}

\begin{notation}
Let $\CL^*$ be a language, $\CI^*$ an $\CL^*$-structure, and $I,J\subseteq \CI^*$.
\begin{itemize}
\item[(i)] We write $I\sim_{\CL^*}^{\CI^*}\!J$ and say $I$ and $J$ are {\it $\CL^*\!$-similar in $\CI^*$} if $\qftp_{\CL^*}^{\CI^*}(I)=\qftp_{\CL^*}^{\CI^*}(J)$. If $\CL^*$ or $\CI^*$ is clear from the context, then we omit $\CL^*$ or $\CI^*$ and write $I\sim_{\CL^*}\!\!J$, $I\sim^{\CI^*}\!\!\!J$, or just simply $I\sim J$.
\item[(ii)] We write $I\sim_{\CL^*}^\text{fin}J$ if  
\begin{itemize}
\item[$\ast$] for each finite $I_0\subseteq I$, there exists $J_0\subseteq J$ such that $J_0\sim_{\CL^*} I_0$, and
\item[$\ast$] for each finite $J_0\subseteq J$, there exists $I_0\subseteq I$ such that $I_0\sim_{\CL^*} J_0$,
\end{itemize}
\item[(iii)] $\la I\ra_{\CL^*}^{\CI^*}$ is the substructure of $\CI^*$ generated by $I$. As above, we omit $\CL^*$ or $\CI^*$ and write $\la I\ra_{\CL^*}$, $\la I\ra^{\CI^*}$, or $\la I\ra$, when $\CL^*$ or $\CI^*$ is clear from the context.
\end{itemize}
\end{notation}

\begin{definition}\cite[Definition 3.1.1.]{Sco10}\label{def: generalized indiscernible}
Let $\CL^*$ and $\CI^*$ be as above. We say $\CI^*_{\CL^*}$-indiscernibles have the {\it modeling property} if for any given language $\CL$, an $\CL$-theory $T$, a monster model $\mathbb{M}$ of $T$, a small subset $A$ of $\mathbb{M}$, and an $\CI^*$-indexed set $(a_i)_{i\in\CI^*}$ in $\mathbb{M}$ whose elements have the same length, we can always find $(b_i)_{i\in\CI^*}$ in $\mathbb{M}$ such that 
\begin{itemize}
\item[(i)] $(b_i)_{i\in\CI^*}$ is $\CI^*_{\CL^*}$-indiscernible over $A$,
\item[(ii)] for any finite $I\subseteq\CI^*$ and $\varphi(x)\in \tp_\CL^\mathbb{M}(  (b_i)_{i\in I} /A)$, there exists $J\subseteq\CI^*$ such that $J\sim_{\CL^*}\!I$ and $\mathbb{M}\models \varphi((a_i)_{i\in J})$.
\end{itemize}

If some $\CI^*$-indexed set satisfies (ii), then we say the $\CI^*$-indexed set is {\it $\CL^*\!$-locally based} on $(a_i)_{i\in\CI^*}$.
\end{definition}

\begin{notation}
$\CL_0 := \{\unlhd, \lex, \wedge\}$. We can regard $\omega^{<\omega}$ and $\Tau_\omega$ as $\CL_0$-structures by interpreting the symbols of $\CL_0$ as in \Cref{notation: language of omega^<omega} and \Cref{notation: language of Tau_omega}. We will write a subscript $\CL_0$ on $\omega^{<\omega}$ and $\Tau_\omega$, namely $\omega^{<\omega}_{\CL_0}$ and $\Tau_{\omega,\CL_0}$, when we want to emphasize the language that the structures are based on. We say $(a_\eta)_{\eta \in \omega^{<\omega}}$ and $(a_\eta)_{\eta \in \Tau_\omega}$ are {\it strongly indiscernible} if they are $\omega^{<\omega}_{\CL_0}$-indiscernible and $\Tau_{\omega,\CL_0}$-indiscernible, respectively.

Similarly, we can give $\CL_s:=\{\unlhd,\lex,\wedge\}\cup\{P_i\}_{i<\omega}$-structures and $\CL_l:=\{\unlhd,\lex,\wedge,<_l\}$-structures on $\omega^{<\omega}$ and $\Tau_\omega$. We say that $(a_\eta)_{\eta \in \omega^{<\omega}}$ and $(a_\eta)_{\eta \in \Tau_\omega}$ are {\it s-indiscernible} if they are $\omega^{<\omega}_{\CL_s}$-indiscernible and $\Tau_{\omega,\CL_s}$-indiscernible, respectively. We can define {\it l-indiscernibility} in the same manner.
\end{notation}

\begin{fact}\cite{TT12,KKS14}
The strong indiscernibility of $\omega^{<\omega}$ and s-indiscernibility of $\omega^{<\omega}$ have the modeling property.
\end{fact}

\begin{notation}
Let $\CL^*$ be a language and $\CI^*$ an $\CL^*$-structure. $\age_{\CL^*}(\CI^*)$ is the class of all $\CL^*$-structures that are isomorphic to some finitely generated $\CL^*$-substructures of $\CI^*$. If $X$ is a subset of $\CI^*$, then $\age_{\CL^*}(X)$ is the class of all $\CL^*$-structures that are isomorphic to $\la X_0 \ra_{\CL^*}^{\CI^*}$ for some finite $X_0\subseteq X$.
 If it is clear from the context, then we omit the subscript $\CL^*$ and just write $\age(\CI^*)$ and $\age(X)$.
\end{notation}

\begin{definition}
For a language $\CL^*$, an $\CL^*$-structure $\CI^*$ is said to be {\it locally finite} if every finitely generated substructure of $\CI^*$ is  finite.
\end{definition}

\begin{definition}\label{def: Ramsey}
Let $\CL^*$ be a language and $\CK$ a class of finitely generated $\CL^*$-structures. For $A,B\in\CK$, let $\binom{B}{A}$ be the set of all substructures of $B$ that are $\CL^*$-isomorphic to $A$. For $0<r<\omega$, we write $C\to(B)^A_r$ if for every coloring $\sigma:\binom{C}{A}\to r$, there exists $B'\in\binom{C}{B}$ such that $\sigma|_{\binom{B'}{A}}$ is constant. 

For $A,B\in\CK$, let $\operatorname{Emb}_{\CL^*}(A,B)$ be the set of all $\CL^*$-embeddings from $A$ to $B$.  For $0<r<\omega$, we write
$C\to_{\operatorname{emb}}(B)^A_r$
if for every coloring $\sigma:\operatorname{Emb}_{\CL^*}(A,C)\to r$, there exists $f\in\operatorname{Emb}_{\CL^*}(B,C)$ such that
$\sigma|_{\lbrace f\circ g:g\in\operatorname{Emb}_{\CL^*}(A,B)\rbrace}$
is constant.
\begin{itemize}
\item[(i)] We say $\CK$ has the {\it Ramsey property} if for all $A,B\in\CK$ and $0<r<\omega$, there exists $C\in\CK$ such that $C\to(B)^A_r$.
\item[(ii)] We say $\CK$ has the {\it embedding Ramsey property} if for all $A,B\in\CK$ and $0<r<\omega$, there exists $C\in\CK$ such that $C\to_{\operatorname{emb}}(B)^A_r$.
\end{itemize}
\end{definition}

The following facts are useful for determining whether a given $\CL^*$-structure has the modeling property.

\begin{fact}\cite{MP23,Sco11}
Let $\CL^*$ be a language and $\CI^*$ an infinite locally finite $\CL^*$-structure. 
\begin{itemize}
\item[(i)] $\age_{\CL^*}(\CI^*)$ has the embedding Ramsey property if and only if $\CI^*_{\CL^*}$-indiscernibles have the modeling property.
\item[(ii)] If $\CI^*$ is linearly ordered by some binary relation symbol in $\CL^*$, then $\age_{\CL^*}(\CI^*)$ has the Ramsey property if and only if $\CI^*_{\CL^*}$-indiscernibles have the modeling property.
\end{itemize}
\end{fact}

\begin{fact}\cite{MP23,Sco11}\label{fact: age(I)=age(J)}
Let $\CL^*$ be a language, $\CI^*$ and $\CJ^*$ be locally finite
$\CL^*$-structures with $\age(\CI^*)=\age(\CJ^*)$.
If $\CI^*$-indiscernibles have the modeling property, then
$\CJ^*$-indiscernibles also have the modeling property.
Moreover, for any $\CI^*$-indexed set of parameters
$(a_\eta)_{\eta\in\CI^*}$, there is a $\CJ^*$-indexed set of
parameters $(b_\eta)_{\eta\in\CJ^*}$ such that it is
$\CJ^*$-indiscernible and for any $\bar{\eta}\in\CJ^*$
and $\varphi(\bar{x})$ with $\models\varphi(\bar{b}_{\bar{\eta}})$,
there exists $\bar{\nu}\in\CI^*$ such that
$\bar{\nu}\sim_{\CL^*}\bar{\eta}$ and
$\models\varphi(\bar{a}_{\bar{\nu}})$.
Thus the strong indiscernibility of $\Tau_\omega$ has the
modeling property.
The $s$-modeling property of $\Tau_\omega$ follows from that of
$\omega^{<\omega}$ by relabeling the finitely many relevant levels
in reverse order and applying compactness.
\end{fact}

\subsection{Algebraic dimension and geometric theories}

The notion of algebraic dimension below is based on the algebraic independence considered in \cite{CP98} and the dimension notion used in \cite{BV14}, but we formulate it in a slightly more general form suitable for our purposes.

\begin{definition}\cite{CP98,BV14}
Let $\CL$ be a language, $T$ an $\CL$-theory, and $\mathbb{M}$ a sufficiently saturated model of $T$.
\begin{itemize}
\item[(i)] For $A,B\subseteq \mathbb{M}^1$, we say $B$ is {\it algebraically independent} over $A$ if $b\notin\acl_\CL(A\cup (B\setminus\{b\}))$ for all $b\in B$.
\item[(ii)] For $\bar{b}:=(b_0,...,b_{n-1})\in\mathbb{M}^n$ with $n<\omega$ and $A\subseteq\mathbb{M}^1$, the {\it algebraic dimension} of $\bar{b}$ over $A$, denoted by $\dim_\CL(\bar{b}/A)$, is defined by
\[
\max\{|B|: B\subseteq\{b_0,...,b_{n-1}\}, B\text{ is algebraically independent over }A\}.
\]
\item[(iii)] For a formula $\varphi(x_0,...,x_{n-1},y)$ with $|x_0|=\cdots=|x_{n-1}|=1$ and $a$ with $|a|=|y|$, the {\it algebraic dimension} of $\varphi(\bar{x},a)$, denoted by $\dim_\CL(\varphi(\bar{x},a))$, is defined by 
\[
\max\{\dim_\CL(\bar{b}/a):\bar{b}\models\varphi(\bar{x},a)\}.
\]
\item[(iv)] For a type $\Phi(x_0,...,x_{n-1},y)$ without parameters and $a$ such that $|x_0|=\cdots=|x_{n-1}|=1$ and $|a|=|y|$, the {\it algebraic dimension} of $\Phi(\bar{x},a)$, denoted by $\dim_\CL(\Phi(\bar{x},a))$, is defined by
\[
\max\{\dim_\CL(\bar{b}/a):\bar{b}\models\Phi(\bar{x},a)\}.
\]
\item[(v)] When a defining formula $\varphi(\bar{x},a)$ for $D$
is fixed by the context, we also write $\dim_\CL(D)$ for
$\dim_\CL(\varphi(\bar{x},a))$.
\end{itemize}
If $\CL$ is clear from the context, we omit the subscript $\CL$ and just write $\dim(\bar{b}/A)$, $\dim(\varphi)$, $\dim(\Phi)$, and $\dim(D)$.
\end{definition}

\begin{definition}\cite{BV14,BV16}
We say a complete theory $T$ is {\it geometric} if it eliminates $\exists^{\infty}$ and $\acl$ has the exchange property.
\end{definition}

\begin{definition}\cite{BV14,BV16}
Let $T$ be a geometric theory, $\mathbb{M}\models T$ a sufficiently saturated model, and $A\subseteq\mathbb{M}^1$.
\begin{itemize}
\item[(i)] We say $A$ satisfies the {\it density property} if $p(x)\in S_1(B)$ always has a realization in $A$ whenever $p(x)$ is non-algebraic and $\dim(B/\emptyset)<\omega$.
\item[(ii)] We say $A$ satisfies the {\it extension property} if $p(x)\in S_1(B)$ always has a realization in $\mathbb{M}\setminus\acl(AB)$ whenever $p(x)$ is non-algebraic and $\dim(B/\emptyset)<\omega$.
\end{itemize}
\end{definition}

\begin{definition}\cite{BV10,BV14,BV16}
Let $T$ be a geometric $\CL$-theory, let $H$ be a new unary predicate
symbol, and let $\CL_H:=\CL\cup\{H\}$. For $A\subseteq\mathbb{M}$,
we write $H(A):=A\cap H(\mathbb{M})$.
\begin{itemize}
\item[(i)] An $\CL_H$-structure $(\mathbb{M},H(\mathbb{M}))$ is an
{\it $H$-structure} of $T$ if $\mathbb{M}\models T$ and
$H(\mathbb{M})$ is algebraically independent and satisfies the
density and extension properties. The common complete theory of $H$-structures of $T$ is
denoted by $T^{ind}$.
\item[(ii)] The theory of the structure induced on $H(\mathbb{M})$ by
the parameter-free $\CL$-definable relations is called the
{\it generic trivialization} of $T$ and is denoted by $T^{gt}$.
\item[(iii)] $(\mathbb{M},H(\mathbb{M}))$ is a {\it lovely pair} of
$T$ if $H(\mathbb{M})\preccurlyeq_{\CL}\mathbb{M}$ and
$H(\mathbb{M})$ satisfies the density and extension properties.
The common complete theory of lovely pairs of $T$ is denoted by $T_P$.
\end{itemize}
\end{definition}

\begin{fact}\cite[Proposition 2.13 and Remark 2.14]{BV14}\label{fact: acl(H(M))-def -> H(M)-def}
Let $T$ be a geometric theory, $(\mathbb{M},H(\mathbb{M}))\models T^{ind}$ a sufficiently saturated model, and $n<\omega$.
 If $D\subseteq\mathbb{M}^n$ is a set $\CL$-definable over $\acl_\CL(H(\mathbb{M}))$ and $\dim(D)=n$, then there exists $D'\subseteq D$, $\CL$-definable over $H(\mathbb{M})$, such that $\dim(D\setminus D')<n$.
\end{fact}

\begin{definition}\cite[Definition 2.4]{BK16}
Let $T^*$ be $T_P$ or $T^{ind}\!$ and $(\mathbb{M},H(\mathbb{M}))\models T^*$ a sufficiently saturated model. We say a tuple $\bar{a}\in\mathbb{M}$ is {\it $H$-independent} if $\dim_\CL(\bar{a}/H(\bar{a}))=\dim_\CL(\bar{a}/H(\mathbb{M}))$.
\end{definition}

\begin{fact}\cite[Proposition 2.5 and Lemma 2.6]{BK16}\label{fact: H-independent tuples}
Let $T^*$ be $T_P$ or $T^{ind}\!$ and $(\mathbb{M},H(\mathbb{M}))$ a sufficiently saturated model of $T^*$.
\begin{itemize}
\item[(i)] For any tuple $\bar{a}$, there exists $\bar{b}$ such that $\bar{a}\bar{b}$ is $H$-independent.
\item[(ii)] If a tuple $\bar{a}$ is $H$-independent, then $\bar{a}\bar{h}$ is $H$-independent for any $\bar{h}\in H(\mathbb{M})$.
\item[(iii)] For any $H$-independent tuples $\bar{a}$ and $\bar{b}$,
$\tp_{\CL_H}(\bar{a})=\tp_{\CL_H}(\bar{b})
$ if and only if $
\tp_\CL(\bar{a},H(\bar{a}))=\tp_\CL(\bar{b},H(\bar{b})).
$
\end{itemize}
\end{fact}

\section{A class of dividing lines satisfying some consistency-inconsistency configurations}\label{sec: D-n-var}

In this section, we introduce the class $\mathfrak{D}_{n\text{-var}}$ of dividing lines that can be characterized in a uniform way in terms of positive consistency-inconsistency configurations and generalized indiscernibles. We then show that many of the dividing lines recalled in the previous section belong to $\mathfrak{D}_{n\text{-var}}$.

Note that most of the dividing lines introduced in the previous section involve three ingredients, namely, an index set for the parameters, subsets of the index set corresponding to consistent sets of formulas, and subsets corresponding to inconsistent sets of formulas. For example, in the case of TP$_1$, these can be taken to be $\omega^{<\omega}$, paths, and incomparable pairs, respectively.

We add one more ingredient to this picture, namely,  a language where the index set can be viewed as a structure with the modeling property. Using these four ingredients, we introduce a uniform framework for defining dividing lines. As we will see in the following sections, this framework provides useful tools for studying dividing lines, particularly in connection with preservation problems.

We begin by introducing some terminology for the conditions that these ingredients are required to satisfy.

\begin{definition}\label{def: weakly monochromatic pseudo-Ramsey cofinality}
Let $\CL^*$ be a language and $\CI^*$ an $\CL^*$-structure.
\begin{itemize}
\item[(i)] We say $\CI^*$ is {\it monochromatic} if  $\{\la \eta \ra\subseteq \CI^*: \eta\in \CI^*, |\eta|=1\}$ has only one element up to $\CL^*$-isomorphism.
\item[(ii)] We say $\CI^*$ is {\it weakly monochromatic} if  $\{\la \eta \ra\subseteq \CI^*: \eta\in \CI^*, |\eta|=1\}$ has only finitely many elements up to $\CL^*$-isomorphism.
\item[(iii)] Let $J\subseteq\CI^*$. We say $\CI^*$ is {\it monochromatically extendable} over $J$ if for any cardinal $\lambda$, there exists $\CI^\dagger$ such that $\age_{\CL^*}(\CI^\dagger)=\age_{\CL^*}(\CI^*)$ and for any coloring $c:\CI^\dagger\to\lambda$, there exists an $\CL^*$-embedding $f:\CI^*\to\CI^\dagger$ such that $c(f(\eta))=c(f(\nu))$ for $\eta,\nu\in J$ with $\la\eta\ra\sim_{\CL^*}\la\nu\ra$. When $J=\CI^*$, we simply say that $\CI^*$ is
monochromatically extendable. We call such $\CI^\dagger$ a {\it monochromatic extension} of $\CI^*$ over $J$ in $\lambda$. 
\item[(iv)] We say a subset $X$ of $\CI^*$ is {\it cofinal} in $\CI^*$ if for any finite subset $Z$ of $\CI^*$, there exists an $\CL^*$-embedding $f:\CI^*\to\CI^*$ such that
\begin{itemize}
\item[$\ast$] $f(\CI^*)\cap Z=\emptyset$,
\item[$\ast$] $\bar{\eta}\sim_{\CL^*}\bar{\nu} \Rightarrow \bar{\eta} Z\sim_{\CL^*}\bar{\nu}Z$ for all $\bar{\eta},\bar{\nu}\in f(\CI^*)$ with $|\bar{\eta}|=|\bar{\nu}|<\omega$,
\item[$\ast$] $(f(\CI^*)\cap X)\sim_{\CL^*}^\text{fin}X$.
\end{itemize}
\end{itemize}
\end{definition}

\begin{remark}\label{rmk: basic facts for LIXY}
The following are easy to check.
\begin{itemize}
\item[(i)] If $\CL^*$ is a finite relational language, then $\CI^*$ is weakly monochromatic.
\item[(ii)] For any infinite cardinals $\lambda$ and $\kappa$, the path $\{\la0\ra^i:i<\kappa\}$ is cofinal in the $\CL_0$-structure on $\lambda^{<\kappa}$. 
\item[(iii)] For any infinite cardinals $\lambda$ and $\kappa$, $\lambda^\kappa$ is cofinal in the $\CL_0$-structure on $\lambda^{\le \kappa}$.
\item[(iv)] For any infinite cardinals $\kappa_0$ and $\kappa_1$ and any cardinal $\lambda$, there exists a cardinal $\kappa_2$ such that $\kappa_2^{<\kappa_2}$ is a monochromatic extension of $\kappa_0^{<\kappa_1}$ in $\lambda$, with respect to the language $\CL_0$. This follows by the arguments in \cite[Remark~3.22 and Corollary~3.23(a)]{AKL21}.
\item[(v)] If $\CI^*$ is monochromatically extendable over $\CI^*$, then it is monochromatically extendable over $\mathbbof{X}$ for any $\mathbbof{X}\subseteq \CI^*$.
\end{itemize}
\end{remark}

\begin{definition}\label{def: n-var quadruple}
We say a quadruple $(\CL^*\!\!,\,\CI^*\!\!,\,X,\bar{Y})$ satisfies the {\it $n$-variable theorem} if
\begin{itemize}
\item[(i)] $\CL^*$ is a language, $\CI^*$ is a locally finite $\CL^*$-structure having the modeling property, $X$ is a subset of $\CI^*$, and $\bar{Y}$ is a finite set of subsets of $\CI^*$,
\item[(ii)] $\CI^*$ is weakly monochromatic,
\item[(iii)] $X$ is cofinal in $\CI^*$,
\item[(iv)] $\CI^*$ is monochromatically extendable over $\mathbbof{X}:=\{\eta\in\CI^*:\la\eta\ra\sim_{\CL^*}\la\nu\ra\text{ for some }\nu\in X\}$,
\item[(v)] $Y\subseteq\mathbbof{X}$ for each $Y\in\bar{Y}$,
\item[(vi)] $|Y|>1$ for each $Y\in\bar{Y}$,
\item[(vii)] there exists a disjoint union $W\subseteq\CI^*$ of $X'$ and $Y'$ such that 
\begin{itemize}
\item[$\ast$] $X'\sim^\text{fin}_{\CL^*}X$ and $Y'\sim^\text{fin}_{\CL^*}Y$ for some $Y\in\bar{Y}$,
\item[$\ast$] $\{\eta\}\cup X'\sim^\text{fin}_{\CL^*}X$ for each $\eta\in Y'$.
\end{itemize}
\end{itemize}
If $|\bar{Y}|=1$ and $\bar{Y}=\{Y\}$, then we just say $(\CL^*\!\!,\,\CI^*\!\!,\,X,Y)$ satisfies  the $n$-variable theorem.
\end{definition}

\begin{definition}\label{def: D_n-var}
For a given quadruple $(\CL^*\!\!,\,\CI^*\!\!,\,X,\bar{Y})$ satisfying  the $n$-variable theorem, let
$D^{\CL^*\!\!\!,\,\CI^*}_{X,\bar{Y}}$ be the class of complete first-order theories $T$ such that there do not exist $\varphi(x,y)$ in the language of $T$ and $(a_\eta)_{\eta\in\CI^*}$ in a model of $T$ such that
\begin{itemize}
\item[(i)] $(a_\eta)_{\eta\in\CI^*}$ is $\CI^*_{\CL^*}$-indiscernible,
\item[(ii)] $\{\varphi(x,a_\eta)\}_{\eta\in X}$ is consistent,
\item[(iii)] $\{\varphi(x,a_\eta)\}_{\eta\in Y}$ is inconsistent for each $Y\in \bar{Y}$.
\end{itemize}

Let $\mathfrak{D}_{n\text{-var}}$ be a class of classes of complete first-order theories such that
\[
\mathfrak{D}_{n\text{-var}}:=\left\lbrace D^{\CL^*\!\!\!,\,\CI^*}_{X,\bar{Y}}: (\CL^*\!\!,\,\CI^*\!\!,X,\bar{Y})\text{ satisfies the }n\text{-variable theorem} \right\rbrace.
\]
\end{definition}

\begin{definition}\label{def: witnessing neg D}
In \Cref{def: D_n-var}, if such a formula $\varphi(x,y)$ and $(a_\eta)_{\eta\in\CI^*}$ exist, then we say $\varphi(x,y)$ {\it witnesses} $T\notin D^{\CL^*\!\!\!,\,\CI^*}_{X,\bar{Y}}$ with $(a_\eta)_{\eta\in\CI^*}$, or it {\it witnesses} $\neg D^{\CL^*\!\!\!,\,\CI^*}_{X,\bar{Y}}$ with $(a_\eta)_{\eta\in\CI^*}$.
\end{definition}

\begin{remark}
We emphasize that the symbol $n$ in $\mathfrak{D}_{n\text{-var}}$ is not a parameter in the definition, but is simply part of the name. The reason for this terminology will become clear in \Cref{sec: n-var thm}: for every dividing line in $\mathfrak{D}_{n\text{-var}}$, we will prove a result that may be viewed as a weak form of a one-variable theorem. Here, `$n$' is just meant to indicate that these results are weaker than the corresponding `1'-variable theorems.
\end{remark}

\begin{notation}
Let $D_\text{stable}$, $D_\text{simple}$, $D_\text{NIP}$, $D_{\text{NTP}_1}$, $D_{\text{NTP}_2}$, $D_\text{NATP}$, $D_\text{NCTP}$, $D_\text{NBTP}$, $D_\text{NWP}$, $D_\text{NGP}$, and $D_{\text{NPM}^{(k)}}$ be the classes of stable theories, simple theories, NIP theories, NTP$_1$ theories, NTP$_2$ theories, NATP theories, NCTP theories, NBTP theories, NWP theories, NGP theories, and NPM$^{(k)}$ theories,  respectively.
\end{notation}

\begin{lemma}\label{lem: equiv cond stable}
$T$ is unstable if and only if there exist $\varphi(x,y)$ and $(a_{n,i})^{n<\omega}_{i<2}$ such that 
\begin{itemize}
\item[(i)] $\{\varphi(x,a_{n',0})\}_{n'<n}\cup\{\varphi(x,a_{n',1})\}_{n<n'}$ is consistent for each $n<\omega$,
\item[(ii)] $\{\varphi(x,a_{n,0}),\varphi(x,a_{n,1})\}$ is inconsistent for each $n<\omega$.
\end{itemize}
\begin{proof}
Suppose $T$ is unstable. Then there exist $\varphi(x,y)$ and $(a_q)_{q\in\mathbb{Q}}$ such that 
\begin{itemize}
\item[(iii)] $\{\varphi(x,a_{q'})\}_{q'<q}\cup\{\neg\varphi(x,a_{q'})\}_{q<q'}$ is consistent for each $q\in\mathbb{Q}$.
\end{itemize}
Let $\varphi'(x,y_0,y_1):=\varphi(x,y_0)\wedge\neg\varphi(x,y_1)$. Choose any sequences of rational numbers $(p_n)_{n<\omega}$, $(q_n)_{n<\omega}$, and $(r_n)_{n<\omega}$ such that
\begin{itemize}
\item[(iv)] $p_n<q_{n'}<r_{n''}$ for any $n,n',n''<\omega$,
\item[(v)] $p_n<p_{n+1}$ and $q_{n+1}<q_n$ for each $n<\omega$.
\end{itemize}
For each $n<\omega$ and $i<2$, let
\[b_{n,i}:=
\begin{cases}
a_{p_n}a_{q_n}&\text{if }i=0   \\
a_{q_n}a_{r_n}&\text{if }i=1.
\end{cases}
\]
Then $\varphi'$ and $(b_{n,i})^{n<\omega}_{i<2}$ satisfy (i) and (ii). The other direction is clear.
\end{proof}
\end{lemma}

\begin{proposition}\label{prop: stable is D_n-var}
$D_\text{stable}\in\mathfrak{D}_{n\text{-var}}$.
\begin{proof}
We use \Cref{lem: equiv cond stable}. Let $<^*$ be a binary relation symbol, $C_0,C_1$ unary relation symbols, and $\CL^*:=\{<^*,C_0,C_1\}$. We give an $\CL^*$-structure $\CI^*$ on $\mathbb{Q}\times2$ as follows.
\begin{itemize}
\item[(i)] $(q,i)<^*(q',i')$ if and only if $q<q'$.
\item[(ii)] $C_i((q,i'))$ if and only if $i=i'$.
\end{itemize}
Then $\CI^*$ has the modeling property. Choose any irrational number $r$ and let $X:=\{(q,0):q<r\}\cup\{(q,1):q>r\}$ and $Y:=\{(0,0),(0,1)\}$. It is easy to check that $(\CL^*\!,\CI^*\!,X,Y)$ satisfies the $n$-variable theorem and $D_\text{stable}=D^{\CL^*\!,\CI^*}_{X,Y}\!$.
\end{proof}
\end{proposition}

\begin{proposition}\label{prop: NTP1 is D_n-var}
$D_{\text{NTP}_1} \in \mathfrak{D}_{n\text{-var}}$. 
\begin{proof}
Choose any uncountable cardinal $\kappa$ with $\cf(\kappa)=\kappa$. Consider the $\CL_0$-structure on $\omega^{<\kappa}$ and subsets $X,Y$ of $\omega^{<\kappa}$ given by $X:=\{\la 0\ra^i:i<\kappa\}$ and $Y:=\{\la0\ra,\la1\ra\}$. The monochromatic extendability of $\omega^{<\kappa}$ follows from \Cref{rmk: basic facts for LIXY} (iv). The cofinality of $X$ follows from \Cref{rmk: basic facts for LIXY} (ii). Thus  $(\CL_0,\omega^{<\kappa},X,Y)$ satisfies the $n$-variable theorem and $D_{\text{NTP}_1}=D^{\CL_0,\omega^{<\kappa}}_{X,Y}\!$. 
\end{proof}
\end{proposition}

\begin{proposition}\label{prop: NATP is D_n-var}
$D_{\text{NATP}}\in \mathfrak{D}_{n\text{-var}}$. 
\begin{proof}
Choose any uncountable cardinal $\kappa$ with $\cf(\kappa)=\kappa$. Consider the $\CL_0$-structure on $\omega^{\le\kappa}$ and subsets $X,Y$ of $\omega^{\le\kappa}$ given by $X:=\{\eta:l(\eta)=\kappa\}$ and $Y:=\{\emptyset,\la0\ra\}$. The monochromatic extendability is as in \Cref{prop: NTP1 is D_n-var} and the cofinality of $X$ follows from \Cref{rmk: basic facts for LIXY} (iii). Thus $(\CL_0,\omega^{\le\kappa},X,Y)$ satisfies the $n$-variable theorem and $D_{\text{NATP}}=D^{\CL_0,\omega^{\le\kappa}}_{X,Y}\!$. 
\end{proof}
\end{proposition}

\begin{proposition}\label{prop: NCTP is D_n-var}
$D_\text{NCTP}\in \mathfrak{D}_{n\text{-var}}$. 
\begin{proof}
Choose any uncountable cardinal $\kappa$ with $\cf(\kappa)=\kappa$. Consider the $\CL_0$-structure on $\omega^{<\kappa}$ and subsets $X,Y$ of $\omega^{<\kappa}$ given by $X:=\{\la0\ra^i{}^{\frown}\la1\ra:i<\kappa\}$ and $Y:=\{\la0\ra^i\}_{i<\omega}$. The monochromatic extendability follows from \Cref{rmk: basic facts for LIXY} (iv). It is easy to check that $X$ is cofinal in $\omega^{<\kappa}$. Thus $(\CL_0,\omega^{<\kappa},X,Y)$ satisfies the $n$-variable theorem and $D_{\text{NCTP}}=D^{\CL_0,\omega^{<\kappa}}_{X,Y}\!$.
\end{proof}
\end{proposition}

\begin{proposition}\label{prop: NTP2 is D_n-var}
$D_{\text{NTP}_2}\in \mathfrak{D}_{n\text{-var}}$. 
\begin{proof}
Let $E,\le^*$ be binary relation symbols and put $\CL^*:=\{E,\le^*\}$. We consider the class of finite $\CL^*$-structures $A$ such that $\le^*$ is a linear order on $A$ and $E$ is a $\le^*$-convex equivalence relation on $A$ (i.e., $E$ is an equivalence relation satisfying $\forall xyz(E(x,z)\wedge x\le^* y\le^* z \rightarrow E(x,y))$). This class is known as a \Fraisse and Ramsey class. Let $\CI^*$ be its \Fraisse limit. Consider an $\CL^*$-structure on $\omega\times\omega$ with the lexicographic order $\le^*$ and an equivalence relation $E$ such that $E( (i,j),(i',j'))$ if and only if $i=i'$. Then $\age(\CI^*)=\age(\omega\times\omega)$ and hence the $\CL^*$-structure on $\omega\times\omega$ has the modeling property. To show the monochromatic extendability of $\omega\times\omega$, fix any cardinal $\lambda$ and choose any infinite cardinal $\kappa>\lambda$. Define an $\CL^*$-structure on $\kappa\times\kappa$ in the same way as on $\omega\times\omega$. Then we can show that $\kappa\times\kappa$ is a monochromatic extension by applying the pigeonhole principle twice. Let $X:=\{(i,0):i<\omega\}$ and $Y:=\{(0,0),(0,1)\}$. It is easy to check that $X$ is cofinal. Thus $(\CL^*\!,\omega\times\omega,X,Y)$ satisfies the $n$-variable theorem and $D_{\text{NTP}_2}=D^{\CL^*\!\!,\,\omega\times\omega}_{X,Y}$. 
\end{proof}
\end{proposition}

\begin{proposition}\label{prop: NIP is D_n-var}
$D_\text{NIP}\in\mathfrak{D}_{n\text{-var}}$.
\begin{proof}
Let $E,\le^*$ be binary relation symbols, $C_0, C_1$ unary relation symbols, and put $\CL^*:=\{E,\le^*,C_0,C_1\}$. We consider the class of finite $\CL^*$-structures $A$ such that $\le^*$ is a linear order on $A$, $E$ is a $\le^*$-convex equivalence relation on $A$, and $\{C_0,C_1\}$ is a partition of $A$. This class is a \Fraisse and Ramsey class. Let $\CI^*$ be its \Fraisse limit. Consider an $\CL^*$-structure on $\omega\times\omega$ such that $E((i,j),(i',j'))$ if and only if $i=i'$, $\le^*$ is the lexicographic order, $C_0((i,j))$ if $i+j\equiv_{\rm{mod}\;2}0$, and $C_1((i,j))$ if $i+j\equiv_{\rm{mod}\;2}1$. Then $\age(\CI^*)=\age(\omega\times\omega)$ and hence the $\CL^*$-structure on $\omega\times\omega$ has the modeling property. To show the monochromatic extendability of $\omega\times\omega$, fix any cardinal $\lambda$ and choose any infinite cardinal $\kappa>\max\{\lambda,\omega\}$. Define an $\CL^*$-structure on $\kappa\times\kappa$ in the same way as on $\omega\times\omega$. Let 
\[
(\omega\times\omega)/2^2:=\{(2i,2j):i,j<\omega\}
\] 
and 
\[(\kappa\times\kappa)/2^2:=\{(\alpha+2i,\beta+2j):\alpha,\beta<\kappa,\;i,j<\omega,\;\alpha,\beta\text{ are limit ordinals}\}.
\]
Let $c$ be a $\lambda$-coloring on $\kappa\times\kappa$. Then we can define a $\lambda^2$-coloring $d$ on $(\kappa\times\kappa)/2^2$ such that $d(\alpha+2i,\beta+2j):=(c(\alpha+2i,\beta+2j),c(\alpha+2i,\beta+2j+1))$. Then by the same argument as in \Cref{prop: NTP2 is D_n-var}, there exists an $\{E,\le^*\}$-embedding $g:(\omega\times\omega)/2^2\to(\kappa\times\kappa)/2^2$ such that $d(g(2i,2j))=d(g(2i',2j'))$ for all $(2i,2j),(2i',2j')\in(\omega\times\omega)/2^2$. For each $(i,j)\in\omega\times\omega$, write $g(2i,2j)=(\gamma,\delta)$ and define a map $f:\omega\times\omega\to\kappa\times\kappa$ such that 
\[
f(i,j)=
\begin{cases}
(\gamma,\delta) &\text{ if } i+j\text{ is even} \\
(\gamma,\delta+1) &\text{ if } i+j\text{ is odd}.
\end{cases}
\]
Then $f$ is an $\CL^*$-embedding and $c(f(i,j))=c(f(i',j'))$ if $\models C_0(i,j)\leftrightarrow C_0(i',j')$. Thus $\kappa\times\kappa$ is a monochromatic extension of $\omega\times\omega$.
Let $X:=\{(i,0):i<\omega\}$ and $Y:=\{(0,i):i<\omega\}$. It is easy to check that $X$ is cofinal in $\omega\times\omega$. Thus $(\CL^*\!,\omega\times\omega,X,Y)$ satisfies the $n$-variable theorem. $D_{\text{NIP}}=D^{\CL^*\!\!,\,\omega\times\omega}_{X,Y}$ follows from \cite{Day25}. 
\end{proof}
\end{proposition}

Before we prove that $D_\text{NGP}\in\mathfrak{D}_{n\text{-var}}$, we show that having a $k$-grid is equivalent to having a $2$-grid for any $k\ge 2$, which answers the question posed in \cite[Question 5.15]{Han25}. The discussion below concerning NGP was developed with helpful advice from James Hanson.

\begin{proposition}\label{prop: k-grid -> 2-grid}
If $T$ has a $k$-grid for some $k>2$, then $T$ has a $2$-grid.
\begin{proof}
Let $\le_0,\le_1$ be binary symbols and put $\CL^*:=\{\le_0,\le_1\}$. Consider the class of finite $\CL^*$-structures $A$ such that $\le_0,\le_1$ are linear orders on $A$. This class is a \Fraisse and Ramsey class. Let $\CI^*$ be its \Fraisse limit. Give an $\CL^*$-structure on $\mathbb{Q}\times\mathbb{Q}$ such that $(p,q)\le_0(p',q')$ if $p< p'$ or $p=p'\wedge q\le q'$,  $(p,q)\le_1(p',q')$ if $q< q'$ or $q=q'\wedge p\le p'$. Then $\age(\CI^*)=\age(\mathbb{Q}\times\mathbb{Q})$ and hence the $\CL^*$-structure on $\mathbb{Q}\times\mathbb{Q}$ has the modeling property.

 Now suppose $\varphi(x,y)$ has a $k$-grid with $(a_{p,q})_{p,q\in\mathbb{Q}}$, for some $k>2$. We may assume $(a_{p,q})_{p,q\in\mathbb{Q}}$ is $\CL^*$-indiscernible. Let 
\[
\Phi(x):=\{\varphi(x,a_{-1,1})\}\cup\{\varphi(x,a_{n,-{1\over n}}):n\in\mathbb{N}\setminus\{0\}\}.
\]
If $\Phi(x)$ is consistent, then $T$ has a $(k-1)$-grid. If $\Phi(x)$ is inconsistent, then $T$ has a $2$-grid.
\end{proof}
\end{proposition}

\begin{lemma}\label{lem: coloring lemma of dense linear order}
For any cardinal $\lambda$, let $I_\lambda$ be a $\lambda^+$-saturated $\{<\}$-structure such that $I_\lambda\models\Th(\mathbb{Q},<)$. Then for any $c:I_\lambda\to \lambda$, there exist $\lambda'<\lambda$ and $q,r\in I_\lambda$ such that $q<r$, and $c^{-1}(\lambda')$ is dense in $(q,r)$.
\begin{proof}
 Choose any  $c:I_\lambda\to\lambda$. To get a contradiction, assume that there do not exist $\lambda'<\lambda$ and $q,r\in I_\lambda$ such that $q<r$ and $c^{-1}(\lambda')$ is dense in $(q,r)$. We construct $(q_ir_i)_{i<\lambda}$ such that $q_i<q_j<r_j<r_i$ for all $i<j<\lambda$, and $(q_i,r_i)\cap c^{-1}(i)=\emptyset$ for all $i<\lambda$. Choose any $q'_0<r'_0$. By the assumption, $c^{-1}(0)$ is not dense in $(q'_0,r'_0)$. So we can find $q_0<r_0$ such that $ (q_0,r_0)\cap c^{-1}(0)=\emptyset$. Suppose that we have constructed $(q_i,r_i)_{i<\lambda_0}$ for some $\lambda_0<\lambda$. By the saturation, we can find $q'_{\lambda_0},r'_{\lambda_0}\in I_\lambda$ such that $q_i<q'_{\lambda_0}<r'_{\lambda_0}<r_i$ for all $i<\lambda_0$. Since $c^{-1}(\lambda_0)$ is not dense in $(q'_{\lambda_0},r'_{\lambda_0})$, there exist $q_{\lambda_0},r_{\lambda_0}\in I_\lambda$ such that $q'_{\lambda_0}<q_{\lambda_0}<r_{\lambda_0}<r'_{\lambda_0}$ and $(q_{\lambda_0},r_{\lambda_0})\cap c^{-1}(\lambda_0)=\emptyset$. By repeating this we can complete $(q_ir_i)_{i<\lambda}$. By the saturation, we can find $d\in I_\lambda$ such that $q_i<d<r_i$ for all $i<\lambda$. But then $d\notin c^{-1}(i)$ for all $i<\lambda$, so we have a contradiction.
\end{proof}
\end{lemma}

\begin{lemma}\label{lem: monochromatic extendability of NGP}
The $\CL^*$-structure on $\mathbb{Q}\times\mathbb{Q}$ in \Cref{prop: k-grid -> 2-grid} is monochromatically extendable.
\begin{proof}
Choose any infinite cardinal $\lambda$. Let $I_\lambda$ be a $\lambda^+$-saturated $\{<\}$-structure such that $I_\lambda\models\Th(\mathbb{Q},<)$. Let $\{a_i\}_{i<\kappa}$ be an enumeration of $I_\lambda$. Let $I_\kappa$ be a $\kappa^+$-saturated $\{<\}$-structure such that $I_\kappa\models\Th(\mathbb{Q},<)$. Note that $\lambda<\kappa$. Define an $\CL^*$-structure on $I_\lambda\times I_\kappa$ as in \Cref{prop: k-grid -> 2-grid}. Choose any coloring $c:I_\lambda\times I_\kappa\to \lambda$. For each $i<\kappa$, define a coloring $c_i:I_\kappa\to \lambda$ such that $c_i(b)=c(a_i,b)$. Then by applying \Cref{lem: coloring lemma of dense linear order} repeatedly, we can construct $(q_i,r_i,\lambda_i)_{i<\kappa}$ such that
\begin{itemize}
\item[(i)] $\lambda_i<\lambda$ and $q_i<q_j<r_j<r_i$ for all $i<j<\kappa$, 
\item[(ii)] $c_i^{-1}(\lambda_i)$ is dense in $(q_i,r_i)$ for all $i<\kappa$.
\end{itemize}
By the $\kappa^+$-saturation, there exist $q,r\in I_\kappa$ such that $q_i<q<r<r_i$ for all $i<\kappa$. Thus $c^{-1}_i(\lambda_i)$ is dense in $(q,r)$ for all $i<\kappa$.

Define a coloring $c':I_\lambda\to \lambda$ such that $c'(a_i)=\lambda_i$. By \Cref{lem: coloring lemma of dense linear order}, there exist $\lambda'<\lambda$ and $q',r'\in I_\lambda$ such that $q'<r'$, $c'^{-1}(\lambda')$ is dense in $(q',r')$. Then we can find an $\CL^*$-embedding $f$ from $\mathbb{Q}\times\mathbb{Q}$ into $(q',r')\times (q,r)\subseteq I_\lambda\times I_\kappa$ such that $c(f(x,y))=\lambda'$ for all $(x,y)\in\mathbb{Q}\times\mathbb{Q}$.
\end{proof}
\end{lemma}

\begin{proposition}\label{prop: NGP is D_n-var}
$D_\text{NGP}\in\mathfrak{D}_{n\text{-var}}$.
\begin{proof}
We use the same notation as in \Cref{prop: k-grid -> 2-grid}. By \Cref{lem: monochromatic extendability of NGP}, the $\CL^*$-structure on $\mathbb{Q}\times\mathbb{Q}$ is monochromatically extendable. Let $X:=\{(i,i):i\in\mathbb{Q}\}$ and $Y:=\{(i,-i):i\in\mathbb{Q}\}$. It is easy to check that $X$ is cofinal. Thus $(\CL^*\!,\mathbb{Q}\times\mathbb{Q},X,Y)$ satisfies the $n$-variable theorem by \Cref{prop: k-grid -> 2-grid} and $D_{\text{NGP}}=D^{\CL^*\!\!,\,\mathbb{Q}\times\mathbb{Q}}_{X,Y}$.
\end{proof}
\end{proposition}

\begin{proposition}\label{prop: NPMk is D_n-var}
$D_{\text{NPM}^{(k)}}\in\mathfrak{D}_{n\text{-var}}$ for each $1<k<\omega$.
\begin{proof}
Let $<$ be a binary relation symbol, $E$ a $k$-ary relation symbol, $P,Q$ unary relation symbols, $\CL^*:=\{<,E\}$, and $\CL^*_{P,Q}:=\CL^*\cup\{P,Q\}$. Let $\CK:=\{A_i\}_{i<\omega}$ be a class of finite $\CL^*_{P,Q}$-structures such that
\begin{itemize}
\item[(i)] $E$ is irreflexive and symmetric in $A_i$,
\item[(ii)] $<$ is a linear order in $A_i$,
\item[(iii)] $P$ is a clique in $A_i$,
\item[(iv)] $Q$ is an anticlique (i.e., $\models\neg E(x_0,...,x_{k-1})$ for all $x_0,...,x_{k-1}\in Q$) in $A_i$,
\item[(v)] $P\cap Q=\emptyset$ in $A_i$,
\item[(vi)] $E(x_0,...,x_{k-1})$ for all distinct $x_0,...,x_{k-2}\in P$ and $x_{k-1}\in Q$ in $A_i$,
\end{itemize}
for each $i<\omega$, and 
\begin{itemize}
\item[(vii)] if $A$ is a finite $\CL^*_{P,Q}$-structure satisfying (i), (ii), (iii), (iv), (v), and (vi), then there exists $i<\omega$ such that $A_i\sim_{\CL^*_{P,Q}}A$,
\item[(viii)] $A_i\not\sim_{\CL^*_{P,Q}}A_j$ for all $i<j<\omega$.
\end{itemize}
For finite $\CL^*_{P,Q}$-structures $A$, $A'$  satisfying (i), (ii), (iii), (iv), (v), and (vi), we define the following notations.
\begin{itemize}
\item[(ix)] For $X,Y\subseteq A$, we write $X<Y$ if $x<y$ for all $x\in X$ and $y\in Y$.
\item[(x)] $A\oplus_< A'$ is an $\CL^*_{P,Q}$-structure on the disjoint union of $A$ and $A'$ such that 
\begin{itemize}
\item[$\ast$] $A<A'$,
\item[$\ast$] $(A\oplus_<A')|_A\sim_{\CL^*_{P,Q}} A$ and $(A\oplus_<A')|_{A'}\sim_{\CL^*_{P,Q}} A'$,
\item[$\ast$] for all distinct $x_0,...,x_{k-1}$, if $\{x_0,...,x_{k-1}\}\subseteq A\cup A'$, $\{x_0,...,x_{k-1}\}\not\subseteq A$, and $\{x_0,...,x_{k-1}\}\not\subseteq A'$, then $E(x_0,...,x_{k-1})$ if and only if they satisfy one of the following conditions.
\begin{itemize}
\item[$\ast\ast$] $P(x_i)$ for all $i<k$,
\item[$\ast\ast$] $Q(x_i)$ and $\bigwedge_{j\neq i}P(x_j)$ for some $i<k$.
\end{itemize}
\end{itemize}
\item[(xi)] $A^-$ is the substructure of $A$ on $A\setminus\{a\}$ for the $<$-maximal element $a$ of $A$.
\end{itemize}
Note that $A\oplus_<A'$ still satisfies (i), (ii), (iii), (iv), (v), and (vi).

We construct a chain $(\CI_i)_{i<\omega}$ of $\CL^*_{P,Q}$-structures as follows. $\CI_0:=A_0$. Suppose that  we have constructed $\CI_{2n}$. Let $\CK_{2n}:=\{B^{2n}_i\}_{i<d_{2n}}$ be a set of $\CL^*_{P,Q}$-structures such that 
\begin{itemize}
\item[(xii)] $d_{2n}<\omega$
\item[(xiii)] $B_i^{2n-}\subseteq \CI_{2n}$ for each $i<d_{2n}$,
\item[(xiv)] for each $i<d_{2n}$, there exists $j<\omega$ such that $B_i^{2n}\sim_{\CL^*_{P,Q}}A_j$,
\item[(xv)] for any $\CL^*_{P,Q}$-structure $B$ satisfying (i), (ii), (iii), (iv), (v), and (vi), if $B^-\subseteq \CI_{2n}$, then there exists $i<d_{2n}$ such that $B\sim_{\CL^*_{P,Q}}B^{2n}_i$
\item[(xvi)] if $i<j<d_{2n}$, then $B^{2n-}_i\neq B^{2n-}_j$ or $B^{2n}_i\not\sim_{\CL^*_{P,Q}}B^{2n}_j$.
\end{itemize}
Let $\CI_{2n+1}$ be the $\CL^*_{P,Q}$-structure on $\CI_{2n}\cup\{b^{2n}_i\}_{i<d_{2n}}$, where $b^{2n}_i$ is a new single element, such that
\begin{itemize}
\item[(xvii)] $\CI_{2n}<b^{2n}_0<\cdots<b^{2n}_{d_{2n}-1}$,
\item[(xviii)] $B^{2n-}_ib^{2n}_i\sim_{\CL^*_{P,Q}}B^{2n}_i$ for each $i<d_{2n}$,
\item[(xix)] for all distinct $x_0,...,x_{k-1}$, if $\{x_0,...,x_{k-1}\}\subseteq \CI_{2n}\cup\{b^{2n}_i\}_{i<d_{2n}}$, $\{x_0,...,x_{k-1}\}\not\subseteq \CI_{2n}$, and $\{x_0,...,x_{k-1}\}\not\subseteq B^{2n-}_ib^{2n}_i$ for all $i<d_{2n}$, then $E(x_0,...,x_{k-1})$ if and only if they satisfy one of the following conditions.
\begin{itemize}
\item[$\ast$] $P(x_i)$ for all $i<k$,
\item[$\ast$] $Q(x_i)$ and $\bigwedge_{j\neq i}P(x_j)$ for some $i<k$.
\end{itemize}
\end{itemize}
Let $\CI_{2n+2}:=\CI_{2n+1}\oplus_< A_0 \oplus_<\cdots\oplus_<A_{n+1}$. Let $\CI:=\bigcup_{n<\omega}\CI_n$, $\CI^*:=\CI|_{\CL^*}$, and $X:=P(\CI)$. Choose any $Y\subseteq Q(\CI)$ such that $|Y|=k$. Clearly $X$ is cofinal.

We show that $\CI^*$ is monochromatically extendable. Fix a cardinal $\lambda$ and choose any $(\lambda+2^\omega)^+$-saturated $\{<,E\}$-structure $\CJ\models\Th(\CI^*)$. Let $c$ be a $\lambda$-coloring on $\CJ$.

\medskip

\noindent\underline{Claim.} There exist $\lambda'<\lambda$, $\delta<\lambda$, and $(r_i)_{i\le\delta}\subseteq\CJ$ such that 
\begin{itemize}
\item[(xx)] $r_i<r_j$ for all $i< j\le\delta$,
\item[(xxi)] for all $s_0,...,s_n\in\CJ$, with $r_\delta<s_0<\cdots<s_n$, there exists $s'$ such that
\begin{itemize}
\item[$\ast$] $s'\equiv^{\rm qf}_{s_{<n}}s_n$,
\item[$\ast$] $s'\equiv^{\rm qf}_{r_{<\delta}}r_\delta$,
\item[$\ast$] $c(s')=\lambda'$.
\end{itemize}
\end{itemize}

\smallskip

\noindent{\it Proof of Claim.} Suppose not. Then by the choice of $\CI^*$ and $\CJ$, we can construct a sequence of finite tuples $(\bar s^i:=(s^i_0,...,s^i_{n_i}))_{i<\lambda}\subseteq\CJ$ such that
\begin{itemize}
\item[(xxii)] $s^i_{k_0}<s^i_{k_1}$ for all $i<\lambda$ and $k_0<k_1\le n_i$,
\item[(xxiii)] $s^i_k<s^{i'}_{k'}$ for all  $k\le n_i$, $k'\le n_{i'}$ and $i<i'<\lambda$,
\item[(xxiv)] $s^{i}_{n_{i}}\equiv^{\rm qf}_{\bar s^{<i} s^i_{<n_i}}s^{i'}_{n_{i'}}$ for all $i<i'<\lambda$,
\item[(xxv)] for all $i<\lambda$, there is no $s'\equiv^{\rm qf}_{\bar s^{<i} s^i_{<n_i}}s^i_{n_i}$ such that $c(s')=i$.
\end{itemize}
Since $\CJ$ is $(\lambda+2^\omega)^+$-saturated, we can find $s$ such that $s\equiv^{\rm qf}_{\bar s^{<i} s^i_{<n_i}}s^{i}_{n_{i}}$ for all $i<\lambda$. By (xxv), $c(s)\neq i$ for all $i<\lambda$. This yields a contradiction. $\dashv$

\medskip

Let $\lambda'$ and $(r_i)_{i\le\delta}$ be given in Claim. Then we can find an $\CL^*$-embedding $f:\CI^*\to\CJ$ such that $f(i)\equiv^{\rm qf}_{r_{<\delta}}r_\delta$ and $c(f(i))=\lambda'$ for all $i\in\CI^*$. Since $\lambda$ is arbitrary, $\CI^*$ is monochromatically extendable. It is easy to check that $(\CL^*\!,\CI^*\!,X,Y)$ satisfies the remaining conditions for the $n$-variable theorem as well, and that $D_{\text{NPM}^{(k)}}=D^{\CL^*\!\!,\,\CI^*}_{X,Y}\!$.
\end{proof}
\end{proposition}

\begin{proposition}\label{prop: NWP is D_n-var}
$D_\text{NWP}\in\mathfrak{D}_{n\text{-var}}$.
\begin{proof}
As we mentioned in \Cref{def: path antichain descending comb}, $(2^2)^\omega$ can be regarded as $4^\omega$. Let $\CL_4:=\CL_0\cup\{\lhd_i:i<4\}$, with the usual
interpretations on $4^{\le\omega}$.
We write $\sim_4$ for equality of quantifier-free $\CL_4$-types.
Let $\Delta_\unlhd,\Delta_0,\Delta_1,\Delta_2,$ and $\Delta_3$ be quaternary relation symbols and $\CL_{\Delta,4}:=\{\Delta_\unlhd,\Delta_0,\Delta_1,\Delta_2,\Delta_3\}$. We give an $\CL_{\Delta,4}$-structure on $4^\omega$ by interpreting each symbol in $\CL_{\Delta,4}$ as follows.
\begin{itemize}
\item[(i)] $\Delta_\unlhd(\eta_0,\eta_1,\eta_2,\eta_3)$ if and only if $\eta_0\wedge\eta_1\unlhd\eta_2\wedge\eta_3$ in $4^{\le\omega}$.
\item[(ii)] For $i<4$, $\Delta_i(\eta_0,\eta_1,\eta_2,\eta_3)$ if and only if $(\eta_0\wedge\eta_1)^\frown{\la i\ra}\unlhd\eta_2\wedge\eta_3$ in $4^{\le\omega}$.
\end{itemize}
For $\bar{\eta}\in 4^{\le\omega}$, let ${\bf 0}(\bar{\eta})\in 4^{\le\omega}$ be the tuple obtained by replacing $\eta$ with $\eta^\frown{\la0\ra^\omega}$ for all $\unlhd$-maximal $\eta$ in $\bar{\eta}$ with $l(\eta)<\omega$.
It is easy to check that for $\bar{\eta},\bar{\nu}\in 4^\omega$ and $\bar\xi\in4^{\le\omega}$,
\begin{itemize}
\item[(iii)] $\bar{\eta}\sim_{\Delta,4}\bar{\nu}$ if and only if $\cl(\bar{\eta})\sim_4\cl(\bar{\nu})$,
\item[(iv)] $\bar{\xi}\sim_4 {\bf 0}(\bar{\xi})$.
\end{itemize}

First we show that the $\CL_{\Delta,4}$-structure on $4^\omega$ has the modeling property. Fix $(a_\eta)_{\eta\in 4^\omega}$. For each $\eta\in 4^{<\omega}$, let $a_\eta:=a_{\bf0(\eta)}$.
By \cite{KK11}, the $\CL_4$-structure on $4^{\le\omega}$ has the modeling property.  So we can find $\CL_4$-indiscernible $(b_\eta)_{\eta\in 4^{\le\omega}}$ $\CL_4$-locally based on $(a_\eta)_{\eta\in 4^{\le\omega}}$. It is easy to check that $(b_\eta)_{\eta\in 4^\omega}$ is $\CL_{\Delta,4}$-indiscernible and $\CL_{\Delta,4}$-locally based on $(a_\eta)_{\eta\in 4^\omega}$ by using (iii) and (iv). This proves that the $\CL_{\Delta,4}$-structure has the modeling property.

Let $4^\omega_\text{fin}:=\{\eta\in 4^\omega:\eta(i)\neq0\text{ for only finitely many }i<\omega\}$. We can give an $\CL_{\Delta,4}$-structure on $4^\omega_\text{fin}$ in the same way as above. Since $\age(4^\omega_\text{fin})=\age(4^\omega)$, the $\CL_{\Delta,4}$-structure on $4^\omega_\text{fin}$ also has the modeling property.

\medskip

\noindent\underline{Claim.} $4^\omega_\text{fin}$ is monochromatically extendable.

\smallskip

\noindent{\it Proof of Claim.} Fix a cardinal $\lambda$ and choose any infinite cardinal $\kappa>\max\{\lambda,\omega\}$ such that $\cf(\kappa)=\kappa$. Clearly $\age(4^\kappa)=\age(4^\omega_\text{fin})$. Let $c$ be a $\lambda$-coloring on $4^\kappa$. Then by the same argument of \Cref{lem: coloring lemma of dense linear order}, we can find $\lambda'<\lambda$ and $\xi\in 4^{<\kappa}$ such that 
\begin{itemize}
\item[(v)] for any $\nu\in 4^{<\kappa}$ with $\xi\lhd\nu$, there exists $\nu'\in 4^\kappa$ such that $\nu\lhd\nu'$ and $c(\nu')=\lambda'$.
\end{itemize}
Thus for each $\eta\in 4^\omega_\text{fin}$, there exists $\nu'_\eta\in 4^\kappa$ such that $\xi^\frown \eta\lhd \nu'_\eta$ and $c(\nu'_\eta)=\lambda'$. Define a map $f:4^\omega_\text{fin}\to 4^\kappa$ such that $f(\eta)=\nu'_\eta$. Then $f$ is an $\CL_{\Delta,4}$-embedding and $c(f(\eta))=\lambda'$ for all $\eta\in 4^\omega_\text{fin}$. Since $\lambda$ is arbitrary, this proves that $4^\omega_\text{fin}$ is monochromatically extendable. $\dashv$

\medskip

 Let $\eta^*$, $\nu^*_0$, $\nu^*_1$, and $\nu^*_2$ be maps from $\omega$ to $4$ such that 
\[\eta^*(2i)=2\text{ and }\eta^*(2i+1)=3\text{ for each }i<\omega,\] 
\[\nu^*_0(2i)=1\text{ and }\nu^*_0(2i+1)=3\text{ for each }i<\omega,\] 
\[\nu^*_1(i)=1\text{ for each }i<\omega,\] 
\[\nu^*_2(i)=3\text{ for each }i<\omega,\] 
and let 
\[X:=\{{\bf0}(({\eta^*}_{|_{2i}})^\frown{\la 0\ra}):i<\omega\}\cup\{{\bf0}(({\eta^*}_{|_{2i+1}})^\frown{\la 1\ra}):i<\omega\},\]
\[Y_0:=\{{\bf0}(({\nu^*_0}_{|_{2i}})^\frown{\la 0\ra}):i<\omega\}\cup\{{\bf0}(({\nu^*_0}_{|_{2i+1}})^\frown{\la 2\ra}):i<\omega\},\]
\[Y_1:=\{{\bf0}(({\nu^*_1}_{|_{i}})^\frown{\la 0\ra}):i<\omega\},\]
\[Y_2:=\{{\bf0}(({\nu^*_2}_{|_{i}})^\frown{\la 2\ra}):i<\omega\}.\]
Let $\bar{Y}:=(Y_0,Y_1,Y_2)$.
It is easy to check $X$ is cofinal, $(\CL_{\Delta,4},4^\omega_\text{fin},X,\bar{Y})$ satisfies the $n$-variable theorem, and $D_\text{NWP}=D_{X,\bar{Y}}^{\CL_{\Delta,4},\,4^\omega_\text{fin}}$.
\end{proof}
\end{proposition}

\subsection{The classes of simple theories and NBTP theories}
Let $t_i$ be a unary relation symbol for each $i<\omega$
and $\unlhd^-$, $<_{ds}$, $<_{ll}$, $<_{rv}$ binary relation symbols.
Let $\Delta^-_\unlhd$, $\Delta^-_{lex}$, and $\Delta^-_l$
be quaternary relation symbols, and let $\Delta^-_{\lhd_i}$
be a six-ary relation symbol for each $i<\omega$.
\begin{itemize}
\item[]
$\CL_{ds}:=\{t_i\}_{i<\omega}\cup\{\unlhd^-\}
\cup\{\Delta^-_{\lhd_i}\}_{i<\omega}
\cup\{\Delta^-_\unlhd,\Delta^-_l\}$.
\item[]
$\CL_{kr^*}:=\{<_{ds},<_{ll},<_{rv}\}
\cup\{\Delta^-_{lex},\Delta^-_\unlhd,\Delta^-_l\}$.
\item[]
$\CL_{kr}:=\{<_{ds},<_{ll},<_{rv}\}
\cup\{\Delta^-_{lex},\Delta^-_\unlhd\}$.
\end{itemize}

\begin{notation}\label{notation: level closure}
Let $\kappa$ be an infinite cardinal, $\lambda$ a cardinal,
and $\wedge_l$ a binary function symbol.
On $\lambda^{<\kappa}$, we interpret $\wedge_l$ by
\[
\eta\wedge_l\nu:=
\begin{cases}
\nu|_{l(\eta)} & \text{if }l(\eta)\le l(\nu),\\
\nu & \text{otherwise}.
\end{cases}
\]
On $\Tau_\kappa$, we interpret $\wedge_l$ by
\[
\eta\wedge_l\nu:=
\begin{cases}
\nu|_{[l(\eta),\kappa)} & \text{if }l(\eta)\ge l(\nu),\\
\nu & \text{otherwise}.
\end{cases}
\]
\end{notation}

\begin{notation}\label{notation: ds kr language}
We give an $\CL_{ds}$-structure, an $\CL_{kr^*}$-structure, and an $\CL_{kr}$-structure on $\Tau_\omega$ by giving an interpretation of each symbol as follows.
\begin{itemize}
\item[(i)] By $\eta\in t_i$, we mean $t(\eta)=i$. 
\item[(ii)] By $\eta\unlhd^-\nu$, we mean $\eta^-\unlhd\nu$ and $\eta\perp\nu$.
\item[(iii)] By $\eta<_{ds}\nu$, we mean $\eta\lex\nu$, $\eta\neq\nu$, and that $\{\eta,\nu\}$ forms a direct sibling set.
\item[(iv)] By $\eta<_{ll}\nu$, we mean $\eta<_{lex}\nu$ and that $\{\eta,\nu\}$ forms a strict left-leaning path.
\item[(v)] By $\eta<_{rv}\nu$, we mean $\eta\lex\nu$ and that $\{\eta,\nu\}$ forms a strict right-veering path.
\item[(vi)] By $(\eta_0,\eta_1,\eta_2,\eta_3)\in\Delta^-_\unlhd$, we mean $\eta^-_0\wedge\eta^-_1\unlhd\eta^-_2\wedge\eta^-_3$.
\item[(vii)] By $(\eta_0,\eta_1,\eta_2,\eta_3)\in\Delta^-_{lex}$, we mean $\eta^-_0\wedge\eta^-_1\lex\eta^-_2\wedge\eta^-_3$.
\item[(viii)] By $(\eta_0,\eta_1,\eta_2,\eta_3)\in\Delta^-_l$, we mean $\eta^-_0\wedge\eta^-_1<_l\eta^-_2\wedge\eta^-_3$.
\item[(ix)]
By $(\eta_0,\eta_1,\eta_2,\eta_3,\eta_4,\eta_5)\in\Delta^-_{\lhd_i}$, we mean
$(\eta^-_0\wedge\eta^-_1)\wedge_l\eta^-_2
\lhd_i
(\eta^-_3\wedge\eta^-_4)\wedge_l\eta^-_5.
$
\end{itemize}
\end{notation}

\smallskip

We will argue that the $\CL_{ds}$-structure, the $\CL_{kr^*}$-structure, and the $\CL_{kr}$-structure  on $\Tau_\omega$ have the modeling property. First we recall the modeling property of $0$-fti. The idea of $0$-fti was introduced by D{\v z}amonja and Shelah \cite{DS04} under the name $0$-fbti on $2^{<\omega}$, and was later developed by Kim and Kim \cite{KK11} for $n^{<\omega}$ for an arbitrary natural number $n$. 
For $n\in\omega$, let 
\begin{itemize}
\item[] $\CL^n_{0\text{-fti}}:=\{\wedge,\unlhd,\wedge_l,<_l\}\cup\{\lhd_i\}_{i<n}$\vspace{-4pt}
\end{itemize}
and\vspace{-4pt}
\begin{itemize}
\item[] $\CL_{0\text{-fti}}:=\{\wedge,\unlhd,\wedge_l,<_l\}\cup\{\lhd_i\}_{i<\omega}$.\vspace{2pt}
\end{itemize}
 Then we can give an $\CL^n_{0\text{-fti}}$-structure on $n^{<\omega}$ and $\CL_{0\text{-fti}}$-structures on $\omega^{<\omega}$ and $\Tau_\omega$ with the interpretations given above. Note that $\age_{\CL_{0\text{-fti}}}(\omega^{<\omega})=\age_{\CL_{0\text{-fti}}}(\Tau_\omega)$.

\begin{fact}\cite[Proposition 2.9]{KK11}\label{fact: modeling property of 0-fti on n^<omega}
The $\CL^n_\text{0-fti}$-structure on $n^{<\omega}$ has the modeling property for each $n<\omega$.
\end{fact}

\begin{lemma}\label{lem: modeling property of 0-fti tau_omega}
The $\CL_{0\text{-fti}}$-structure on $\Tau_\omega$ has the modeling property.
\begin{proof}
By compactness and \Cref{fact: modeling property of 0-fti on n^<omega},  the $\CL_{0\text{-fti}}$-structure on $\omega^{<\omega}$ has the modeling property.
By \Cref{fact: age(I)=age(J)}, the $\CL_{0\text{-fti}}$-structure on $\Tau_\omega$ has the modeling property.
\end{proof}
\end{lemma}

\begin{lemma}\label{lem: ds-sim => 1-sim}
Let $\bar{\eta}:=(\eta_0,...,\eta_{n-1})$ and $\bar{\nu}:=(\nu_0,...,\nu_{n-1})$ be tuples in $\Tau_\omega$. Then $\bar{\eta}\sim_{ds}\bar{\nu}$, if and only if $\bar{\eta}^-\sim_{0\text{-fti}}\bar{\nu}^-$ and $t(\eta_i)=t(\nu_i)$ for all $i<n$, where $\bar{\eta}^-:=(\eta^-_0,...,\eta^-_{n-1})$ and $\bar{\nu}^-:=(\nu^-_0,...,\nu^-_{n-1})$.
\begin{proof}
Clear.
\end{proof}
\end{lemma}

\begin{notation}
Let $(a_\eta)_{\eta\in\omega^{<\omega}}$ be an $\omega^{<\omega}$-indexed set of parameters and $\bar{\eta}:=(\eta_0,...,\eta_{n-1})\in(\omega^{<\omega})^n$. By $\bar{a}_{\bar{\eta}}$, we mean $(a_{\eta_0},...,a_{\eta_{n-1}})$. We can give the similar notation for $\Tau_\omega$-indexed sets of parameters.
\end{notation}

The main idea of the proof of the following proposition comes from D\v{z}amonja and Shelah \cite{DS04}.

\begin{proposition}\label{prop: modeling property of ds-indiscernibility}
The $\CL_{ds}$-structure on $\Tau_\omega$ has the modeling property.
\begin{proof}
Choose any $\Tau_\omega$-indexed set of parameters $(a_\eta)_{\eta\in\Tau_\omega}$. Let $\Tau^-_\omega:=\{\eta\in\Tau_\omega:l(\eta)>0\}$. For each $\eta\in\Tau^-_\omega$, let $b_\eta:=(a_{\eta^\frown\la i\ra})_{i<\omega}$. By applying \Cref{fact: age(I)=age(J)} and \Cref{lem: modeling property of 0-fti tau_omega}, we can find $\CL_{0\text{-fti}}$-indiscernible $(c_\eta)_{\eta\in\Tau^-_\omega}$ which is $\CL_{0\text{-fti}}$-locally based on $(b_\eta)_{\eta\in\Tau^-_\omega}$. Note that each $c_\eta$ is of the form $(c^i_\eta)_{i<\omega}$. For each $\eta\in\Tau_\omega$, let $d_\eta:=c^{t(\eta)}_{\eta^-}$.

First we show that $(d_\eta)_{\eta\in\Tau_\omega}$ is ds-indiscernible. Choose any $\bar{\eta}:=(\eta_0,...,\eta_{n-1}),\bar{\nu}:=(\nu_0,...,\nu_{n-1})$ in $\Tau_\omega$. Suppose $\bar{\eta}\sim_{ds}\bar{\nu}$. Then by \Cref{lem: ds-sim => 1-sim}, we have $\bar{c}_{\bar{\eta}^-}\equiv \bar{c}_{\bar{\nu}^-}$. In particular, $c^{t(\eta_0)}_{\eta^-_0}...c^{t(\eta_{n-1})}_{\eta^-_{n-1}}\equiv c^{t(\eta_0)}_{\nu^-_0}...c^{t(\eta_{n-1})}_{\nu^-_{n-1}}$. Since $\bar{\eta}\sim_{ds}\bar{\nu}$, we have $t(\eta_i)=t(\nu_i)$ for each $i<n$, and hence $c^{t(\eta_0)}_{\eta^-_0}...c^{t(\eta_{n-1})}_{\eta^-_{n-1}}\equiv c^{t(\nu_0)}_{\nu^-_0}...c^{t(\nu_{n-1})}_{\nu^-_{n-1}}$. Thus $d_{\eta_0}...d_{\eta_{n-1}}\equiv d_{\nu_0}...d_{\nu_{n-1}}$.

Now we show that $(d_\eta)_{\eta\in\Tau_\omega}$ is $\CL_{ds}$-locally based on $(a_\eta)_{\eta\in\Tau_\omega}$. Suppose $\models\varphi(d_{\eta_0},...,d_{\eta_{n-1}})$. Then $\models\varphi(c^{t(\eta_0)}_{\eta^-_0},...,c^{t(\eta_{n-1})}_{\eta^-_{n-1}})$. Since $(c_\eta)_{\eta\in\Tau^-_\omega}$ is $\CL_{0\text{-fti}}$-locally based on $(b_\eta)_{\eta\in\Tau^-_\omega}$, there exists $\bar{\nu}:=(\nu_0,...,\nu_{n-1})$ such that $(\nu_0,...,\nu_{n-1}) \sim_{0\text{-fti}} (\eta^-_0,...,\eta^-_{n-1})$ and $\models\varphi(b^{t(\eta_0)}_{\nu_0},...,b^{t(\eta_{n-1})}_{\nu_{n-1}})$. For each $i<n$, let $\mu_i:=\nu_i ^\frown \la t(\eta_i)\ra$. Then by \Cref{lem: ds-sim => 1-sim}, $\bar{\mu}\sim_{ds}\bar{\eta}$. Clearly $\models\varphi(a_{\mu_0},...,a_{\mu_{n-1}})$.
\end{proof}
\end{proposition}

\begin{notation}
For a cardinal $\kappa$ and $I\subseteq \kappa$, we put $\Tau_\kappa|_I:=\{\eta\in\Tau_\kappa:l(\eta)\in I,\eta(i)=0\text{ for all }i\notin I\}$.
\end{notation}

\begin{definition}
$(a_\eta)_{\eta\in\Tau_\omega}$ is said to be {\it level-s-indiscernible} if it is s-indiscernible and $(a_\eta)_{\eta\in \Tau_\omega|_I}\equiv(a_\eta)_{\eta\in \Tau_\omega|_J}$ for all $I,J\subseteq\omega$ with $|I|=|J|<\omega$.
\end{definition}

\begin{lemma}\label{lem: s-sim ds-sim => krstar-sim => kr-sim}
If $\bar{\eta}\sim_s\bar{\nu}$ or $\bar{\eta}\sim_{ds}\bar{\nu}$, then $\bar{\eta}\sim_{kr^*}\bar{\nu}$. If $\bar{\eta}\sim_{kr^*}\bar{\nu}$, then $\bar{\eta}\sim_{kr}\bar{\nu}$.
\end{lemma}

\begin{notation}
For a finite tuple $\bar\eta:=(\eta_0,...,\eta_{n-1})\in\Tau_\omega$ with $\eta_0\lex\cdots\lex\eta_{n-1}$, we define $\bar\eta^{nor}$ as follows.
\begin{itemize}
\item[(i)] $(\eta_0)^{nor}:=(\la 0\ra^{l(\eta_0)})$.
\item[(ii)] If $n>1$, then let $(\eta'_1,...,\eta_{n-1}'):=(\eta_1,...,\eta_{n-1})^{nor}$ and $d_i:=l(\eta_i\wedge\eta_0)-1$. Put 
 \[
\eta''_i=
\begin{cases}
 \la 0\ra^{l(\eta_0)}& \text{ if } i=0\\
\eta'_i & \text{ if } \eta_0\lhd\eta_i\\
\bigg(\eta'_i\setminus\Big\lbrace
\Big(d_i,\eta'_i(d_i)\Big)\Big\rbrace\bigg)\cup\Big\lbrace\Big(d_i,\eta'_i(d_i)+1\Big)\Big\rbrace & \text{ otherwise}
\end{cases}
\]
and $(\eta_0,...,\eta_{n-1})^{nor}:=(\eta''_0,...,\eta''_{n-1})$.
\end{itemize}
\end{notation}

\begin{lemma}\label{lem: s-similarity of left-normal}
$\bar\eta\sim_s\bar\eta^{nor}$.
\begin{proof}
Clear.
\end{proof}
\end{lemma}

\begin{lemma}\label{lem: level-s = krstar}
Let $(a_\eta)_{\eta\in\Tau_\omega}$ be a $\Tau_\omega$-indexed set of parameters. Then $(a_\eta)_{\eta\in\Tau_\omega}$ is kr$^*\!$-indiscernible if and only if it is level-s-indiscernible.
\begin{proof}
Suppose that $(a_\eta)_{\eta\in\Tau_\omega}$ is kr$^*\!$-indiscernible. Then by \Cref{lem: s-sim ds-sim => krstar-sim => kr-sim}, it is s-indiscernible. It is easy to check that $\Tau_\omega|_I\sim_{kr^*}\Tau_\omega|_J$ for any $I,J\subset\omega$ with $|I|=|J|<\omega$. Thus $(a_\eta)_{\eta\in\Tau_\omega}$ is level-s-indiscernible.

Conversely, suppose that $(a_\eta)_{\eta\in\Tau_\omega}$ is level-s-indiscernible. Let $\bar{\eta}\sim_{kr^*}\bar{\nu}$ for some $\bar{\eta}:=(\eta_0,...,\eta_{n-1})$ and $\bar{\nu}:=(\nu_0,...,\nu_{n-1})$. By \Cref{lem: s-similarity of left-normal}, it is enough to show that $\bar a_{\bar{\eta}^{nor}}\equiv \bar a_{\bar\nu^{nor}}$. Let 
\[
\begin{aligned}
I&:=\{k<\omega:k=l(\mu)-1\text{ for some }
\mu\in\cl((\bar{\eta}^{nor})^-)\},\\
J&:=\{k<\omega:k=l(\mu)-1\text{ for some }
\mu\in\cl((\bar{\nu}^{nor})^-)\}.
\end{aligned}
\]
Then $\bar\eta^{nor}\subseteq\Tau_\omega|_I$, $\bar\nu^{nor}\subseteq\Tau_\omega|_J$, and $|I|=|J|$. Let $f$ be the bijective $\CL_{kr^*}$-embedding from $\Tau_\omega|_I$ to $\Tau_\omega|_J$.  Note that such a bijection is unique and $f(\bar\eta^{nor})=\bar\nu^{nor}$. Since we assume that $(a_\eta)_{\eta\in\Tau_\omega}$ is level-s-indiscernible, we have
$\bar a_{\bar\eta^{nor}}\equiv\bar a_{\bar\nu^{nor}}$.
\end{proof}
\end{lemma}

\begin{proposition}\label{prop: modeling property of krstar-indiscernibility}
The $\CL_{kr^*}$-structure on $\Tau_\omega$ has the modeling property.
\begin{proof}
Let $(a_\eta)_{\eta\in\Tau_\omega}$ be a $\Tau_\omega$-indexed set of parameters. By \Cref{prop: modeling property of ds-indiscernibility}, we can find ds-indiscernible $(b_\eta)_{\eta\in\Tau_\omega}$ which is ds-locally based on $(a_\eta)_{\eta\in\Tau_\omega}$. Choose a sufficiently large cardinal $\kappa$ with $\cf(\kappa)=\kappa$. Since $\age_{ds}(\Tau_\omega)=\age_{ds}(\Tau_\kappa)$, we can extend $(b_\eta)_{\eta\in\Tau_\omega}$ to $(b_\eta)_{\eta\in\Tau_\kappa}$ by \Cref{fact: age(I)=age(J)}. By the modeling property, there exists s-indiscernible $(c_\eta)_{\eta\in\Tau_\kappa}$ which is s-locally based on $(b_\eta)_{\eta\in\Tau_\kappa}$. By using the argument in \cite[Lemma 5.10]{KR20}, we  can find a level-s-indiscernible tree $(d_\eta)_{\eta\in\Tau_\omega}$ such that for each $I\subseteq \omega$ with $|I|<\omega$, there exists $J\subseteq\kappa$ with $|J|=|I|$ such that $(d_\eta)_{\eta\in\Tau_\omega|_I}\equiv (c_\eta)_{\eta\in\Tau_\kappa|_J}$. By \Cref{lem: level-s = krstar}, $(d_\eta)_{\eta\in\Tau_\omega}$ is kr$^*\!$-indiscernible. Note that $\Tau_\omega|_I\sim_{kr^*}\Tau_\kappa|_J$ for all $I\subseteq \omega$ and $J\subseteq\kappa$ with $|I|=|J|<\omega$. Thus by \Cref{lem: s-sim ds-sim => krstar-sim => kr-sim}, $(d_\eta)_{\eta\in\Tau_\omega}$ is kr$^*\!$-locally based on $(a_\eta)_{\eta\in\Tau_\omega}$.
\end{proof}
\end{proposition}

\begin{proposition}\label{prop: simple is D_n-var}
$D_\text{simple}\in\mathfrak{D}_{n\text{-var}}$.
\begin{proof}
Choose any uncountable cardinal $\kappa$ such that $\cf(\kappa)=\kappa$.
Consider the $\CL_{{kr}^*}$-structure on 
\[
\CI^*:=\{\eta\in\omega^{<\kappa}:l(\eta)\text{ is a successor ordinal}\}.
\]
Then $\CI^*$-indiscernibles have the modeling property by \Cref{fact: age(I)=age(J)}. We show $\CI^*$ is monochromatically extendable. Fix a cardinal $\lambda$. Let $\mu:=\max\{\kappa,\lambda\}$ and $\delta:=(2^\mu)^+$. Let 
\[
\delta^{<\delta}_\text{succ}:=\{\nu\in\delta^{<\delta}: l(\nu)\text{ is a successor ordinal}\}.
\]
Choose any $c:\delta^{<\delta}_\text{succ}\to \lambda$. For each $\eta$, there exist $\lambda_\eta<\lambda$ and increasing $(m_{\eta,i})_{i<\omega}$ such that $c(\eta^\frown\la m_{\eta,i}\ra)=\lambda_\eta$ for all $i<\omega$. Let $d$ be a map from $\delta^{<\delta}$ to $\lambda$, sending $\eta$ to $\lambda_\eta$.
Let $g$ be a map from $\omega^{<\delta}$ to $\delta^{<\delta}$ such that
\[
g(\xi)=
\begin{cases}
\emptyset & \text{ if }\xi=\emptyset
\\
g(\xi')^\frown\la m_{g(\xi'),i}\ra & \text{ if }\xi=\xi'^\frown\la i\ra\text{ for some }\xi'\text{ and }i<\omega
\\
\bigcup_{\xi'\lhd\xi}g(\xi') &\text{ if }\xi\text{ is limit}
\end{cases}
\]
 Let
\[
\omega^{<\delta}_\text{supp}:=\{\xi\in\omega^{<\delta}:\, |\{i<l(\xi):\xi(i)\neq 0\}|\le\mu\}.
\]
For each $\alpha<\delta$, let 
\[
P_\alpha:=\{\xi\in \omega^{<\delta}_\text{supp}:l(\xi)=\alpha\}.
\]
Then $|P_\alpha|<\delta$ for each $\alpha<\delta$. Let $Q:=\{\alpha<\delta:\cf(\alpha)=\mu^+\}$. Then $Q$ is a stationary subset of $\delta$.  

\medskip

\noindent\underline{Claim.} For each $\alpha \in Q$, there exist $\xi_\alpha\in\omega^{<\delta}_\text{supp}$ and $\lambda_\alpha<\lambda$ such that 
\begin{itemize}
\item[(i)] $l(\xi_\alpha)<\alpha$,
\item[(ii)] for each $\xi\in\omega^{<\delta}_\text{supp}$ with $\xi_\alpha\unlhd\xi$ and $l(\xi)<\alpha$, there exists $\xi'\rhd\xi$ such that $\xi'\in P_{\alpha}$ and $(d\circ g)(\xi')=\lambda_\alpha$.
\end{itemize}

\smallskip

\noindent{\it Proof of Claim.} Suppose not. Since $\cf(\alpha)=\mu^+=\cf(\mu^+)$ and $\lambda<\mu^+$, we can construct $(\xi_i)_{i<\lambda}$ such that $\xi_i\lhd\xi_j$ and $l(\xi_i)<\alpha$ for all $i<j<\lambda$, and $(d\circ g)(\xi')\neq i$ for all $\xi_i\lhd\xi'\in P_\alpha$.  Since $l(\bigcup_{i<\lambda}\xi_i)<\alpha$, there exists $\xi'\in P_\alpha$ such that $\xi_i\lhd \xi'$ for all $i<\lambda$. Thus $(d\circ g)(\xi')\neq i$ for all $i<\lambda$. But it is not possible. $\dashv$

\medskip

Let $h$ be a map from $Q$ to $\delta$, sending $\alpha$ to $l(\xi_\alpha)$. Then $h(\alpha)<\alpha$ for all $\alpha\in Q$. By Fodor's lemma, there exists a stationary $Q_0\subseteq Q$ and $\alpha_0$ such that $h(\alpha)=\alpha_0$ for all $\alpha\in Q_0$. Note that $|Q_0|=\delta$ since $\cf(\delta)=\delta$. Since $|P_{\alpha_0}|<\delta$ and $\xi_\alpha\in P_{\alpha_0}$ for all $\alpha\in Q_0$, there exist $Q_1\subseteq Q_0$, $\xi_1\in\omega^{<\delta}_\text{supp}$, and $\lambda_1<\lambda$ such that $|Q_1|=\delta$, and $\xi_\alpha=\xi_1 \wedge \lambda_\alpha=\lambda_1$ for all $\alpha\in Q_1$.
Let $\{\beta_i\}_{i<\delta}$ be the increasing enumeration of $Q_1$.

Now we construct $f:\CI^*\to\delta^{<\delta}_\text{succ}$ such that 
\begin{itemize}
\item[(iii)] $g(\xi_1)\lhd f(\eta)$ for all $\eta\in\CI^*$,
\item[(iv)] $l(f(\eta))=\beta_{l(\eta)-1}+1$ for all $\eta\in \CI^*$,
\item[(v)] $c(f(\eta))=\lambda_1$ for all $\eta\in \CI^*$.
\end{itemize}
 Choose any $\xi_\emptyset\in P_{\beta_0}$ with $\xi_1\lhd\xi_\emptyset$, $(d\circ g)(\xi_\emptyset)=\lambda_1$, and let $f(\la i\ra):=g(\xi_\emptyset^\frown\la i \ra)$ for each $i<\omega$. Choose $\eta\in\omega^{<\kappa}\setminus\{\emptyset\}$. 
 Suppose we have constructed $f(\eta)$. Then there exists $\xi_\eta\in P_{\beta_{l(\eta)}}$ such that $f(\eta)\lhd g(\xi_\eta)$ and $(d\circ g)(\xi_\eta)=\lambda_1$. Let $f(\eta^\frown\la i\ra):=g(\xi_\eta^\frown\la i\ra)$ for each $i<\omega$. 
 Now suppose $l(\eta)$ is limit and we have constructed $f(\eta')$ for all $\eta'\lhd\eta$. Then there exists $\xi_{\eta}\in P_{\beta_{l(\eta)}}$ such that $\bigcup_{\eta'\lhd\eta}f(\eta')\lhd g(\xi_\eta)$ and $(d\circ g)(\xi_\eta)=\lambda_1$. For each $i<\omega$, let $f(\eta^\frown\la i\ra):=g(\xi_\eta^\frown\la i\ra)$.
  Then $f$ is an $\CL_{kr^*}$-embedding from $\CI^*$ to $\delta^{<\delta}_\text{succ}$ such that $c(f(\eta))=\lambda_1$ for all $\eta\in\CI^*$. This proves the monochromatic extendability of $\CI^*$.

 Let 
$X:=\{\la0\ra^\alpha:\alpha<\kappa\text{ and }\alpha\text{ is a successor ordinal}\}
$ 
and $Y:=\{\la i\ra:i<\omega\}$. Then $(\CL_{kr^*}\!,\CI^*\!,X,Y)$ satisfies the $n$-variable theorem and
$D_\text{simple}=D^{\CL_{kr^*}\!,\CI^*}_{X,Y}$.
\end{proof}
\end{proposition}

\begin{proposition}\label{prop: modeling property of kr-indiscernibility}
The $\CL_{kr}$-structure on $\Tau_\omega$ has the modeling property.
\begin{proof}
For each $n,m<\omega$, let \[\Tau_n^m:=\{\eta\in\Tau_\omega:\text{ran}(\eta)\subseteq m, l(\eta)\le n, \eta(i)=0 \text{ for all }i\ge n\}.\]
We can regard $\Tau_n^m$ as an $\CL_{kr}$-substructure of $\Tau_\omega$. By compactness and \Cref{prop: modeling property of krstar-indiscernibility},
 it is enough to show that for each $n,m<\omega$, there exists an $\CL_{kr}$-embedding $f$ from $\Tau_n^m$ to $\Tau_\omega$ such that $f(\bar{\eta})\sim_{kr^*} f(\bar{\nu})$ for any $\bar{\eta}\sim_{kr}\bar{\nu}$. Fix $0<n,m<\omega$. Let $g$ be an $\{\unlhd,\{\lhd_i\}_{i<m},\wedge\}$-embedding from $\Tau^m_n$ to $\Tau_\omega$ such that $\eta\lex\nu$ if and only if $g(\eta)<_l g(\nu)$. We can find such an embedding by using the arguments in  \cite[Lemma 3.1]{KK11} and \cite[Theorem 16]{TT12}. Let $f$ be a map from $\Tau_n^m$ to $\Tau_\omega$ such that 
 \[
f(\eta)=
\begin{cases}
g(\eta^-)^\frown \la t(\eta)\ra & \text{ if } \eta\neq\la 0\ra^n\\
g(\la 0\ra^n) & \text{ if } \eta=\la 0\ra^n.
\end{cases}
\]
Then $f$ satisfies all conditions we want. 
\end{proof}
\end{proposition}

\begin{proposition}\label{prop: NBTP is D_n-var}
$D_\text{NBTP}\in\mathfrak{D}_{n\text{-var}}$.
\begin{proof}
Let $\CI^*$ be the $\CL_{kr}$-structure on $\omega^{<\omega}\setminus\{\emptyset\}$. Then $\age_{kr}(\Tau_\omega)=\age_{kr}(\CI^*)$ and hence $\CI^*$ has the modeling property by \Cref{fact: age(I)=age(J)} and \Cref{prop: modeling property of kr-indiscernibility}. We show that $\CI^*$ is monochromatically extendable. Fix a cardinal $\lambda$ and choose any $\kappa>\max\{\lambda,\omega\}$ with $\cf(\kappa)=\kappa$. Let $\CJ^*$ be the $\CL_{kr}$-structure on $\kappa^{<\kappa}_\text{succ}:=\{\eta\in\kappa^{<\kappa}:l(\eta)\text{ is a successor}\}$. Let $c$ be a $\lambda$-coloring on $\CJ^*$. For each $\eta\in\kappa^{<\kappa}$, let $c_0(\eta)$ be an ordinal $\alpha<\lambda$ such that $|\{i<\kappa:c(\eta^\frown\la i\ra)=\alpha\}|=\kappa$. Then $c_0$ is a $\lambda$-coloring on $\kappa^{<\kappa}$. By the same argument in \cite[Remark 3.22 and Corollary 3.23(a)]{AKL21}, there exist $\alpha_0<\lambda$ and an $\eta_0\in\kappa^{<\kappa}$ such that for all $\eta\rhd\eta_0$, there exists $\nu\in\kappa^{<\kappa}$ with $\eta\lhd\nu$ such that $c_0(\nu)=\alpha_0$. 

We recursively construct  $f_0:\omega^{<\omega}\to\kappa^{<\kappa}$ and $f:\CI^*\to\kappa^{<\kappa}_\text{succ}$ as follows. Choose any $f_0(\emptyset)\in\kappa^{<\kappa}$ such that $\eta_0\lhd f_0(\emptyset)$ and $c_0(f_0(\emptyset))=\alpha_0$. Choose an increasing $(m_i)_{i<\omega}\subseteq \kappa$ such that $c(f_0(\emptyset)^\frown\la m_i\ra)=\alpha_0$ and for each $i<\omega$, let $f(\la i\ra)=f_0(\emptyset)^\frown\la m_i\ra$.
For each $\eta\in\omega^{<\omega}$, suppose that we have defined $f_0(\eta)$ and $f(\eta^\frown\la i\ra)$ for all $i<\omega$. For each $i<\omega$, choose any $f_0(\eta^\frown\la i\ra)\in\kappa^{<\kappa}$ such that $f(\eta^\frown\la i\ra)\lhd f_0(\eta^\frown\la i\ra)$ and $c_0(f_0(\eta^\frown\la i\ra))=\alpha_0$. For each $i<\omega$, choose any increasing $(m_{i,j})_{j<\omega}\subseteq\kappa$ such that $c(f_0(\eta^\frown\la i\ra)^\frown\la m_{i,j}\ra)=\alpha_0$ for each $j<\omega$. Let $f(\eta^\frown\la i\ra^\frown\la j\ra)=f_0(\eta^\frown\la i\ra)^\frown\la m_{i,j}\ra$ for each $i,j<\omega$. Then the resulting map $f:\CI^*\to\kappa^{<\kappa}_\text{succ}$ is an $\CL_{kr}$-embedding and $c(f(\eta))=\alpha_0$ for all $\eta\in\CI^*$. Since $\lambda$ is arbitrary, $\CI^*$ is monochromatically extendable.

Let $X:=\{\langle 0\rangle^n{}^\frown\langle1\rangle:n<\omega\}$, $Y_0:=\{\langle1\rangle^n{}^\frown\langle0\rangle:n<\omega\}$, and $Y_1:=\{\langle i\rangle:i<\omega\}$.
Then $(\CL_{kr},\CI^*\!,X,(Y_0,Y_1))$ satisfies the $n$-variable theorem and $D_\text{NBTP}=D^{\CL_{kr},\CI^*}_{X,(Y_0,Y_1)}$.
\end{proof}
\end{proposition}

The definition of $\mathfrak{D}_{n\text{-var}}$ is preserved under interpretations in the following sense. This also gives the corresponding result for $T^{eq}$.

\begin{proposition}\label{prop: bi-interpretable}
If $T'$ is interpretable in $T$ and $T\in D$ for some $D\in\mathfrak{D}_{n\text{-var}}$, then $T'\in D$. In particular,  $T\in D$ if and only if $T^{eq}\in D$.
\begin{proof}
Suppose $T'\notin D$. Let $D=D^{\CL^*\!\!\!,\,\CI^*}_{X,\bar{Y}}$, and choose $\varphi(x,y)$ and an $\CI^*_{\CL^*}$-indiscernible $(a_\eta)_{\eta\in\CI^*}$ witnessing $T'\notin D$. Let $\psi(x',y')$ be the formula in the language of $T$ obtained from $\varphi(x,y)$ by the interpretation, and choose representatives $(b_\eta)_{\eta\in\CI^*}$ of $(a_\eta)_{\eta\in\CI^*}$. By the modeling property, there exists an $\CI^*_{\CL^*}$-indiscernible $(c_\eta)_{\eta\in\CI^*}$ which is $\CL^*$-locally based on $(b_\eta)_{\eta\in\CI^*}$. Then $\{\psi(x',c_\eta)\}_{\eta\in X}$ is consistent and $\{\psi(x',c_\eta)\}_{\eta\in Y}$ is inconsistent for each $Y\in\bar{Y}$. Thus $\psi(x',y')$ witnesses $T\notin D$ with $(c_\eta)_{\eta\in\CI^*}$.
\end{proof}
\end{proposition} 

\begin{remark}
The ordered field of real numbers $(\mathbb{R},<,+,\cdot)$ and the rotation group $(\mathrm{SO}(3,\mathbb{R}),\cdot)$ are bi-interpretable \cite{PPS00}, but only $(\mathbb{R},<,+,\cdot)$ is o-minimal. Thus, by \Cref{prop: bi-interpretable}, the class of complete o-minimal theories does not belong to $\mathfrak{D}_{n\text{-var}}$.
\end{remark}

\begin{remark}
It is natural to ask whether the previous proposition can be extended to hyperimaginaries. Recall that, although $M^{heq}$ is well-defined, obtaining an analogue of $T^{eq}$ for hyperimaginaries within first-order logic is more delicate. One possible way to deal with this is to work in positive logic. In this setting, to a collection $\mathcal E$ of type-definable equivalence relations, one can associate a positive theory $T^{\mathcal E}$ where the corresponding hyperimaginaries are represented as additional sorts \cite{DK22}. In particular, this yields a version of $T^{heq}$ in positive logic. Some results concerning the modeling property for certain index structures have been developed in positive logic \cite{Kam24}. However, a general framework for the modeling property for arbitrary index structures, of the form needed in our definition of $\mathfrak{D}_{n\text{-var}}$, does not seem to be available yet. It would be interesting to develop such a framework in positive logic and investigate whether the previous proposition admits an analogue for hyperimaginaries.
\end{remark}

\section{$n$-variable theorems}\label{sec: n-var thm}

\begin{notation}
By a {\it single variable}, we mean a variable $x$ with $|x|=1$.
For an $n$-tuple of single variables $\bar{x}:=(x_0,...,x_{n-1})$, an $n$-tuple of parameters $\bar{a}:=(a_0,...,a_{n-1})$, a single variable $y$, and $i<n$,
let 
\begin{itemize}
\item[(i)] $\check{x}^i:=(x_0,...,x_{i-1},x_{i+1},...,x_{n-1})$,
\item[(ii)] $\bar{x}^{i/y}:=(x_0,...,x_{i-1},y,x_{i+1},...,x_{n-1})$,
\item[(iii)] $\bar{a}^{i/y}:=(a_0,...,a_{i-1},y,a_{i+1},...,a_{n-1})$.
\end{itemize}
\end{notation}

\begin{lemma}\label{lem: 1-var lemma}
Let $D^{\CL^*\!\!\!,\,\CI^*}_{X,\bar{Y}}\!\!\in\mathfrak{D}_{n\text{-var}}$. Let $\CL$ be a language, $T$ an $\CL$-theory, and $\mathbb{M}\models T$ a sufficiently saturated model. If  there exist $\varphi(x,y)\in\CL$ and $\CI^*_{\CL^*}$-indiscernible $(a_\eta)_{\eta\in \CI^*}\!$ in $\mathbb{M}$ such that
\begin{itemize}
\item[(i)] $\varphi(x,y)$ witnesses $\neg D^{\CL^*\!\!\!,\,\CI^*}_{X,\bar{Y}}$ with $(a_\eta)_{\eta\in \CI^*}$,
\item[(ii)] $|x|=1$,
\end{itemize}
then there exists $b\models\{\varphi(x,a_\eta)\}_{\eta\in X}$ such that $\dim(b/(a_\eta)_{\eta\in\CI^*}\!)=1$.
\begin{proof}
It is enough to show that $\{\varphi(x,a_\eta)\}_{\eta\in X}$ has infinitely many realizations. Suppose not. Then there exists a finite subset $X_0$ of $X$ such that 
\[
\bigcap_{\eta\in X_0}\varphi(\mathbb{M},a_\eta)=\bigcap_{\eta\in X}\varphi(\mathbb{M},a_\eta).
\]
Let $W$, $X'$ and $Y'$ be subsets of $\CI^*$ satisfying the condition (vii) in \Cref{def: n-var quadruple}. There exists $X'_0\subseteq X'$ such that $X'_0\sim_{\CL^*}X_0$. Since
\[
\bigcap_{\eta\in X'_0}\varphi(\mathbb{M},a_\eta)=\bigg(\bigcap_{\eta\in X'_0}\varphi(\mathbb{M},a_\eta)\bigg)\cap\varphi(\mathbb{M},a_{\nu})
\]
for each $\nu\in Y'$, $\{\varphi(x,a_\nu)\}_{\nu\in Y'}$ is consistent. This yields a contradiction.
\end{proof}
\end{lemma}

\begin{theorem}\label{thm: n-var theorem}\!{\rm($n$-variable theorem)}
Let $D^{\CL^*\!\!\!,\,\CI^*}_{X,\bar{Y}}\!\!\in\mathfrak{D}_{n\text{-var}}$. Let $\CL$ be a language, $T$ a complete $\CL$-theory, and $\mathbb{M}\models T$ a sufficiently saturated model. If $T\notin D^{\CL^*\!\!\!,\,\CI^*}_{X,\bar{Y}}$, then there exist $\varphi(\bar{x},y)\in\CL$ and $\CI^*_{\CL^*}$-indiscernible $(a_\eta)_{\eta\in \CI^*}\!$ in $\mathbb{M}$ such that
\begin{itemize}
\item[(i)] $\varphi(\bar{x},y)$ witnesses $\neg D^{\CL^*\!\!\!,\,\CI^*}_{X,\bar{Y}}$ with $(a_\eta)_{\eta\in \CI^*}$,
\item[(ii)] there exists $\bar{b}\models\{\varphi(\bar{x},a_\eta)\}_{\eta\in X}$ such that $\dim(\bar{b}/(a_\eta)_{\eta\in\CI^*}\!)=|\bar{x}|$.
\end{itemize}
\begin{proof}
Suppose $T\notin D^{\CL^*\!\!\!,\,\CI^*}_{X,\bar{Y}}$. Then there exist $\varphi(\bar{x},y)\in\CL$ and $(a_\eta)_{\eta\in\CI^*}\!\!\subseteq\mathbb{M}$ such that
\begin{itemize}
\item[(iii)] $(a_\eta)_{\eta\in\CI^*}$ is $\CI^*_{\CL^*}$-indiscernible,
\item[(iv)] $\{\varphi(\bar{x},a_\eta)\}_{\eta\in X}$ is consistent,
\item[(v)] $\{\varphi(\bar{x},a_\nu)\}_{\nu\in Y}$ is inconsistent for each $Y\in\bar{Y}$.
\end{itemize}
If there exists $\bar{b}\models\{\varphi(\bar{x},a_\eta)\}_{\eta\in X}$ such that $\dim(\bar{b}/(a_\eta)_{\eta\in\CI^*}\!)=|\bar{x}|$, then there is nothing to prove. Suppose not and let $n:=|\bar{x}|$, $\bar{x}:=(x_0,...,x_{n-1})$. Then $n>1$ by \Cref{lem: 1-var lemma}. Choose any $\bar{b}\models\{\varphi(\bar{x},a_\eta)\}_{\eta\in X}$. Then $\bar{b}$ can be written of the form $(b_0,...,b_{n-1})$ where $|b_i|=1$ for each $i<n$. Since $\dim(\bar{b}/(a_\eta)_{\eta\in\CI^*}\!)<|\bar{x}|$, there exist $i<n$, $\theta\in\CL$, $(\nu_0,...,\nu_{l-1})\in \CI^*$, and $k<\omega$ such that 
\begin{itemize}
\item[(vi)] $\models \theta(b_i,\check{b}^i,a_{\nu_0},...,a_{\nu_{l-1}})\wedge \exists^{\le k}x\theta(x,\check{b}^i,a_{\nu_0},...,a_{\nu_{l-1}})$.
\end{itemize}
Let 
\begin{itemize}
\item[(vii)] $\varphi'(\bar{x},y,y''_0,...,y''_{l-1}):=\varphi(\bar{x},y)\wedge\theta(x_i,\check{x}^i,y''_0,...,y''_{l-1})$.
\end{itemize}
Then $\{\varphi'(b_0,...,b_{i-1},x,b_{i+1},...,b_{n-1},a_\xi,a_{\nu_0},...,a_{\nu_{l-1}})\}_{\xi\in X}$ has only finitely many realizations. Thus there exists a finite subset $X_0$ of $X$ such that
\[
\bigcap_{\xi\in X_0}\varphi'(b_0,...,b_{i-1},\mathbb{M},b_{i+1},...,b_{n-1},a_\xi,a_{\nu_0},...,a_{\nu_{l-1}})
\]\vspace{-8pt}
\[\;\;\;\;\;=\bigcap_{\xi\in X}\varphi'(b_0,...,b_{i-1},\mathbb{M},b_{i+1},...,b_{n-1},a_\xi,a_{\nu_0},...,a_{\nu_{l-1}}).
\]
Let $\bar{\nu}:=(\nu_0,...,\nu_{l-1})$ and $\bar{a}_{\bar{\nu}}:=(a_{\nu_0},...,a_{\nu_{l-1}})$. Note that 
\begin{itemize}
\item[(viii)]
\makebox[\linewidth][c]{$\displaystyle
\exists x\!\! \bigwedge_{\xi\in X_0}\!\!\! \varphi'(\bar{b}^{i/x}\!,a_\xi,\bar{a}_{\bar{\nu}})\wedge\forall x\bigg(\Big(\bigwedge_{\xi\in X_0}\!\! \varphi'(\bar{b}^{i/x}\!,a_\xi,\bar{a}_{\bar{\nu}})\wedge\varphi'(\bar{b}^{i/x}\!,a_{\xi'},\bar{a}_{\bar{\nu}})\Big)\!\leftrightarrow\!\! \bigwedge_{\xi\in X_0}\!\! \varphi'(\bar{b}^{i/x}\!,a_\xi,\bar{a}_{\bar{\nu}}) \!\bigg)
$}
\end{itemize}
for all $\xi'\in X\setminus X_0$. Let $\{\eta_0,...,\eta_{m-1}\}$ be an enumeration of $X_0$, $\bar{\eta}:=(\eta_0,...,\eta_{m-1})$, and $\bar{a}_{\bar{\eta}}:=(a_{\eta_0},...,a_{\eta_{m-1}})$. Let $\bar{y}':=(y'_0,...,y'_{m-1})$ and $\bar{y}'':=(y''_0,...,y''_{l-1})$ be tuples of variables. Put
\begin{itemize}
\item[(ix)] \makebox[\linewidth][c]{$\displaystyle
\begin{aligned}[t]
&\varphi''(\check{x}^i,y,\bar{y}',\bar{y}'')
\\
&\quad :=\exists x\!\bigwedge_{j<m}\!
\varphi'(\bar{x}^{i/x}\!,y'_j,\bar{y}'')
\wedge\forall x\bigg(
\Big(
\bigwedge_{j<m}\!\varphi'(\bar{x}^{i/x}\!,y'_j,\bar{y}'')
\wedge\varphi'(\bar{x}^{i/x}\!,y,\bar{y}'')
\Big)
\leftrightarrow
\bigwedge_{j<m}\!\varphi'(\bar{x}^{i/x}\!,y'_j,\bar{y}'')
\bigg).
\end{aligned}
$}
\end{itemize}
Then by (viii), we have 
\begin{itemize}
\item[(x)] $\models\varphi''(\check{b}^i,a_{\xi'},\bar{a}_{\bar{\eta}},\bar{a}_{\bar{\nu}})$ for each $\xi'\in X\setminus X_0$. 
\end{itemize}
\smallskip
Since $D^{\CL^*\!\!\!,\,\CI^*}_{X,\bar{Y}}\!\!\in\mathfrak{D}_{n\text{-var}}$, there exists an $\CL^*$-embedding $f:\CI^*\to\CI^*$ such that 
\begin{itemize}
\item[(xi)] $f(\CI^*)\cap \bar{\eta}\bar{\nu}=\emptyset$,
\item[(xii)] $\bar{\xi}\sim_{\CL^*}\!\bar{\xi}'\;\Rightarrow\; \bar{\xi}\bar{\eta}\bar{\nu}\sim_{\CL^*}\!\bar{\xi}'\bar{\eta}\bar{\nu}$ for all $\bar{\xi},\bar{\xi}'\in f(\CI^*)$,
\item[(xiii)] $(f(\CI^*)\cap X)\sim_{\CL^*}^\text{fin}X$.
\end{itemize} 
Let $a'_\xi:=a_{f(\xi)}\bar{a}_{\bar{\eta}}\bar{a}_{\bar{\nu}}$ for each $\xi\in\CI^*$.
Then $(a'_\xi)_{\xi\in\CI^*}$ is $\CI^*_{\CL^*}\!$-indiscernible by (xii). 

We prove that $\varphi''(\check{x}^i;y,\bar{y}',\bar{y}'')$ witnesses $\neg D^{\CL^*\!\!\!,\,\CI^*}_{X,\bar{Y}}$ with $(a'_\xi)_{\xi\in\CI^*}$. First we show that $\{\varphi''(\check{x}^i;a'_\xi)\}_{\xi\in X}$ is consistent. Choose any finite subset $X_1$ of $X$. By compactness, it is enough to show that $\{\varphi''(\check{x}^i;a'_\xi)\}_{\xi\in X_1}$ is consistent. By (xi), (xii), and (xiii), there exists $X'_1\subset X\setminus X_0$ such that $X'_1\bar{\eta}\bar{\nu}\sim_{\CL^*}\!\!f(X_1)\bar{\eta}\bar{\nu}$. Thus it is enough to show that $\{\varphi''(\check{x}^i;a_\xi,\bar{a}_{\bar{\eta}},\bar{a}_{\bar{\nu}})\}_{\xi\in X'_1}$ is consistent. But it is clear by (x).

Now we show that $\{\varphi''(\check{x}^i;a'_\xi)\}_{\xi\in Y}$ is inconsistent for each $Y\in\bar{Y}$. Suppose $\{\varphi''(\check{x}^i;a'_\xi)\}_{\xi\in Y}$ is realized by $\check{c}^i:=(c_0,...,c_{i-1},c_{i+1},...,c_{n-1})$. Let $\bar{c}^{i/x}:=(c_0,...,c_{i-1},x,c_{i+1},...,c_{n-1})$ for a variable $x$.
Then for each $\xi\in Y$, we have
\begin{itemize}
\item[(xiv)]\makebox[\linewidth][c]{$\displaystyle
\exists x\!\! \bigwedge_{j<m}\!\! \varphi'(\bar{c}^{i/x}\!,a_{\eta_j},\bar{a}_{\bar{\nu}})\wedge\forall x \bigg(\Big( \bigwedge_{j<m}\!\varphi'(\bar{c}^{i/x}\!,a_{\eta_j},\bar{a}_{\bar{\nu}})\wedge \varphi'(\bar{c}^{i/x}\!,a_{f(\xi)},\bar{a}_{\bar{\nu}})\! \Big)\!\leftrightarrow\! \bigwedge_{j<m}\!\varphi'(\bar{c}^{i/x}\!,a_{\eta_j},\bar{a}_{\bar{\nu}})\bigg).
$}
\end{itemize}
Thus there exists $c_i$ such that $\models\varphi'(\bar{c}^{i/c_i}\!,a_{f(\xi)},\bar{a}_{\bar{\nu}})$ for each $\xi\in Y$, where $\bar{c}^{i/c_i}=(c_0,...,c_{n-1})$. By the choice of $\varphi'$, $\{\varphi(\bar{x},a_\xi)\}_{\xi\in f(Y)}$ is consistent. It is a contradiction since $f(Y)\sim_{\CL^*}Y$. Thus $\varphi''$ witnesses $\neg D^{\CL^*\!\!\!,\,\CI^*}_{X,\bar{Y}}$ with $(a'_\xi)_{\xi\in\CI^*}$.

Note that the length of the free variable part of the new witness of $\neg D^{\CL^*\!\!\!,\,\CI^*}_{X,\bar{Y}}$ is $n-1$. So we can repeat this process until we get a witness of $\neg D^{\CL^*\!\!\!,\,\CI^*}_{X,\bar{Y}}$ satisfying (ii), by \Cref{lem: 1-var lemma}.
\end{proof}
\end{theorem}

\subsection{Comparing with one-variable theorems}

When a dividing line is defined by the existence of a formula
$\varphi(x,y)$ satisfying certain conditions, we say that it satisfies
a {\it one-variable theorem} if such a formula can always be chosen
with $|x|=1$. One-variable theorems are known for various dividing lines.

\begin{fact}
Let $T$ be a complete theory.
\begin{itemize}
\item[(i)] \cite{She90}\label{fact: one-var TP}
If $T$ has TP, then there exists a formula $\varphi(x,y)$ with $|x|=1$ that witnesses TP.
\item[(ii)] \cite{CR16}\label{fact: one-var TP1}
If $T$ has TP$_1$, then there exists a formula $\varphi(x,y)$
with $|x|=1$ that witnesses TP$_1$.
\item[(iii)] \cite{Che13}\label{fact: one-var TP2}
If $T$ has TP$_2$, then there exists a formula $\varphi(x,y)$
with $|x|=1$ that witnesses TP$_2$.
\item[(iv)] \cite{Ram19}\label{fact: one-var SOP1}
If $T$ has SOP$_1$, then there exists a formula $\varphi(x,y)$
with $|x|=1$ that witnesses SOP$_1$.
\item[(v)] \cite{AKL21}\label{fact: one-var ATP}
If $T$ has ATP, then there exists a formula $\varphi(x,y)$
with $|x|=1$ that witnesses ATP.
\end{itemize}
\end{fact}

The one-variable theorems for OP and IP \cite{She90}, together
with the characterizations in
\Cref{prop: stable is D_n-var} and \ref{prop: NIP is D_n-var},
yield the following. 

\begin{fact}\cite[combined with
\Cref{prop: stable is D_n-var} and \ref{prop: NIP is D_n-var}]{She90}
Let $T$ be a complete theory.
\begin{itemize}
\item[(vi)]
If $T$ is unstable, then there exists a formula $\varphi(x,y)$
with $|x|=1$ that witnesses $\neg D^{\CL^*\!\!,\,\CI^*}_{X,Y}$,
where $(\CL^*\!,\CI^*\!,X,Y)$ is the configuration given in
the proof of \Cref{prop: stable is D_n-var}.
\item[(vii)]
If $T$ has IP, then there exists a formula $\varphi(x,y)$
with $|x|=1$ that witnesses
$\neg D^{\CL^*\!,\,\omega\times\omega}_{X,Y}$,
where $(\CL^*\!,\omega\times\omega,X,Y)$ is the configuration
given in the proof of \Cref{prop: NIP is D_n-var}.
\end{itemize}
\end{fact}

For configurations of the form introduced in \Cref{sec: D-n-var},
a one-variable theorem provides a witnessing formula $\varphi(x,y)$ with $|x|=1$ 
such that the set of formulas required to be consistent has a
realization $b$ which is not algebraic over the entire parameter
set $A$, by \Cref{lem: 1-var lemma}.
Thus the algebraic dimension of this realization equals the
number of free variables ($|x|=\dim(b/A)=1$).

A quadruple satisfying the $n$-variable theorem also provides
a witnessing formula $\varphi(\bar{x},y)$ such that the set of
formulas required to be consistent has a realization $\bar b$
whose algebraic dimension over the entire parameter set $A$
equals the number of free variables ($|\bar{x}|=\dim(\bar b/A)$).
The difference is that this number need not be one.
In this sense, the $n$-variable theorem is a weaker form of
a one-variable theorem.

\begin{question}
Garc\'ia and Mennuni \cite[Question~4.2]{GM22} ask whether
one-variable theorems for dividing lines defined via posets
can be proved in a uniform way.
In view of the discussion above, \Cref{def: n-var quadruple}
partially addresses their question by giving conditions
where the $n$-variable theorem holds.
It would be interesting to see whether this definition can be
strengthened further so that the corresponding dividing lines
satisfy a one-variable theorem.
\end{question}

\subsection{Customizing $\mathfrak{D}_{n\text{-var}}$}

Note that the proof of \Cref{thm: n-var theorem} uses only conditions (iii) and (vii) in \Cref{def: n-var quadruple}. The preservation results in the following sections also use different parts of this definition. We can adjust the conditions depending on the preservation problem or the dividing lines we want to cover. We may strengthen the conditions to use more powerful techniques in solving preservation problems, or weaken them to include more dividing lines.

For example, allowing $\bar{Y}$ to be countably infinite enables us to include the class of NPM theories \cite{Bai24} in the framework, while retaining the $n$-variable theorem. However, the arguments in \Cref{sec: T gt}, \ref{sec: H-structure lovely pair}, \ref{sec: T^G}, and \ref{sec: generic derivation fields} do not apply directly, since they rely on the compactness in \Cref{prop: n-var theorem exchange}, which uses the finiteness of $\bar{Y}$.

\begin{table}[H]
\centering
\small
\renewcommand{\arraystretch}{1.3}
\setlength{\tabcolsep}{4pt}
\begin{tabular}{@{}p{0.39\linewidth}cccl@{}}
\hline
\noalign{\vskip 2.5pt}
Result
& \shortstack{Uses the\\$n$-variable theorem}
& \shortstack{Modeling\\property}
& $|\bar{Y}|<\omega$
& \shortstack[l]{Other\\conditions} \\
\noalign{\vskip 0pt}
\hline
\noalign{\vskip 0pt}
$n$-variable theorem
& N/A & No & No & (iii), (vii) \\

Preservation under generic trivialization
& Yes & Yes & Yes & (ii), (iii), (iv), (v), (vii) \\

Preservation under $H$-structure and lovely pair expansions
& Yes & Yes & Yes & (ii), (iii), (vii) \\

\raggedright Preservation under generic subspace expansions
& Yes & Yes & Yes & (ii), (iii), (vii) \\

Preservation for $\mathrm{ACF}_T$
& Yes & Yes & No & (ii), (iii), (vii) \\

\raggedright Preservation under generic non-commutative derivations
& No & Yes & Yes & (iii) \\
\noalign{\vskip 0.5pt}
\hline
\end{tabular}
\caption{Conditions in \Cref{def: n-var quadruple} used in the proofs.}
\label{tab: conditions used}
\end{table}

\section{Generic trivializations}\label{sec: T gt}

From \Cref{sec: T gt} to \Cref{sec: generic derivation fields}, we show that dividing lines in
$\mathfrak{D}_{n\text{-var}}$ are preserved under various
model-theoretic constructions. We begin with generic trivializations, showing that
a geometric theory and its generic trivialization lie on
the same side of each dividing line in
$\mathfrak{D}_{n\text{-var}}$.

\begin{proposition}\label{prop: n-var theorem exchange}
Let $D^{\CL^*\!\!\!,\,\CI^*}_{X,\bar{Y}}\!\!\in\mathfrak{D}_{n\text{-var}}$. Let $\CL$ be a language, $T$ a complete $\CL$-theory, and $\mathbb{M}\models T$ a sufficiently saturated model. Suppose that $\acl$ in $T$ satisfies the exchange property. Then the following are equivalent.
\begin{itemize}
\item[(i)] $T\notin D^{\CL^*\!\!\!,\,\CI^*}_{X,\bar{Y}}\!$.
\item[(ii)] there exist $\varphi(x,y)\in \CL$ and $\CI^*_{\CL^*}\!$-indiscernible $(a_\eta)_{\eta\in\CI^*}$ such that 
\begin{itemize}
\item[$\ast$] $\dim(\{\varphi(x,a_\eta)\}_{\eta\in X})=|x|$,
\item[$\ast$] $\dim(\{\varphi(x,a_\nu)\}_{\nu\in Y})=0$ for all $Y\in\bar{Y}$. 
\end{itemize}
\item[(iii)] there exist $\varphi(x,y)\in \CL$ and $\CI^*_{\CL^*}\!$-indiscernible $(a_\eta)_{\eta\in\CI^*}$ such that 
\begin{itemize}
\item[$\ast$] $\dim(\{\varphi(x,a_\eta)\}_{\eta\in X})=|x|$,
\item[$\ast$] $\dim(\{\varphi(x,a_\nu)\}_{\nu\in Y})<|x|$ for all $Y\in\bar{Y}$. 
\end{itemize}
\end{itemize}
\begin{proof}
The implication from (i) to (ii) follows from \Cref{thm: n-var theorem}. From (ii) to (iii) is trivial.  Suppose (iii). Let $n:=|x|$. Then $\varphi$ can be written of the form $\varphi(x_0,...,x_{n-1},y)$ where $|x_i|=1$ for all $i<n$. Consider 
\[
\Phi((x_\xi)_{\xi\in\omega^{\le n}},y):=\{\varphi(x_{\xi_1},...,x_{\xi_{n}},y):\emptyset\lhd\xi_1\lhd\cdots\lhd\xi_n\}\cup\{x_\eta\neq x_\nu\}_{\eta\neq\nu}.
\] 
Then $\bigcup_{\eta\in X}\Phi( (x_\xi)_{\xi\in\omega^{\le n}}, a_\eta)$ is consistent. Fix $Y\in\bar{Y}$ and suppose $\bigcup_{\nu\in Y}\Phi( (x_\xi)_{\xi\in\omega^{\le n}}, a_\nu)$ is realized by $(b_\xi)_{\xi\in\omega^{\le n}}$. By applying the s-modeling property, we may assume that $(b_\xi)_{\xi\in\omega^{\le n}}$ is s-indiscernible over $(a_\nu)_{\nu\in Y}$. Then for any $\emptyset\lhd\xi_1\lhd\cdots\lhd\xi_n$, $b_{\xi_i}\notin\acl((a_\nu)_{\nu\in Y} b_{\xi_{<i}})$ for all $0<i\le n$. By the exchange property, we have $\dim(b_{\xi_1}...b_{\xi_n}/(a_\nu)_{\nu\in Y})=n=|x|$, which yields a contradiction. Thus $\bigcup_{\nu\in Y}\Phi( (x_\xi)_{\xi\in\omega^{\le n}}, a_\nu)$ is inconsistent for all $Y\in \bar{Y}$. Since $\bar{Y}$ is finite, compactness gives a finite conjunction of formulas in $\Phi$ witnessing
$\neg D^{\CL^*\!\!\!,\,\CI^*}_{X,\bar{Y}}$
with $(a_\eta)_{\eta\in\CI^*}$. Thus $T\notin D^{\CL^*\!\!\!,\,\CI^*}_{X,\bar{Y}}\!$.
\end{proof}
\end{proposition}

\begin{theorem}\label{thm: T^gt preservation}
Let $\CL$ be a language, $T$ be a geometric $\CL$-theory, and $D^{\CL^*\!\!\!,\,\CI^*}_{X,\bar{Y}}\!\!\in\mathfrak{D}_{n\text{-var}}$. Then $T\in D^{\CL^*\!\!\!,\,\CI^*}_{X,\bar{Y}}$ if and only if $T^{gt}\in\! D^{\CL^*\!\!\!,\,\CI^*}_{X,\bar{Y}}$.
\begin{proof}
Suppose $T^{gt}\!\notin \!D^{\CL^*\!\!\!,\,\CI^*}_{X,\bar{Y}}\!$. Choose any sufficiently saturated model $(\mathbb{M},H(\mathbb{M}))\models T^{ind}\!$. By the construction of $T^{gt}$ and \Cref{thm: n-var theorem}, there exist $\varphi(x,y)\in \CL$ and $\CI^*_{\CL^*}\!$-indiscernible $(a_\eta)_{\eta\in\CI^*}\!\subseteq H(\mathbb{M})$ such that
\begin{itemize}
\item[(i)] $\dim_\CL(\{\varphi(x,a_\eta)\wedge H(x)\}_{\eta\in X})=|x|$,
\item[(ii)] $\{\varphi(x,a_\nu)\wedge H(x)\}_{\nu\in Y}$ is inconsistent for all $Y\in\bar{Y}$.
\end{itemize} 
Thus $\dim_\CL(\{\varphi(x,a_\eta)\}_{\eta\in X})=|x|$ modulo $T$. We show that $\dim_\CL(\{\varphi(x,a_\nu)\}_{\nu\in Y})<|x|$ modulo $T$ for each $Y\in\bar{Y}$. If $\dim_\CL(\{\varphi(x,a_\nu)\}_{\nu\in Y})=|x|$ for some $Y\in\bar{Y}$, then by the density property and compactness, there exists $d$ such that $d\models \{\varphi(x,a_\nu)\wedge H(x)\}_{\nu\in Y}$. It is a contradiction and hence $\dim_\CL(\{\varphi(x,a_\nu)\}_{\nu\in Y})<|x|$ for all $Y\in\bar{Y}$. Thus $T\notin D^{\CL^*\!\!\!,\,\CI^*}_{X,\bar{Y}}\!$ by \Cref{prop: n-var theorem exchange}.

\medskip

Now suppose $T\notin D^{\CL^*\!\!\!,\,\CI^*}_{X,\bar{Y}}\!$ and let $(\mathbb{M},H(\mathbb{M}))\models T^{ind}\!$ be a sufficiently saturated model. Then by \Cref{thm: n-var theorem}, there exist $\varphi(x,y)\in\CL$ and $\CL^*$-indiscernible $(a_\eta)_{\eta\in\CI^*}\!\subseteq\mathbb{M}$ in $\CL$ such that
\begin{itemize}
\item[(iii)] $\dim_\CL(\{\varphi(x,a_\eta)\}_{\eta\in X})=|x|$,
\item[(iv)] $\{\varphi(x,a_\nu)\}_{\nu\in Y}$ is inconsistent for each $Y\in\bar{Y}$.
\end{itemize}

\smallskip

Since $\CI^*$ is monochromatically extendable over $\mathbbof{X}:=\{\eta\in\CI^*:\la\eta\ra\sim_{\CL^*}\la\nu\ra\text{ for some }\nu\in X\}$, there exists  $\CI^\dagger$ with $\age_{\CL^*}(\CI^\dagger)=\age_{\CL^*}(\CI^*)$ such that for any coloring $c:\CI^\dagger\to |T|$, there exists an $\CL^*$-embedding $f:\CI^*\to\CI^\dagger$ such that $c(f(\eta))=c(f(\nu))$ for $\eta,\nu\in\mathbbof{X}$ with $\la\eta\ra\sim_{\CL^*}\la\nu\ra$. Since $Y\subseteq\mathbbof{X}$ for each $Y\in\bar Y$, every type of the form $\{\varphi(x,a_\eta)\}_{\eta\in Z}$ required to be  consistent or inconsistent is a subset of $\{\varphi(x,a_\eta)\}_{\eta\in\mathbbof{X}}$.
So we may assume $\mathbbof{X}=\CI^*$. We can find $\CL^*$-indiscernible $(\dot a_\eta)_{\eta\in\CI^\dagger}$ which is locally based on $(a_\eta)_{\eta\in\CI^*}$. By applying the density property repeatedly (by pulling singletons $\dot a\in \bigcup_{\eta\in\CI^\dagger} \dot a_\eta$ into $\acl_\CL(H(\mathbb{M}))$ one-by-one), we may assume $(\dot a_\eta)_{\eta\in\CI^\dagger}\!\subseteq\acl_\CL(H(\mathbb{M}))$. Note that
\begin{itemize}
\item[(v)] $\dim_\CL(\{\varphi(x,\dot a_\eta)\}_{\eta\in \dot X})=|x|$ for any finite subset  $\dot X$ of $\CI^\dagger$ such that  $\dot X\sim X'$ for some $X'\subseteq X$,
\item[(vi)] for each $Y\in\bar Y$, there exists a finite subset $Y'$ of $Y$ such that $\{\varphi(x,\dot a_\nu)\}_{\nu\in \dot Y}$ is inconsistent for any $\dot Y\subseteq\CI^\dagger$ with $\dot Y\sim Y'$.
\end{itemize}

\smallskip

By weak monochromaticity, we may assume that there exists $m<\omega$ such that $m:=|\dot a_\eta|$ for all $\eta\in\CI^\dagger$. Then $\dot a_\eta$ is of the form $(\dot a^\eta_0,...,\dot a^\eta_{m-1})$ for each $\eta\in\CI^\dagger$. For each $\eta\in \CI^\dagger$, there exist $\dot\psi^{\eta}_{0}(y_0,z^{\eta}_{0}),...,\dot\psi^\eta_{m-1}(y_{m-1},z^{\eta}_{m-1})\in\CL$ and $\dot b^{\eta}_{0},...,\dot b^{\eta}_{m-1}\in H(\mathbb{M})$ such that $\models\dot\psi^{\eta}_{j}(\dot a^\eta_j,\dot b^{\eta}_{j})$ and $\dot\psi^\eta_j(y_j,\dot b^\eta_j)$ is algebraic for each $\eta\in\CI^\dagger$ and $j<m$. 

\smallskip

Since $\CI^\dagger$ is a monochromatic extension  in $|T|$, there exists an $\CL^*$-embedding $f:\CI^*\to\CI^\dagger$ such that 
\begin{itemize}
\item[(vii)] $\dot\psi^{f(\eta)}_j\!=\dot\psi^{f(\nu)}_j$ and $|\dot \psi^{f(\eta)}_j(\mathbb{M},\dot b^{f(\eta)}_j)|=|\dot \psi^{f(\nu)}_j(\mathbb{M},\dot b^{f(\nu)}_j)|$  for any $\la\eta\ra\sim\la\nu\ra$ and any $j<m$.
\end{itemize}
Let $z'$ be a variable with $|z'|=\max\{|z^{f(\eta)}_j|:\eta\in\CI^*,j<m\}$. Then $|z'|<\omega$ by the weak monochromaticity and (vii). By adding dummy variables and parameters, we may assume that $\dot\psi^{f(\eta)}_j$ is of the form $\dot \psi^{f(\eta)}_j(y_j,z')$  and $|\dot b^{f(\eta)}_j|=|z'|$ for each $\eta\in\CI^*$ and $j<m$. 
For each $\eta\in\CI^*$ and $j<m$, let $\ddot a_\eta:=(\ddot a^\eta_0,...,\ddot a^\eta_{m-1}):=\dot a_{f(\eta)}$, $\ddot b^\eta_j:=\dot b^{f(\eta)}_j$, and $\ddot\psi^\eta_j(y_j,z'):=\dot\psi^{f(\eta)}_j(y_j,z')$. 

\smallskip

Summarizing so far, we have
\begin{itemize}
\item[(viii)] $(\ddot a_\eta=(\ddot a^\eta_0,...,\ddot a^\eta_{m-1}))_{\eta\in\CI^*}$ is $\CL^*$-indiscernible in $\CL$,
\item[(ix)] $(\ddot b^\eta_j)^{\eta\in\CI^*}_{j<m}\subseteq H(\mathbb{M})$ and $|\ddot b^\eta_j|=|\ddot b^{\eta'}_{j'}|=|z'|$ for all $\eta,\eta'\in\CI^*$ and $j,j'<m$,
\item[(x)] $\models\ddot\psi^\eta_j(\ddot a^\eta_j,\ddot b^\eta_j)$ and $\ddot \psi^\eta_j(y_j,\ddot b^\eta_j)$ is algebraic for each $\eta\in\CI^*$ and $j<m$,
\item[(xi)] $\ddot \psi^\eta_j=\ddot\psi^\nu_j$ if $\la \eta\ra\sim\la\nu\ra$ for all $\eta,\nu\in\CI^*$ and $j<m$,
\item[(xii)] $\dim_\CL(\{\varphi(x,\ddot a_\eta)\}_{\eta\in X})=|x|$,
\item[(xiii)] $\{\varphi(x,\ddot a_\nu)\}_{\nu\in Y}$ is inconsistent for each $Y\in \bar{Y}$.
\end{itemize}
Let $\ddot b_\eta:=(\ddot b^\eta_0,...,\ddot b^\eta_{m-1})$ for each $\eta\in\CI^*$.
By (vii) and applying the modeling property, we may assume additionally that
\begin{itemize}
\item[(xiv)] $(\ddot a_\eta\ddot b_\eta)_{\eta\in\CI^*}$ is $\CL^*$-indiscernible in $\CL_H$.
\end{itemize}
By the weak monochromaticity, there exist $l<\omega$ and $\eta_0,...,\eta_{l-1}\in\CI^*$, such that 
\begin{itemize}
\item[(xv)] $\la\eta_{k}\ra\not\sim\la\eta_{k'}\ra$ for any $k<k'<l$, 
\item[(xvi)] for each $\eta\in\CI^*$, there exists $k<l$ such that $\la\eta_{k}\ra\sim\la\eta\ra$.
\end{itemize}
Let $z^0,...,z^{l-1}$ be variables with $|z^k|=|z'|$ for each $k<l$. Let $\ddot z:=(z^0,...,z^{l-1})$. For each $j<m$, let $\ddot \psi_j(y_j,\ddot z):=\bigvee_{k<l}\ddot\psi^{\eta_k}_j(y_j,z^k)$. By adding dummy parameters in $H(\mathbb{M})$, we may assume $|\ddot b^\eta_j|=|\ddot z|$ and $(\ddot b^\eta_j)_{j<m}^{\eta\in\CI^*}$ satisfies 
\begin{itemize}
\item[(xvii)] $\models \ddot\psi_j(\ddot a^\eta_j,\ddot b^\eta_j)$ and $\ddot\psi_j(y_j,\ddot b^\eta_j)$ is algebraic for each $\eta$ and $j$.
\end{itemize}
We may still assume $(\ddot a_\eta \ddot b_\eta)_{\eta\in\CI^*}$ is $\CL^*$-indiscernible in $\CL_H$.

\smallskip

Recall that $y$ is of the form $(y_0,...,y_{m-1})$ and $\ddot b_\eta=(\ddot b^\eta_0,...,\ddot b^\eta_{m-1})$ for each $\eta$.
Let $\ddot z_0,...,\ddot z_{m-1}$ be variables with $|\ddot z_j|=|\ddot z|$ for each $j<m$. Let $\ddot z:=(\ddot z_0,...,\ddot z_{m-1})$. 
Let 
\[
\varphi'(x,y,\ddot z):=\varphi(x,y)\wedge\bigwedge_{j<m}\ddot\psi_j(y_j,\ddot z_j).
\]
Then by the argument in \Cref{prop: n-var theorem exchange} using the s-modeling property, we have
\begin{itemize}
\item[(xviii)] $\dim_\CL(\{\varphi'(x,\ddot a_\eta,\ddot b_\eta)\}_{\eta\in X})=|x|$,
\end{itemize}
There exists $N\in\omega$ such that $|\ddot \psi_j(\mathbb{M},\ddot b^\eta_j)|\le N$ for all $\eta$ and $j$. For each $\eta$ and $j$, allowing duplication, choose any $\ddot a^\eta_{j,0},...,\ddot a^\eta_{j,N-1}\in\mathbb{M}$ such that $\{\ddot a^\eta_{j,0},...,\ddot a^\eta_{j,N-1}\}=\ddot\psi_j(\mathbb{M},\ddot b^\eta_j)$ and $\ddot a^\eta_{j,0}=\ddot a^\eta_j$. Let $g_0$ be the map from $m$ to $N$ such that $g_0(j)=0$ for all $j<m$. For each $g:m\to N$ and $\eta\in\CI^*$, let 
\[
D^g_\eta:=\varphi'(\mathbb{M},\ddot a^\eta_{0,g(0)},...,\ddot a^\eta_{m-1,g(m-1)},\ddot b_\eta).
\]
For a set $G$ of maps from $m$ to $N$, let 
\[
D^G_\eta:=\bigcap_{g\in G} D^g_\eta \cap \bigcap_{g\notin G}  (\mathbb{M}\setminus \! D^g_\eta).
\]
For each $G$ with $\dim_\CL(D^G_\eta)=|x|$, choose any $c^G_\eta\in H(\mathbb{M})\cap D^G_\eta$. We can find such $c^G_\eta$ by the density property. Then $D^G_\eta$ is defined by 
\begin{itemize}
\item[(xix)]\makebox[\linewidth][c]{$\displaystyle
 \forall y\bigg(\bigwedge_{j<m}\ddot\psi_j(y_j,\ddot{b}^\eta_j)\rightarrow \left(\varphi'(x,y,\ddot b_\eta)\leftrightarrow\varphi'(c^G_\eta,y,\ddot b_\eta)\right)\bigg).
$}
\end{itemize}
Thus $D^G_\eta$ is definable over $H(\mathbb{M})$ if $\dim_\CL(D^G_\eta)=|x|$. Let 
\[
D'_\eta:=\bigcup_{g_0\in G, \dim_\CL(D^G_\eta)=|x|} D^G_\eta.
\]
Then $D'_\eta\subset D^{g_0}_\eta=\varphi'(\mathbb{M},\ddot a_\eta,\ddot b_\eta)=\varphi(\mathbb{M},\ddot a_\eta)$ and $\dim_\CL(D^{g_0}_\eta\setminus D'_\eta)<|x|$. Let 
\[
\mathbbof X:=\{\eta\in\CI^*:\exists \eta'\in X \text{ such that }\la\eta'\ra\sim\la\eta\ra\}.
\]
Note that if $\eta\in{\mathbbof X}$, then $D'_\eta$ is defined by a finite disjunction of $\CL(H(\mathbb{M}))$-formulas of the form in (xix).  Since the number of the $\CL(H(\mathbb{M}))$-formulas in the disjunction is bounded, we may assume that these numbers for $D'_\eta$ and $D'_\nu$ are the same for all $\eta,\nu\in{\mathbbof X}$. So there exists $\varphi''(x,w)\in\CL$ and $(a''_\eta)_{\eta\in\CI^*}\subseteq H(\mathbb{M})$ such that
\begin{itemize}
\item[(xx)] $\dim_\CL(\varphi(x,{\ddot a_\eta})\wedge\neg\varphi''(x,a''_\eta))<|x|$ and $\varphi''(\mathbb{M},a''_\eta)\subseteq\varphi(\mathbb{M},\ddot a_\eta)$ for each $\eta\in{\mathbbof X}$.
\end{itemize}

\smallskip

Since $T$ is geometric, we may assume that $(a''_\eta)_{\eta\in\CI^*}$ is $\CL^*$-indiscernible by applying the $\CL^*$-modeling property to $(\ddot a_\eta,a''_\eta)_{\eta\in\CI^*}$ in $\CL_H$. 
 Now we show that $\varphi''$ witnesses $\neg D^{\CL^*\!\!\!,\,\CI^*}_{X,\bar{Y}}$ with $(a''_\eta)_{\eta\in\CI^*}$ modulo $T^{gt}$. Clearly $\{\varphi''(x,a''_\nu)\}_{\nu\in Y}$ is inconsistent in $H(\mathbb{M})$ for all $Y\in\bar{Y}$ since $\varphi''(\mathbb{M},a''_\nu)\subseteq\varphi(\mathbb{M},\ddot a_\nu)$. It is enough to show that $\{\varphi''(x,a''_\eta)\}_{\eta\in X}$ is consistent in $H(\mathbb{M})$. By (xx) and the argument using compactness and the s-modeling property, we have $\dim_\CL(\{\varphi''(x,a''_\eta)\}_{\eta\in X})=|x|$. By the density property and compactness, there exists $h\in H(\mathbb{M})$ realizing $\{\varphi''(x,a''_\eta)\}_{\eta\in X}$.
\end{proof}
\end{theorem}

Applying \Cref{thm: T^gt preservation} to the dividing lines considered in \Cref{sec: D-n-var}, we obtain the following preservation results. Some of these results have already been proved, as indicated below.

\begin{corollary}\label{cor: new T^gt}
Let $T$ be a geometric theory.
\begin{itemize}
\item[(i)] $T$ is stable if and only if $T^{gt}$ is stable
(previously proved in \cite{BV14}).
\item[(ii)] $T$ is simple if and only if $T^{gt}$ is simple
(previously proved in \cite{BV14} by Itay Kaplan).
\item[(iii)] $T$ is NIP if and only if $T^{gt}$ is NIP
(previously proved in \cite{BV14}).
\item[(iv)] $T$ is NTP$_1$ if and only if $T^{gt}$ is NTP$_1$.
\item[(v)] $T$ is NTP$_2$ if and only if $T^{gt}$ is NTP$_2$
(previously proved in \cite{BV14}).
\item[(vi)] $T$ is NATP if and only if $T^{gt}$ is NATP.
\item[(vii)] $T$ is NCTP if and only if $T^{gt}$ is NCTP.
\item[(viii)] $T$ is NBTP if and only if $T^{gt}$ is NBTP.
\item[(ix)] $T$ is NWP if and only if $T^{gt}$ is NWP.
\item[(x)] $T$ is NGP if and only if $T^{gt}$ is NGP.
\item[(xi)] For every $1<k<\omega$, $T$ is NPM$^{(k)}$
if and only if $T^{gt}$ is NPM$^{(k)}$.
\end{itemize}
\end{corollary}

\section{$H$-structure and lovely pair expansions}\label{sec: H-structure lovely pair}

We show that lovely pair expansions and $H$-structure expansions
preserve every dividing line in $\mathfrak{D}_{n\text{-var}}$.
The proof uses only properties shared by the two expansions,
so we can treat both cases together.

\begin{notation}
Let $P$ be a unary predicate symbol, $\bar{x}:=(x_0,...,x_{n-1})$ an $n$-tuple of variables with $|x_i|=1$ for all $i<n$, $\bar{a}:=(a_0,...,a_{n-1})$ an $n$-tuple of parameters with $|a_i|=1$ for all $i<n$, $y$ a single variable, and $b$ a parameter with $|b|=1$.
\begin{itemize}
\item[(i)] By $P(\bar{x})$, we mean $\bigwedge_{i<n}\!P(x_i)$. 
\item[(ii)] By $\neg P(\bar{x})$, we mean $\bigwedge_{i<n}\!\!\neg P(x_i)$. 
\end{itemize}
As in \Cref{sec: n-var thm}, 
\begin{itemize}
\item[(iii)] $\check{x}^i:=(x_0,...,x_{i-1},x_{i+1},...,x_{n-1})$,
\item[(iv)] $\check{a}^i:=(a_0,...,a_{i-1},a_{i+1},...,a_{n-1})$,
\item[(v)] $\bar{x}^{i/y}:=(x_0,...,x_{i-1},y,x_{i+1},...,x_{n-1})$,
\item[(vi)] $\bar{x}^{i/b}:=(x_0,...,x_{i-1},b,x_{i+1},...,x_{n-1})$,
\item[(vii)] $\bar{a}^{i/y}:=(a_0,...,a_{i-1},y,a_{i+1},...,a_{n-1})$.
\item[(viii)] $\bar{a}^{i/b}:=(a_0,...,a_{i-1},b,a_{i+1},...,a_{n-1})$.
\end{itemize}
\end{notation}

\begin{remark}\label{rmk: in the proof of n-var thm}
Let us recall the strategy of the proof of \Cref{thm: n-var theorem}. If $\varphi(\bar x;y)$ witnesses $\neg D^{\CL^*\!\!\!,\,\CI^*}_{X,\bar{Y}}$ with $(a_\eta)_{\eta\in\CI^*}$ and $\dim(\{\varphi(x,a_\eta)\}_{\eta\in X})<|x|$, then we can find a witness $\varphi''(\check{x}^i;y,\bar y',\bar y'')$ of $\neg D^{\CL^*\!\!\!,\,\CI^*}_{X,\bar{Y}}$ such that $\varphi''(\check{x}^i;y,\bar y',\bar y'')\rightarrow\exists x_i\varphi(\bar x;y)$. 

Using this, we obtain the following.
Let $P$ be a unary predicate symbol in $\CL$. If $\varphi(\bar{x}_0,\bar{x}_1;y)\wedge
\neg P(\bar{x}_0)\wedge P(\bar{x}_1)$ witnesses
$\neg D^{\CL^*\!\!\!,\,\CI^*}_{X,\bar{Y}}$ and $\dim(\{\varphi(\bar x_0,\bar x_1;a_\eta)\wedge \neg P(\bar x_0)\wedge P(\bar x_1)\}_{\eta\in X})<|\bar x_0|+|\bar x_1|$, then we can find a witness $\varphi^*(\bar{x}^*_0,\bar{x}^{*}_1;\bar{y}^*)$ of $\neg D^{\CL^*\!\!\!,\,\CI^*}_{X,\bar{Y}}$ such that 
\begin{itemize}
\item[(i)] $\bar{x}^*_0$ and $\bar x^{*}_1$ are subtuples of $\bar{x}_0$ and $\bar{x}_1$, respectively, 
\item[(ii)] $\bar{x}^*_0$ or $\bar x^{*}_1$ is a proper subtuple of $\bar{x}_0$ or $\bar{x}_1$, respectively, 
\item[(iii)] $\varphi^*(\bar{x}^*_0,\bar x^*_1;\bar{y}^*)\rightarrow \neg P(\bar{x}^*_0)\wedge P(\bar{x}^*_1)$.
\end{itemize}
Note that in this case, $\varphi^*(\bar{x}^*_0,\bar x^*_1;\bar{y}^*)\wedge \neg P(\bar{x}^*_0)\wedge P(\bar{x}^*_1)$ also witnesses $\neg D^{\CL^*\!\!\!,\,\CI^*}_{X,\bar{Y}}$.
\end{remark}

\begin{proposition}\label{prop: n-var theorem 1-predicate}
Let $D^{\CL^*\!\!\!,\,\CI^*}_{X,\bar{Y}}\!\!\in\mathfrak{D}_{n\text{-var}}$. Let $\CL$ be a language, $T$ a complete $\CL$-theory, and $\mathbb{M}\models T$ a sufficiently saturated model. Assume that $\CL$ has a unary predicate symbol $P$. If $T\notin D^{\CL^*\!\!\!,\,\CI^*}_{X,\bar{Y}}$, then there exist $\varphi(\bar{x},\bar{x}',y)\in\CL$ and $\CI^*_{\CL^*}$-indiscernible $(a_\eta)_{\eta\in \CI^*}\!$ such that
\begin{itemize}
\item[(i)] $\varphi'(\bar{x},\bar{x}',y):=\varphi(\bar{x},\bar{x}',y)\wedge \neg P(\bar{x})\wedge P(\bar{x}')$ witnesses $\neg D^{\CL^*\!\!\!,\,\CI^*}_{X,\bar{Y}}$ with $(a_\eta)_{\eta\in \CI^*}$,
\item[(ii)] there exists $\bar{b}\bar{b}'\models\{\varphi'(\bar{x},\bar{x}',a_\eta)\}_{\eta\in X}$ such that 
\begin{itemize}
\item[$\ast$] $\dim_\CL(\bar{b}\bar{b}'/(a_\eta)_{\eta\in\CI^*}\!)=|\bar{x}|+|\bar{x}'|$,
\item[$\ast$] $\dim_\CL(\bar{b}/P(\mathbb{M})(a_\eta)_{\eta\in\CI^*}\!)=|\bar{x}|$.
\end{itemize}
\end{itemize}
\begin{proof}
By \Cref{thm: n-var theorem}, there exists $\varphi(\bar{x}^*\!,y)$ witnessing $\neg D^{\CL^*\!\!\!,\,\CI^*}_{X,\bar{Y}}$ with $\CI^*$-indiscernible $(a_\eta)_{\eta\in\CI^*}$ such that $\{\varphi(\bar{x}^*\!,a_\eta)\}_{\eta\in X}$ is realized by some $\bar{b}^*$ with $\dim_\CL(\bar{b}^*\!/(a_\eta)_{\eta\in\CI^*}\!)=|\bar{x}^*|$. There are two disjoint subtuples $\bar{b}$ and $\bar{b}'$ of $\bar{b}^*$ such that $\bar{b}\cup\bar{b}'=\bar{b}^*$ and $\models \neg P(\bar{b})\wedge P(\bar{b}')$. Let $\bar{x}$ and $\bar{x}'$ be the subtuples in $\bar{x}^*$ corresponding to $\bar{b}$ and $\bar{b}'$ in $\bar{b}^*$. Then $\{\varphi(\bar{x},\bar{x}',a_\eta)\}_{\eta\in X}$ is realized by $\bar{b}\bar{b}'$ and $\dim_\CL(\bar{b}\bar{b}'/(a_\eta)_{\eta\in\CI^*}\!)=|\bar{x}|+|\bar{x}'|$. If $\dim_\CL(\bar{b}/(a_\eta)_{\eta\in\CI^*}\!\cup P(\mathbb{M}))=|\bar{x}|$, then there is nothing to prove. So we assume $\dim_\CL(\bar{b}/(a_\eta)_{\eta\in\CI^*}\!\cup P(\mathbb{M}))<|\bar{x}|$. 

Let $|\bar{x}|=n$, $|\bar{x}'|=m$, $\bar{x}=(x_0,...,x_{n-1})$, and $\bar{x}'=(x'_0,...,x'_{m-1})$. There exist $i<n$, $\theta\in\CL$, $l,q,r\in\omega$, $h_0,...,h_{l-1}\in P(\mathbb{M})$, and $\nu_0,...,\nu_{q-1}\in\CI^*$ such that 
\begin{itemize}
\item[(iii)] $\models\theta(b_i,\check{b}^i,\bar{b}',h_0,...,h_{l-1},a_{\nu_0},...,a_{\nu_{q-1}})\wedge\exists^{\le r}x\theta(x,\check{b}^i,\bar{b}',h_0,...,h_{l-1},a_{\nu_0},...,a_{\nu_{q-1}})$.
\end{itemize}
Let $\bar{x}'':=(x''_0,...,x''_{l-1})$, $\bar{y}'':=(y''_0,...,y''_{q-1})$, and
\begin{itemize}
\item[(iv)] $\varphi'(\bar{x},\bar{x}',\bar{x}'',y,\bar{y}''):=\varphi(\bar{x},\bar{x}',y)\wedge\theta(x_i,\check{x}^i,\bar{x}', \bar{x}'',\bar{y}'')$.
\end{itemize}
Let $\bar{h}:=(h_0,...,h_{l-1})$, $\bar{\nu}:=(\nu_0,...,\nu_{q-1})$, and $\bar{a}_{\bar{\nu}}=(a_{\nu_0},...,a_{\nu_{q-1}})$. Then the partial type \[\{\varphi'(b_0,...,b_{i-1},x,b_{i+1},...,b_{n-1},\bar{b}',\bar{h},a_\xi,\bar{a}_{\bar{\nu}})\}_{\xi\in X}\] has only finitely many realizations. Thus there exists a finite subset $X_0$ of $X$ such that 
\[
\bigcap_{\xi\in X_0}\varphi'(b_0,...,b_{i-1},\mathbb{M},b_{i+1},...,b_{n-1},\bar{b}',\bar{h},a_\xi,\bar{a}_{\bar{\nu}})
\]\vspace{-8pt}
\[\;\;\;\;\;=\bigcap_{\xi\in X}\varphi'(b_0,...,b_{i-1},\mathbb{M},b_{i+1},...,b_{n-1},\bar{b}',\bar{h},a_\xi,\bar{a}_{\bar{\nu}}).
\]
Note that 
\begin{itemize}
\item[(v)] for all $\xi'\in X\setminus X_0$, we have
\[\exists x\!\!\!\! \bigwedge_{\xi\in X_0}\!\!\! \varphi'(\bar{b}^{i/x}\!\!,\bar{b}'\!,\bar{h},a_\xi,\bar{a}_{\bar{\nu}})\wedge\forall x\!\!\left(\!\!\!\left(\!\bigwedge_{\xi\in X_0}\!\!\! \varphi'(\bar{b}^{i/x}\!\!,\bar{b}'\!,\bar{h},a_\xi,\bar{a}_{\bar{\nu}})\wedge\varphi'(\bar{b}^{i/x}\!\!,\bar{b}'\!,\bar{h},a_{\xi'}\!,\bar{a}_{\bar{\nu}})\!\!\right)\!\!\leftrightarrow\!\!\!\! \bigwedge_{\xi\in X_0}\!\!\!\varphi'(\bar{b}^{i/x}\!\!,\bar{b}'\!,\bar{h},a_\xi,\bar{a}_{\bar{\nu}}) \!\!\right)\!\!.
\]
\end{itemize}
Let $\{\eta_0,...,\eta_{p-1}\}$ be an enumeration of $X_0$ and $\bar{y}':=(y'_0,...,y'_{p-1})$ be a tuple of variables. Put
\begin{itemize}
\item[(vi)] $\varphi''(\check{x}^i,\bar{x}',\bar{x}'',y,\bar{y}',\bar{y}''):=$\vspace{-2pt}
\[\exists x\!\! \bigwedge_{j<p}\!\! \varphi'(\bar{x}^{i/x}\!\!,\bar{x}'\!,\bar{x}''\!\!,y'_j,\bar{y}'')\wedge\forall x\!\!\left(\!\!\!\left(\bigwedge_{j<p}\!\! \varphi'(\bar{x}^{i/x}\!\!,\bar{x}'\!,\bar{x}''\!\!,y'_j,\bar{y}'')\wedge\varphi'(\bar{x}^{i/x}\!\!,\bar{x}'\!,\bar{x}''\!,y,\bar{y}'')\!\!\right)\!\!\leftrightarrow\!\!\bigwedge_{j<p}\!\varphi'(\bar{x}^{i/x}\!\!,\bar{x}'\!,\bar{x}''\!,y'_j,\bar{y}'') \!\!\right)\!\!.
\]
\end{itemize}
Let $\bar{\eta}:=(\eta_0,...,\eta_{p-1})$ and $\bar{a}_{\bar{\eta}}:=(a_{\eta_0},...,a_{\eta_{p-1}})$.
Then by (v), we have
\begin{itemize}
\item[(vii)] $\models \varphi''(\check{b}^i,\bar{b}',\bar{h}, a_{\xi'},\bar{a}_{\bar{\eta}},\bar{a}_{\bar{\nu}})$ for each $\xi'\in X\setminus X_0$.
\end{itemize}

Since $(\CL^*\!,\CI^*\!,X,\bar{Y})$ satisfies the $n$-variable theorem, there exists an $\CL^*$-embedding $f:\CI^*\to\CI^*$ such that
\begin{itemize}
\item[(viii)] $f(\CI^*)\cap\bar{\eta}\bar{\nu}=\emptyset$,
\item[(ix)] $\bar{\xi}\sim_{\CL^*}\!\bar{\xi}'\;\Rightarrow\; \bar{\xi}\bar{\eta}\bar{\nu}\sim_{\CL^*}\!\bar{\xi}'\bar{\eta}\bar{\nu}$ for all $\bar{\xi},\bar{\xi}'\in f(\CI^*)$,
\item[(x)] $(f(\CI^*)\cap X)\sim_{\CL^*}^\text{fin}X$.
\end{itemize}
For each $\xi\in\CI^*$, let $a'_\xi:=a_{f(\xi)}\bar{a}_{\bar{\eta}}\bar{a}_{\bar{\nu}}$. Then $(a'_\xi)_{\xi\in\CI^*}$ is $\CI^*_{\CL^*}$-indiscernible by (ix).

We prove that $\varphi''(\check{x}^i,\bar{x}',\bar{x}'';y,\bar{y}',\bar{y}'')$ witnesses $\neg D^{\CL^*\!\!\!,\,\CI^*}_{X,\bar{Y}}\!$ with $(a'_\xi)_{\xi\in\CI^*}$. Choose any finite subset $X_1\subseteq X$. By (viii), (ix), and (x), there exists $X'_1\subseteq X\setminus X_0$ such that
\begin{itemize}
\item[(xi)] $X'_1\bar{\eta}\bar{\nu}\sim_{\CL^*} f(X_1)\bar{\eta}\bar{\nu}$.
\end{itemize}
 By (vii), $\{\varphi''(\check{x}^i,\bar{x}',\bar{x}'';a_{\xi'},\bar{a}_{\bar{\eta}},\bar{a}_{\bar{\nu}})\}_{\xi'\in X'_1}$ is consistent. So $\{\varphi''(\check{x}^i,\bar{x}',\bar{x}'';a_{f(\xi)},\bar{a}_{\bar{\eta}},\bar{a}_{\bar{\nu}})\}_{\xi\in X_1}$ is consistent by (xi). Thus $\{\varphi''(\check{x}^i,\bar{x}',\bar{x}'';a'_\xi)\}_{\xi\in X_1}$ is consistent and  hence $\{\varphi''(\check{x}^i,\bar{x}',\bar{x}'';a'_\xi)\}_{\xi\in X}$ is consistent by compactness.
 
Now we show that $\{\varphi''(\check{x}^i,\bar{x}',\bar{x}'';a'_\xi)\}_{\xi\in Y}$ is inconsistent for each $Y\in \bar{Y}$. To get a contradiction, suppose that $\{\varphi''(\check{x}^i,\bar{x}',\bar{x}'';a'_\xi)\}_{\xi\in Y}$ is realized by $\check{c}^i\bar{c}'\bar{c}''$. Then for each $\xi\in Y$, we have
\begin{itemize}
\item[(xii)] 
\[
\exists x\!\! \bigwedge_{j<p}\!\! \varphi'(\bar{c}^{i/x}\!\!,\bar{c}'\!,\bar{c}''\!\!,a_{\eta_j},\bar{a}_{\bar{\nu}})\wedge\forall x \!\!\left(\!\!\!\left(\! \bigwedge_{\,j<p}\!\varphi'(\bar{c}^{i/x}\!\!,\bar{c}'\!,\bar{c}''\!\!,a_{\eta_j},\bar{a}_{\bar{\nu}})\wedge \varphi'(\bar{c}^{i/x}\!\!,\bar{c}'\!,\bar{c}''\!\!,a_{f(\xi)},\bar{a}_{\bar{\nu}})\!\! \right)\!\!\leftrightarrow\!\! \bigwedge_{j<p}\!\!\varphi'(\bar{c}^{i/x}\!\!,\bar{c}'\!,\bar{c}''\!\!, a_{\eta_j},\bar{a}_{\bar{\nu}})\!\! \right)\!\!.
\]
\end{itemize}
Thus there exists $c_i$ such that $\models \varphi'(\bar{c}^{i/c_i}\!,\bar{c}',\bar{c}''\!; a_{f(\xi)},\bar{a}_{\bar{\nu}})$ for each $\xi\in Y$, where $\bar{c}^{i/c_i}=(c_0,...,c_{n-1})$. By the choice of $\varphi'$, $\{\varphi(\bar{x},\bar{x}';a_{f(\xi)})\}_{\xi\in Y}$ is consistent. Since $f$ is an $\CL^*$-embedding, $\{\varphi(\bar{x},\bar{x}';a_\xi)\}_{\xi\in Y}$ is consistent. It is a contradiction. Thus $\{\varphi''(\check{x}^i,\bar{x}',\bar{x}'';a'_\xi)\}_{\xi\in Y}$ is inconsistent for each $Y\in\bar{Y}$ and hence $\varphi''(\check{x}^i,\bar{x}',\bar{x}'';y,\bar{y}',\bar{y}'')$ witnesses $\neg D^{\CL^*\!\!\!,\,\CI^*}_{X,\bar{Y}}$ with $(a'_\xi)_{\xi\in\CI^*}$.

Note that $\varphi''(\check{x}^i,\bar{x}',\bar{x}'';y,\bar{y}',\bar{y}'')\wedge \neg P(\check{x}^i)\wedge P(\bar{x}'\bar{x}'')$ witnesses $\neg D^{\CL^*\!\!\!,\,\CI^*}_{X,\bar{Y}}$ with $(a'_\xi)_{\xi\in\CI^*}$ by (vii). Also note that the length of the free variable part in $\neg P$ of the new witness is $n-1$.
Thus we can repeat this process until we get a witness of $\neg D^{\CL^*\!\!\!,\,\CI^*}_{X,\bar{Y}}$ satisfying (i) and (ii), by \Cref{rmk: in the proof of n-var thm}.
\end{proof}
\end{proposition}

\begin{remark}
Note that the same argument in \Cref{prop: n-var theorem 1-predicate} works for a definable partition of $\mathbb{M}^1$, as follows.

Let $D^{\CL^*\!\!\!,\,\CI^*}_{X,\bar{Y}}\!\!\in\mathfrak{D}_{n\text{-var}}$. Let $\CL$ be a language, $T$ a complete $\CL$-theory, and $\mathbb{M}\models T$ a sufficiently saturated model. Assume that $\emptyset$-definable sets $P_0,...,P_{l-1}$ form a partition of $\mathbb{M}^1$ (i.e., for all $m\in\mathbb{M}^1$, there exists $k<l$ such that $\models P_k(m)$, and there do not exist $m$ and $k<k'<l$ such that $\models P_k(m)\wedge P_{k'}(m)$). If $T\notin D^{\CL^*\!\!\!,\,\CI^*}_{X,\bar{Y}}$, then there exist $\varphi(\bar{x}_0,...,\bar{x}_{l-1},y)\in\CL$ and $\CI^*_{\CL^*}$-indiscernible $(a_\eta)_{\eta\in \CI^*}\!$ such that
\begin{itemize}
\item[(i)] $\varphi'(\bar{x}_0,...,\bar{x}_{l-1},y):=\varphi(\bar{x}_0,...,\bar{x}_{l-1},y)\wedge  P_0(\bar{x}_0)\wedge\cdots\wedge P_{l-1}(\bar{x}_{l-1})$ witnesses $\neg D^{\CL^*\!\!\!,\,\CI^*}_{X,\bar{Y}}$ with $(a_\eta)_{\eta\in \CI^*}$,
\item[(ii)] there exists $\bar{b}_0...\bar{b}_{l-1}\models\{\varphi'(\bar{x}_0,...,\bar{x}_{l-1},a_\eta)\}_{\eta\in X}$ such that 
\begin{itemize}
\item[$\ast$] $\dim_\CL(\bar{b}_0...\bar{b}_{l-1}/(a_\eta)_{\eta\in\CI^*}\!)=|\bar{x}_0|+\cdots+|\bar{x}_{l-1}|$,
\item[$\ast$] $\dim_\CL(\bar{b}_k/P_{>k}(\mathbb{M})(a_\eta)_{\eta\in\CI^*}\!)=|\bar{x}_k|$ for each $k<l$.
\end{itemize}
\end{itemize}
\end{remark}

\begin{lemma}\label{lem: weak monochromatic 1}
For $\psi_0(x,a_0),...,\psi_{n-1}(x,a_{n-1})\in\CL(\mathbb{M})$, there exist $\psi(x,y)\in\CL$ and $a'_0,...,a'_{n-1}$ such that $\psi_i(\mathbb{M},a_i)=\psi(\mathbb{M},a'_i)$ for all $i<n$.
\begin{proof}
For each $f:n\to 2$, let 
\[
D_f:=\bigcap_{f(i)=0}\psi_i(\mathbb{M},a_i) \cap \bigcap_{f(i)=1} \Big(\mathbb{M}\setminus \psi_i(\mathbb{M},a_i)\Big).
\]
If $D_f\neq\emptyset$, choose any $c_f\in D_f$. Then $D_f$ is defined by $\bigwedge_{i<n}(\psi_i(x,a_i)\leftrightarrow\psi_i(c_f,a_i))$. Note that $\psi_i(\mathbb{M},a_i)=\bigcup_{f(i)=0}D_f$. Thus $\psi_i(x,a_i)$ can be defined by a finite disjunction of formulas of the form $\bigwedge_{i<n}(\psi_i(x,a_i)\leftrightarrow\psi_i(c_f,a_i))$. By allowing duplication, we may assume that the numbers of the formulas in the disjunctions for $\psi_0,...,\psi_{n-1}$ are the same.
\end{proof}
\end{lemma}

\begin{theorem}\label{thm: T_P T^ind preservation}
Let $T$ be a geometric theory and $D^{\CL^*\!\!\!,\,\CI^*}_{X,\bar{Y}}\!\!\in\mathfrak{D}_{n\text{-var}}$. If $T^*$ is $T_P$ or $T^{ind}$, then $T\in D^{\CL^*\!\!\!,\,\CI^*}_{X,\bar{Y}}$ if and only if $T^*\in D^{\CL^*\!\!\!,\,\CI^*}_{X,\bar{Y}}$. 
\begin{proof}
If $T\notin D^{\CL^*\!\!\!,\,\CI^*}_{X,\bar{Y}}$, then $T^*\notin D^{\CL^*\!\!\!,\,\CI^*}_{X,\bar{Y}}$ since $T^*$ is an expansion of $T$. Suppose $T^*\notin D^{\CL^*\!\!\!,\,\CI^*}_{X,\bar{Y}}$. By \Cref{prop: n-var theorem 1-predicate}, there exist $\CI^*$-indiscernible $(a_\eta)_{\eta\in\CI^*}$ in $\CL_H$, $\varphi(\bar{x},\bar{x}',y)\in\CL_H$ with $\bar{x}:=(x_0,...,x_{n-1})$, $\bar{x}':=(x'_0,...,x'_{m-1})$, and $|x_i|=|x'_j|=1$ for all $i<n$, $j<m$ such that
\begin{itemize}
\item[(i)] $\varphi'(\bar{x},\bar{x}',y):=\varphi(\bar{x},\bar{x}',y)\wedge\neg H(\bar{x})\wedge H(\bar{x}')$ witnesses $\neg D^{\CL^*\!\!\!,\,\CI^*}_{X,\bar{Y}}$ with $(a_\eta)_{\eta\in\CI^*}$,
\item[(ii)] there exists $\bar{b}\bar{b}'\models\{\varphi'(\bar{x},\bar{x}',a_\eta)\}_{\eta\in X}$ such that \vspace{2pt}
\begin{itemize}
\item[$\ast$] $\dim_{\CL_H}(\bar{b}\bar{b}'/(a_\eta)_{\eta\in \CI^*}\!)=n+m$,
\item[$\ast$] $\dim_{\CL_H}(\bar{b}/H(\mathbb{M})(a_\eta)_{\eta\in \CI^*}\!)=n$. \vspace{2pt}
\end{itemize}
\end{itemize}
By \Cref{fact: H-independent tuples}, the $\CI^*$-modeling property, and the weak monochromaticity, we may assume that each $a_\eta$ is $H$-independent.
For each $\eta\in\CI^*$, let 
\[
Z^0_\eta:=\{\bar{c}\bar{c}'\in\mathbb{M}:\models\varphi(\bar{c},\bar{c}',a_\eta)\wedge \neg H(\bar{c})\wedge H(\bar{c}'),\dim_{\CL}(\bar{c}\bar{c}'/a_\eta)=n+m,\; \dim_{\CL}(\bar{c}/a_\eta H(\mathbb{M}))=n\}
\]
and
\[
Z^1_\eta:=\{\bar{c}\bar{c}'\in\mathbb{M}:\models \neg\varphi(\bar{c},\bar{c}',a_\eta)\wedge \neg H(\bar{c})\wedge H(\bar{c}'),\dim_{\CL}(\bar{c}\bar{c}'/a_\eta)=n+m,\; \dim_{\CL}(\bar{c}/a_\eta H(\mathbb{M}))=n\}.
\]
Note that 
\begin{itemize}
\item[(iii)] $Z^0_\eta$ and $Z^1_\eta$ are type definable,
\item[(iv)] $\bar{b}_0\bar{b}'_0a_\eta$ is $H$-independent and $\qftp_{\{H\}}(\bar{b}_0\bar{b}'_0a_\eta)=\qftp_{\{H\}}(\bar{b}_1\bar{b}'_1a_\eta)$,
\end{itemize}
for all $\bar{b}_0\bar{b}'_0,\bar{b}_1\bar{b}'_1\in Z^0_\eta\cup Z^1_\eta$.
By weak monochromaticity, there exist $l<\omega$ and $\eta_0,...,\eta_{l-1}\in\CI^*$ such that 
\begin{itemize}
\item[(v)] $\eta_i\not\sim_{\CL^*}\eta_j$ for all $i<j<l$,
\item[(vi)] for each $\eta\in\CI^*$, there exists $i<l$ such that $\eta\sim_{\CL^*}\eta_i$.
\end{itemize}
By (iii), (iv), \Cref{fact: H-independent tuples}, and compactness, there exist $\psi_0(\bar{x},\bar{x}'\!,y),...,\psi_{l-1}(\bar{x},\bar{x}'\!,y)\in\CL$ such that
\begin{itemize}
\item[(vii)] $Z^0_{\eta_i}\subseteq\psi_i(\mathbb{M}^{n+m}\!,a_{\eta_i})$,
\item[(viii)] $Z^1_{\eta_i}\cap \psi_i(\mathbb{M}^{n+m}\!,a_{\eta_i})=\emptyset$,
\end{itemize}
for each $i<l$.
 By \Cref{lem: weak monochromatic 1}, there exist $\psi(\bar{x},\bar{x}'\!,z)$ and $a'_{\eta_0},...,a'_{\eta_{l-1}}$ such that $\psi_i(\mathbb{M}^{n+m}\!,a_{\eta_i})=\psi(\mathbb{M}^{n+m}\!,a'_{\eta_i})$ for each $i<l$.  By (vi), there exists $f:\CI^*\to l$ such that $\eta\sim_{\CL^*}\eta_{f(\eta)}$ for each $\eta\in\CI^*$. Note that 
\begin{itemize}
\item[(ix)] $f(\eta)=f(\eta')$ if $\eta\sim_{\CL^*}\eta'$.
\end{itemize}
 For each $\eta\in\CI^*$, let $\sigma_\eta$ be an $\CL_H$-automorphism sending $a_\eta$ to $a_{\eta_{f(\eta)}}$. For each $\eta\in\CI^*$, let $a'_\eta:=\sigma_\eta^{-1}(a'_{\eta_{f(\eta)}})$. Note that for all $\eta\in\CI^*$,
\begin{itemize}
\item[(x)] $a_\eta a'_\eta\equiv a_{\eta_{f(\eta)}}a'_{\eta_{f(\eta)}}$,
\item[(xi)] $\psi_{f(\eta)}(\mathbb{M}^{n+m}\!,a_{\eta})=\psi(\mathbb{M}^{n+m}\!,a'_{\eta})$,
\item[(xii)] $Z^0_\eta\subseteq \psi_{f(\eta)}(\mathbb{M}^{n+m}\!,a_\eta)$,
\item[(xiii)] $Z^1_\eta \cap \psi_{f(\eta)}(\mathbb{M}^{n+m}\!,a_\eta)=\emptyset$.
\end{itemize}
Let 
$\Psi((x_\xi)_{\xi\in \omega^{\le n+m}},z):=\{\psi(x_{\xi_1},...,x_{\xi_{n+m}},z):\emptyset\lhd \xi_1\lhd\cdots\lhd\xi_{n+m}\}\cup\{x_\eta\neq x_\nu\}_{\eta\neq\nu}$.

\medskip

\noindent\underline{Claim.} $\bigcup_{\eta\in X'_0}\!\Psi((x_\xi)_{\xi\in \omega^{\le n+m}},a'_\eta)$ is consistent for all finite $X'_0\subseteq \CI^*$ such that $X'_0\sim_{\CL^*}\!X_0$ for some $X_0\subseteq X$.

\smallskip

\noindent{\it Proof of Claim.} Choose any finite $X'_0\subseteq \CI^*$ such that $X'_0\sim_{\CL^*}X_0$ for some $X_0\subseteq X$. Then by (ii) and the indiscernibility of $(a_\eta)_{\eta\in\CI^*}$, there exists $\bar{c}\bar{c}'\models\{\varphi(\bar{x},\bar{x}'\!,a_\eta)\wedge\neg H(\bar{x})\wedge H(\bar{x}')\}_{\eta\in X'_0}$ such that $\dim_{\CL_H}(\bar{c}\bar{c}'/(a_\eta)_{\eta\in X'_0})=n+m$ and $\dim_{\CL_H}(\bar{c}/H(\mathbb{M})(a_\eta)_{\eta\in X'_0})=n$. Note that for any $\CL_H$-automorphic image $\bar{d}\bar{d}'$ of $\bar{c}\bar{c}'$ over $(a_\eta)_{\eta\in X'_0}$, we have $\bar{d}\bar{d}'\models\{\varphi(\bar{x},\bar{x}'\!,a_\eta)\wedge\neg H(\bar{x})\wedge H(\bar{x}')\}_{\eta\in X'_0}$, $\dim_{\CL_H}(\bar{d}\bar{d}'/(a_\eta)_{\eta\in X'_0})=n+m$, and $\dim_{\CL_H}(\bar{d}/H(\mathbb{M})(a_\eta)_{\eta\in X'_0})=n$. So there exists $(c_\xi)_{\xi\in\omega^{\le n+m}}$ such that
\begin{itemize}
\item[(xiv)] $(c_{\xi_1},...,c_{\xi_{n+m}})\models\{\varphi(\bar{x},\bar{x}'\!,a_\eta)\wedge\neg H(\bar{x})\wedge H(\bar{x}')\}_{\eta\in X'_0}$,
\item[(xv)] $\dim_{\CL_H}(c_{\xi_1}...c_{\xi_{n+m}}/(a_\eta)_{\eta\in X'_0})=n+m$, 
\item[(xvi)] $\dim_{\CL_H}(c_{\xi_1}...c_{\xi_n}/H(\mathbb{M})(a_\eta)_{\eta\in X'_0})=n$,
\end{itemize} 
for each $\emptyset\lhd\xi_1\lhd\cdots\lhd\xi_{n+m}$. Thus $(c_{\xi_1},...,c_{\xi_{n+m}})\in\psi(\mathbb{M}^{n+m}\!,a'_\eta)$ for each $\eta\in X'_0$ and $\emptyset\lhd\xi_1\lhd\cdots\lhd\xi_{n+m}$. Thus $\bigcup_{\eta\in X'_0}\!\Psi((x_\xi)_{\xi\in \omega^{\le n+m}},a'_\eta)$ is consistent. $\dashv$

\medskip

By the $\CI^*_{\CL^*}$-modeling property, we can find $\CI^*_{\CL^*}$-indiscernible $(a''_\eta a'''_\eta)_{\eta\in\CI^*}$ which is $\CI^*_{\CL^*}$-locally based on $(a_\eta a'_\eta)_{\eta\in\CI^*}$. We show $\psi(\bar{x},\bar{x}'\!,z)$ and $(a'''_\eta)_{\eta\in\CI^*}$ satisfy condition (iii) in \Cref{prop: n-var theorem exchange}, which implies $T\notin D^{\CL^*\!\!\!,\,\CI^*}_{X,\bar{Y}}\!$.
First we show that $\{\psi(\bar{x},\bar{x}'\!,a'''_\eta)\}_{\eta\in X}$ has a realization $\bar{c}\bar{c}'$ such that $\dim_\CL(\bar{c}\bar{c}'/\{a'''_\eta\}_{\eta\in X})=n+m$. If there is no realization $\bar{c}\bar{c}'$ of $\{\psi(\bar{x},\bar{x}'\!,a'''_\eta)\}_{\eta\in X}$ such that $\dim_\CL(\bar{c}\bar{c}'/\{a'''_\eta\}_{\eta\in X})=n+m$, then 
\[
\bigcup_{\eta\in X_0}\Psi((x_\xi)_{\xi\in \omega^{\le n+m}},a'''_\eta)
\]
is not consistent for some finite $X_0\subseteq X$. Since $(a'''_\eta)_{\eta\in\CI^*}$ is $\CI^*_{\CL^*}$-locally based on $(a'_\eta)_{\eta\in\CI^*}$, there is $X'_0\sim_{\CL^*}X_0$ such that 
\[
\bigcup_{\eta\in X'_0}\Psi((x_\xi)_{\xi\in \omega^{\le n+m}},a'_\eta)
\]
is not consistent. But it is a contradiction with Claim.

\smallskip

It only remains to show that $\dim_\CL(\{\psi(\bar{x},\bar{x}',a'''_\nu) \}_{\nu\in Y})<n+m$ for each $Y\in\bar{Y}$. Since $(a''_\eta a'''_\eta)_{\eta\in\CI^*}$ is $\CI^*_{\CL^*}$-locally based on $(a_\eta a'_\eta)_{\eta\in\CI^*}$, we have
 $\psi_{f(\eta)}(\mathbb{M}^{n+m}\!,a''_\eta)=\psi(\mathbb{M}^{n+m}\!,a'''_\eta)$ for each $\eta\in\CI^*$ by (ix) and (xi). Since $(a''_\eta a'''_\eta)_{\eta\in\CI^*}$ is $\CI^*_{\CL^*}$-locally based on $(a_\eta a'_\eta)_{\eta\in\CI^*}$ and $(a_\eta)_{\eta\in\CI^*}$ is already indiscernible, there exists an $\CL_H$-automorphism $\tau$ sending $(a_\eta)_{\eta\in\CI^*}$ to $(a''_\eta)_{\eta\in\CI^*}$. By (xiii), we have $\tau(Z^1_\eta)\cap \psi(\mathbb{M}^{n+m}\!,a'''_\eta)=\emptyset$ for each $\eta\in\CI^*$. Note that 
\[
\tau(Z^1_\eta):=\{\bar{c}\bar{c}'\in\mathbb{M}:\models \neg\varphi(\bar{c},\bar{c}',a''_\eta)\wedge \neg H(\bar{c})\wedge H(\bar{c}'),\dim_{\CL}(\bar{c}\bar{c}'/a''_\eta)=n+m,\; \dim_{\CL}(\bar{c}/a''_\eta H(\mathbb{M}))=n\}.
\]

 If $\dim_\CL(\{\psi(\bar{x},\bar{x}',a'''_\nu) \}_{\nu\in Y})=n+m$, then by the density property, the extension property, and compactness, we can find $\bar{c}\bar{c}'\models\{\psi(\bar{x},\bar{x}',a'''_\nu)\}_{\nu\in Y}$ such that $\bar c \bar c'\models\neg H(\bar{x})\wedge H(\bar{x}')$, $\dim_{\CL}(\bar{c}\bar{c}'/a'''_\nu)=n+m$, and $\dim_{\CL}(\bar{c}/H(\mathbb{M})a'''_\nu)=n$ for all $\nu\in Y$. Moreover, we may assume $\dim_{\CL}(\bar{c}\bar{c}'/a''_\nu a'''_\nu)=n+m$ and $\dim_{\CL}(\bar{c}/H(\mathbb{M})a''_\nu a'''_\nu)=n$ by applying the argument using the s-modeling property and compactness. For that $\bar{c}\bar{c}'$, we have $\bar{c}\bar{c}'\notin \tau(Z^1_\nu)$ for each $\nu\in Y$, and hence $\bar{c}\bar{c}'\models\{\varphi(\bar{x},\bar{x}',a''_\nu)\wedge \neg H(\bar{x})\wedge H(\bar{x}')\}_{\nu\in Y}$ and $\tau^{-1}(\bar{c}\bar{c}')\models \{\varphi(\bar{x},\bar{x}',a_\nu)\wedge \neg H(\bar{x})\wedge H(\bar{x}')\}_{\nu\in Y}$, a contradiction. Therefore $\dim_\CL(\{\psi(\bar{x},\bar{x}',a'''_\nu)\}_{\nu\in Y})<n+m$ and $T\notin D^{\CL^*\!\!\!,\,\CI^*}_{X,\bar{Y}}\!$ by \Cref{prop: n-var theorem exchange}.
\end{proof}
\end{theorem}

As in \Cref{sec: T gt}, applying \Cref{thm: T_P T^ind preservation} to the dividing lines considered in \Cref{sec: D-n-var} yields the following preservation results. For the results that have already been proved, we include the relevant references.

\begin{corollary}\label{cor: new T_P T^ind}
Let $T$ be a geometric theory.
\begin{itemize}
\item[(i)] $T$ is stable if and only if $T_P$ is stable, and if and only if $T^{ind}$ is stable
(follows from (ii) and (iii)).
\item[(ii)] \textls[-2]{$T$ is simple if and only if $T_P\!$ is simple,
and if and only if $T^{ind}\!$ is simple
(follows from (iv) and (v)).}
\item[(iii)] $T$ is NIP if and only if $T_P$ is NIP, and if and only if $T^{ind}$ is NIP (previously proved in \cite{BDO11,BV16}).
\item[(iv)] $T$ is NTP$_1$ if and only if $T_P$ is NTP$_1$,
and if and only if $T^{ind}$ is NTP$_1$
(previously proved in \cite{DK17}).
\item[(v)] $T$ is NTP$_2$ if and only if $T_P$ is NTP$_2$, and if and only if $T^{ind}$ is NTP$_2$
(previously proved in \cite{BK16}).
\item[(vi)] \textls[-1]{$T$ is NATP if and only if $T_P$ is NATP,
and if and only if $T^{ind}$ is NATP
(previously proved in \cite{AKLL25}).}
\item[(vii)] $T$ is NCTP if and only if $T_P$ is NCTP,
and if and only if $T^{ind}$ is NCTP.
\item[(viii)] $T$ is NBTP if and only if $T_P$ is NBTP, and if and only if $T^{ind}$ is NBTP.
\item[(ix)] $T$ is NWP if and only if $T_P$ is NWP,
and if and only if $T^{ind}$ is NWP.
\item[(x)] $T$ is NGP if and only if $T_P$ is NGP,
and if and only if $T^{ind}$ is NGP.
\item[(xi)] For each $1<k<\omega$, $T$ is NPM$^{(k)}$
if and only if $T_P$ is NPM$^{(k)}$, and if and only if $T^{ind}$ is NPM$^{(k)}$.
\end{itemize}
\end{corollary}

\begin{remark}
$\Th(\mathbb{R},<,+)$ is distal, but its lovely pair expansion $\Th(\mathbb{R},<,+,\mathbb{Q})$ is not distal \cite{HN17}. Thus, by \Cref{thm: T_P T^ind preservation}, the class of complete distal theories does not belong to $\mathfrak{D}_{n\text{-var}}$.
\end{remark}

\section{Vector spaces with a dense-codense generic subspace}\label{sec: T^G}

Fix a field $\mathbb{F}$ of characteristic $0$, its subring $R$, and a language $\CL\supseteq\CL_0:=\{+,0,\{\lambda\cdot\}_{\lambda\in\mathbb{F}}\}$. We denote $\hat{R}:=Frac(R)$. Let $T$ be a complete $\CL$-theory expanding the theory of vector spaces over $\mathbb{F}$ which has quantifier elimination in $\CL$ for which $\dcl_\CL=\acl_\CL=\span_\mathbb{F}$ and such that it eliminates the quantifier $\exists^\infty$. Let $G$ be a new unary predicate symbol. For each formula $\varphi(\bar{x})$ in $\CL_{R\text{-mod}}:=\{+,0,\{r\cdot\}_{r\in R}\}\subseteq\CL_0$, let $P_\varphi(\bar{x})$ be a new predicate. Let $\CL_G:=\CL\cup\{G\}\cup\{P_\varphi:\varphi\in\CL_{R\text{-mod}}\}$.

In \cite{BdV22}, Berenstein, d'Elb\'ee, and Vassiliev suggested an expansion $T^G_R$ of $T$ into $\CL_G$, satisfying the following first order conditions\footnote{They denote this theory by $T^G$.}.
\begin{itemize}
\item[(A)] $G$ is a proper $R$-submodule of the universe, and for all $\bar{a}$, $P_\varphi(\bar{a})$ if and only if $\bar{a}\in G$ and $G\models\varphi(\bar{a})$ as an $R$-module.
\item[(B)] If $\lambda_0,...,\lambda_{n-1}\in\mathbb{F}$ are $\hat{R}$-linearly independent, then for all $g_0,...,g_{n-1}\in G$, 
\[ \lambda_0g_0+\cdots+\lambda_{n-1}g_{n-1}=0\;\;\;\Rightarrow\;\;\bigwedge_{i<n}g_i=0.\]
\item[(C)] (Density property) for all $r\in R\setminus\{0\}$, $rG$ is dense in the universe. This is a first order property that can be axiomatized through the scheme: for every $\CL$-formula $\varphi(x,\bar{y})$, add the sentence $\exists^\infty x\varphi(x,\bar{y})\rightarrow \exists x(\varphi(x,\bar{y})\wedge rG(x))$;
\item[(D)] (Extension/co-density property) for any $\CL$-formulas $\varphi(x,\bar{y})$ and $\psi(x,\bar{y},\bar{z})$ and $n\ge 1$, the following sentence
\[ (\exists^\infty x\varphi(x,\bar y)\wedge \forall \bar{z} \exists^{\le n}x\psi(x,\bar y, \bar z))\rightarrow\exists x(\varphi(x,\bar y)\wedge\forall \bar z(G(\bar z)\rightarrow \neg \psi(x, \bar y, \bar z))).\]
\end{itemize}

In their original work, Berenstein, d'Elb\'ee, and Vassiliev \cite{BdV22} proved that this expansion preserves stability, NIP, NTP$_1$, NTP$_2$, and simplicity. Ahn, the author, Lee, and Lee \cite{AKLL25} later proved that it also preserves NATP. In this section, we will show that when $R$ is a subfield of $\mathbb{F}$, this expansion preserves every dividing line in $\mathfrak{D}_{n\text{-var}}$.

 Note that $T$ is geometric, $T^G_R$ is a dense/codense expansion of $T$, but $T^G_R$ is not the lovely pair or the $H$-structure expansion in general. If $R\subsetneq\mathbb{F}$, $(V,H)$ is a sufficiently saturated model of the $H$-structure expansion of $T$, and $G:=\span_R(H)$, then $(V,G)$ is a model of $T^G_R$ but not a model of the $H$-structure expansion or the lovely pair expansion, by the argument in \cite[Lemma 3.25, Lemma 3.26]{BdV22}. 
 
 But if we assume additionally $R$ is a subfield of $\mathbb{F}$, then by the same argument in \Cref{thm: T_P T^ind preservation}, we still can prove that $T\in D$ implies $T^G_R\in D$ for each $D\in\mathfrak{D}_{n\text{-var}}$. The only part we have to modify is \Cref{fact: H-independent tuples}.

\begin{definition}
We say a tuple $\bar a$ is {\it $G$-independent} if  $\dim_\CL(\bar a/G(\bar a))=\dim_\CL(\bar a/ G(V))$.
\end{definition}

\begin{remark}\label{rmk: G-independent tuple}
For any tuple $\bar a$, there exists $\bar b\in G(V)$ such that $\dim_{\CL}(\bar a/G(\bar{a})\bar b)=\dim_\CL(\bar{a}/G(V))$ so that $\bar a \bar b$ is $G$-independent.
\end{remark}

\begin{notation}
For each $\hat{R}$-independent tuple $\bar{\lambda}:=(\lambda_0,...,\lambda_{n-1})\in \mathbb{F}$ and $i<n$, let $G_{\bar\lambda}$ and $f_{\bar{\lambda},i}$ be a unary predicate symbol and a unary function symbol, respectively. We can expand $T^G_R$ to the language $\CL_G^+:=\CL_G\cup\{G_{\bar\lambda}:\bar\lambda\in\mathbb{F}, \hat{R}\text{-independent}\}\cup\{f_{\bar \lambda,i}:\bar\lambda=(\lambda_0,...,\lambda_{n-1})\in\mathbb{F}, \hat{R}\text{-independent}\}$ by interpreting the new symbols as follows.
\begin{itemize}
\item[(i)] $G_{\bar\lambda}(m)$ if $m=\bar\lambda\cdot\bar g$ for some $\bar g\in G^{|\bar\lambda |}$.
\item[(ii)] $f_{\bar\lambda,i}(m)=0$ if $\neg G_{\bar\lambda}(m)$.
\item[(iii)] If $G_{\bar\lambda}(m)$ and $m=\sum_{j<|\bar\lambda|}\lambda_jg_j$ for some $g_0,...,g_{|\bar\lambda|-1}\in G$, then $f_{\bar\lambda,i}(m)=g_i$.
\end{itemize}
We call such an expansion $T^{G+}_R$.
\end{notation}

\begin{rem/def}\label{rem/def: G-independent tuple}
If $\bar{a}$ is $G$-independent, then by rearranging, we may assume that $\bar{a}$ is of the form $(a_0,...,a_{n-1},a_n,...a_{m-1},a_m,...,a_{l-1})$ for some $n\le m\le l<\omega$ such that
\begin{itemize}
\item[(i)] $a_i\in G$ for all $i<n$,
\item[(ii)] $a_i\notin\span_\mathbb{F}(G(V)a_{<i})$  for all $n\le i<m$,
\item[(iii)] $a_i\in\span_\mathbb{F}(a_{<i})$ for all $m\le i<l$.
\end{itemize}

We say two $G$-independent tuples $\bar{a}$ and $\bar{b}$ have the same {\it independent code} if $\bar{a}$ and $\bar b$ are of the form $(a_0,...,a_{n-1},a_n,...a_{m-1},a_m,...,a_{l-1})$ and $(b_0,...,b_{n-1},b_n,...b_{m-1},b_m,...,b_{l-1})$ for some $n\le m\le l<\omega$ such that
\begin{itemize}
\item[(iv)] $a_i,b_i\in G$ for all $i<n$,
\item[(v)] $a_i\notin\span_\mathbb{F}(G(V)a_{<i})$ and $b_i\notin\span_\mathbb{F}(G(V)b_{<i})$ for all $n\le i<m$,
\item[(vi)] $a_i\in\span_\mathbb{F}(a_{<i})$ and $b_i\in\span_\mathbb{F}(b_{<i})$ for all $m\le i<l$.
\end{itemize}
\end{rem/def}

\begin{lemma}\label{lem: T^G quantifier elimination}
Assume $R$ is a subfield of $\mathbb{F}$. Suppose $\bar
a$ and $\bar b$ are $G$-independent having the same independent code. If $\tp_\CL(\bar{a})=\tp_\CL(\bar{b})$, then $\tp_{\CL_G^+}(\bar{a})=\tp_{\CL_G^+}(\bar{b})$.
\begin{proof}
Let $\bar{a}$ and $\bar{b}$ be of the form in \Cref{rem/def: G-independent tuple}.

\medskip

\noindent\underline{Claim 1.} $\la \bar{a}\ra_{\CL_G^+}=\span_\mathbb{F}(\bar{a})$ and $\la \bar{b}\ra_{\CL_G^+}=\span_\mathbb{F}(\bar{b})$.

\smallskip

\noindent{\it Proof of Claim 1.} By \cite[Lemma 3.10]{BdV22}, it is enough to show that $f_{\bar\lambda,i}(\bar\mu\cdot\bar{a})\in\span_\mathbb{F}(\bar{a})$ for all $\bar\mu\in\mathbb{F}$, $\hat{R}$-independent $\bar\lambda$, and $i<|\bar\lambda|$. If $\neg G_{\bar\lambda}(\bar\mu\cdot\bar a)$, then $f_{\bar\lambda,i}(\bar\mu\cdot\bar{a})=0\in\span_\mathbb{F}(\bar{a})$. Suppose $G_{\bar{\lambda}}(\bar\mu\cdot\bar a)$ and $\bar{\mu}\cdot\bar a=\bar{\lambda}\cdot\bar{g}$ for some $\bar{g}\in G$. Since $a_i\in\span_\mathbb{F}(a_{<i})$ for each $m\le i<l$, we may assume that $\mu_i=0$ for each $m\le i<l$. If $\mu_{m-1}\neq 0$, then $a_{m-1}={{\bar\lambda\cdot\bar g}\over{\mu_{m-1}}} -\sum_{ i<m-1}{\mu_i\over \mu_{m-1}}a_i $. But $a_{m-1}\notin\span_\mathbb{F}(G(V)a_{<m-1})$, so it is impossible. $\mu_{m-1}=0$. Similarly, $\mu_i=0$ for all $n\le i<m$. So 
\[\mu_0a_0+\cdots+\mu_{n-1}a_{n-1}=\bar{\lambda}\cdot\bar g\]
and by rearranging the tuple, we may assume that there is $n'\le n$ such that $\mu_i\notin\span_{\hat{R}}(\bar\lambda\mu_{<i})$ if and only if $i< n'$. By (B), $g_i\in\span_{\hat{R}}(a_{n'},...,a_{n-1})$ for each $i<|\bar{\lambda}|$. If $n'=n$, then $g_i=0$ for all $i$. Thus $f_{\bar\lambda,i}(\bar\mu\cdot\bar a)\in\span_\mathbb{F}(\bar{a})$ for all $i$. $\dashv$

\medskip

Since $T_R^{G+}$ has quantifier elimination \cite[Theorem 3.16]{BdV22}, it is enough to show that $\qftp_{\CL_G^+}(\span_\mathbb{F}(\bar a))=\qftp_{\CL_G^+}(\span_\mathbb{F}(\bar b))$. Note that every element of $\span_\mathbb{F}(\bar a)$ can be written of the form $\mu_0a_0+\cdots+\mu_{m-1}a_{m-1}$ since $a_i\in\span_\mathbb{F}(a_{<i})$ for all $m\le i<l$. Thus the lemma can be proved by showing that the $\CL$-isomorphism from $\span_\mathbb{F}(\bar{a})$ to $\span_\mathbb{F}(\bar{b})$ sending $\mu_0a_0+\cdots+\mu_{m-1}a_{m-1}$ to $\mu_0b_0+\cdots+\mu_{m-1}b_{m-1}$ is an $\CL_G^+$-isomorphism.

\medskip

\noindent\underline{Claim 2.} $\models G(\mu_0a_0+\cdots+\mu_{m-1}a_{m-1})$ if and only if $\models G(\mu_0b_0+\cdots+\mu_{m-1}b_{m-1})$ for all $\bar\mu\in\mathbb{F}$.

\smallskip

\noindent{\it Proof of Claim 2.} Suppose $\mu_0a_0+\cdots+\mu_{m-1}a_{m-1}=g\in G(V)$. Since $a_i\notin\span_\mathbb{F}(G(V)a_{<i})$ for all $n\le i <m$, we have $\mu_i=0$ for all $n\le i<m$. So $\mu_0a_0+\cdots+\mu_{n-1}a_{n-1}=g$. By \cite[Lemma 3.7]{BdV22}, there exist $r_0,...,r_{n-1}\in\hat{R}$ such that $r_0a_0+\cdots+r_{n-1}a_{n-1}=g$. Thus we have
\[\mu_0a_0+\cdots+\mu_{n-1}a_{n-1}=r_0a_0+\cdots+r_{n-1}a_{n-1}.\]
Since $\tp_\CL(\bar{a})=\tp_\CL(\bar{b})$, we also have
\[\mu_0b_0+\cdots+\mu_{n-1}b_{n-1}=r_0b_0+\cdots+r_{n-1}b_{n-1}.\]
Thus $\mu_0b_0+\cdots+\mu_{m-1}b_{m-1}=\mu_0b_0+\cdots+\mu_{n-1}b_{n-1}=r_0b_0+\cdots+r_{n-1}b_{n-1}\in G(V)$ since $\hat{R}=R$ and $G(V)$ is an $R$-module. $\dashv$

\medskip

\noindent\underline{Claim 3.} $\models P_\varphi(\bar\mu_0\cdot \bar a,...,\bar\mu_{k-1}\cdot\bar a)$ if and only if $\models P_\varphi(\bar\mu_0\cdot \bar b,...,\bar{\mu}_{k-1}\cdot\bar b)$ for all $\bar{\mu}_0,...,\bar{\mu}_{k-1}\in\mathbb{F}$ and $\varphi\in\CL_{R\text{-mod}}$.

\smallskip

\noindent{\it Proof of Claim 3.} Suppose $\models P_\varphi(\bar\mu_0\cdot \bar a,...,\bar\mu_{k-1}\cdot\bar a)$. Then $\models G(\bar\mu_i\cdot \bar{a})$ for all $i<k$ and $G\models \varphi(\bar\mu_0\cdot\bar a,...,\bar\mu_{k-1}\cdot \bar a)$. By Claim 2, $\models G(\bar\mu_i\cdot \bar b)$ for all $i<k$. Since $R$ is a field, $G\prec_{\CL_{R\text{-mod}}}V$, and hence $V\models\varphi(\bar\mu_0\cdot\bar a,...,\bar\mu_{k-1}\cdot\bar a)$. Since $\CL_{R\text{-mod}}\subseteq\CL$, we have  $V\models\varphi(\bar\mu_0\cdot\bar b,...,\bar\mu_{k-1}\cdot\bar b)$ and hence $G\models\varphi(\bar\mu_0\cdot\bar b,...,\bar\mu_{k-1}\cdot\bar b)$. Thus $\models P_\varphi(\bar\mu_0\cdot\bar b,...,\bar\mu_{k-1}\cdot\bar b)$. $\dashv$.

\medskip

\noindent\underline{Claim 4.} $\models G_{\bar\lambda}(\mu_0a_0+\cdots+\mu_{m-1}a_{m-1})$ if and only if $\models G_{\bar\lambda}(\mu_0b_0+\cdots+\mu_{m-1}b_{m-1})$ for all $\bar\mu\in \mathbb{F}$ and $\hat{R}$-independent $\bar{\lambda}$.

\smallskip

\noindent{\it Proof of Claim 4.} Suppose $\mu_0a_0+\cdots+\mu_{m-1}a_{m-1}=\bar\lambda\cdot\bar{g}$ for some $\bar g\in G(V)$. By the same argument in Claim 1, we have $\mu_i=0$ for all $n\le i<m$, and $g_i\in\span_{\hat{R}}(a_0,...,a_{n-1})$ for each $i<|\bar\lambda|$. Thus there exists $(\bar{r}_i)_{i<|\bar\lambda|}$ such that
\[
\mu_0a_0+\cdots+\mu_{n-1}a_{n-1}=\sum_{i<|\bar\lambda|}\lambda_i(\bar{r}_i\cdot(a_0,...,a_{n-1})).\]
Since $\tp_\CL(\bar{a})=\tp_\CL(\bar{b})$, we also have
\[
\mu_0b_0+\cdots+\mu_{n-1}b_{n-1}=\sum_{i<|\bar\lambda|}\lambda_i(\bar{r}_i\cdot (b_0,...,b_{n-1})  ).\]
Since we assume $R$ is a field and $G$ is an $R$-module, $\bar{r}_i\cdot\bar b\in G$ for each $i<|\bar\lambda|$. Thus $\models G_{\bar{\lambda}}(\mu_0b_0+\cdots+\mu_{m-1}b_{m-1})$. $\dashv$

\medskip

Note that, as in Claim 1, $f_{\bar\lambda,i}(\mu_0a_0+\cdots+\mu_{m-1}a_{m-1})$ is always of the form $r_0a_0+\cdots+r_{n-1}a_{n-1}$ for some $r_0,...,r_{n-1}\in\hat{R}$.

\medskip

\noindent\underline{Claim 5.} $f_{\bar\lambda,i}(\mu_0a_0+\cdots+\mu_{m-1}a_{m-1})=r_0a_0+\cdots+r_{n-1}a_{n-1}$ if and only if $f_{\bar\lambda,i}(\mu_0b_0+\cdots+\mu_{m-1}b_{m-1})=r_0b_0+\cdots+r_{n-1}b_{n-1}$.

\smallskip

\noindent{\it Proof of Claim 5.} By Claim 4, we may assume $\models G_{\bar\lambda}(\mu_0a_0+\cdots+\mu_{m-1}a_{m-1})$ and $\mu_0a_0+\cdots+\mu_{m-1}a_{m-1}=\mu_0a_0+\cdots+\mu_{n-1}a_{n-1}=\sum_{i<|\bar\lambda|}\lambda_i(\bar{r}_i\cdot(a_0,...,a_{n-1}))$ for some $(\bar r_i)_{i<|\bar \lambda|}$ in $\hat{R}$. Note that $\bar r_i=(r_0,...,r_{n-1})$. Since $\tp_\CL(\bar a)=\tp_\CL(\bar b)$, we have $\models G_{\bar\lambda}(\mu_0b_0+\cdots+\mu_{m-1}b_{m-1})$ and $\mu_0b_0+\cdots+\mu_{m-1}b_{m-1}=\mu_0b_0+\cdots+\mu_{n-1}b_{n-1}=\sum_{i<|\bar\lambda|}\lambda_i(\bar{r}_i\cdot(b_0,...,b_{n-1}))$. Since $\bar\lambda$ is $\hat{R}$-independent, we have $f_{\bar\lambda,i}(\mu_0b_0+\cdots+\mu_{m-1}b_{m-1})=r_0b_0+\cdots+r_{n-1}b_{n-1}$. $\dashv$

\medskip

Thus $\la \bar a\ra_{\CL_G^+}$ and $\la \bar b\ra_{\CL_G^+}$ are $\CL_G^+$-isomorphic, and hence $\tp_{\CL_G^+}(\bar{a})=\tp_{\CL_G^+}(\bar b)$.
\end{proof}
\end{lemma}

\begin{lemma}\label{lem: G-independent tuple extend}
Let $\bar{a}$ be $G$-independent, $\bar{b}\in G(V)$ and $\bar{c}\in \mathbb{M}\setminus G(V)$. Suppose that $c_i\notin \acl(G(V)\bar{a}c_{<i})$ for all $i$. Then $\bar{a}\bar{b}\bar{c}$ is $G$-independent.
\end{lemma}

\begin{theorem}\label{thm: T^G_R preservation}
Let $\mathbb{F}$ be a field of characteristic zero,   $\CL\supseteq\CL_0:=\{+,0,\{\lambda\cdot\}_{\lambda\in\mathbb F}\}$, and $T$ a complete $\CL$-theory expanding the theory of vector spaces over $\mathbb{F}$ which has quantifier elimination in $\CL$ for which $\dcl_\CL=\acl_\CL=\span_\mathbb{F}$ and such that it eliminates the quantifier $\exists^\infty$.
Let $D^{\CL^*\!\!\!,\,\CI^*}_{X,\bar{Y}}\!\!\in\mathfrak{D}_{n\text{-var}}$ and assume that $R$ is a subfield of $\mathbb{F}$. Then $T\in D^{\CL^*\!\!\!,\,\CI^*}_{X,\bar{Y}}$ if and only if $T^G_R\in D^{\CL^*\!\!\!,\,\CI^*}_{X,\bar{Y}}$. 
\begin{proof}
It can be proved by the same argument in \Cref{thm: T_P T^ind preservation}, replacing the role of \Cref{fact: H-independent tuples} with \Cref{rmk: G-independent tuple}, \Cref{lem: G-independent tuple extend}, and \Cref{lem: T^G quantifier elimination}.
\end{proof}
\end{theorem}

\begin{corollary}
Let $T$ and $R$ be as above.
\begin{itemize}
\item[(i)] $T$ is stable if and only if $T^G_R$ is stable (previously proved in \cite{BdV22}).
\item[(ii)] $T$ is simple if and only if $T^G_R$ is simple (previously proved in \cite{BdV22}).
\item[(iii)] $T$ is NIP if and only if $T^G_R$ is NIP (previously proved in \cite{BdV22}).
\item[(iv)] $T$ is NTP$_1$ if and only if $T^G_R$ is NTP$_1$ (previously proved in \cite{BdV22}).
\item[(v)] $T$ is NTP$_2$ if and only if $T^G_R$ is NTP$_2$ (previously proved in \cite{BdV22}).
\item[(vi)] $T$ is NATP if and only if $T^G_R$ is NATP (previously proved in \cite{AKLL25}).
\item[(vii)] $T$ is NCTP if and only if $T^G_R$ is NCTP.
\item[(viii)] $T$ is NBTP if and only if $T^G_R$ is NBTP.
\item[(ix)] $T$ is NWP if and only if $T^G_R$ is NWP.
\item[(x)] $T$ is NGP if and only if $T^G_R$ is NGP.
\item[(xi)] For every $1<k<\omega$, $T$ is NPM$^{(k)}$
if and only if $T^G_R$ is NPM$^{(k)}$.
\end{itemize}
\end{corollary}

\subsection{A special case of Mordell-Lang theories}\label{subsection: ML-theories}
T$_P$, T$^{ind}$, and T$^G_R$ are special cases of {\it Mordell-Lang theories} ({\it ML-theories}) introduced by Block Gorman, Hieronymi, and Kaplan in \cite{BHK20}. In this subsection, we give a sufficient condition for ML-theories that makes the construction preserve every dividing line in $\mathfrak{D}_{n\text{-var}}$. 

We consider languages $\CL_\alpha$ and $\CL_\beta$ such that all function symbols in $\CL_\alpha$ belong to $\CL_\beta$. Let $T_\alpha$ and $T_\beta$ be an $\CL_\alpha$-theory and an $\CL_\beta$-theory respectively, and assume that $T_\beta$ is geometric. Let $\CL:=\CL_\alpha\cap\CL_\beta$ and $\CL^2:=\CL_\alpha\cup\CL_\beta\cup\{A\}$, where $A$ is a new unary predicate. For each $\CL_\alpha$-formula $\theta$, let $\theta_A$ be the $\CL_\alpha\cup\{A\}$-formula obtained by relativizing all quantifiers in $\theta$ to $A$. Let $T^2$ be an $\CL^2$-theory extending $T_\beta\cup\{\theta_A:\theta\in T_\alpha\}$ and containing sentences saying ``$A$ is closed under every function in $\CL_\alpha$'' and ``every relation in $\CL_\alpha\setminus\CL_\beta$ holds only on tuples from $A$''.

We denote models of $T^2$ by $(\mathcal B,\mathcal A)$, where $\mathcal B$ is the $\CL_\beta$-reduct and $\mathcal A$ is the $\CL_\alpha$-structure induced on the interpretation of the predicate $A$.

\begin{definition}
Let $(\mathcal{B},\mathcal{A})\models T^2$. We say a tuple $\bar c\in(\mathcal{B},\mathcal{A})$ is {\it $A$-independent} if $
\dim_{\CL_\beta}(\bar c/A(\bar c))=\dim_{\CL_\beta}(\bar c/\mathcal{A})$.
\end{definition}

\begin{definition}
Let $T\supseteq T^2$ be an $\CL^2$-theory. A {\it Mordell--Lang challenge (for $T$)} is a tuple
\[
(p(\bar x_\alpha),q(\bar x),\varphi(\bar x,y),\psi(\bar x_\alpha,y))
\]
such that
\begin{itemize}
\item[(i)] $p$ is a complete $\CL_\alpha$-type and $q$ is a complete $\CL_\beta$-type,
\item[(ii)] $\varphi$ is an $\CL_\beta$-formula and $\psi$ is an $\CL_\alpha$-formula,
\item[(iii)] $q(\bar x)\models\exists^{<\infty}y\,\varphi(\bar x,y)$.
\end{itemize}

A {\it contender} to a Mordell-Lang challenge is a tuple $((\mathcal B,\mathcal A),\bar c)$ where $(\mathcal B,\mathcal A)\models T$ and where $\bar c$ is a tuple in $B$ such that $\bar c$ realizes $q$, the subtuple $\bar c_\alpha:=A(\bar c)$ realizes $p_A$, and $\bar c$ is $A$-independent.
A {\it solution} to a Mordell-Lang challenge is a tuple $((\mathcal B,\mathcal A),\bar c,a)$ such that $((\mathcal B,\mathcal A),\bar c)$ is a contender, $a\in\mathcal A$, and
\[
(\mathcal B,\mathcal A)\models\varphi(\bar c,a)\wedge\psi_A(\bar c_\alpha,a).
\]
A Mordell-Lang challenge is {\it solvable} if it has a solution.
\end{definition}

\begin{definition}
Let $T\supseteq T^2$ be an $\CL^2$-theory. We say that $T$ satisfies the {\it Mordell--Lang condition} if for every solvable Mordell-Lang challenge $(p,q,\varphi,\psi)$ for $T$ and for every contender $((\mathcal B,\mathcal A),\bar c)$, there is $a\in A$ such that $((\mathcal B,\mathcal A),\bar c,a)$  is a solution.
\end{definition}

\begin{definition}
An $\CL^2$-theory $T$ is a {\it Mordell--Lang theory of pairs}, or an {\it ML-theory}, if
\begin{itemize}
\item[(i)] $T$ extends $T^2$,
\item[(ii)] $T$ satisfies the Mordell--Lang condition,
\item[(iii)] for every $\kappa$-saturated model $(\mathcal B,\mathcal A)\models T$ where $\kappa>|T^2|$, every $C\subseteq\mathcal B$ with $|C|<\kappa$, and every non-algebraic unary $\CL_\beta(C)$-type $q(x)$, the following conditions hold.
\begin{itemize}
\item[$\ast$] (Density) If $p(x)$ is a unary $\CL_\alpha(\mathcal A\cap C)$-type such that $q\models\operatorname{qf}(p|_\CL)$, where $p|_\CL$ restricts to only $\CL(A\cup C)$-formulas, then there exists $a\in\mathcal A$ realizing $p_A\cup q$.
\item[$\ast$] (Codensity) There exists $b\in\mathcal B\setminus\acl_{\CL_\beta}(\mathcal A\cup C)$ realizing $q$.
\end{itemize}
\end{itemize}
A model $(\mathcal B,\mathcal A)$ of an ML-theory is called an {\it ML-pair}.
\end{definition}

A technique for decomposing $\CL^2$-types of $A$-independent tuples in ML-pairs into $\CL_\alpha$-types and $\CL_\beta$-types was already developed in \cite{BHK20}. We recall that below. Under additional assumptions, this will give an analogue of \Cref{fact: H-independent tuples}.

\begin{fact}\cite[Section 3]{BHK20}\label{fact: qe for ML-theories}
Let $(\mathcal B,\mathcal A)$ be an ML-pair. If $\bar c$ and $\bar c'$ are $A$-independent, $\tp_\beta(\bar c)=\tp_\beta(\bar c')$, $\qftp_{\{A\}}(\bar c)=\qftp_{\{A\}}(\bar c')$, and $\tp_\alpha(A(\bar c))=\tp_\alpha(A(\bar c'))$, then $\tp_{\CL^2}(\bar c)=\tp_{\CL^2}(\bar c')$.
\end{fact}

\begin{lemma}\label{lem: qe for ML-theories additional condition}
Let $(\mathcal B,\mathcal A)$ be an ML-pair and additionally assume that $\CL_\alpha\subseteq \CL_\beta$ and $\mathcal{A}\prec_{\CL_\alpha}\!\mathcal{B}$.
If $\bar c$ and $\bar c'$ are $A$-independent, $\tp_\beta(\bar c)=\tp_\beta(\bar c')$, and $\qftp_{\{A\}}(\bar c)=\qftp_{\{A\}}(\bar c')$, then $\tp_{\CL^2}(\bar c)=\tp_{\CL^2}(\bar c')$.
\end{lemma}

\begin{theorem}
In the setting above, assume in addition that $\CL_\alpha\subseteq\CL_\beta$. Let $T$ be a Mordell-Lang theory of pairs satisfying $T\vdash \forall x,\bar y(\psi(x,\bar y)\wedge A(\bar y)\rightarrow\exists x(A(x)\wedge\psi(x,\bar y)))$ for each $\psi\in\CL_\alpha$. Let $D\in\mathfrak{D}_{n\text{-var}}$. Then $T_\beta\in D$ if and only if $T\in D$.
\begin{proof}
As in \Cref{thm: T^G_R preservation}, it can be proved by the same argument in \Cref{thm: T_P T^ind preservation}, replacing the role of \Cref{fact: H-independent tuples} with \Cref{lem: qe for ML-theories additional condition}.
\end{proof}
\end{theorem}

\section{Algebraically closed fields with a distinguished subfield}\label{sec: ACFT}

Let $\CL$ be a language such that $\CL\supseteq\CL_\text{ring}:=\{+,-,\times,0,1\}$ and $T$ be a complete $\CL$-theory such that $T|_{\CL_\text{ring}}$ is a theory of a field. Let $P$ be a new predicate symbol and $\CL_P:=\CL\cup\{P\}$. 
For each $\CL$-formula $\theta(x)$, let $\theta^P(x)$ be the $\CL_P$-formula obtained by restricting every quantifier to $P$ and asserting $x\in P$.
Let $\text{ACF}_T$ be an $\CL_P$-theory such that
\begin{itemize}
\item[(i)] $\text{ACF}_T |_{\CL_\text{ring}}$ is the theory of algebraically closed fields,
\item[(ii)] for any $\mathbb{M}\models \text{ACF}_T$, $P(\mathbb{M})$ is an $\CL$-structure (closed under all function symbols in $\CL$) and $P(\mathbb{M})\models T$,
\item[(iii)] $f(\bar{a})=0$ for every $n$-ary function symbol $f\in \CL\setminus\CL_\text{ring}$ and $\bar{a}\in\mathbb{M}^n\setminus P(\mathbb{M})^n$,
\item[(iv)] $\mathbb{M}\models\neg R(\bar{a})$ for every $n$-ary relation symbol $R\in \CL\setminus\CL_\text{ring}$ and $\bar{a}\in\mathbb{M}^n\setminus P(\mathbb{M})^n$,
\item[(v)] the degree of the field extension of $\mathbb{M}$ over $P(\mathbb{M})$ is infinite.
\end{itemize}
This construction was studied by Keisler \cite{Kei64} and investigated further by d'Elb\'ee, Kaplan, and Neuhauser \cite{dKN21}. We show that this construction preserves every dividing line in $\mathfrak{D}_{n\text{-var}}$. 

\begin{fact}\cite{dKN21}
Let $A\subseteq\mathbb{M}$. If $a,a'\notin\acl_{\CL_\text{ring}}(P(\mathbb{M})A)$, then $a\equiv_{P(\mathbb{M})A}a'$.
\end{fact}

\begin{fact}\cite{dKN21}\label{fact: stably embedded ACF_T}
$P(\mathbb{M})$ is stably embedded. Moreover, every $\CL_P(\mathbb{M})$-definable set in $P(\mathbb{M})^n$ is defined in the $\CL$-structure $P(\mathbb{M})$ with parameters from $P(\mathbb{M})$. In other words, for each $\varphi(x)\in\CL_P(\mathbb{M})$, there exists an $\CL(P(\mathbb{M}))$-formula $\psi(x)$ such that $\models \forall x(\varphi(x)\wedge P(x) \leftrightarrow \psi^P(x))$.
\end{fact}

\begin{fact}\cite{AKLL25}\label{fact: lifting strong indiscernible ACF_T}
Suppose an $\CL^*$-structure $\CI^*$ has the modeling property.
If $(a_\eta)_{\eta\in\CI^*}$ is $\CI^*_{\CL^*}$-indiscernible over $C$ and $b\notin\acl_{\CL_\text{ring}}(P(\mathbb{M})(a_\eta)_{\eta\in \CI^*}C )$, then $(a_\eta)_{\eta\in\CI^*}$ is $\CI^*_{\CL^*}$-indiscernible over $Cb$.
\end{fact}

\begin{theorem}\label{thm: ACF_T preservation}
Let $T$ be a complete theory of fields, possibly with an additional structure. If $D^{\CL^*\!\!\!,\,\CI^*}_{X,\bar{Y}}\!\!\in\mathfrak{D}_{n\text{-var}}$,  then $T\in D^{\CL^*\!\!\!,\,\CI^*}_{X,\bar{Y}}$ if and only if ACF$_T\in D^{\CL^*\!\!\!,\,\CI^*}_{X,\bar{Y}}$.
\begin{proof}

The implication from right to left is clear. Suppose $\text{ACF}_T\notin D^{\CL^*\!\!\!,\,\CI^*}_{X,\bar{Y}}$. Then by \Cref{prop: n-var theorem 1-predicate}, we have $\varphi(\bar{x},\bar{x}',y)$ and $\CI^*$-indiscernible $(a_\eta)_{\eta\in\CI^*}$ such that
\begin{itemize}
\item[(i)] $\varphi(\bar{x},\bar{x}',y)\wedge \neg P(\bar{x})\wedge P(\bar{x}')$ witnesses $\neg D^{\CL^*\!\!\!,\,\CI^*}_{X,\bar{Y}}$ with $(a_\eta)_{\eta\in\CI^*}$,
\item[(ii)] there exist $\bar{b}\bar{b}'\models \{\varphi(\bar{x},\bar{x}',a_\eta)\wedge \neg P(\bar{x})\wedge P(\bar{x}')\}_{\eta\in X}$ such that
\begin{itemize}
\item[$\ast$] $\dim(\bar{b}\bar{b}'/(a_\eta)_{\eta\in\CI^*}\!)=|\bar{x}|+|\bar{x}'|$,
\item[$\ast$] $\dim(\bar{b}/P(\mathbb{M})(a_\eta)_{\eta\in\CI^*}\!)=|\bar{x}|$.
\end{itemize}
\end{itemize}

Fix any $\bar{b}\bar{b}'$ satisfying (ii). By \Cref{fact: lifting strong indiscernible ACF_T}, $(a_\eta)_{\eta\in\CI^*}$ is $\CI^*$-indiscernible over $\bar{b}$. Thus $(\bar ba_\eta)_{\eta\in\CI^*}$ is $\CI^*$-indiscernible and hence we may assume $|\bar{x}|=0$. By \Cref{fact: stably embedded ACF_T}, indiscernibility of $(a_\eta)_{\eta\in\CI^*}$, weak monochromaticity of $\CI^*$, the $\CI^*$-modeling property, and \Cref{lem: weak monochromatic 1}, there is an $\CL$-formula $\psi(\bar{x}',y')$ and $\CI^*$-indiscernible $(a'_\eta)_{\eta\in\CI^*}\!\subseteq P(\mathbb{M})$ such that $\psi^P(\bar{x}',y')$ witnesses $\neg D^{\CL^*\!\!\!,\,\CI^*}_{X,\bar{Y}}$ with $(a'_\eta)_{\eta\in\CI^*}$. Thus $\psi$ witnesses $T\notin D^{\CL^*\!\!\!,\,\CI^*}_{X,\bar{Y}}$.
\end{proof}
\end{theorem}

\begin{corollary}\label{cor: ACF_T}
Let $T$ be a complete theory of fields, possibly with additional structure.
\begin{itemize}
\item[(i)] $T$ is stable if and only if $\text{ACF}_T$ is stable (previously proved in \cite{dKN21}).
\item[(ii)] $T$ is simple if and only if $\text{ACF}_T$ is simple (previously proved in \cite{dKN21}).
\item[(iii)] $T$ is NIP if and only if $\text{ACF}_T$ is NIP (previously proved in \cite{dKN21}).
\item[(iv)] $T$ is NTP$_1$ if and only if $\text{ACF}_T$ is NTP$_1$ (previously proved in \cite{AKLL25}).
\item[(v)] $T$ is NTP$_2$ if and only if $\text{ACF}_T$ is NTP$_2$ (previously proved in \cite{AKLL25}).
\item[(vi)] $T$ is NATP if and only if $\text{ACF}_T$ is NATP (previously proved in \cite{AKLL25}).
\item[(vii)] $T$ is NCTP if and only if $\text{ACF}_T$ is NCTP.
\item[(viii)] $T$ is NBTP if and only if $\text{ACF}_T$ is NBTP.
\item[(ix)] $T$ is NWP if and only if $\text{ACF}_T$ is NWP.
\item[(x)] $T$ is NGP if and only if $\text{ACF}_T$ is NGP.
\item[(xi)] For each $1<k<\omega$, $T$ is NPM$^{(k)}$
if and only if $\text{ACF}_T$ is NPM$^{(k)}$.
\end{itemize}
\end{corollary}

For a field $K$, let $SG(K)$ denote the complete system of its absolute Galois group \cite[Section~2.1]{Cha08}, and let $T^K_{\mathrm{VS},\infty}$ be the theory of infinite-dimensional vector spaces over fields elementarily equivalent to $K$.
It is known that both $\Th(SG(K))$ and $T^K_{\mathrm{VS},\infty}$ are interpretable in $\text{ACF}_{\Th(K)}$ by \cite[Fact 6.6]{dKN21} and \cite[proof of Corollary 4.4]{CH24}, respectively. Moreover, $\Th(K)$ is interpreted in the field sort of $T^K_{\mathrm{VS},\infty}$ \cite[Remark~2.19(1)]{CH24}.
Combining with \Cref{thm: ACF_T preservation} and \Cref{prop: bi-interpretable}, we obtain the following.

\begin{theorem}\label{thm: complete system}
Let $K$ be a field and $D\in\mathfrak{D}_{n\text{-var}}$. If $\Th(K)\in D$, then $\mathcal{SG}(K)\in D$.
\end{theorem}

\begin{corollary}\label{cor: complet system corollary}
Let $K$ be a field.
\begin{itemize}
\item[(i)] If $T$ is stable, then $\Th(SG(K))$ is stable
(previously proved in \cite{dKN21}).
\item[(ii)] If $T$ is simple, then $\Th(SG(K))$ is simple
(previously proved in \cite{dKN21}).
\item[(iii)] If $T$ is NIP, then $\Th(SG(K))$ is NIP
(previously proved in \cite{dKN21}).
\item[(iv)] If $T$ is NTP$_1$, then $\Th(SG(K))$ is NTP$_1$
(previously proved in \cite{AKLL25}).
\item[(v)] If $T$ is NTP$_2$, then $\Th(SG(K))$ is NTP$_2$
(previously proved in \cite{AKLL25}).
\item[(vi)] If $T$ is NATP, then $\Th(SG(K))$ is NATP
(previously proved in \cite{AKLL25}).
\item[(vii)] If $T$ is NCTP, then $\Th(SG(K))$ is NCTP.
\item[(viii)] If $T$ is NBTP, then $\Th(SG(K))$ is NBTP.
\item[(ix)] If $T$ is NWP, then $\Th(SG(K))$ is NWP.
\item[(x)] If $T$ is NGP, then $\Th(SG(K))$ is NGP.
\item[(xi)] For each $1<k<\omega$, if $T$ is NPM$^{(k)}$,
then $\Th(SG(K))$ is NPM$^{(k)}$.
\end{itemize}
\end{corollary}

\begin{theorem}
Let $K$ be a field and $D\in\mathfrak{D}_{n\text{-var}}$. Then $\Th(K)\in D$ if and only if
$T^K_{\mathrm{VS},\infty}\in D$.
\end{theorem}

\begin{corollary}
Let $K$ be a field.
\begin{itemize}
\item[(i)] $T$ is stable if and only if
$T^K_{\mathrm{VS},\infty}$ is stable
(previously proved in \cite{CH24}).
\item[(ii)] $T$ is simple if and only if
$T^K_{\mathrm{VS},\infty}$ is simple
(previously proved in \cite{CH24}).
\item[(iii)] $T$ is NIP if and only if
$T^K_{\mathrm{VS},\infty}$ is NIP
(previously proved in \cite{CH24}).
\item[(iv)] $T$ is NTP$_1$ if and only if
$T^K_{\mathrm{VS},\infty}$ is NTP$_1$
(follows from \cite{AKLL25} and \cite{CH24}).
\item[(v)] $T$ is NTP$_2$ if and only if
$T^K_{\mathrm{VS},\infty}$ is NTP$_2$
(follows from \cite{AKLL25} and \cite{CH24}).
\item[(vi)] $T$ is NATP if and only if
$T^K_{\mathrm{VS},\infty}$ is NATP
(follows from \cite{AKLL25} and \cite{CH24}).
\item[(vii)] $T$ is NCTP if and only if
$T^K_{\mathrm{VS},\infty}$ is NCTP.
\item[(viii)] $T$ is NBTP if and only if
$T^K_{\mathrm{VS},\infty}$ is NBTP.
\item[(ix)] $T$ is NWP if and only if
$T^K_{\mathrm{VS},\infty}$ is NWP.
\item[(x)] $T$ is NGP if and only if
$T^K_{\mathrm{VS},\infty}$ is NGP.
\item[(xi)] For each $1<k<\omega$, $T$ is NPM$^{(k)}$
if and only if $T^K_{\mathrm{VS},\infty}$ is NPM$^{(k)}$.
\end{itemize}
\end{corollary}

The following phenomenon is observed by Yvon Bossut.
\begin{remark}[Bossut]
Let $D\in\mathfrak{D}_{n\text{-var}}$. If $T\in D$ for some complete theory $T$ of fields of characteristic $p$, possibly with extra structure, then $\mathrm{ACF}_p\in D$.
\end{remark}
This observation is interesting since it is difficult to think of a dividing line that is characterized by a straight definition and contains no theory of fields. So it is natural to ask the following question.
\begin{question}
Does every $D\in\mathfrak{D}_{n\text{-var}}$ contain ACF? Or, is there a dividing line $D\in\mathfrak{D}_{n\text{-var}}$ that does not contain any theory of fields?
\end{question}
Or we can ask a more general question.
\begin{question}\label{question Dstable minimal}
Does every $D\in\mathfrak{D}_{n\text{-var}}$ contain $D_{\text{stable}}$?
\end{question}
If \Cref{question Dstable minimal} has a positive answer, then $D_{\text{stable}}$ would be the least element of $\mathfrak{D}_{n\text{-var}}$ with respect to inclusion.

\section{Generic non-commutative derivations on algebraically bounded fields}\label{sec: generic derivation fields}

In this section, we prove preservation results for expansions by generic non-commutative derivations $T^{\bar\delta,nc}_g$. We first consider the case of a single derivation and then extend the argument to the cases of several non-commutative derivations of arbitrary cardinality.

We recall basic notions about generic derivation from \cite{FT26} and \cite{KK26}. 
Fix $\CL\supseteq\CL_\text{ring}$ and a complete $\CL$-theory $T$ whose restriction to $\CL_\text{ring}$ is a theory of fields of characteristic $0$.

\begin{definition}
We say $T$ is {\it algebraically bounded} if
\[a\in\acl_{\CL}(KB)\;\;\Leftrightarrow\;\; \text{trdeg}(a|K(B))=0\]
for any $K\prec\mathbb{K}\models T$, $a\in \mathbb{K}$, and $B\subseteq\mathbb{K}$.

For a field $F$, we say a map $f:F\to F$ is a {\it derivation} if it satisfies
\begin{itemize}
\item[(i)] $f(x+y)=f(x)+f(y)$,
\item[(ii)] $f(xy)=f(x)y+xf(y)$.
\end{itemize}
\end{definition}

\begin{notation}
Let $\delta$ be a new unary function symbol, $\CL^\delta:=\CL\cup\{\delta\}$, and \[T^\delta:=T\cup\{\forall xy(\delta(x+y)=\delta(x)+\delta(y)),\forall xy(\delta(xy)=\delta(x)y+x\delta(y))\}.\]
\end{notation}

\begin{definition}
Let $K\models T^\delta$.
We say $\delta$ is {\it generic} if for all $\CL(K)$-definable $X\subseteq K^{r+1}$, if the projection  of $X$ onto the first $r$ coordinates has dimension $r$, then there is $a\in K$ such that $(a,\delta(a),...,\delta^r(a))\in X$. Let $T^\delta_g:=T^\delta+\text{``}\delta\text{ is generic"}$. If $T$ is algebraically bounded, then the genericity condition
is first-order axiomatizable \cite{FT26}.
\end{definition}

\begin{remark}
Let $K\prec\mathbb{K}\models T$. Note that if $T$ is algebraically bounded, then $\acl_\CL$ has the exchange property over $K$. That is, $a\in\acl_\CL(KBc)\setminus\acl_\CL(KB)$ implies $c\in\acl_\CL(KBa)$ for all $a,c\in \mathbb{K}$ and $B\subseteq \mathbb{K}$.
\end{remark}

\begin{fact}\cite[Corollary 3.18]{FT26}\label{fact: generic derivative}
Let $T$ be a complete algebraically bounded theory of fields of characteristic $0$, possibly with additional structure. Suppose $(K,\delta)\prec(\mathbb K,\delta)\models T_g^\delta$.
For every $\CL^\delta(K)$-formula $\varphi(\bar{x})$, there is an $\CL(K)$-formula $\psi(\bar{x}_0,...,\bar{x}_r)$ such that
\[(\mathbb{K},\delta)\models \forall \bar{x}(\varphi(\bar{x})\leftrightarrow \psi(\bar{x},\delta(\bar{x}),...,\delta^r(\bar{x}))).\]
\end{fact}

\begin{notation}
For a variable $x$ and $r<\omega$, $\nabla^r(x):=(x,\delta(x),...,\delta^r(x))$. Let $\delta^0(x):=x$ and $\nabla^0(x):=x$. We use the similar notation for a parameter $a$ and $r<\omega$.
\end{notation}

\begin{lemma}\label{lem: generic derivative}
Assume $T$ is algebraically bounded. Suppose $(\mathbb{K},\delta)$ is a saturated model of $T^\delta_g$. Let $\kappa$ be a small cardinal and $\Phi((x_i)_{i<\kappa})$ a partial type in $\CL$, over a small set $A\subseteq \mathbb{K}$. Suppose that there exist $(b_i)_{i<\kappa}\models\Phi$ and $n<\min\{\omega,\kappa\}$ such that $b_i\notin\acl_\CL(Ab_{<i})$ for all $i<n$. Then there exists $(c_i)_{i<\kappa}\models\Phi$ such that $c_i=\delta^i(c_0)$ for all $i\le n$.
\begin{proof}
Let $n<\omega$. It is enough to show that $\Phi(x_0,\delta(x_0),...,\delta^n(x_0),(x_i)_{n<i<\kappa})$ is consistent. Let $\Phi_0$ be a finite subset of $\Phi$ and $\varphi:=\bigwedge \Phi_0$. We may assume $\varphi$ is of the form $\varphi(x_0,...,x_n,x_{i_0},...,x_{i_{m-1}})$ for some $n<i_0<\cdots<i_{m-1}<\kappa$.

 By the assumption, $\exists x_n x_{i_0}...x_{i_{m-1}} \varphi(x_0,...,x_n,x_{i_0},...,x_{i_{m-1}})$ has dimension $n$. Thus there exists $c$ such that $\nabla^n(c)\models \exists x_{i_0}...x_{i_{m-1}} \varphi(x_0,...,x_n,x_{i_0},...,x_{i_{m-1}})$. There exists $(c_{i_k})_{k<m}$ such that $\nabla^n(c)(c_{i_k})_{k<m}\models \varphi(x_0,...,x_n,x_{i_0},...,x_{i_{m-1}})$. By compactness, $\Phi(x_0,\delta(x_0),...,\delta^n(x_0),(x_i)_{n<i<\kappa})$ is consistent.
\end{proof}
\end{lemma}

\begin{theorem}\label{thm: T^delta_g preservation}
Assume $T$ is algebraically bounded. If $T\in D$ for some $D\in\mathfrak{D}_{n\text{-var}}$, then every (some) completion of $T^\delta_g$ belongs to $D$.
\begin{proof}
Assume $D\in\mathfrak{D}_{n\text{-var}}$ and $T\in D$. Suppose $(K,\delta)\prec(\mathbb K,\delta)\models T_g^\delta$ and $(\mathbb{K},\delta)$ is $|K|^+$-saturated. To get a contradiction, suppose that $Th_{\CL^\delta}(\mathbb{K},\delta)\notin D$. By \Cref{fact: generic derivative}, there exist $n<\omega$, $r_0,...,r_{n-1}<\omega$, 
\[
\psi(x_0^0,...,x_0^{r_0},...,x_{n-1}^0,...,x_{n-1}^{r_{n-1}},y)\in\CL,
\] 
and indiscernible $(a_\eta)_{\eta\in\CI^*}$ over $K$ with $|a_\eta|=|y|$ and $|x_i^j|=1$ for all 
$i<n$, $j\le r_i$, such that \[\psi(\nabla^{r_0}(x_0),...,\nabla^{r_{n-1}}(x_{n-1}),y)\] witnesses $Th_{\CL^\delta}(\mathbb{K},\delta)\notin D$ with $(a_\eta)_{\eta\in\CI^*}$. Since we assume $T\in D$, $r_i\neq 0$ for some $i<n$. We may assume that $n$ is minimal. That is, there is no witness of $Th_{\CL^\delta}(\mathbb{K},\delta)\notin D$ of the form
\[\psi'(\nabla^{s_0}(x_0),...,\nabla^{s_{n'-1}}(x_{{n'-1}}),y)\]
for $\psi'\in\CL$ and $n'<n$.
 Define a linear order $\sqsubset$ on $\omega^{n}$ by $(s_0,...,s_{n-1})\sqsubset(r_0,...,r_{n-1})$ if there exists $i< n$ such that $s_i<r_i$ and $s_j=r_j$ for all $j>i$. Since there is no infinite descending $\sqsubset$-chain in $\omega^n$, we may assume that $(r_0,...,r_{n-1})$ is minimal with respect to $\sqsubset$. That is, there is no witness of $Th_{\CL^\delta}(\mathbb{K},\delta)\notin D$ of the form
\[\psi'(\nabla^{s_0}(x_0),...,\nabla^{s_{n-1}}(x_{{n-1}}),y)\]
for $\psi'\in\CL$ and $(s_0,...,s_{n-1}) \sqsubset (r_0,...,r_{n-1})$.

\medskip

\noindent\underline{Claim.}
There exists $(b_0^0,...,b_0^{r_0},...,b_{n-1}^0,...,b_{n-1}^{r_{n-1}})\models\{\psi(\bar{x},a_\eta)\}_{\eta\in X}$ such that \[b_i^j\notin\acl_\CL((a_\eta)_{\eta\in X} (b_{i'}^{j'})_{i'<i}^{j'\le r_{i'}} (b_i^{j'})^{j'<j})\] for all $i<n$ and $j<r_i$.

\smallskip

\noindent{\it Proof of Claim.}
Suppose not. Choose any $(\nabla^{r_0}(b_0),...,\nabla^{r_{n-1}}(b_{n-1}))\models\{\psi(\bar{x},a_\eta)\}_{\eta\in X}$. By algebraic boundedness, there exist $i<n$, $j<r_i$, a finite subset $X_0$ of $X$, and a polynomial \[
p(x_0^0,...,x_0^{r_0},...,x_{i-1}^0
,...,x_{i-1}^{r_{i-1}},x_i^0,...,x_i^j)\in  K((a_\eta)_{\eta\in X_0})[\bar{x}]\]
such that
\[
p=\sum_{e\le d}q_e(x^0_0,...,x^{r_0}_0,...,x^0_{i-1},...,x^{r_{i-1}}_{i-1},x^0_i,...,x^{j-1}_i)(x^j_i)^e
\]
for some $0<d<\omega$ and $q_0,...,q_d\in K((a_\eta)_{\eta\in X_0})[\bar x]$, and
\[
(\nabla^{r_0}(b_0),...,\nabla^{j}(b_{i}))\models  p(\bar{x})=0 \wedge {\partial p\over \partial x^j_i}(\bar x)\neq 0.
\]
Let
\[
\psi'(\bar{x},y):=\psi(\bar{x},y)\wedge p(\bar{x})=0\wedge {\partial p\over \partial x^j_i}(\bar x)\neq 0.
\]
Since $\psi'\in\CL\big(K((a_\eta)_{\eta\in X_0})\big)$, $\psi'(\nabla^{r_0}(x_0),...,\nabla^{r_{n-1}}(x_{n-1}),y)$ witnesses $Th_{\CL^\delta}(\mathbb{K},\delta)\notin D$ by the cofinality of $X$ in $\CI^*$.

By applying $\delta$ on $p$ $(r_i\!-j)$-times, we obtain a rational function 
\[
s(x^0_0,...,x^{r_0+r_i-j}_0,...,x^0_{i-1},...,x^{r_{i-1}+r_i-j}_{i-1},x^0_i,...,x^{r_i-1}_i)\in  K\la(a_\eta)_{\eta\in X_0}\ra_\delta(\bar x)
\]
such that
\[
\vdash \forall x_0^0...x^0_{i} \Big( p\big(\overline{\nabla}x_{\le i}\big)=0\wedge {\partial p\over \partial x^j_i}\big(\overline{\nabla}x_{\le i}\big)\neq 0\rightarrow \delta^{r_i}(x^0_i)=s\big(\overline{\nabla}^*\hspace{-5pt}x_{\le i}\big)\Big)
\]
where 
\begin{itemize}
\item[] \makebox[\linewidth][l]{$\displaystyle
\hspace{21pt}\overline{\nabla}x_{\le i}:=(\nabla^{r_0}(x_0),...,\nabla^{r_{i-1}}(x_{i-1}),\nabla^j(x_i)),
$}\vspace{2pt}
\item[] \makebox[\linewidth][l]{$\displaystyle
\hspace{21pt}\overline{\nabla}^*\hspace{-5pt}x_{\le i}:=(\nabla^{r_0+r_i-j}(x_0),...,\nabla^{r_{i-1}+r_i-j}(x_{i-1}),\nabla^{r_i-1}(x_i)),
$}\vspace{1pt}
\end{itemize}
and $K\la (a_\eta)_{\eta\in X_0}\ra_\delta$ is the smallest subfield of $\mathbb K$ containing $K\cup (a_\eta)_{\eta\in X_0}$
and closed under $\delta$. Let \vspace{3pt}
\begin{itemize}
\item[] \makebox[\linewidth][l]{$\displaystyle
\hspace{21pt}\overline{\nabla}\bar x^{\delta^{r_i}}_x\hspace{-3pt}:=(\nabla^{r_0}(x_0),...,\nabla^{r_{i-1}}(x_{i-1}),\nabla^{r_i-1}(x_i),x,\nabla^{r_{i+1}}(x_{i+1}),...,\nabla^{r_{n-1}}(x_{n-1}))
$}\vspace{1pt}
\end{itemize}
and $\psi^*(x_0,...,x_{n-1},y)$ be an $\CL^\delta((a_\eta)_{\eta\in X_0})$-formula given by
\[
\psi^*:=\exists x\bigg(\psi'\Big(\overline{\nabla}\bar{x}^{\delta^{r_i}}_x\!\!,y\Big)
\wedge\bigg(\bigwedge_{\eta\in X_0} \psi'\Big(\overline{\nabla}\bar{x}^{\delta^{r_i}}_x\!\!,a_\eta\Big)\bigg)
\wedge
\Big(x=s\big(\overline{\nabla}^*\hspace{-5pt}x_{\le i}\big)\Big)
\bigg).
\]
Note that 
\[
\psi^*(x_0,...,x_{n-1},y):=\psi^{**}(\nabla^{s_0}(x_0),...,\nabla^{s_{i-1}}(x_{i-1}),\nabla^{r_i-1}(x_i),\nabla^{r_{i+1}}(x_{i+1}),...,\nabla^{r_{n-1}}(x_{n-1}),y)
\]
for some $\CL(K\la(a_\eta)_{\eta\in X_0}\ra_\delta)$-formula $\psi^{**}$ and $s_0,...,s_{i-1}\in\omega$. Since $(s_0,...,s_{i-1},r_i-1,r_{i+1},...,r_{n-1})\sqsubset(r_0,...,r_{n-1})$ and we assume the minimality of $(r_0,...,r_{n-1})$, it is enough to show that $\psi^*$ witnesses $Th_{\CL^\delta}(\mathbb{K},\delta)\notin D$. 

Since $X$ is cofinal in $\CI^*$, there exists an $\CL^*$-embedding $f:\CI^*\to \CI^*$ such that
\begin{itemize}
\item[(i)] $f(\CI^*)\cap X_0=\emptyset$,
\item[(ii)] $\bar{\eta}\sim_{\CL^*}\bar{\nu} \Rightarrow \bar{\eta} X_0\sim_{\CL^*}\bar{\nu}X_0$ for all $\bar{\eta},\bar{\nu}\in f(\CI^*)$ with $|\bar{\eta}|=|\bar{\nu}|<\omega$,
\item[(iii)] $(f(\CI^*)\cap X)\sim_{\CL^*}^\text{fin}X$.
\end{itemize}
For each $\eta\in\CI^*$, let $a^*_\eta:=a_{f(\eta)}$. Clearly $(a^*_\eta)_{\eta\in\CI^*}$ is indiscernible and \[\psi^*(x_0,...,x_{n-1},a^*_\eta)\vdash\psi(\nabla^{r_0}(x_0),...,\nabla^{r_{n-1}}(x_{n-1}),a^*_\eta)\] for each $\eta\in\CI^*$. Thus $\{\psi^*(x_0,...,x_{n-1},a^*_\nu)\}_{\nu\in Y}$ is inconsistent for each $Y\in \bar{Y}$. Let $X_1$ be a finite subset of $X$. By compactness, it is enough to show that $\{\psi^*(x_0,...,x_{n-1},a^*_\eta)\}_{\eta\in X_1}$ is consistent. Equivalently, it is enough to show that $\{\psi^*(x_0,...,x_{n-1},a_\eta)\}_{\eta\in f(X_1)}$ is consistent.

By (ii) and (iii), there exists $X'_1\subseteq f(\CI^*)\cap X$ such that  $X'_1X_0\sim_{\CL^*}f(X_1)X_0$. Since $X'_1X_0\subseteq X$, there exist $b_0,...,b_{n-1}$ such that $\models\{\psi'(\nabla^{r_0}(b_0),...,\nabla^{r_{n-1}}(b_{n-1}),a_\eta)\}_{\eta\in X'_1 X_0}$. By the choice of $\psi^*$,  we have $\models\{\psi^*(b_0,...,b_{n-1},a_\eta)\}_{\eta\in X'_1}$. Note that $\{\psi^*(x_0,...,x_{n-1},a_\eta)\}_{\eta\in X'_1}$ is definable over $K(a_\eta)_{\eta\in X'_1X_0}$ and $(a_\eta)_{\eta\in X'_1}\equiv_{K(a_\eta)_{\eta\in X_0}}(a_\eta)_{\eta\in f(X_1)}$. There exists an automorphism over $K(a_\eta)_{\eta\in X_0}$  sending $(a_\eta)_{\eta\in X'_1}$ to $(a_\eta)_{\eta\in f(X_1)}$. Thus $\models\{\psi^*(x_0,...,x_{n-1},a_\eta)\}_{\eta\in f(X_1)}$ is consistent. This proves the claim. $\dashv$

\medskip

Since we assume $T\in D$, by the same argument in \Cref{prop: n-var theorem exchange}, there exist $Y\in\bar{Y}$ and 
\[(b_0^0,...,b_0^{r_0},...,b_{n-1}^0,...,b_{n-1}^{r_{n-1}})\models\{\psi(\bar{x},a_\eta)\}_{\eta\in Y}\] 
such that 
\begin{itemize}
\item[(iv)] \makebox[\linewidth][c]{$\displaystyle
  b_i^j\notin\acl((a_\eta)_{\eta\in Y} (b_{i'}^{j'})_{i'<i}^{j'\le r_{i'}} (b_i^{j'})^{j'<j})
$}
\end{itemize}
 for all $i<n$ and $j<r_i$. Let 
\[I\!:=\!\{\xi\!\in\!\omega^{< r_0+\cdots+r_{n-1}+n+1}\!:\emptyset\lhd\xi\text{ and } \xi(r_0\!+\!\cdots\!+\!r_i\!+\!i)\!=\!0 \text{ for all }i\!<\!n\text{ with }r_0\!+\!\cdots\!+\!r_i\!+\!i\!\in\!\dom(\xi)\}\]
and
\[I_i:=\{\xi\in I: \xi\unlhd \la0\ra^{r_0+\cdots+r_i+i+1}\text{ or } \la0\ra^{r_0+\cdots+r_i+i+1}\lhd \xi\}\]
for each $i<n$.
Let
\[
\Psi((x_\xi)_{\xi\in I},y):=\{\psi(x_{\xi_0},...,x_{\xi_{r_0+\cdots+r_{n-1}+n-1}},y):\emptyset\lhd\xi_0\lhd\cdots\lhd\xi_{r_0+\cdots+r_{n-1}+n-1}\}\cup\{x_\xi\neq x_{\xi'}\}_{\xi <_{ds} \xi'}
\]
and
\[
\Psi_i((x_\xi)_{\xi\in I_i},y):=\{\psi(x_{\xi_0},...,x_{\xi_{r_0+\cdots+r_{n-1}+n-1}},y):\emptyset\lhd\xi_0\lhd\cdots\lhd\xi_{r_0+\cdots+r_{n-1}+n-1}\}\cup\{x_\xi\neq x_{\xi'}\}_{\xi<_{ds}\xi'}
\]
for each $i<n$.
Then $\bigcup_{\eta\in Y}\Psi((x_\xi)_{\xi\in I},a_\eta)$ is consistent by (iv). In particular, $\bigcup_{\eta\in Y}\Psi_0((x_\xi)_{\xi\in I_0},a_\eta)$ is realized by $(b_\xi)_{\xi\in I_0}$ such that $b_{\la0\ra^i}\notin\acl((a_\eta)_{\eta\in Y}b_{\la0\ra^{<i}})$ for all $0<i\le r_0$. By \Cref{lem: generic derivative}, there exists $(c_\xi)_{\xi\in I_0}\models \bigcup_{\eta\in Y}\Psi_0((x_\xi)_{\xi\in I_0},a_\eta)$ such that $c_{\la 0\ra^i}=\delta^{i-1}(c_{\la0\ra})$ for all $0<i< r_0+2$. By the s-modeling property, we may assume that $(c_\xi)_{\xi\unrhd \la0\ra^{r_0+1}}$ is s-indiscernible over $(a_\eta)_{\eta\in Y}\nabla^{r_0}c_{\la 0\ra}$. Thus \[c_{\la0\ra^{r_0+2+j}}\notin\acl((a_\eta)_{\eta\in Y}\nabla^{r_0}(c_{\la 0\ra})c_{\la0\ra^{<r_0+2+j}})\] for all $j<r_1$. By \Cref{lem: generic derivative} again, there exists $(c'_\xi)_{\xi\in I_1}\models \cup_{\eta\in Y}\Psi_1((x_\xi)_{\xi\in I_1},a_\eta)$ such that $c'_{\la0\ra^i}=c_{\la0\ra^i}$ for all $0<i<r_0+2$ and $c'_{\la0\ra^i}=\delta^{i-r_0-2}(c'_{\la0\ra^{r_0+2}})$ for all $r_0+1<i<r_0+r_1+3$. By repeating this, we can find a realization of $\bigcup_{\eta\in Y}\{\psi(\nabla^{r_0}x_0,...,\nabla^{r_{n-1}}x_{n-1},a_\eta)\}$ which yields a contradiction.
\end{proof}
\end{theorem}

\begin{corollary}
Let $T$ be a complete algebraically bounded theory of fields of characteristic $0$, possibly with additional structure.
\begin{itemize}
\item[(i)] $T$ is stable if and only if every (some) completion of $T^\delta_g$ is stable
(previously proved in \cite{FT26}).
\item[(ii)] $T$ is simple if and only if every (some)  completion of $T^\delta_g$ is simple
(previously proved in \cite{LSM25}).
\item[(iii)] $T$ is NIP if and only if every (some)  completion of $T^\delta_g$ is NIP
(previously proved in \cite{FT26}).
\item[(iv)] $T$ is NTP$_1$ if and only if every (some)  completion of  $T^\delta_g$ is NTP$_1$
(previously proved in \cite{LSM25}).
\item[(v)] $T$ is NTP$_2$ if and only if every (some)  completion of  $T^\delta_g$ is NTP$_2$
(previously proved in \cite{KK26}).
\item[(vi)] $T$ is NATP if and only if every (some)  completion of  $T^\delta_g$ is NATP
(previously proved in \cite{KK26}).
\item[(vii)] $T$ is NCTP if and only if every (some)  completion of  $T^\delta_g$ is NCTP.
\item[(viii)] $T$ is NBTP if and only if every (some)  completion of  $T^\delta_g$ is NBTP.
\item[(ix)] $T$ is NWP if and only if  every (some)  completion of $T^\delta_g$ is NWP.
\item[(x)] $T$ is NGP if and only if  every (some)  completion of $T^\delta_g$ is NGP.
\item[(xi)] For each $1<k<\omega$, $T$ is NPM$^{(k)}$
if and only if every (some) completion of  $T^\delta_g$ is NPM$^{(k)}$.
\end{itemize}
\end{corollary}

\subsection{The cases of several non-commutative derivations}\label{subsection: non-commutative}

We extend the previous result to several non-commutative derivations, considering the cases where $|\bar\delta|$ is an arbitrary cardinal. We first briefly recall the notation used by Fornasiero and Terzo \cite{FT26}.

As in the beginning of this section, we fix $\CL\supseteq\CL_\text{ring}$ and a complete $\CL$-theory $T$ whose restriction to $\CL_\text{ring}$ is a theory of fields of characteristic $0$. We additionally assume that $T$ is algebraically bounded. Let $\delta_0,...\delta_{k-1}$ be finitely many unary function symbols, $\bar\delta:=(\delta_0,...,\delta_{k-1})$, and $\CL^{\bar\delta}:=\CL\cup\{\delta_0,...,\delta_{k-1}\}$. We \linebreak will explain that all the arguments in this subsection can be extended to the cases of $|\bar\delta|=\lambda$ where $\lambda$ is an arbitrary cardinal, in \Cref{rmk: Walsberg}. 

\begin{definition}\label{def: monoid}
Let $\Gamma$ be the monoid freely generated by $\delta_0,...,\delta_{k-1}$. We form words by successively adding generators on the left, so that a word of length $q$ is written as $\delta_{m_{q-1}}\cdots\delta_{m_0}$, for some $m_{q-1},...,m_0\!<k$.
For words $\theta_0,\theta_1$ in $\Gamma$, we write $\theta_0\preceq\theta_1$ if there exists $\theta$ such that $\theta_1=\theta\theta_0$. Then $\preceq$ is a partial order on $\Gamma$. 
We can extend $\preceq$ to a total order $\preceq^*$ such that $\theta_0\preceq^*\theta_1$ for $\theta_0,\theta_1\in\Gamma$ with $\theta_0\not\preceq\theta_1$ and $\theta_1\not\preceq\theta_0$ if
\begin{itemize}
\item[(i)] $|\theta_0|<|\theta_1|$ or
\item[(ii)] $|\theta_0|=|\theta_1|=q$, $\theta_0:=\delta_{m_{q-1}}\cdots\delta_{m_0}$, $\theta_1:=\delta_{l_{q-1}}\cdots\delta_{l_0}$, and  there exists $p< q$ such that $m_p<l_p$ and $m_r=l_r$ for all $r<p$.
\end{itemize}
Note that 
\begin{itemize}
\item[(iii)] if $\theta_0\preceq^*\theta_1$, then $\theta\theta_0\preceq^*\theta\theta_1$ for any $\theta$.
\end{itemize}
We say a subset $\mathcal{V}$ of $\Gamma$ is {\it $\preceq$-initial} if $\theta\preceq\theta'\in \mathcal{V}$ implies $\theta\in \mathcal{V}$.
\end{definition}

\begin{definition}\label{def: T bar delta nc, T bar delta}
Let $T^{\bar\delta}$ be an $\CL^{\bar\delta}$-theory such that
\begin{itemize}
\item[(i)] $T\subseteq T^{\bar\delta}$,
\item[(ii)] $\delta_i$ is a derivation for each $i<k$.
\end{itemize}
\end{definition}

\begin{notation}
For $K\models T^{\bar\delta}$ and $\mathcal{V}\subseteq\Gamma$, let $K^\mathcal{V}:=\{(a_\theta)_{\theta\in\mathcal{V}}:a_\theta\in K\}$. 
For $\mathcal{F}\subseteq \mathcal{V}\subseteq\Gamma$, let $\Pi_{\mathcal{V},\mathcal{F}}:K^\mathcal{V}\to K^\mathcal{F}$ be the map such that $\Pi_{\mathcal{V},\mathcal{F}}((a_\theta)_{\theta\in \mathcal{V}})=(a_\theta)_{\theta\in \mathcal{F}}$.
 For a variable $x$ and a parameter $a\in K$, let $x^\mathcal{V}:=(\theta(x))_{\theta\in \mathcal{V}}$ and $a^\mathcal{V}:=(\theta(a))_{\theta\in \mathcal{V}}$. For natural numbers $n<m$, let $\Pi_{m,n}$ be the map such that $\Pi_{m,n}((a_i)_{i<m})=(a_i)_{i<n}$.
\end{notation}

\begin{definition}\label{def: T delta nc g}
Let $T^{\bar\delta,nc}_g$ be an $\CL^{\bar\delta}$-theory such that
\begin{itemize}
\item[(i)] $T^{\bar\delta}\subseteq T^{\bar\delta,nc}_g$,
\item[(ii)] \texttt{(nc-Wide)}
For a sufficiently saturated model $\mathbb{K}\models T^{\bar\delta,nc}_g$, let $W\subseteq \mathbb{K}^{(k+1)n}$ be $\CL(A)$-definable for some small $A\subseteq \mathbb{K}$. If there exists $(b_i)_{i<n}\in \Pi_{(k+1)n,n}(W)$ such that $b_i\notin\acl_\CL(Ab_{\neq i})$ for all $i<n$, then there exists $\bar c\in\mathbb{K}^n$ such that $(\bar c,\delta_0(\bar c),...,\delta_{k-1}(\bar c))\in W$.
\end{itemize}
\end{definition}

\begin{fact}\cite[Theorem 4.5]{FT26}
$T^{\bar\delta,nc}_g$ is consistent.
\end{fact}

\begin{notation}
For $\mathcal{F}_0,...,\mathcal{F}_{n-1},\mathcal{V}_0,...,\mathcal{V}_{n-1}\subseteq\Gamma$ with $\mathcal{F}_i\subseteq \mathcal{V}_i$ for $i<n$, let $\Pi_{\mathcal{V}_0\times\cdots\times\mathcal{V}_{n-1},\mathcal{F}_0\times\cdots\times\mathcal{F}_{n-1}}$ be the map from $\mathbb{K}^{\mathcal{V}_0}\times\cdots\times\mathbb{K}^{\mathcal{V}_{n-1}}$ to $\mathbb{K}^{\mathcal{F}_0}\times\cdots\times\mathbb{K}^{\mathcal{F}_{n-1}}$  such that 
\[
\Pi_{\mathcal{V}_0\times\cdots\times\mathcal{V}_{n-1},\mathcal{F}_0\times\cdots\times\mathcal{F}_{n-1}}\Big(\big((a_0^{\theta})_{\theta\in \mathcal{V}_0},...,(a_{n-1}^{\theta})_{\theta\in \mathcal{V}_{n-1}}\big)\Big)=\big((a_0^{\theta})_{\theta\in \mathcal{F}_0},...,(a_{n-1}^{\theta})_{\theta\in \mathcal{F}_{n-1}}\big).
\]
\end{notation}

\begin{lemma}\label{lem: nc-deep}
Let $\mathbb{K}$ be a sufficiently saturated model of $T^{\bar\delta,nc}_g$.
Let $n<\omega$ and $\mathcal{V}_0,...,\mathcal V_{n-1}\subset \Gamma$ be finite and $\preceq$-initial. Let $\mathcal{P}_i\subseteq\mathcal{V}_i$ be the set of $\preceq$-maximal elements of $\mathcal{V}_i$, and $\mathcal{F}_i:=\mathcal{V}_i\setminus\mathcal{P}_i$. Let $Z\subseteq \mathbb{K}^{\mathcal{V}_0}\times\cdots\times\mathbb{K}^{\mathcal{V}_{n-1}}$ be $\CL(A)$-definable for some $A\subseteq\mathbb{K}$. If there exists
\[\big((b_0^{\theta})_{\theta\in\mathcal{F}_0},...,(b_{n-1}^{\theta})_{\theta\in\mathcal{F}_{n-1}}\big)\in \Pi_{\mathcal{V}_0\times\cdots\times\mathcal{V}_{n-1},\mathcal{F}_0\times\cdots\times\mathcal{F}_{n-1}}(Z) 
\]
such that 
\[
b_i^{\theta}\notin\acl_\CL\!\big(A\{b_j^{\theta'}\}_{j\neq i\vee\theta'\neq\theta}\big)\]
for all $i<n$ and $\theta\in\mathcal{F}_i$, then there exists $(c_0,...,c_{n-1})\in\mathbb{K}^n$ such that $(c_0^{\mathcal{V}_0},...,c_{n-1}^{\mathcal{V}_{n-1}})\in Z$.
\begin{proof}
Suppose $\varphi\big((x_0^{\theta})_{\theta\in\mathcal{V}_0},...,(x_{n-1}^{\theta})_{\theta\in\mathcal{V}_{n-1}}\big)\in\CL(A)$ defines $Z$. Let
\[
\psi\big((x_{0,0}^{\theta})_{\theta\in\mathcal{V}_0},...,(x_{0,n-1}^{\theta})_{\theta\in\mathcal{V}_{n-1}},
(x_{1,0}^{\theta})_{\theta\in\mathcal{F}_0},...,(x_{1,n-1}^{\theta})_{\theta\in\mathcal{F}_{n-1}},...,
(x_{k,0}^{\theta})_{\theta\in\mathcal{F}_0},...,(x_{k,n-1}^{\theta})_{\theta\in\mathcal{F}_{n-1}}\big)
\]
be the $\CL(A)$-formula given by
\[
\varphi\big((x_{0,0}^{\theta})_{\theta\in\mathcal{V}_0},...,(x_{0,n-1}^{\theta})_{\theta\in\mathcal{V}_{n-1}}\big)
\wedge
\bigwedge_{i<n,j<k,\delta_j\theta\in\mathcal{V}_i}
(x_{0,i}^{\delta_j\theta}=x_{j+1,i}^{\theta}).
\]
Let 
\[
\chi\big((x_{0,0}^{\theta})_{\theta\in\mathcal{F}_0},...,(x_{0,n-1}^{\theta})_{\theta\in\mathcal{F}_{n-1}},...,
(x_{k,0}^{\theta})_{\theta\in\mathcal{F}_0},...,(x_{k,n-1}^{\theta})_{\theta\in\mathcal{F}_{n-1}}\big)
\]
be the $\CL(A)$-formula given by
\[
\exists \bar x_{0,\mathcal{P}}\psi\big(\bar x_{0,\mathcal{P}},\bar x_{0,\mathcal{F}},...,\bar x_{k,\mathcal{F}}\big),
\]
where 
\[
\bar x_{0,\mathcal{P}}:=\big( (x^\theta_{0,0})_{\theta\in\mathcal{P}_0},...,(x^\theta_{0,n-1})_{\theta\in\mathcal{P}_{n-1}}\big)
\] 
and 
\[
\bar x_{i,\mathcal{F}}:=\big( (x^\theta_{i,0})_{\theta\in\mathcal{F}_0},...,(x^\theta_{i,n-1})_{\theta\in\mathcal{F}_{n-1}}\big)
\] 
for each $i\le k$.
Then we can apply \texttt{(nc-Wide)} on $\chi$ and obtain $\bar c:=\big((c_0^\theta)_{\theta\in\mathcal{F}_0},...,(c_{n-1}^\theta)_{\theta\in\mathcal{F}_{n-1}}\big)$ such that 
\[
\models\chi\big(\bar c,\delta_0(\bar c),...,\delta_{k-1}(\bar c)\big).
\]
By the choice of $\varphi$, $\psi$, and $\chi$, there exists $\bar d_{\mathcal{P}}:=\big((d_0^\theta)_{\theta\in\mathcal{P}_0},...,(d_{n-1}^\theta)_{\theta\in\mathcal{P}_{n-1}}\big)$ such that 
\[
\models\psi\big(\bar d_{\mathcal{P}},\bar c,\delta_0(\bar c),...,\delta_{k-1}(\bar c)\big).
\]
For each $i<n$, let
\[
c_i:=
\begin{cases}
c^\emptyset_i & \text{if }\mathcal{F}_i\neq\emptyset\\
d^\emptyset_i & \text{if }\mathcal{F}_i=\emptyset
\end{cases}
\]
Then $(c_0^{\mathcal{V}_0},...,c_{n-1}^{\mathcal{V}_{n-1}})\in Z$.
\end{proof} 
\end{lemma}

\begin{fact}\cite[Theorem 4.5]{FT26}\label{fact: generic derivative infinite}
Let $T$ be a complete algebraically bounded theory of fields of characteristic $0$, possibly with additional structure. Suppose $(K,\bar\delta)\prec(\mathbb K,\bar\delta)\models T_g^{\bar\delta,nc}$.
For every $\CL^{\bar\delta}(K)$-formula $\varphi(\bar{x})$, there exist a finite $\preceq$-initial $\mathcal{V}\subseteq\Gamma$ and an $\CL(K)$-formula $\psi$ such that
\[(\mathbb{K},\bar\delta)\models \forall \bar{x}\big(\varphi(\bar{x})\leftrightarrow \psi(\bar{x}^\mathcal{V})\big).\]
\end{fact}

\begin{remark}\label{rmk: Walsberg}
Walsberg \cite[Fact~6.4]{Wal26} observed that the construction of $T^{\bar\delta,nc}_g$ extends to the case of $\bar\delta=(\delta_i)_{i<\lambda}$, for an arbitrary cardinal $\lambda$. For an arbitrary cardinal $\lambda$ and $\bar\delta=(\delta_i)_{i<\lambda}$, $T^{\bar\delta,nc}_g$ is given by the axioms of $T^{(\delta_{i_0},...,\delta_{i_{m-1}}),nc}_g$ for finite sequences $i_0<\cdots<i_{m-1}<\lambda$. By compactness, $T^{\bar\delta,nc}_g$ is consistent. 

The conclusion of \Cref{lem: nc-deep} also holds in this setting. We will freely use this lemma assuming $|\bar\delta|$ is an arbitrary cardinal.

Moreover, $\preceq^*$ is still a well-ordering of $\Gamma$. Thus there exists an ordinal $\gamma$ and an enumeration $\Gamma=\{\theta_i\}_{i<\gamma}$ such that $\theta_\alpha\preceq^*\theta_\beta$ for all $\alpha<\beta<\gamma$. 
For each finite subset $A$ of $\Gamma$, we can give a map $f_A:\gamma\to 2$ such that $f_A(i)=1$ if and only if $\theta_i\in A$. For finite subsets $A$ and $B$ of $\Gamma$, we write $A\sqsubset B$ if there exists $i<\gamma$ such that $f_A(i)<f_B(i)$ and $f_A(j)=f_B(j)$ for all $j>i$. Then $\sqsubset$ is a well-ordering on the set of all finite subsets of $\Gamma$. 
\end{remark}

\begin{notation}
For $A\subseteq\mathbb K$, let $K\langle A\rangle_\sigma$
be the smallest subfield of $\mathbb K$ containing $K\cup A$
and closed under every derivation occurring in $\sigma$.
\end{notation}

In the proof of \Cref{thm: T^delta_g preservation several nc}, we use almost the same strategy as in \Cref{thm: T^delta_g preservation}. The only part we have to be careful about is the Claim in the proof of \Cref{thm: T^delta_g preservation}. 

\begin{theorem}\label{thm: T^delta_g preservation several nc}
Let $\lambda$ be a cardinal and $\bar\delta:=(\delta_i)_{i<\lambda}$. Assume $T$ is algebraically bounded. If $T\in D$ for some $D\in\mathfrak{D}_{n\text{-var}}$, then every (some) completion of $T^{\bar\delta,nc}_g$ belongs to $D$.
\begin{proof}
Assume $D\in\mathfrak{D}_{n\text{-var}}$ and $T\in D$. Suppose $(K,\bar\delta)\prec(\mathbb K,\bar\delta)\models T_g^{\bar\delta,nc}$ and $(\mathbb{K},\bar\delta)$ is $|K|^+$-saturated. To get a contradiction, suppose that $Th_{\CL^{\bar\delta}}(\mathbb{K},\bar\delta)\notin D$. 
By \Cref{fact: generic derivative infinite}, there exist $n<\omega$, finite $\preceq$-initial $\mathcal{V}_0,...,\mathcal{V}_{n-1}\subseteq\Gamma$, $\psi\in\CL$, and indiscernible $(a_\eta)_{\eta\in\CI^*}$ over $K$ such that $\psi(x_0^{\mathcal{V}_0},...,x_{n-1}^{\mathcal{V}_{n-1}},y)$ witnesses $Th_{\CL^{\bar\delta}}(\mathbb{K},\bar\delta)\notin D$ with $(a_\eta)_{\eta\in\CI^*}$.
 We may assume that $n$ is minimal. That is, there is no witness of $Th_{\CL^{\bar\delta}}(\mathbb{K},\bar\delta)\notin D$ of the form
\[
\psi'(x_0^{\mathcal V'_0},...,x_{n'-1}^{\mathcal V'_{n'-1}},y)
\]
for $\psi'\in\CL$ and $n'<n$.

 For each $i<n$, let $\mathcal{P}_i\subseteq\mathcal{V}_i$ be the set of $\preceq$-maximal elements of $\mathcal{V}_i$ and $\mathcal{F}_i:=\mathcal{V}_i\setminus\mathcal{P}_i$. We give an order $\Subset$ on $n$-sequences of finite subsets of $\Gamma$ such that $(A_0,...,A_{n-1})\Subset (B_0,...,B_{n-1})$ if there exists $i<n$ such that $A_i\sqsubset B_i$ and $A_j=B_j$ for all $j>i$. Then $\Subset$ is a well-ordering on the set of $n$-sequences of finite subsets of $\Gamma$. 
Thus we may assume that $(\mathcal{P}_0,...,\mathcal{P}_{n-1})$ is $\Subset$-minimal. That is, there is no witness of $Th_{\CL^{\bar\delta}}(\mathbb{K},\bar\delta)\notin D$ of the form
\[
\psi'(x_0^{\mathcal V'_0},...,x_{n-1}^{\mathcal V'_{n-1}},y)
\]
for $\psi'\in\CL$ and $(\mathcal P'_0,...,\mathcal P'_{n-1})\Subset (\mathcal P_0,...,\mathcal P_{n-1})$, where $\mathcal P'_i$ is the set of $\preceq$-maximal elements of $\mathcal V'_i$ for each $i<n$.

\medskip

\noindent\underline{Claim.}
There exists $((b_0^\theta)_{\theta\in\mathcal{F}_0},...,(b_{n-1}^\theta)_{\theta\in\mathcal{F}_{n-1}})\models  \Pi_{\mathcal{V}_0\times\cdots\times\mathcal{V}_{n-1},\mathcal{F}_0\times\cdots\times\mathcal{F}_{n-1}}\big(\{\psi(\bar{x},a_\eta)\}_{\eta\in X}\big)$ such that \[
b_i^{\theta}\notin\acl_\CL\!\big((a_\eta)_{\eta\in X}(b_j^{\theta'})_{j< i\vee(j=i\wedge\theta'\prec^*\theta)}\big)\]
for all $i<n$ and $\theta\in\mathcal{F}_i$

\smallskip

\noindent{\it Proof of Claim.} Suppose not. Since $\psi(x_0^{\mathcal{V}_0},...,x_{n-1}^{\mathcal{V}_{n-1}},y)$ witnesses $Th_{\CL^{\bar\delta}}(\mathbb{K},\bar\delta)\notin D$ with $(a_\eta)_{\eta\in\CI^*}$, there exists $(b_0,...,b_{n-1})$ such that $(b_0^{\mathcal{V}_0},...,b_{n-1}^{\mathcal{V}_{n-1}})\models \{\psi(\bar x,a_\eta)\}_{\eta\in X}$. Then there exist $i<n$ and $\theta\in\mathcal{F}_i$ such that 
\[
\theta(b_i)\in\acl_\CL\!\big((a_\eta)_{\eta\in X}(\theta'(b_j))_{j< i\vee(j=i\wedge\theta'\prec^*\theta)}\big).
\] 
There exists finite $X_0\subseteq X$ and a polynomial \[
p\big((x_j^{\theta'})_{j<i,\theta'\in\mathcal{F}_j},(x_i^{\theta'})_{\theta\succeq^*\theta'\in\mathcal{F}_i}\big)\in K( (a_\eta)_{\eta\in X_0})[\bar x]
\]
such that
\[
p=\sum_{e\le d}
 q_e\big(
(x_j^{\theta'})_{j<i,\theta'\in\mathcal{F}_j},
(x_i^{\theta'})_{\theta\succ^*\theta'\in\mathcal{F}_i}
\big)(x_i^\theta)^e
\]
for some $0<d<\omega$ and $q_0,...,q_d\in K((a_\eta)_{\eta\in X_0})[\bar x]$, and
\[
(b_0^{\mathcal{V}_0},...,b_{n-1}^{\mathcal{V}_{n-1}})\models p(\bar x)=0\wedge {\partial p \over \partial x^\theta_i}(\bar x) \neq 0.
\]
Let 
\[
\psi'(\bar x,y):=\psi(\bar x,y)\wedge p(\bar x)=0\wedge {\partial p\over\partial x^\theta_i}(\bar x)\neq 0.
\]
Then by the cofinality, $\psi'$ witnesses $Th_{\CL^{\bar\delta}}(\mathbb{K},\bar\delta)\notin D$.

There exists $\sigma\in\Gamma\setminus\{\emptyset\}$ such that $\sigma\theta\in\mathcal{P}_i$. Let $\mathcal{S}_\sigma$ be the set of all $\tau\in\Gamma$ such that $|\tau|<|\sigma|$ and every component of $\tau$ also appears in $\sigma$.

Let $\mathcal{V}^*_i$ be the $\preceq$-closure of 
\[
(\mathcal{V}_i\!\setminus\!\{\sigma\theta\})\cup\{\sigma\theta':\theta\succ^*\theta'\in\mathcal{F}_i\}\cup \{\tau\theta':\tau\in \mathcal{S}_\sigma,\theta\succ^*\theta'\in\mathcal{F}_i\}.
\] 
For each $j< i$, let $\mathcal{V}^*_j$ be the $\preceq$-closure of 
\[
\mathcal{V}_j\cup\{\sigma\theta':\theta'\in\mathcal{V}_j\}\cup\{\tau\theta':\tau\in \mathcal{S}_\sigma,\theta'\in\mathcal{V}_j\}.
\]
 For each $j\le i$, let $\mathcal{P}_j^*$ be the set of $\preceq$-maximal elements of $\mathcal{V}_j^*$, and $\mathcal{F}^*_j:=\mathcal{V}^*_j\setminus\mathcal{P}_j^*$. 
 
By applying $\sigma$ on $p$, we can find a rational function
\[
s\big(  (x^{\theta'}_j)_{j\le i,\theta'\in\mathcal{V}^*_j}\big)\in K\langle (a_\eta)_{\eta\in X_0}\rangle_\sigma(\bar x)
\]
such that
\[
\vdash \forall x_0^0...x^0_{i} \Big( p\big(\overline{\Gamma}x_{\le i}\big)=0\wedge {\partial p\over \partial x^\theta_i}\big(\overline{\Gamma}x_{\le i}\big)\neq 0\rightarrow \sigma\theta(x^0_i)=s\big(\overline{\Gamma}^*\hspace{-3pt}x_{\le i}\big)\Big)
\]
where 
\begin{itemize}
\item[] \makebox[\linewidth][l]{$\displaystyle
\hspace{77pt}\overline{\Gamma}\hspace{2pt}x_{\le i}:=\big((x_j^{\theta'})_{j<i,\theta'\in\mathcal{F}_j},(x_i^{\theta'})_{\theta\succeq^*\theta'\in\mathcal{F}_i}\big),
$}\vspace{2pt}
\item[] \makebox[\linewidth][l]{$\displaystyle
\hspace{77pt}\overline{\Gamma}^*\hspace{-3pt}x_{\le i}:=\big( (x^{\theta'}_j)_{j\le i,\theta'\in\mathcal{V}^*_j}\big).
$}\vspace{1pt}
\end{itemize}
Let \vspace{1pt}
\begin{itemize}
\item[] \makebox[\linewidth][l]{$\displaystyle
\hspace{77pt}\overline{\Gamma}\bar x^{\sigma\theta}_x\hspace{1pt}:=(x_0^{\mathcal{V}_0},...,x_{i-1}^{\mathcal{V}_{i-1}},x_i^{\mathcal{V}_i\setminus\{\sigma\theta\}},x,x_{i+1}^{\mathcal{V}_{i+1}},...,x_{n-1}^{\mathcal{V}_{n-1}})
$}
\end{itemize}
and $\psi^*(x_0,...,x_{n-1},y)$ be an $\CL^{\bar\delta}((a_\eta)_{\eta\in X_0})$-formula given by
\[
\psi^*:=\exists x\bigg(\psi'\Big(\overline{\Gamma}\bar{x}^{\sigma\theta}_x\!,y\Big)
\wedge\bigg(\bigwedge_{\eta\in X_0} \psi'\Big(\overline{\Gamma}\bar{x}^{\sigma\theta}_x\!,a_\eta\Big)\bigg)
\wedge
\Big(x=s\big(\overline{\Gamma}^*\hspace{-3pt}x_{\le i}\big)\Big)
\bigg).
\]
Note that 
\[
\psi^*(x_0,...,x_{n-1},y):=\psi^{**}(x_0^{\mathcal{V}^*_0},...,x_i^{\mathcal{V}^*_i},x_{i+1}^{\mathcal{V}_{i+1}},...,x_{n-1}^{\mathcal{V}_{n-1}},y)
\]
for some $\CL(K\la(a_\eta)_{\eta\in X_0}\ra_\sigma)$-formula $\psi^{**}$ and $(\mathcal{P}^*_0,...,\mathcal{P}^*_i,\mathcal{P}_{i+1},...,\mathcal{P}_{n-1})\Subset(\mathcal{P}_0,...,\mathcal{P}_{n-1})$ by the choice of $\mathcal{P}^*_0,...,\mathcal{P}^*_i$ and \Cref{def: monoid} (iii). 
As in \Cref{thm: T^delta_g preservation}, $\psi^*$ witnesses $Th_{\CL^{\bar\delta}}(\mathbb{K},\bar\delta)\notin D$ by the cofinality of $X$ in $\CI^*$. But this yields a contradiction with the $\Subset$-minimality of $(\mathcal{P}_0,...,\mathcal{P}_{n-1})$. $\dashv$

\smallskip

The rest of the proof is the same as \Cref{thm: T^delta_g preservation}. We can apply \Cref{lem: nc-deep} and the argument using s-modeling property to obtain a witness of $T\notin D$.
\end{proof}
\end{theorem}

\begin{corollary}
Let $T$ be a complete algebraically bounded theory of fields of
characteristic $0$, possibly with additional structure, and let
$\bar\delta=(\delta_i)_{i<\lambda}$ for an arbitrary cardinal $\lambda$.
Then the following hold.
\begin{itemize}
\item[(i)] $T$ is stable if and only if every (some) completion of $T^{\bar\delta,nc}_g$ is stable (previously proved in \cite{FT26}).
\item[(ii)] $T$ is simple if and only if every (some) completion of $T^{\bar\delta,nc}_g$ is simple (previously proved in \cite{FTII26}).
\item[(iii)] $T$ is NIP if and only if every (some) completion of $T^{\bar\delta,nc}_g$ is NIP (previously proved in \cite{FT26}).
\item[(iv)] $T$ is NTP$_1$ if and only if every (some) completion of $T^{\bar\delta,nc}_g$ is NTP$_1$ (previously proved in
\cite{FKM26}).
\item[(v)] $T$ is NTP$_2$ if and only if every (some) completion of $T^{\bar\delta,nc}_g$ is NTP$_2$.
\item[(vi)] $T$ is NATP if and only if every (some) completion of $T^{\bar\delta,nc}_g$ is NATP.
\item[(vii)] $T$ is NCTP if and only if every (some) completion of $T^{\bar\delta,nc}_g$ is NCTP.
\item[(viii)] $T$ is NBTP if and only if every (some) completion of $T^{\bar\delta,nc}_g$ is NBTP.
\item[(ix)] $T$ is NWP if and only if every (some) completion of $T^{\bar\delta,nc}_g$ is NWP.
\item[(x)] $T$ is NGP if and only if every (some) completion of $T^{\bar\delta,nc}_g$ is NGP.
\item[(xi)] For each $1<k<\omega$, $T$ is NPM$^{(k)}$ if and only if every (some) completion of $T^{\bar\delta,nc}_g$ is NPM$^{(k)}$.
\end{itemize}
\end{corollary}

\begin{question}
Does the construction of $T^{\bar\delta}_g$, generic commutative derivations also preserve $D\in\mathfrak{D}_{n\text{-var}}$?
\end{question}

\section{$n$-variable theorems for higher arity cases}\label{sec: higher arity}

We now show that our framework and several of the results from the previous sections also apply to some higher-arity dividing lines.

\begin{definition}\label{def: higher-arity cofinality}
Let $\CL^*$ be a language, $\CI^*$ an $\CL^*$-structure, and $k$ a natural number.
\begin{itemize}
\item[(i)] We say $\CI^*$ is {\it $k$-monochromatic} if  $\{\la \eta \ra\subseteq \CI^*: \eta\in \CI^*, |\eta|=k\}$ has only one element up to $\CL^*$-isomorphism.
\item[(ii)] We say $\CI^*$ is {\it $k$-weakly monochromatic} if  $\{\la \eta \ra\subseteq \CI^*: \eta\in \CI^*, |\eta|=k\}$ has only finitely many elements up to $\CL^*$-isomorphism.
\item[(iii)] Let $J\subseteq(\CI^*)^k$. We say $\CI^*$ is {\it $k$-monochromatically extendable} over $J$ if for any cardinal $\lambda$, there exists $\CI^\dagger$ such that $\age_{\CL^*}(\CI^\dagger)=\age_{\CL^*}(\CI^*)$ and for any coloring $c:(\CI^\dagger)^k\to\lambda$, there exists an $\CL^*$-embedding $f:\CI^*\to\CI^\dagger$ such that $c(f(\bar\eta))=c(f(\bar\nu))$ for $\bar\eta,\bar\nu\in J$ with $\la\bar\eta\ra\sim_{\CL^*}\la\bar\nu\ra$. When $J=(\CI^*)^k$, we simply say that $\CI^*$ is
$k$-monochromatically extendable. We call such $\CI^\dagger$ a {\it $k$-monochromatic extension} of $\CI^*$ over $J$ in $\lambda$. 
\item[(iv)] We say a subset $X$ of $(\CI^*)^k$ is {\it cofinal} in $\CI^*$ if for any finite subset $Z$ of $\CI^*$, there exists an $\CL^*$-embedding $f:\CI^*\to\CI^*$ such that
\begin{itemize}
\item[$\ast$] $f(\CI^*)\cap Z=\emptyset$,
\item[$\ast$] $\bar{\eta}\sim_{\CL^*}\bar{\nu} \Rightarrow \bar{\eta} Z\sim_{\CL^*}\bar{\nu}Z$ for all $\bar{\eta},\bar{\nu}\in f(\CI^*)$ with $|\bar{\eta}|=|\bar{\nu}|<\omega$,
\item[$\ast$] $(f(\CI^*)^k\cap X)\sim_{\CL^*}^\text{fin}X$.
\end{itemize}
\end{itemize}
\end{definition}

\begin{definition}\label{def: higher-arity n-var quadruple}
Let $\CL^*$ be a language, $\CI^*$ an $\CL^*$-structure, $k\in\omega$, $X$ a subset of $(\CI^*)^k$, and $\bar{Y}$ a finite set of subsets of $(\CI^*)^k$. We say a quintuple $(\CL^*\!\!,\,\CI^*\!\!,\,k,X,\bar{Y})$ satisfies the {\it $n$-variable theorem} if
\begin{itemize}
\item[(i)] $\CI^*$ has the modeling property,
\item[(ii)] $\CI^*$ is $k$-weakly monochromatic,
\item[(iii)] $X$ is cofinal in $\CI^*$,
\item[(iv)] $\CI^*$ is $k$-monochromatically extendable over $\mathbbof{X}:=\{\bar\eta\in\CI^*:\la\bar\eta\ra\sim_{\CL^*}\la\bar\nu\ra\text{ for some }\bar\nu\in X\}$,
\item[(v)] $Y\subseteq\mathbbof{X}$ for each $Y\in\bar{Y}$,
\item[(vi)] $|Y|>1$ for each $Y\in\bar{Y}$,
\item[(vii)] there exists a disjoint union $W\subseteq(\CI^*)^k$ of $X'$ and $Y'$ such that 
\begin{itemize}
\item[$\ast$] $X'\sim^\text{fin}_{\CL^*}X$ and $Y'\sim^\text{fin}_{\CL^*}Y$ for some $Y\in\bar{Y}$,
\item[$\ast$] $\{\nu\}\cup X'\sim^\text{fin}_{\CL^*}X$ for each $\nu\in Y'$.
\end{itemize}
\end{itemize}
\end{definition}

\begin{notation}
Let $\CL^*$ be a language and $\CI^*$ an $\CL^*$-structure. 
For an $\CI^*$-indexed set of parameters $(a_\eta)_{\eta\in\CI^*}$ and a tuple $\bar{\eta}:=(\eta_0,...,\eta_{k-1})\in(\CI^*)^k$, we put $\bar{a}_{\bar{\eta}}:=(a_{\eta_0},...,a_{\eta_{k-1}})$.
\end{notation}

\begin{definition}\label{def: D^h_n-var}
For a given quintuple $(\CL^*\!\!,\,\CI^*\!\!,\,k,X,\bar{Y})$ satisfying the $n$-variable theorem, let
$D^{\CL^*\!\!\!,\,\CI^*}_{k,X,\bar{Y}}$ be the class of all complete theories $T$ such that there do not exist $\varphi(x,\bar{y}):=\varphi(x,y_0,...,y_{k-1})$ in the language of $T$ and $(a_\eta)_{\eta\in\CI^*}$ in a model of $T$ such that
\begin{itemize}
\item[(i)] $(a_\eta)_{\eta\in\CI^*}$ is $\CI^*_{\CL^*}$-indiscernible,
\item[(ii)] $\{\varphi(x,\bar{a}_{\bar{\eta}})\}_{\bar{\eta}\in X}$ is consistent,
\item[(iii)] $\{\varphi(x,\bar{a}_{\bar{\nu}})\}_{\bar{\nu}\in Y}$ is inconsistent for each $Y\in \bar{Y}$.
\end{itemize}

Let $\mathfrak{D}^h_{n\text{-var}}$ be a class of classes of complete theories such that
\[
\mathfrak{D}^h_{n\text{-var}}:=\left\lbrace D^{\CL^*\!\!\!,\,\CI^*}_{k,X,\bar{Y}}: (\CL^*\!\!,\,\CI^*\!\!,k,X,\bar{Y})\text{ satisfies the }n\text{-variable theorem} \right\rbrace.
\]
\end{definition}

\begin{remark}
$\mathfrak{D}_{n\text{-var}}\subseteq\mathfrak{D}^h_{n\text{-var}}$.
\end{remark}

\subsection{NOP$_k$, NFOP$_k$, NIP$_k$, and ${l\over k}$NP$_k$}\label{subsection: l over kNPk}

\begin{notation}
Let $(I,<)$ be a linearly ordered set and $k\in \omega$. For $X\subseteq I^k$ and $0\le l\le k$, let 
\[
\overline{X}^l\!:=\!\bigg\lbrace\!(i_0,...,i_{k-1})\!\in\! I^k: \exists(j_0,...,j_{k-1})\!\in\! X \Big(\forall k'\big(k'<k-l\Rightarrow i_{k'}=j_{k'}\big)\wedge \forall k'\big( k'\ge k-l\Rightarrow i_{k'}\le j_{k'}\big)\!\Big) \!\bigg\rbrace.
\]
We call $\overline{X}^l$ the {\it $l$-th down closure\footnote{The name follows \cite[Definition A.1]{Han25b}.} of $X$}. Note that $\overline{\emptyset}^l=\emptyset$, $\overline{X}^0=X$, and $\overline{ \overline{X}^l}^{l'}=\overline{X}^l$ for $l'<l$. $X$ is said to be {\it $l$-down closed} if $\overline{X}^l=X$. 
We say $X$ is {\it simple $l$-down closed} if there exist finitely many points $(j^0_0,...,j^0_{k-1}),...,(j^{m-1}_0,...,j^{m-1}_{k-1})$ in $I^k$ such that 
\[
X=\bigg\lbrace(i_0,...,i_{k-1})\in I^k: \exists m'<m \Big(\forall k'\big(k'<k-l\Rightarrow i_{k'}=j^{m'}_{k'}\big)\wedge\forall k' \big( k'\ge k-l\Rightarrow i_{k'}\le j^{m'}_{k'}\big)\Big) \bigg\rbrace.
\]
\end{notation}

\vspace{-22pt}

\begin{figure}[H]
\centering
\includegraphics[width=\linewidth]{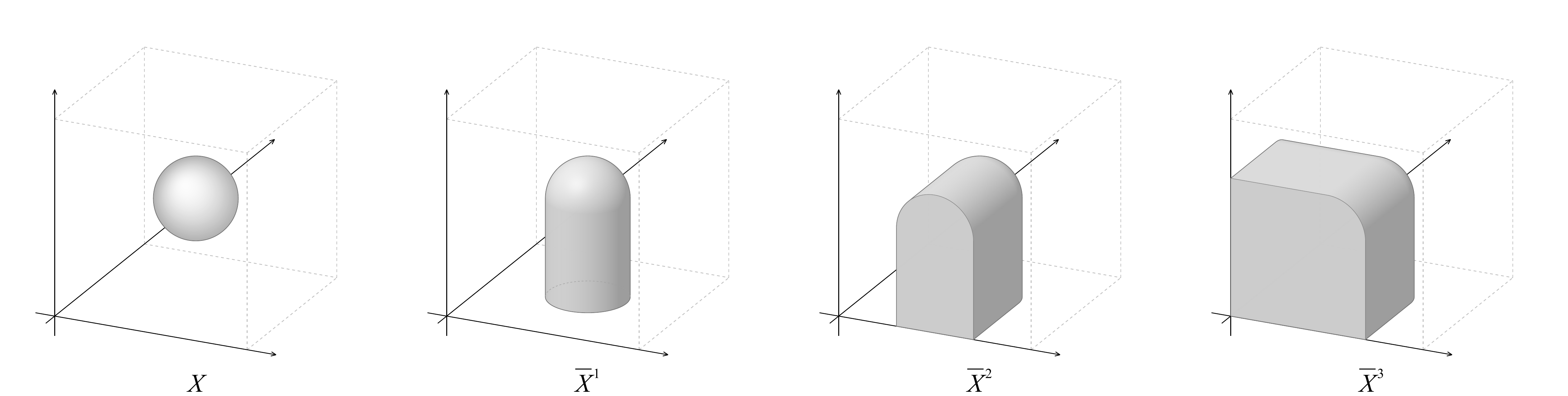}
\vspace{-17pt}
\caption{\footnotesize $X$ and its $l$-th down closures in $(0,1)^3$}
\label{fig: down closures}
\end{figure}

\begin{definition}
Let $0<k<\omega$ and $0\le l \le k$. We say a complete theory $T$ has {\it${l\over k}$P$_k$} if there exist $\varphi(x,y_0,...,y_{k-1})$ and $(a^0_i)_{i<\omega},...,(a^{k-1}_i)_{i<\omega}$ such that for any $X\subseteq \omega^k$, there exists $b_X$ such that 
\[
\models \varphi(b_X,a^0_{i_0},...,a^{k-1}_{i_{k-1}})\;\text{ if and only if }\;(i_0,...,i_{k-1})\in\overline{X}^{k-l}.
\]
We say $T$ is {\it ${l\over k}$NP$_k$} if it does not have ${l\over k}$P$_k$.
\end{definition}
\begin{remark}\label{rmk: m/nPn}
Let $T$ be a complete theory. Let $1<k<\omega$.
\begin{itemize}
\item[(i)] $T$ has OP if and only if $T$ has ${0\over 1}$P$_1$.
\item[(ii)] $T$ has OP$_2$ if and only if $T$ has ${0\over 2}$P$_2$.
\item[(iii)] $T$ has IP if and only if $T$ has ${1\over 1}$P$_1$.
\item[(iv)] $T$ has IP$_k$ if and only if $T$ has ${k\over k}$P$_k$.
\item[(v)] $T$ has FOP$_k$ if and only if $T$ has ${k-1\over k}$P$_k$.
\end{itemize}
\end{remark}

We can extend the definition of OP and OP$_2$ regarding \Cref{rmk: m/nPn} (i) and (ii), as follows.

\begin{definition}
Let $0<k<\omega$. We say $T$ has {\it OP$_k$} if it has ${0\over k}$P$_k$. We say $T$ is {\it NOP$_k$} if it does not have OP$_k$.
\end{definition}

\begin{proposition}\label{prop: l+1 k Pk => l k Pk / l k+1 Pk+1 => l+1 k Pk} 
Let $0<k<\omega$ and $0\le l< k$. Let $T$ be a complete theory.
\begin{itemize}
\item[(i)] If $T$ has ${l+1\over k}$P$_{k}$, then it has ${l\over k}$P$_{k}$.
\item[(ii)] If $T$ has ${l\over k+1}$P$_{k+1}$, then it has ${l+1\over k}$P$_k$.
\end{itemize}
\begin{proof}
(i) Let $\varphi(x,y_0,...,y_{k-1})$, $(a^0_i)_{i<\omega},...,(a^{k-1}_i)_{i<\omega}$ witness ${l+1\over k}$P$_{k}$. Then there exists $(b_X)_{X\subseteq \omega^{k}}$ such that
\[
\models \varphi(b_X,a^0_{i_0},...,a^{k-1}_{i_{k-1}})\;\text{ if and only if }\;(i_0,...,i_{k-1})\in\overline{X}^{k-(l+1)}.
\]
For each $X\subseteq\omega^{k}$, let $b^*_X:=b_{\overline{X}^{k-l}}$. It is easy to check that 
\[
\models \varphi(b^*_X,a^0_{i_0},...,a^{k-1}_{i_{k-1}})\;\text{ if and only if }\;(i_0,...,i_{k-1})\in\overline{X}^{k-l}.
\]

\medskip

\noindent(ii) Suppose $T$ has ${l\over k+1}$P$_{k+1}$. Let $\mathbb{Q}_+$ be the set of all positive rational numbers. By compactness, there exist $\varphi(x,y_0,...,y_k)$, $(a^0_q)_{q\in \mathbb{Q}_+},...,(a^k_q)_{q\in\mathbb{Q}_+}$, and $(b_X)_{X\subseteq\mathbb{Q}_+^{k+1}}$ such that
\[
\models \varphi(b_X,a^0_{q_0},...,a^{k}_{q_{k}})\;\text{ if and only if }\;(q_0,...,q_{k})\in\overline{X}^{k+1-l}
\]
for each $X\subseteq\mathbb{Q}_+^{k+1}$. For each $q\in\mathbb{Q}_+$ and $0<k'<k$, let $c^{k'}_q:=a^{k'-1}_q$ and $c^0_q:=a^{k-1}_qa^k_{1\over q}$. For each $0<k'<k$, let $z_{k'}:=y_{k'-1}$, $z_0:=y_{k-1}y_k$, and $\psi(x,z_0,...,z_{k-1}):=\varphi(x,y_0,...,y_k)$. We show $\psi(x,z_0,...,z_{k-1})$ witnesses ${l+1\over k}$P$_k$ with $(c^0_q)_{q\in\mathbb{Q}_+},...,(c^{k-1}_q)_{q\in\mathbb{Q}_+}$. For each $X\subseteq\mathbb{Q}_+^k$, let 
\[
X':=\{(q_0,...,q_k)\in\mathbb{Q}_+^{k+1}:(q_{k-1},q_0,...,q_{k-2})\in X, q_k={1\over q_{k-1}}\}.
\]
It is easy to check that
\[
\models \psi(b_{X'},c^0_{q_0},...,c^{k-1}_{q_{k-1}})\;\text{ if and only if }\;(q_0,...,q_{k-1})\in\overline{X}^{k-(l+1)}.
\]
Thus $T$ has ${l+1\over k}$P$_k$.
\end{proof}
\end{proposition}

\begin{figure}[H]
\vspace{-15pt}
\[
\scalebox{1}{
\begin{tikzpicture}
[>={stealth[length=1pt,width=6pt]}, x=1.85cm, y=0.9cm]
\node (OP) at (0,0) {OP};
\node (OPfake) at (0.11,-0.05) {};
\node (IP) at (5,0) {IP};
\node (IPfake) at (4.89,-0.05) {};
\node (IPfake2) at (4.9,0.02) {};
\node (IPfake3) at (4.9,0.1) {};
\node (OP2) at (0,1) {OP$_2$};
\node (OP2fake) at (0.19,0.91) {};
\node (FOP2) at (2.5,1) {FOP$_2$};
\node (IP2) at (5,1) {IP$_2$};
\node (IP2fake) at (4.87,1.06) {};
\node (IP2fake2) at (4.87,1.14) {};
\node (OP3)   at (0,2) {OP$_3$};
\node (1/3P3)   at (1.66,2) {${1\over3}$P$_3$};
\node (FOP3)   at (3.33,2) {FOP$_3$};
\node (IP3)   at (5,2) {IP$_3$};
\node (OP4)   at (0,3) {OP$_4$};
\node (1/4P4)   at (1.25,3) {${1\over4}$P$_4$};
\node (2/4P4)   at (2.5,3) {${2\over4}$P$_4$};
\node (FOP4)   at (3.75,3) {FOP$_4$};
\node (IP4)   at (5,3) {IP$_4$};
\node (OP5)   at (0,4) {OP$_5$};
\node (1/5P5)   at (1,4) {${1\over5}$P$_5$};
\node (2/5P5)   at (2,4) {${2\over5}$P$_5$};
\node (3/5P5)   at (3,4) {${3\over5}$P$_5$};
\node (FOP5)   at (4,4) {FOP$_5$};
\node (IP5)   at (5,4) {IP$_5$};
\node (OP6)   at (0,5) {OP$_6$};
\node (1/6P6)   at (0.83,5) {${1\over6}$P$_6$};
\node (2/6P6)   at (1.66,5) {${2\over6}$P$_6$};
\node (3/6P6)   at (2.5,5) {${3\over6}$P$_6$};
\node (4/6P6)   at (3.33,5) {${4\over6}$P$_6$};
\node (FOP6)   at (4.17,5) {FOP$_6$};
\node (IP6)   at (5,5) {IP$_6$};
\node (OP7)   at (0,6) {OP$_7$};
\node (1/7P7)   at (0.71,6) {${1\over7}$P$_7$};
\node (2/7P7)   at (1.42,6) {${2\over7}$P$_7$};
\node (3/7P7)   at (2.14,6) {${3\over7}$P$_7$};
\node (4/7P7)   at (2.85,6) {${4\over7}$P$_7$};
\node (5/7P7)   at (3.54,6) {${5\over7}$P$_7$};
\node (FOP7)   at (4.29,6) {FOP$_7$};
\node (IP7)   at (5,6) {IP$_7$};
\node (OP8)   at (0,7) {OP$_8$};
\node (1/8P8)   at (0.625,7) {${1\over8}$P$_8$};
\node (2/8P8)   at (1.25,7) {${2\over8}$P$_8$};
\node (3/8P8)   at (1.875,7) {${3\over8}$P$_8$};
\node (4/8P8)   at (2.5,7) {${4\over8}$P$_8$};
\node (5/8P8)   at (3.125,7) {${5\over8}$P$_8$};
\node (6/8P8)   at (3.75,7) {${6\over8}$P$_8$};
\node (FOP8)   at (4.375,7) {FOP$_8$};
\node (IP8)   at (5,7) {IP$_8$};
\node (dots)   at (2.5,7.95) {$\vdots$};

\draw[->] (IPfake) -- (OPfake);
\draw[->] (IP2) -- (FOP2);
\draw[->] (FOP2) -- (OP2);
\draw[->] (IP3) -- (FOP3);
\draw[->] (FOP3) -- (1/3P3);
\draw[->] (1/3P3) -- (OP3);
\draw[->] (IP4) -- (FOP4);
\draw[->] (FOP4) -- (2/4P4);
\draw[->] (2/4P4) -- (1/4P4);
\draw[->] (1/4P4) -- (OP4);
\draw[->] (IP5) -- (FOP5);
\draw[->] (FOP5) -- (3/5P5);
\draw[->] (3/5P5) -- (2/5P5);
\draw[->] (2/5P5) -- (1/5P5);
\draw[->] (1/5P5) -- (OP5);
\draw[->] (IP6) -- (FOP6);
\draw[->] (FOP6) -- (4/6P6);
\draw[->] (4/6P6) -- (3/6P6);
\draw[->] (3/6P6) -- (2/6P6);
\draw[->] (2/6P6) -- (1/6P6);
\draw[->] (1/6P6) -- (OP6);
\draw[->] (IP7) -- (FOP7);
\draw[->] (FOP7) -- (5/7P7);
\draw[->] (5/7P7) -- (4/7P7);
\draw[->] (4/7P7) -- (3/7P7);
\draw[->] (3/7P7) -- (2/7P7);
\draw[->] (2/7P7) -- (1/7P7);
\draw[->] (1/7P7) -- (OP7);
\draw[->] (IP8) -- (FOP8);
\draw[->] (FOP8) -- (6/8P8);
\draw[->] (6/8P8) -- (5/8P8);
\draw[->] (5/8P8) -- (4/8P8);
\draw[->] (4/8P8) -- (3/8P8);
\draw[->] (3/8P8) -- (2/8P8);
\draw[->] (2/8P8) -- (1/8P8);
\draw[->] (1/8P8) -- (OP8);
\draw[->] (OP2) -- (OP);
\draw[->] (OP2fake) -- (IPfake2);
\draw[->] (FOP2) -- (IPfake3);
\draw[->] (IP2) -- (IP);
\draw[->] (OP3) -- (OP2);
\draw[->] (OP3) -- (FOP2);
\draw[->] (1/3P3) -- (FOP2);
\draw[->] (1/3P3) -- (IP2fake);
\draw[->] (FOP3) -- (IP2fake2);
\draw[->] (IP3) -- (IP2);
\draw[->] (OP4) -- (OP3);
\draw[->] (OP4) -- (1/3P3);
\draw[->] (1/4P4) -- (1/3P3);
\draw[->] (1/4P4) -- (FOP3);
\draw[->] (2/4P4) -- (FOP3);
\draw[->] (2/4P4) -- (IP3);
\draw[->] (FOP4) -- (IP3);
\draw[->] (IP4) -- (IP3);
\draw[->] (OP5) -- (OP4);
\draw[->] (OP5) -- (1/4P4);
\draw[->] (1/5P5) -- (1/4P4);
\draw[->] (1/5P5) -- (2/4P4);
\draw[->] (2/5P5) -- (2/4P4);
\draw[->] (2/5P5) -- (FOP4);
\draw[->] (3/5P5) -- (FOP4);
\draw[->] (3/5P5) -- (IP4);
\draw[->] (FOP5) -- (IP4);
\draw[->] (IP5) -- (IP4);
\draw[->] (OP6) -- (OP5);
\draw[->] (OP6) -- (1/5P5);
\draw[->] (1/6P6) -- (1/5P5);
\draw[->] (1/6P6) -- (2/5P5);
\draw[->] (2/6P6) -- (2/5P5);
\draw[->] (2/6P6) -- (3/5P5);
\draw[->] (3/6P6) -- (3/5P5);
\draw[->] (3/6P6) -- (FOP5);
\draw[->] (4/6P6) -- (FOP5);
\draw[->] (4/6P6) -- (IP5);
\draw[->] (FOP6) -- (IP5);
\draw[->] (IP6) -- (IP5);
\draw[->] (OP7) -- (OP6);
\draw[->] (OP7) -- (1/6P6);
\draw[->] (1/7P7) -- (1/6P6);
\draw[->] (1/7P7) -- (2/6P6);
\draw[->] (2/7P7) -- (2/6P6);
\draw[->] (2/7P7) -- (3/6P6);
\draw[->] (3/7P7) -- (3/6P6);
\draw[->] (3/7P7) -- (4/6P6);
\draw[->] (4/7P7) -- (4/6P6);
\draw[->] (4/7P7) -- (FOP6);
\draw[->] (5/7P7) -- (FOP6);
\draw[->] (5/7P7) -- (IP6);
\draw[->] (FOP7) -- (IP6);
\draw[->] (IP7) -- (IP6);
\draw[->] (OP8) -- (OP7);
\draw[->] (OP8) -- (1/7P7);
\draw[->] (1/8P8) -- (1/7P7);
\draw[->] (1/8P8) -- (2/7P7);
\draw[->] (2/8P8) -- (2/7P7);
\draw[->] (2/8P8) -- (3/7P7);
\draw[->] (3/8P8) -- (3/7P7);
\draw[->] (3/8P8) -- (4/7P7);
\draw[->] (4/8P8) -- (4/7P7);
\draw[->] (4/8P8) -- (5/7P7);
\draw[->] (5/8P8) -- (FOP7);
\draw[->] (5/8P8) -- (5/7P7);
\draw[->] (6/8P8) -- (FOP7);
\draw[->] (6/8P8) -- (IP7);
\draw[->] (FOP8) -- (IP7);
\draw[->] (IP8) -- (IP7);
\end{tikzpicture}
}
\]
\vspace{-20pt}
\caption{
\footnotesize  
${l\over k+1}$P$_{k+1}$ $\rightarrow$ ${l+1\over k}$P$_k$,\; ${l+1\over k}$P$_{k}$ $\rightarrow$ ${l\over k}$P$_k$}
\label{fig: m/nPn}
\end{figure}

\begin{remark}
Thus we have \Cref{fig: m/nPn} and implications OP$_k \rightarrow$ IP$_{\lfloor{k\over 2}\rfloor}$.
\Cref{fig: m/nPn} and this implication can be compared with Figure B.3 in \cite{Ald25} and the first diagram on page 200 of \cite{Ald25}, respectively.
In his thesis \cite{Ald25}, Abd Aldaim introduces the {\it two-sided functional order property FOP$_{k_1,k_2}$}, another common generalization of FOP$_k$ and IP$_k$. Although its definition is quite different from that of ${l\over k}P_k$, the diagram of implications among the properties FOP$_{k_1,k_2}$ is quite similar to \Cref{fig: m/nPn}. Understanding the relation between these two families would be an interesting direction for future work.
\end{remark}

\begin{notation}\label{notation: universal m-down closed subset}
Fix $1<k<\omega$ and $0\le l\le k$. Put $I:=\{q\in\mathbb{Q}:0<q<1\}$ and let $(X_i)_{i<\omega}$ be a sequence of simple $(k-l)$-down closed subsets of $I^k$ such that for any $N<\omega$ and a simple $(k-l)$-down closed subset $X$ of $I^k$, there exists $i<\omega$ such that $i>N$ and $X_i=X$. Let ${l\over k}X_k$ be a subset of $\mathbb{Q}^k$ given by 
\[
{l\over k}X_k:=\bigcup_{i<\omega}\Big((-i,i,...,i)+X_i\Big),
\] 
where $(-i,i,...,i)+X_i:=\{(x_0-i,x_1+i,...,x_{k-1}+i):(x_0,...,x_{k-1})\in X_i\}$. Let ${l\over k}\hat X_k\subseteq \mathbb{Q}^k$ be the $(k-l)$-th down closure of ${l\over k}X_k$.
\end{notation}

\begin{figure}[H]
\centering
\includegraphics[width=\linewidth]
{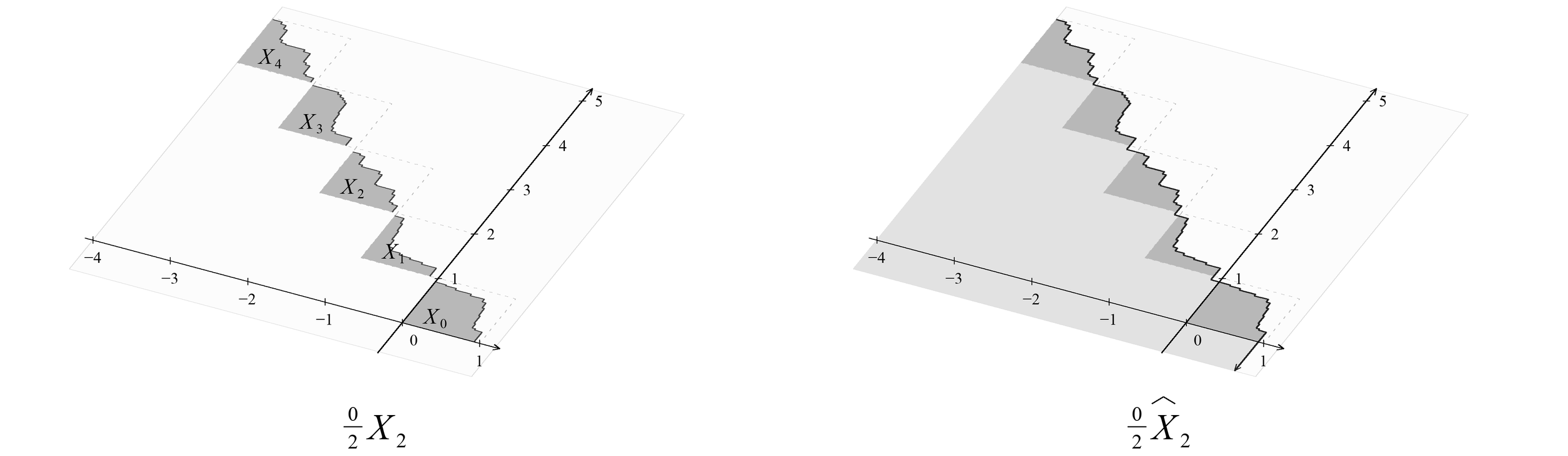}
\vspace{-20pt}
\caption{\footnotesize ${l\over k}X_k$ and ${l\over k}\hat X_k$ when $k=2, l=0$}
\label{fig: l over k hat X k}
\end{figure}

The following statement is immediate from the choice of ${l\over k}\hat X_k$ and mutual indiscernibility of sequences indexed by $k\times \mathbb{Q}$.

\begin{lemma}\label{lem: m/nPn characterized by s-c n-straight pattern}
Let $1<k<\omega$ and $0\le l\le k$. Let ${l\over k}\hat X_k$ be the $(k-l)$-down closed subset of $\mathbb{Q}^k$ defined in \Cref{notation: universal m-down closed subset}. Then the following are equivalent.
\begin{itemize}
\item[(i)] $T$ has ${l\over k}$P$_k$.
\item[(ii)] There exist $\varphi(x,y_0,...,y_{k-1})$ and mutually indiscernible $(a^0_q)_{q\in\mathbb{Q}},...,(a^{k-1}_q)_{q\in\mathbb{Q}}$ such that
\[
\Big\lbrace\varphi(x,a^0_{q_0},...,a^{k-1}_{q_{k-1}}):(q_0,...,q_{k-1})\in {l\over k}\hat X_k\Big\rbrace\cup\Big\lbrace\neg\varphi(x,a^0_{q_0},...,a^{k-1}_{q_{k-1}}):(q_0,...,q_{k-1})\notin {l\over k}\hat X_k\Big\rbrace
\]
is consistent.
\end{itemize}
\end{lemma}

\begin{notation}
Let $<,E$ be binary relation symbols and $P_0,...,P_{k-1}, C$ unary relation symbols, and $\CL_{{x\over k}\text{NP}_k}:=\{<,E,P_0,...,P_{k-1},C\}$. Define an $\CL_{{x\over k}\text{NP}_k}$-structure $\CI_{{x\over k}\text{NP}_k}$ on $k\times\mathbb{Q}\times\mathbb{Q}$ as follows.
\begin{itemize}
\item[(i)] $(i,q,r)<(i',q',r')$ if one of the following holds.
\begin{itemize}
\item[$\ast$] $i<i'$.
\item[$\ast$] $i=i'$ and $q<q'$.
\item[$\ast$] $i=i'$, $q=q'$, and $r<r'$.
\end{itemize}
\item[(ii)] $E((i,q,r),(i',q',r'))$ if $i=i'$ and $q=q'$.
\item[(iii)] $P_j((i,q,r))$ if $i=j$ for each $i,j<k$.
\item[(iv)] $C((i,q,r))$ if $r=0$ or $r={a\over b}$ for $a,b\in\mathbb{Z}$ such that $(a,b)=1$ and $a$ is even.
\end{itemize}
\end{notation}

\begin{lemma}
${\CI_{{x\over k}\text{NP}_k}}$-indiscernibles have the modeling property. 
\begin{proof}
 Let $\CL^-_{{x\over k}\text{NP}_k}:=\CL_{{x\over k}\text{NP}_k}\setminus\{C\}$ and $\CK^-_{{x\over k}\text{NP}_k}$ be the class of finite $\CL^-_{{x\over k}\text{NP}_k}$-structures satisfying
\begin{itemize}
\item[(i)] $<$ is a linear order,
\item[(ii)] $E$ is an equivalence relation,
\item[(iii)] If $E(x,z)$ and $x<y<z$, then $E(x,y)$,
\item[(iv)] $P_0,...,P_{k-1}$ form a partition on the universe,
\item[(v)] If $x\in P_i$, $y\in P_j$ and $i<j<k$, then $x<y$ and $\neg E(x,y)$.
\end{itemize}
It is easy to check that  $\CK^-_{{x\over k}\text{NP}_k}$ is a \Fraisse class with strong amalgamation property.
By \cite[Lemma 1.2, 1.3]{Che13} and \cite[Lemma A.2]{CPT19}, $\CK^-_{{x\over k}\text{NP}_k}$ is Ramsey. Let $\mathcal{K}_{{x\over k}\text{NP}_k}$ be the class of all finite $\CL_{{x\over k}\text{NP}_k}$-structures $A$ such that $A|_{\CL^-_{{x\over k}\text{NP}_k}}\!\!\!\in\mathcal{K}^-_{{x\over k}\text{NP}_k}$. Then by \cite[Theorem 1.3]{Bod14}, $\CK_{{x\over k}\text{NP}_k}$ is a \Fraisse class and satisfies the Ramsey property. Let $\mathbb{M}_{{x\over k}\text{NP}_k}$ be the \Fraisse limit of $\mathcal{K}_{{x\over k}\text{NP}_k}$. Then $\age(\mathbb{M}_{{x\over k}\text{NP}_k})=\age(\CI_{{x\over k}\text{NP}_k})$ and hence ${\CI_{{x\over k}\text{NP}_k}}$-indiscernibles have the modeling property.
\end{proof}
\end{lemma}

\begin{notation}
For $1<k<\omega$ and $0\le l \le k$, let 
\[
{l\over k}X^*_k:=\bigg\lbrace \Big( (0,q_0,0),...,(k-1,q_{k-1},0)\Big)\in (k\times\mathbb{Q}\times\mathbb{Q})^k:(q_0,...,q_{k-1})\in {l\over k}\hat X_k\bigg\rbrace
\]
\[\hspace{55pt}\cup\; \bigg\lbrace \Big( (0,q_0,{1\over2}),...,(k-1,q_{k-1},{1\over2})\Big)\in (k\times\mathbb{Q}\times\mathbb{Q})^k:(q_0,...,q_{k-1})\notin {l\over k}\hat X_k\bigg\rbrace.
\]
\end{notation}

\begin{lemma}\label{lem: cofinality of l/kX_k}
${l\over k}X^*_k$ is cofinal in $\CI_{{x\over k}\text{NP}_k}$.
\begin{proof}
For each $n<\omega$, choose any order isomorphisms $f^+_n:\mathbb{Q}\to(n,\infty)$ and $f^-_n:\mathbb{Q}\to(-\infty,-n)$. Let $g_n$ be a map from $k\times \mathbb{Q}\times\mathbb{Q}$ to $k\times\mathbb{Q}\times\mathbb{Q}$ sending $(0,q,r)$ to $(0,f^-_n(q),r)$ for each $q,r\in\mathbb{Q}$, and $(i,q,r)$ to $(i,f^+_n(q),r)$ for each $0<i<k$ and $q,r\in \mathbb{Q}$. Clearly $g_n$ is an $\CL_{{x\over k}\text{NP}_k}$-embedding. Choose a finite subset $Z$ of $k\times \mathbb{Q}\times\mathbb{Q}$. Then there exists $N<\omega$ such that $g_N(k\times\mathbb{Q}\times\mathbb{Q})\cap Z=\emptyset$. Clearly $\bar\eta Z\sim_{\CL_{{x\over k}\text{NP}_k}}\bar\nu Z$ for each $\bar\eta,\bar\nu\in g_N(k\times\mathbb{Q}\times\mathbb{Q})$ with $\bar\eta\sim_{\CL_{{x\over k}\text{NP}_k}}\bar\nu$. To show $g_N(k\times\mathbb{Q}\times\mathbb{Q})^k\cap {l\over k}X^*_k\sim^\text{fin}_{\CL_{{x\over k}\text{NP}_k}}{l\over k}X^*_k$, choose any finite subset $X_0$ of ${l\over k}X^*_k$. Then we can find a positive rational number $q$ such that $X_0$ is contained in
\[
X^q:=\bigg\lbrace\Big( (0,q_0,0),...,(k-1,q_{k-1},0)  \Big) \in (k\times\mathbb{Q}\times\mathbb{Q})^k:(q_0,...,q_{k-1})\in {l\over k}\hat X_k,\; -q<q_0,...,q_{k-1}<q \bigg\rbrace
\]
\[
\hspace{23pt}\cup \;\bigg\lbrace\hspace{-0.5pt}\Big(\hspace{-1pt} (0,q_0,\hspace{-1pt}{1\over 2}),...,(k-1,q_{k-1},\hspace{-1pt}{1\over 2})  \hspace{-0.5pt}\Big) \in (k\times\mathbb{Q}\times\mathbb{Q})^k:(q_0,...,q_{k-1})\notin {l\over k}\hat X_k,\; -q<q_0,...,q_{k-1}<q\bigg\rbrace.\]
 By using an order isomorphism from $(-q,q)$ to $(0,1)$, we can show that there exists a $(k-l)$-down closed subset $X'$ of $(0,1)^k$ such that \[
\bigg\lbrace\Big( (0,q_0,0),...,(k-1,q_{k-1},0)  \Big) \in (k\times\mathbb{Q}\times\mathbb{Q})^k:(q_0,...,q_{k-1})\in X',\; 0<q_0,...,q_{k-1}<1 \bigg\rbrace
\]
\[
\hspace{20pt}\cup \;\bigg\lbrace\Big( (0,q_0,{1\over 2}),...,(k-1,q_{k-1},{1\over 2})  \Big) \in (k\times\mathbb{Q}\times\mathbb{Q})^k:(q_0,...,q_{k-1})\notin X',\; 0<q_0,...,q_{k-1}<1\bigg\rbrace\]
has the same quantifier-free $\CL_{{x\over k}\text{NP}_k}$-type as $X^q$. Thus by the construction of ${l\over k}X^*_k$, there exists $X_1\subseteq g_N(k\times\mathbb{Q}\times\mathbb{Q})^k\cap {l\over k}X^*_k$ such that $X_1\sim_{\CL_{{x\over k}\text{NP}_k}}X_0$.
\end{proof}
\end{lemma}

\begin{notation}
${l\over k}\mathbbof{X}^*_k:=\{\bar\eta\in\CI_{{x\over k}NP_k}:\la\bar\eta\ra\sim\la\bar\nu\ra\text{ for some }\bar\nu\in{l\over k}X^*_k\}$.
\end{notation}

\begin{lemma}\label{lem: k-monochromatic extendable IxkNPk}
$\CI_{{x\over k}NP_k}$ is $k$-monochromatically extendable over ${l\over k}\mathbbof{X}^*_k$ for any $0<k<\omega$ and $0\le l\le k$.
\begin{proof}
Let $D$ be a unary predicate symbol. For any cardinal $\lambda$, let $I_\lambda$ be a $\lambda^+$-saturated $\{<,D\}$-structure such that $I_\lambda\models \Th(\mathbb{Q},<,D)$ where $D$ is a dense codense subset of $\mathbb{Q}$. Then $I_\lambda$ satisfies
\begin{itemize}
\item[(i)] for any  $(a_i)_{i<\lambda},(b_i)_{i<\lambda}\subseteq I_\lambda$ such that $a_i<b_j$ for all $i,j<\lambda$, there exist $d,e\in I_\lambda$ with $I_\lambda \models D(d)\wedge \neg D(e)$ such that $a_i<d,e<b_j$ for all $i,j<\lambda$. 
\end{itemize}

\medskip
 
\noindent\underline{Claim 1.} For any $c:I_\lambda\to \lambda$, there exist $\lambda'<\lambda$ and $a',b'\in I_\lambda$ such that $a'<b'$, and $c^{-1}(\lambda')\cap D(I_\lambda)$ is dense in $(a',b')$.

\smallskip

\noindent{\it Proof of Claim 1.}
 Suppose not. We construct a sequence $(a_ib_i)_{i<\lambda}$ such that
\begin{itemize}
\item[(ii)] $a_i<a_j<b_j<b_i$ for all $i<j<\lambda$,
\item[(iii)] $(a_i,b_i)\cap c^{-1}(i)\cap D(I_\lambda)=\emptyset$ for each $i<\lambda$.
\end{itemize}
 Choose any $a'_0,b'_0\in I_\lambda$ with $a'_0<b'_0$. Then $c^{-1}(0)\cap D(I_\lambda)$ is not dense in $(a'_0,b'_0)$. Thus there exist $a_0<b_0$ such that $(a_0,b_0)\cap c^{-1}(0)\cap D(I_\lambda)=\emptyset$. Suppose we have constructed $(a_ib_i)_{i<\lambda_0}$ for some $\lambda_0<\lambda$. By (i), we can find $a'_{\lambda_0},b'_{\lambda_0}\in I_\lambda$ such that $a_i<a'_{\lambda_0}<b'_{\lambda_0}<b_i$ for each $i<\lambda_0$. Since $c^{-1}(\lambda_0)\cap D(I_\lambda)$ is not dense in $(a'_{\lambda_0},b'_{\lambda_0})$, we can find $a_{\lambda_0},b_{\lambda_0}\in I_\lambda$ such that $a'_{\lambda_0}<a_{\lambda_0}<b_{\lambda_0}<b'_{\lambda_0}$ and $(a_{\lambda_0},b_{\lambda_0})\cap c^{-1}(\lambda_0)\cap D(I_\lambda)=\emptyset$. By repeating this, we can complete $(a_ib_i)_{i<\lambda}$.
By (i), we can find $d\in I_\lambda$ such that $I_\lambda\models D(d)$ and $a_i<d<b_i$ for all $i<\lambda$. But it yields a contradiction with (iii).  $\dashv$

\medskip 
 
\noindent\underline{Claim 2.} For any $c:I_\lambda\to \lambda$, there exist $\lambda',\lambda''<\lambda$ and $a'',b''\in I_\lambda$ such that $a''<b''$, and $c^{-1}(\lambda')\cap D(I_\lambda)$ and $c^{-1}(\lambda'')\cap\neg D(I_\lambda)$ are dense in $(a'',b'')$.

\smallskip

\noindent{\it Proof of Claim 2.} By Claim 1, we have $\lambda'<\lambda$ and $a'<b'$ such that $c^{-1}(\lambda')\cap D(I_\lambda)$ is dense in $(a',b')$. If Claim 2 is not true, then for any $\lambda''<\lambda$ and $a'',b''\in I_\lambda$ with $a'<a''<b''<b'$, there exist $a''',b'''\in I_\lambda$ such that $a''<a'''<b'''<b''$ and $(a''',b''')\cap c^{-1}(\lambda'')\cap \neg D(I_\lambda)=\emptyset$. Thus, by the same argument of Claim 1, we can construct a sequence $(a'_ib'_i)_{i<\lambda}$ such that 
\begin{itemize}
\item[(iv)] $a'<a'_i<a'_j<b'_j<b'_i<b'$ for all $i<j<\lambda$,
\item[(v)] $(a'_i,b'_i)\cap c^{-1}(i)\cap \neg D(I_\lambda)=\emptyset$ for each $i<\lambda$.
\end{itemize}
By (i) again, we can find $d\in I_\lambda$ such that $a'_i<d<b'_i$ for all $i<\lambda$ and $\neg D(d)$. But this yields a contradiction with (v). $\dashv$

\medskip

For any cardinal $\lambda$, let $J_\lambda$ be a $\lambda^+$-saturated $\{<\}$-structure such that $J_\lambda\models \Th(\mathbb{Q},<)$. Then by the same argument of Claim 1 or \Cref{lem: coloring lemma of dense linear order}, we have
\begin{itemize}
\item[(vi)] for any $c:J_\lambda\to \lambda$, there exist $\lambda'<\lambda$ and $q',r'\in J_\lambda$ such that $q'<r'$, and $c^{-1}(\lambda')$ is dense in $(q',r')$.
\end{itemize}

For an infinite cardinal $\lambda$,
let $\lambda_0:=\lambda$ and $\lambda_{i+1}:=2^{\max\{|J_{\lambda_i}|,|I_{\lambda_i}|\}}$ for each $i<\omega$. For each $\lambda$ and $0<k<\omega$, we can give an $\CL_{{x\over k}\text{NP}_k}$-structure on 
\[
\mathbbof I^\lambda_k:=\bigcup_{i<k}\{ (i,q,a):(q,a)\in  J_{\lambda_i}\times I_{\lambda_i}\}
\] 
such that
\begin{itemize}
\item[(vii)] $(i,q,a)<(i',q',a')$ if one of the following holds.
\begin{itemize}
\item[$\ast$] $i<i'$.
\item[$\ast$] $i=i'$ and $q<q'$.
\item[$\ast$] $i=i'$, $q=q'$, and $a<a'$.
\end{itemize}
\item[(viii)] $E((i,q,a),(i',q',a'))$ if $i=i'$ and $q=q'$.
\item[(ix)] $P_j((i,q,a))$ if $i=j$ for each $i,j<k$.
\item[(x)] $C((i,q,a))$ if $D(a)$.
\end{itemize}
Then $\age(\CI_{{x\over k}\text{NP}_k})=\age(\mathbbof{I}^\lambda_k)$. Note that ${l\over k}\mathbbof X^*_k\subseteq (k\times\mathbb{Q}\times\mathbb{Q})^k$ is a disjoint union of ${l\over k}\mathbbof X^{C}_k$ and ${l\over k}\mathbbof X^{\neg C}_k$, where
\[
{l\over k}\mathbbof X^{C}_k:=\bigg\lbrace\Big((0,x_0,y_0),...,(k-1,x_{k-1},y_{k-1})\Big):\CI_{{x\over k}\text{NP}_k}\!\models C((i,x_i,y_i))\text{ for all }i<k \bigg\rbrace
\]
\[
{l\over k}\mathbbof X^{\neg C}_k:=\bigg\lbrace\Big((0,x_0,y_0),...,(k-1,x_{k-1},y_{k-1})\Big):\CI_{{x\over k}\text{NP}_k}\!\models \neg C((i,x_i,y_i))\text{ for all }i<k \bigg\rbrace.
\]
Also note that ${l\over k}\mathbbof X^*_k$, ${l\over k}\mathbbof X^{C}_k$, and ${l\over k}\mathbbof X^{\neg C}_k$ do not depend on the choice of $l$.

\medskip

\noindent\underline{Claim 3.} For any infinite cardinal $\lambda$ and any $c:(\mathbbof I^\lambda_1)\to\lambda$, there exist $\lambda',\lambda''<\lambda$ and an $\CL_{{x\over 1}\text{NP}_1}$-embedding $f:\CI_{{x\over 1}\text{NP}_1}\to \mathbbof I^\lambda_1$ such that
\begin{itemize}
\item[(xi)] $c(f(0,x_0,y_0))=\lambda'$ for all $(0,x_0,y_0)\in {l\over 1}\mathbbof{X}^{C}_1$, 
\item[(xii)] $c(f(0,x_0,y_0))=\lambda''$ for all $(0,x_0,y_0)\in {l\over 1}\mathbbof{X}^{\neg C}_1$, 
\end{itemize}

\smallskip

\noindent{\it Proof of Claim 3.} For each  $q\in J_\lambda$, let $c_q$ be a $\lambda$-coloring on $I_\lambda$ such that $c_{q}(a)=c(0,q,a)$. 
Then by Claim 2, for each $q\in J_\lambda$, there exist $a_{q},b_{q}\in I_\lambda$ and $\lambda'_{q},\lambda''_{q}<\lambda$ such that $c_{q}^{-1}(\lambda'_{q})\cap D(I_\lambda)$ and  $c_{q}^{-1}(\lambda''_{q})\cap \neg D(I_\lambda)$
are dense in $(a_{q},b_{q})$.
Let $c'$ be a $(\lambda\times\lambda)$-coloring on $J_\lambda$ such that $c'(q)=(\lambda'_{q},\lambda''_{q})$. By (vi), there exist $(\lambda',\lambda'')\in\lambda\times\lambda$ and $q',r'\in J_\lambda$  such that $q'<r'$, and $c'^{-1}(\lambda',\lambda'')$ is dense in $(q',r')$. 

By the choice of $J_\lambda$, we can find a $\{<\}$-embedding $f_0:\mathbb{Q}\to J_\lambda$ such that $q'<f_0(x)<r'$ and $c'(f_0(x))=(\lambda',\lambda'')$ for all $x\in\mathbb{Q}$. By the same argument above, for each $x\in\mathbb{Q}$, we can find a $\{<\}$-embedding $f_x:\mathbb{Q}\to I_\lambda$ such that 
\begin{itemize}
\item[(xiii)] $a_{f_0(x)}<f_x(y)<b_{f_0(x)}$ for all $y\in\mathbb{Q}$, 
\item[(xiv)] $c_{f_0(x)}(f_x(y))=\lambda'_{f_0(x)}=\lambda'$ and $I_\lambda\models D(f_x(y))$ for all $y\in\mathbb{Q}$ with $\CI_{{x\over k}\text{NP}_k}\models C(0,x,y)$,
\item[(xv)] $c_{f_0(x)}(f_x(y))=\lambda''_{f_0(x)}=\lambda''$ and $I_\lambda\models \neg D(f_x(y))$ for all $y\in\mathbb{Q}$ with $\CI_{{x\over k}\text{NP}_k}\models \neg C(0,x,y)$.
\end{itemize}
Let $f$ be a map from $1\times\mathbb{Q}\times\mathbb{Q}$ to $\mathbbof{I}^\lambda_1$ such that $f(0,x,y)=(0,f_0(x),f_x(y))$. Then $f$ is an $\CL_{{x\over 1}\text{NP}_1}$-embedding and satisfies (xi) and (xii).

\medskip

\noindent\underline{Claim 4.} For each $0<k<\omega$, an infinite cardinal $\lambda$, and $c:(\mathbbof I^\lambda_k)^k\to\lambda$, there exist $\lambda',\lambda''<\lambda$ and an $\CL_{{x\over k}\text{NP}_k}$-embedding $f:\CI_{{x\over k}\text{NP}_k}\to \mathbbof I^\lambda_k$ such that
\begin{itemize}
\item[(xvi)] $c(f(0,x_0,y_0),...,f(k-1,x_{k-1},y_{k-1}))=\lambda'$ for all $((0,x_0,y_0),...,(k-1,x_{k-1},y_{k-1}))\in {l\over k}\mathbbof{X}^{C}_k$, 
\item[(xvii)] $c(f(0,x_0,y_0),...,f(k-1,x_{k-1},y_{k-1}))=\lambda''$ for all $((0,x_0,y_0),...,(k-1,x_{k-1},y_{k-1}))\in {l\over k}\mathbbof{X}^{\neg C}_k$, 
\end{itemize}

\smallskip

\noindent{\it Proof of Claim 4.}
We prove this using induction on $k$. The case of $k=1$ follows from Claim 3.
Suppose that we have proved the case of $k$ and  let $c:(\mathbbof I^\lambda_{k+1})^{k+1}\to \lambda$. For each $(q,a)\in J_{\lambda_k}\times I_{\lambda_k}$, let $c_{q,a}$ be a map from $(\mathbbof I^\lambda_k)^k$ to $\lambda$ such that 
\[
c_{q,a}\Big((0,q_0,a_0),...,(k-1,q_{k-1},a_{k-1})\Big)
=c\Big((0,q_0,a_0),...,(k-1,q_{k-1},a_{k-1}),(k,q,a)\Big).
\]
Let $c'$ be a coloring on $\mathbbof{I}^{\lambda_k}_1= 1\times J_{\lambda_k}\times I_{\lambda_k}$ such that $c'(0,q,a)=c_{q,a}$. Note that $c'$ is a $\lambda_k$-coloring. By Claim 3, there exist $c^+$, $c^-$, and an $\CL_{{x\over 1}\text{NP}_1}$-embedding $f': \CI_{{x\over 1}\text{NP}_1}\to\mathbbof I^{\lambda_k}_1$ such that
\begin{itemize}
\item[(xviii)] $c'(f'(0,x,y))=c^+$ if $(0,x,y)\in {l\over 1}\mathbbof{X}^{C}_1$, 
\item[(xix)] $c'(f'(0,x,y))=c^-$ if $(0,x,y)\in {l\over 1}\mathbbof{X}^{\neg C}_1$.
\end{itemize} 
Let $c''$ be a coloring on $(\mathbbof I^\lambda_k)^k$ such that 
\begin{multline*}
c''\big((0,q_0,a_0),...,(k-1,q_{k-1},a_{k-1})\big)
\\
=\Big(c^+\big((0,q_0,a_0),...,(k-1,q_{k-1},a_{k-1})\big),c^-\big((0,q_0,a_0),...,(k-1,q_{k-1},a_{k-1})\big)\Big).
\end{multline*}
Note that $c''$ is a $\lambda\times\lambda$-coloring. By the induction hypothesis, there exist $(\lambda^+_0,\lambda^+_1),(\lambda^-_0,\lambda^-_1)\in\lambda\times\lambda$ and an $\CL_{{x\over k}\text{NP}_k}$-embedding $f'':\CI_{{x\over k}\text{NP}_k}\to \mathbbof I^\lambda_k$ such that
\begin{itemize}
\item[(xx)] $c''(f''(0,x_0,y_0),...,f''(k-1,x_{k-1},y_{k-1}))=
(\lambda^+_0,\lambda^+_1)$ if $((0,x_0,y_0),...,(k-1,x_{k-1},y_{k-1}))\in {l\over k}\mathbbof{X}^{C}_k$, 
\item[(xxi)] $c''(f''(0,x_0,y_0),...,f''(k-1,x_{k-1},y_{k-1}))=
(\lambda^-_0,\lambda^-_1)$ if $((0,x_0,y_0),...,(k-1,x_{k-1},y_{k-1}))\in {l\over k}\mathbbof{X}^{\neg C}_k$, 
\end{itemize}
Define a map $f:\CI_{{x\over k+1}\text{NP}_{k+1}}\to \mathbbof I^\lambda_{k+1}$ such that 
\[
f(i,x,y)=
\begin{cases}
f''(i,x,y) & \text{if } i<k
\\
(k,q,a)&\text{if }i=k\text{ and }f'(0,x,y)=(0,q,a).
\end{cases}
\]
Then
\begin{itemize}
\item[(xxii)] $c(f(0,x_0,y_0),...,f(k,x_k,y_k))=\lambda^+_0$  for all $((0,x_0,y_0),...,(k,x_k,y_k))\in {l\over k+1}\mathbbof{X}^C_{k+1}$,
\item[(xxiii)] $c(f(0,x_0,y_0),...,f(k,x_k,y_k))=\lambda^-_1$ for all $((0,x_0,y_0),...,(k,x_k,y_k))\in {l\over k+1}\mathbbof{X}^{\neg C}_{k+1}$.
\end{itemize}
This proves Claim 4. $\dashv$

\smallskip

This completes the proof.
\end{proof}
\end{lemma}

\begin{notation}
Continuing the construction, let 
\[
Y^*_k:=\bigg\lbrace \Big( (0,0,0),...,(k-1,0,0)    \Big), \Big((0,0,{1\over 2}),...,(k-1,0,{1\over 2})   \Big) \bigg\rbrace\subseteq (k\times\mathbb{Q}\times\mathbb{Q})^k.
\]
\end{notation}

\begin{lemma}\label{lem: l/kCLk l/kCIk k l/kXk Yk satisfies nvariable theorem}
Let $g_1$ be the $\CL_{ {x\over k}\text{NP}_k}$-embedding given in \Cref{lem: cofinality of l/kX_k}. Then $\{\nu\}\cup g_1({l\over k}X^*_k) \sim^\text{fin}_{\CL_{ {x\over k}\text{NP}_k}} {l\over k}X^*_k$ for each $\nu\in Y^*_k$. Consequently, $(\CL_{{x\over k}\text{NP}_k},\CI_{{x\over k}\text{NP}_k}, k, {l\over k}X^*_k,Y^*_k)$ satisfies the $n$-variable theorem for each $1<k<\omega$ and $0\le l\le k$.
\begin{proof}
Let $\nu\in Y^*_k$. If $\nu:=\big( (0,0,0),...,(k-1,0,0)\big)$, then choose any simple $(k-l)$-down closed subset $X'\subseteq (-1,1)^k$ containing $(0,...,0)$. If $\nu:=\big( (0,0,{1\over 2}),...,(k-1,0,{1\over 2})\big)$, then choose  $X'$ such that $(0,...,0)\notin X'$. Let $X^*$ be a finite subset of $\{\nu\}\cup g_1({l\over k}X^*_k)$. Then $X^*$ is contained in 
\[
X^{**}\!:=\!\{\nu\}\cup\bigg\lbrace\!\Big( (0, q_0,0),...,(k-1,q_{k-1},0)\Big)\!:\!(q_0,...,q_{k-1})\in X''\!,\; -q\!<\!q_0\!<\!-1, 1\!<\!q_i\!<\!q\text{ for }0\!<\!i\!<\!k\bigg\rbrace
\]
\[
\hspace{52pt}\cup\;\bigg\lbrace\!\Big( (0, q_0,{1\over 2}),...,(k-1,q_{k-1},{1\over 2})\Big)\!:\!(q_0,...,q_{k-1})\notin X''\!,\;-q\!<\! q_0\!<\!-1, 1\!<\!q_i\!<\!q\text{ for }0\!<\!i\!<\!k\bigg\rbrace
\]
for some rational number $q>1$ and a simple $(k-l)$-down closed subset $X''$ of $(-q,-1)\times(1,q)\times\cdots\times(1,q)$. Let $X'''$ be the $(k-l)$-down closure of $X'\cup X''$ in $ (-q,1)\times (-1,q)\times \cdots\times (-1,q)$. Let 
\[
\hspace{-10pt}X^{***}\!:=\!\bigg\lbrace\!\Big( (0, q_0,0),...,(k-1,q_{k-1},0)\Big)\!:\!(q_0,...,q_{k-1})\in X''',\; -q\!<\!q_0\!<\!1, -1\!<\!q_i\!<\!q\text{ for }0\!<\!i\!<\!k\bigg\rbrace
\]
\[
\hspace{42pt}\cup\;\bigg\lbrace\!\Big( (0, q_0,{1\over 2}),...,(k-1,q_{k-1},{1\over 2})\Big)\!:\!(q_0,...,q_{k-1})\notin X''',\;-q\!<\! q_0\!<\!1, -1\!<\!q_i\!<\!q\text{ for }0\!<\!i\!<\!k\bigg\rbrace.
\]
Then $X^*\subseteq X^{***}$. By using order isomorphisms between countable dense linear orders without endpoints, there exists a simple $(k-l)$-down closed subset $X''''$ of $(0,1)^k$ such that
\[
X^{****}:=\bigg\lbrace\Big( (0, q_0,0),...,(k-1,q_{k-1},0)\Big):(q_0,...,q_{k-1})\in X'''',\; 0<q_i<1\text{ for }i<k\bigg\rbrace
\]
\[
\hspace{56pt}\cup\;\bigg\lbrace\Big( (0, q_0,{1\over 2}),...,(k-1,q_{k-1},{1\over 2})\Big):(q_0,...,q_{k-1})\notin X'''',\;0< q_i<1\text{ for }i<k\bigg\rbrace
\]
is $\CL_{{x\over k}\text{NP}_k}$-isomorphic to $X^{***}$. Note that $X^{****}$ can be $\CL_{{x\over k}\text{NP}_k}$-embedded into ${l\over k}X^*_k$, by the construction of ${l\over k}X^*_k$. In particular, $X^*$ can be $\CL_{{x\over k}\text{NP}_k}$-embedded into ${l\over k}X^*_k$. Thus $\{\nu\}\cup g_1({l\over k}X^*_k) \sim^\text{fin}_{\CL_{ {x\over k}\text{NP}_k}} {l\over k}X^*_k$.
\end{proof}
\end{lemma}

Let $D_{{l\over k}\text{NP}_k}$ be the class of ${l\over k}$NP$_k$ theories. Similarly, let $D_{\text{NOP}_k}$, $D_{\text{NFOP}_k}$, and $D_{\text{NIP}_k}$ be the classes of NOP$_k$, NFOP$_k$, and NIP$_k$ theories, respectively.

\begin{theorem}\label{thm: l-over-kP-k D-n-var}
$D_{{l\over k}\text{NP}_k}\in\mathfrak{D}^h_{n\text{-var}}$ for each $1<k<\omega$ and $0\le l \le k$. In particular,  $D_{\text{NOP}_k}$, $D_{\text{NFOP}_k}$, $D_{\text{NIP}_k}\in\mathfrak{D}^h_{n\text{-var}}$ for each $1<k<\omega$.
\begin{proof}
By \Cref{lem: l/kCLk l/kCIk k l/kXk Yk satisfies nvariable theorem}, it is enough to show that $D_{{l\over k}\text{NP}_k}=D^{\CL_{{x\over k}\text{NP}_k},\CI_{{x\over k}\text{NP}_k}}_{k,{l\over k}X^*_k, Y^*_k}$.

Suppose $T\notin D_{{l\over k}\text{NP}_k}$. Then by \Cref{lem: m/nPn characterized by s-c n-straight pattern}, we have $\varphi(x,y_0,...,y_{k-1})$ and mutually indiscernible $(a^0_q)_{q\in\mathbb{Q}},...,(a^{k-1}_q)_{q\in\mathbb{Q}}$ such that
\[
\big\lbrace\varphi(x,a^0_{q_0},...,a^{k-1}_{q_{k-1}}):(q_0,...,q_{k-1})\in {l\over k}\hat X_k\big\rbrace\cup\big\lbrace\neg
\varphi(x,a^0_{q_0},...,a^{k-1}_{q_{k-1}}):(q_0,...,q_{k-1})\notin {l\over k}\hat X_k\big\rbrace
\]
is consistent. By the mutual indiscernibility of $(a^0_q)_{q\in\mathbb{Q}},...,(a^{k-1}_q)_{q\in\mathbb{Q}}$ and the choice of ${l\over k}\hat X_k$, 
\[
\hspace{-30pt}\text{(i)}\hspace{20pt}\big\lbrace\varphi(x,a^0_{q_0},...,a^{k-1}_{q_{k-1}}):(q_0,...,q_{k-1})\in X\big\rbrace\cup\big\lbrace\neg
\varphi(x,a^0_{q_0},...,a^{k-1}_{q_{k-1}}):(q_0,...,q_{k-1})\notin X\big\rbrace
\]
is consistent for any $(k-l)$-down closed set $X\subseteq\mathbb{Q}^k$. Let 
\[
\psi(x,y^0_0y^1_0,...,y^0_{k-1}y^1_{k-1}):=\varphi(x,y^0_0,...,y^0_{k-1})\wedge\neg\varphi(x,y^1_0,...,y^1_{k-1}).
\] We can find sequences of rational numbers $(u_{(i,q)})_{i<k,q\in\mathbb{Q}}$, $(v_{(i,q)})_{i<k,q\in\mathbb{Q}}$, and $(w_{(i,q)})_{i<k,q\in\mathbb{Q}}$ such that
\begin{itemize}
\item[(ii)] $u_{(i,q_0)}<v_{(i,q_1)}<w_{(i,q_2)}$ for each $i<k$ and $q_0,q_1,q_2\in\mathbb{Q}$,
\item[(iii)] $u_{(i,q)}<u_{(i,q')}$ for $i<k$ and $q<q'$,
\item[(iv)] $v_{(i,q)}>v_{(i,q')}$ for $i<k$ and $q<q'$.
\end{itemize}
For each $(i,q,q')\in k\times \mathbb{Q}\times\mathbb{Q}$, let
\[
    b_{i,q,q'}= 
\begin{cases}
    a^i_{u_{(i,q)}}a^i_{v_{(i,q)}} & \text{if } \;\; q'=0\text{ or }q'={n\over m}\text{ with }(n,m)=1\text{ and }n:\text{even} \\
    a^i_{v_{(i,q)}}a^i_{w_{(i,q)}} & \text{otherwise} \end{cases}
\]

\medskip

\noindent\underline{Claim 1.} $\{\psi(x,b_{\nu_0},... ,b_{\nu_{k-1}})\}_{(\nu_0,...,\nu_{k-1})\in Y}$ is inconsistent for any $Y\sim_{\CL_{{x\over k}\text{NP}_k}} Y^*_k$.

\smallskip

\noindent{\it Proof of Claim 1.}
If $Y\sim_{\CL_{{x\over k}\text{NP}_k}} Y^*_k$, then $Y$ is of the form
\[
\bigg\lbrace\Big(  (0,q_0,q'_0),...,(k-1,q_{k-1},q'_{k-1})\Big),\Big(  (0,q_0,q''_0),...,(k-1,q_{k-1},q''_{k-1})\Big)\bigg\rbrace
\]
for some $(q_iq'_iq''_i)_{i<k}$ such that $\CI_{{x\over k}\text{NP}_k}\models C(i,q_i,q'_i)\wedge\neg C(i,q_i,q''_i)$ for all $i<k$. Thus $\{\psi(x,\bar b_{\bar\nu})\}_{\bar\nu\in Y}$ is of the form
\[
\bigg\lbrace
\psi(x,a^0_{u_{(0,q_0)}}a^0_{v_{(0,q_0)}},...,a^{k-1}_{u_{(k-1,q_{k-1})}}a^{k-1}_{v_{(k-1,q_{k-1})}}),
\psi(x,a^0_{v_{(0,q_0)}}a^0_{w_{(0,q_0)}},...,a^{k-1}_{v_{(k-1,q_{k-1})}}a^{k-1}_{w_{(k-1,q_{k-1})}})
\bigg\rbrace.
\]
It is equivalent to
\[
\bigg\lbrace
\varphi(x,a^0_{u_{(0,q_0)}},...,a^{k-1}_{u_{(k-1,q_{k-1})}}),\neg\varphi(x,a^0_{v_{(0,q_0)}},...,a^{k-1}_{v_{(k-1,q_{k-1})}}),
\]
\[
\hspace{22pt}
\varphi(x,a^0_{v_{(0,q_0)}},...,a^{k-1}_{v_{(k-1,q_{k-1})}}),
\neg\varphi(x,a^0_{w_{(0,q_0)}},...,a^{k-1}_{w_{(k-1,q_{k-1})}}),
\bigg\rbrace.
\]
Thus $\{\psi(x,b_{\nu_0},... ,b_{\nu_{k-1}})\}_{(\nu_0,...,\nu_{k-1})\in Y}$ is inconsistent. $\dashv$

\medskip

\noindent\underline{Claim 2.} If a finite subset $X^*$ of $(k\times \mathbb{Q}\times\mathbb{Q})^k$ can be $\CL_{{x\over k}\text{NP}_k}$-embedded into ${l\over k}X^*_k$, then the set of formulas $\{\psi(x,b_{\eta_0},...,b_{\eta_{k-1}})\}_{(\eta_0,...,\eta_{k-1})\in X^*}$ is consistent.

\smallskip

\noindent{\it Proof of Claim 2.} If $X^*$ can be $\CL_{{x\over k}\text{NP}_k}$-embedded into ${l\over k}X^*_k$, then there exist a positive rational number $q$, a $(k-l)$-down closed $X''\subseteq (-q,q)^k$, and $(f_i:(-q,q)^k\to(-q,q))_{i<k}$ such that 
\[
\CI_{{x\over k}\text{NP}_k}\models \bigwedge_{i<k}C(i, q_i,f_i(q_0,...,q_{k-1}))\;\text{ if }\;(q_0,...,q_{k-1})\in X'',
\]
\[
\CI_{{x\over k}\text{NP}_k}\models \bigwedge_{i<k}\neg C(i,q_i,f_i(q_0,...,q_{k-1}))\;\text{ if }\;(q_0,...,q_{k-1})\notin X'',
\]
and
\[
X^{**}\!:=\!\bigg\lbrace\!\Big( (0,q_0,f_0(\bar q) ),...,(k-1,q_{k-1},f_{k-1}(\bar q) )\Big)\!:\! \bar q=(q_0,...,q_{k-1})\in X'', -q<q_i<q\text{ for }i<k    \bigg\rbrace
\]
\[
\hspace{40pt}\cup\; \bigg\lbrace\!\Big( (0,q_0,f_0(\bar q) ),...,(k-1,q_{k-1},f_{k-1}(\bar q) )\Big)\!:\!\bar q=(q_0,...,q_{k-1})\notin X'', -q<q_i<q\text{ for }i<k    \bigg\rbrace
\]
contains $X^*$. $\{\psi(x,b_{\eta_0},...,b_{\eta_{k-1}})\}_{(\eta_0,...,\eta_{k-1})\in X^{**}}$ is of the form
\[
\bigg\lbrace\varphi(x,a^0_{u_{(0,q_0)}},...,a^{k-1}_{u_{(k-1,q_{k-1})}}):(q_0,...,q_{k-1})\in X''
\bigg\rbrace
\]
\[
\hspace{8pt}\cup\bigg\lbrace\varphi(x,a^0_{v_{(0,q_0)}},...,a^{k-1}_{v_{(k-1,q_{k-1})}}):(q_0,...,q_{k-1})\notin X''
\bigg\rbrace
\]
\[
\hspace{15pt}\cup\bigg\lbrace\neg\varphi(x,a^0_{v_{(0,q_0)}},...,a^{k-1}_{v_{(k-1,q_{k-1})}}):(q_0,...,q_{k-1})\in X''
\bigg\rbrace
\]
\[
\hspace{19pt}\cup\bigg\lbrace\neg\varphi(x,a^0_{w_{(0,q_0)}},...,a^{k-1}_{w_{(k-1,q_{k-1})}}):(q_0,...,q_{k-1})\notin X''
\bigg\rbrace
\]
Note that the $(k-l)$-th down closure of 
\[
\Big\lbrace (u_{(0,q_0)},...,u_{(k-1,q_{k-1})}):(q_0,...,q_{k-1})\in X''\Big\rbrace\cup \Big\lbrace (v_{(0,q_0)},...,v_{(k-1,q_{k-1})}):(q_0,...,q_{k-1})\notin X''\Big\rbrace
\]
does not contain any element of 
\[
\Big\lbrace (v_{(0,q_0)},...,v_{(k-1,q_{k-1})}):(q_0,...,q_{k-1})\in X''\Big\rbrace\cup \Big\lbrace (w_{(0,q_0)},...,w_{(k-1,q_{k-1})}):(q_0,...,q_{k-1})\notin X''\Big\rbrace.
\]
Thus $\{\psi(x,\bar b_{\bar\eta})\}_{\bar\eta\in X^{**}}$ is consistent by (i). In particular, $\{\psi(x,\bar b_{\bar\eta})\}_{\bar\eta\in X^{*}}$ is consistent. $\dashv$

\medskip

Let $(c_{i,q,q'})^{i<k}_{q,q'\in\mathbb{Q}}$ be an $\CL_{{x\over k}\text{NP}_k}$-indiscernible sequence locally based on $(b_{i,q,q'})^{i<k}_{q,q'\in\mathbb{Q}}$. By Claim 1 and Claim 2, $\{\psi(x,\bar{c}_{\bar\eta})\}_{\bar\eta\in Y^*_k}$ is inconsistent and $\{\psi(x,\bar{c}_{\bar\eta})\}_{\bar\eta\in {l\over k}X^*_k}$ is consistent. Thus $T\notin D^{\CL_{{x\over k}\text{NP}_k},\CI_{{x\over k}\text{NP}_k}}_{k,{l\over k}X^*_k, Y^*_k}$.

\medskip

Now we suppose $T\notin D^{\CL_{{x\over k}\text{NP}_k},\CI_{{x\over k}\text{NP}_k}}_{k,{l\over k}X^*_k, Y^*_k}$. Then there exist $\varphi(x,y_0,...,y_{k-1})$ and $\CL_{{x\over k}\text{NP}_k}$-indiscernible $(a_{i,q,q'})^{i<k}_{q,q'\in\mathbb{Q}}$ such that $\{\varphi(x,\bar a_{\bar\eta})\}_{\bar\eta\in {l\over k}X^*_k}$ is consistent and $\{\varphi(x,\bar a_{\bar\eta})\}_{\bar\eta\in Y^*_k}$ is inconsistent. Then 
\[
\bigg\lbrace \varphi(x,a_{0,q_0,0},...,a_{k-1,q_{k-1},0}): (q_0,...,q_{k-1})\in {l\over k}\hat X_k \bigg\rbrace
\]
\[
\hspace{17pt}\cup\;\bigg\lbrace \varphi(x,a_{0,q_0,{1\over2}},...,a_{k-1,q_{k-1},{1\over2}}): (q_0,...,q_{k-1})\notin {l\over k}\hat X_k \bigg\rbrace
\]
is consistent. Thus
\[
\bigg\lbrace \varphi(x,a_{0,q_0,0},...,a_{k-1,q_{k-1},0}): (q_0,...,q_{k-1})\in {l\over k}\hat X_k \bigg\rbrace
\]
\[
\hspace{17pt}\cup\;\bigg\lbrace \neg\varphi(x,a_{0,q_0,0},...,a_{k-1,q_{k-1},0}): (q_0,...,q_{k-1})\notin {l\over k}\hat X_k \bigg\rbrace
\]
is consistent. For each $i<k$ and $q\in\mathbb{Q}$, let $b^i_q:=a_{i,q,0}$. Then
\[
\Big\lbrace \varphi(x,b^0_{q_0},...,b^{k-1}_{q_{k-1}}): (q_0,...,q_{k-1})\in {l\over k}\hat X_k \Big\rbrace
\cup\Big\lbrace \neg\varphi(x,b^0_{q_0},...,b^{k-1}_{q_{k-1}}): (q_0,...,q_{k-1})\notin {l\over k}\hat X_k \Big\rbrace
\]
is consistent. Since $(a_{i,q,q'})^{i<k}_{q,q'\in\mathbb{Q}}$ is $\CL_{{x\over k}\text{NP}_k}$-indiscernible, $(b^0_q)_{q\in\mathbb{Q}},...,(b^{k-1}_q)_{q\in\mathbb{Q}}$ are mutually indiscernible. By \Cref{lem: m/nPn characterized by s-c n-straight pattern}, $T$ has ${l\over k}P_k$. Thus $T\notin D_{{l\over k}\text{NP}_k}$.
\end{proof}
\end{theorem}

Note that the characterization of ${l\over k}$P$_k$ in \Cref{lem: m/nPn characterized by s-c n-straight pattern} is a special form of a straight pattern. The following notation is taken from \cite{Bai24} and \cite{DM26}, with slight modifications for our setting.

\begin{definition}
Let $\CL^*$ be a language, $\CI^*$ an $\CL^*$-structure, and $n<\omega$. A pair $(\CC,\CI)$ is called a {\it straight $n$-pattern} in $\CI^*$ if $\CC$ and $\CI$ are both subsets of $(\CP((\CI^*)^n)\times \CP((\CI^*)^n))\setminus\{(\emptyset,\emptyset)\}$ and $X^+\cap X^-=\emptyset$ for all $(X^+,X^-)\in\CC\cup\CI$. 
\begin{itemize}
\item[(i)] We say a straight $n$-pattern $(\CC,\CI)$ is {\it complete} if $X^+\cup X^-=(\CI^*)^n$ for all $(X^+,X^-)\in\CC\cup\CI$.
\item[(ii)] We say a straight $n$-pattern $(\CC,\CI)$ is {\it positive} if $X^-\!\!=\emptyset$ for all $(X^+,X^-)\in\CC\cup\CI$.
\item[(iii)]  We say a straight $n$-pattern $(\CC,\CI)$ is {\it single-consistency} if $|\CC|=1$ and $|\CI|=0$.
\end{itemize}
We say a complete theory $T$ {\it exhibits} a straight $n$-pattern $(\CC,\CI)$ in $\CI^*$ if there exist $\varphi(x,y_0,...,y_{n-1})$ and $\CI^*$-indiscernible $(a_\eta)_{\eta\in\CI^*}$ such that 
\begin{itemize}
\item[(iv)] $\{\varphi(x,\bar{a}_{\bar{\eta}}):\bar\eta\in X^+\}\cup\{\neg\varphi(x,\bar{a}_{\bar{\eta}}):\bar{\eta}\in X^-\}$ is consistent for each $(X^+,X^-)\in\CC$,
\item[(v)] $\{\varphi(x,\bar{a}_{\bar{\eta}}):\bar\eta\in X^+\}\cup\{\neg\varphi(x,\bar{a}_{\bar{\eta}}):\bar{\eta}\in X^-\}$ is inconsistent for each $(X^+,X^-)\in\CI$.
\end{itemize}
\end{definition}

By \Cref{lem: m/nPn characterized by s-c n-straight pattern}, ${l\over k}$P$_k$ can be characterized by a single-consistency straight $k$-pattern. This leads to the following question of whether \Cref{thm: l-over-kP-k D-n-var} can be generalized further.

\begin{question}
Let $(\CC,\CI)$ be a single-consistency straight $n$-pattern in an $\CL^*$-structure $\CI^*$.  Let $D$ be the class of all complete theories that do not exhibit $(\CC,\CI)$. Does $D$ belong to $\mathfrak{D}^h_{n\text{-var}}$? Or, is there any sufficient condition for $\CI^*$ that makes $D$ belong to $\mathfrak{D}^h_{n\text{-var}}$?
\end{question}

\subsection{$n$-variable theorems for $\mathfrak{D}^h_{n\text{-var}}$}

\begin{lemma}\label{lem: 1-var lemma higher arity}
Let $D^{\CL^*\!\!\!,\,\CI^*}_{k,X,\bar{Y}}\!\!\in\mathfrak{D}^h_{n\text{-var}}$. Let $\CL$ be a language, $T$ an $\CL$-theory, and $\mathbb{M}\models T$ a sufficiently saturated model. If  there exist $\varphi(x,y_0,...,y_{k-1})\in\CL$ and $\CI^*_{\CL^*}$-indiscernible $(a_\eta)_{\eta\in \CI^*}\!$ in $\mathbb{M}$ such that
\begin{itemize}
\item[(i)] $\varphi(x,y_0,...,y_{k-1})$ witnesses $\neg D^{\CL^*\!\!\!,\,\CI^*}_{k,X,\bar{Y}}$ with $(a_\eta)_{\eta\in \CI^*}$,
\item[(ii)] $|x|=1$,
\end{itemize}
then there exists $b\models\{\varphi(x,\bar a_{\bar\eta})\}_{\bar\eta\in X}$ such that $\dim(b/(a_\eta)_{\eta\in\CI^*}\!)=1$.
\begin{proof}
It is enough to show that $\{\varphi(x,\bar{a}_{\bar{\eta}})\}_{\bar{\eta}\in X}$ has infinitely many realizations. Suppose not. Then there exists a finite subset $X_0$ of $X$ such that
\[
\bigcap_{\bar{\eta}\in X_0}\varphi(\mathbb{M},\bar{a}_{\bar{\eta}})
=
\bigcap_{\bar{\eta}\in X}\varphi(\mathbb{M},\bar{a}_{\bar{\eta}}).
\]

Let $W$, $X'$, and $Y'$ be subsets of $(\CI^*)^k$ satisfying condition (vii) in \Cref{def: higher-arity n-var quadruple}. There exists $X'_0\subseteq X'$ such that $X'_0\sim_{\CL^*}X_0$. Since 
\[
\bigcap_{\bar{\eta}\in X'_0}\varphi(\mathbb{M},\bar{a}_{\bar{\eta}})
=
\bigg(\bigcap_{\bar{\eta}\in X'_0}\varphi(\mathbb{M},\bar{a}_{\bar{\eta}})\bigg)
\cap\varphi(\mathbb{M},\bar{a}_{\bar{\nu}})
\]
for each $\bar{\nu}\in Y'$, $\{\varphi(x,\bar{a}_{\bar{\nu}})\}_{\bar{\nu}\in Y'}$ is consistent. This yields a contradiction.
\end{proof}
\end{lemma}

\begin{theorem}\label{thm: n-var theorem higher arity}
Let $D^{\CL^*\!\!\!,\,\CI^*}_{k,X,\bar{Y}}\!\!\in\mathfrak{D}^h_{n\text{-var}}$. Let $\CL$ be a language, $T$ a complete $\CL$-theory, and $\mathbb{M}\models T$ a sufficiently saturated model. If $T\notin D^{\CL^*\!\!\!,\,\CI^*}_{k,X,\bar{Y}}$, then there exist $\varphi(\bar{x},\bar{y}):=\varphi(\bar{x},y_0,...,y_{k-1})\in\CL$ and $\CI^*_{\CL^*}$-indiscernible $(a_\eta)_{\eta\in \CI^*}\!$ in $\mathbb{M}$ with $|y_0|=\cdots=|y_{k-1}|=|a_\eta|$ for all $\eta\in\CI^*$ such that
\begin{itemize}
\item[(i)] $\varphi(\bar{x},\bar{y})$ witnesses $\neg D^{\CL^*\!\!\!,\,\CI^*}_{k,X,\bar{Y}}$ with $(a_\eta)_{\eta\in \CI^*}$,
\item[(ii)] there exists $\bar{b}\models\{\varphi(\bar{x},\bar{a}_{\bar{\eta}})\}_{\bar{\eta}\in X}$ such that $\dim(\bar{b}/(a_\eta)_{\eta\in\CI^*}\!)=|\bar{x}|$.
\end{itemize}
\begin{proof}
Suppose $T\notin D^{\CL^*\!\!\!,\,\CI^*}_{k,X,\bar{Y}}$. Then there exist $n<\omega$, $\varphi(\bar{x},\bar{y}):=\varphi(x_0,...,x_{n-1},y_0,...,y_{k-1})\in\CL$, and $\CI^*_{\CL^*}$-indiscernible $(a_\eta)_{\eta\in \CI^*}\!$ in $\mathbb{M}$ with $|x_0|=\cdots=|x_{n-1}|=1$ and $|y_0|=\cdots=|y_{k-1}|=|a_\eta|$ for all $\eta\in\CI^*$ such that
\begin{itemize}
\item[(iii)] $\{\varphi(\bar{x},\bar{a}_{\bar{\eta}})\}_{\bar{\eta}\in X}$ is consistent,
\item[(iv)] $\{\varphi(\bar{x},\bar{a}_{\bar{\nu}})\}_{\bar{\nu}\in Y}$ is inconsistent for each $Y\in\bar{Y}$.
\end{itemize}
If there exists $\bar{b}\models\{\varphi(\bar{x},\bar{a}_{\bar{\eta}})\}_{\bar{\eta}\in X}$ such that $\dim(\bar{b}/(a_\eta)_{\eta\in\CI^*}\!)=n$, then there is nothing to prove. Suppose not. Then $n>1$ by \Cref{lem: 1-var lemma higher arity}. Choose any $\bar{b}\models\{\varphi(\bar{x},\bar{a}_{\bar{\eta}})\}_{\bar{\eta}\in X}$. Then $\bar{b}$ can be written of the form $(b_0,...,b_{n-1})$ where $|b_i|=1$ for each $i<n$. Since $\dim(\bar{b}/(a_\eta)_{\eta\in\CI^*}\!)<n$, there exist $i_0<n$, a natural number $l$, $\theta(x_{i_0},\check{x}^{i_0},\bar{y}'):=\theta(x_{i_0},\check{x}^{i_0},y'_0,...,y'_{l-1})\in\CL$ with $|y'_i|=|a_\nu|$ for each $i<l$ and $\nu\in\CI^*$, $\bar{\nu}:=(\nu_0,...,\nu_{l-1})\in \CI^*$, and $q<\omega$ such that 
\begin{itemize}
\item[(v)] $\models \theta(b_{i_0},\check{b}^{i_0},\bar{a}_{\bar{\nu}})\wedge \exists^{\le q}x\theta(x,\check{b}^{i_0},\bar{a}_{\bar{\nu}})$.
\end{itemize}
Let 
\begin{itemize}
\item[(vi)] $\begin{aligned}[t]
\varphi'(\bar{x},\bar{y},\bar{y}')
&:=\varphi'(x_0,...,x_{n-1},y_0,...,y_{k-1},y'_0,...,y'_{l-1})\\
&:=\varphi(x_0,...,x_{n-1},y_0,...,y_{k-1})
\wedge\theta(x_{i_0},\check{x}^{i_0},y'_0,...,y'_{l-1}).
\end{aligned}$
\end{itemize}

Then $\{\varphi'(b_0,...,b_{i_0-1},x,b_{i_0+1},...,b_{n-1},\bar{a}_{\bar{\xi}},\bar{a}_{\bar{\nu}})\}_{\bar{\xi}\in X}$ has only finitely many realizations. Thus there exists a finite subset $X_0$ of $X$ such that
\[
\bigcap_{\bar{\xi}\in X_0}\varphi'(b_0,...,b_{i_0-1},\mathbb{M},b_{i_0+1},...,b_{n-1},\bar{a}_{\bar{\xi}},\bar{a}_{\bar{\nu}})
\]\vspace{-8pt}
\[\;\;\;\;\;=\bigcap_{\bar{\xi}\in X}\varphi'(b_0,...,b_{i_0-1},\mathbb{M},b_{i_0+1},...,b_{n-1},\bar{a}_{\bar{\xi}},\bar{a}_{\bar{\nu}}).
\]
 Note that 
 \[\hspace{-10pt}\textrm{(vii)}\hspace{8pt}\exists x\!\!\! \bigwedge_{\bar{\xi}\in X_0}\!\!\! \varphi'(\bar{b}^{i_0/x}\!,\bar{a}_{\bar{\xi}},\bar{a}_{\bar{\nu}})\wedge\forall x\!\left(\!\!\left(\bigwedge_{\bar{\xi}\in X_0}\!\! \varphi'(\bar{b}^{i_0/x}\!,\bar{a}_{\bar{\xi}},\bar{a}_{\bar{\nu}})\wedge\varphi'(\bar{b}^{i_0/x}\!,\bar{a}_{\bar{\xi}'},\bar{a}_{\bar{\nu}})\!\right)\!\leftrightarrow\!\! \bigwedge_{\bar{\xi}\in X_0}\!\! \varphi'(\bar{b}^{i_0/x}\!,\bar{a}_{\bar{\xi}},\bar{a}_{\bar{\nu}}) \!\right)
\]
for all $\bar{\xi}'\in X\setminus X_0$. Let $\{\bar{\eta}_0,...,\bar{\eta}_{m-1}\}$ be an enumeration of $X_0$, $\bar{\bar{\eta}}:=(\bar{\eta}_0,...,\bar{\eta}_{m-1})$, and $\bar{\bar{a}}_{\bar{\bar{\eta}}}:=(\bar{a}_{\bar{\eta}_0},...,\bar{a}_{\bar{\eta}_{m-1}})$. Let $\bar{\bar{y}}'':=(\bar{y}''_0,...,\bar{y}''_{m-1})$  be a tuple of variables such that $|\bar{y}''_i|=k|a_\eta|$ for each $i<m$. Put
\begin{itemize}
\item[(viii)] $\varphi''(\check{x}^{i_0}\!\!,\bar{y},\bar{y}'\!,\bar{\bar{y}}'')$

\vspace{7pt}

$\hspace{-10pt}:=\varphi''(x_0,...,x_{i_0-1},x_{i_0+1},...,x_{n-1},y_0,...,y_{k-1},y'_0,...,y'_{l-1},\bar{y}''_0,...,\bar{y}''_{m-1})$

\vspace{-11pt}

\[
\hspace{8pt}:=\exists x\! \bigwedge_{i<m}\! \varphi'(\bar{x}^{i_0/x}\!,\bar{y}''_i,\bar{y}')\wedge\forall x \!\left(\!\!\left(\bigwedge_{i<m}\!\!\varphi'(\bar{x}^{i_0/x}\!,\bar{y}''_i,\bar{y}')\wedge\varphi'(\bar{x}^{i_0/x}\!,\bar{y},\bar{y}')\!\right)\!\leftrightarrow\!\!\bigwedge_{i<m}\!\varphi'(\bar{x}^{i_0/x}\!,\bar{y}''_i,\bar{y}')\!\right).
\]
\end{itemize}
Then by (vii), we have 
\begin{itemize}
\item[(ix)] $\models\varphi''(\check{b}^{i_0},\bar{a}_{\bar{\xi}'},\bar{a}_{\bar{\nu}},\bar{\bar{a}}_{\bar{\bar{\eta}}})$ for each $\bar{\xi}'\in X\setminus X_0$. 
\end{itemize}

\noindent Since $D^{\CL^*\!\!\!,\,\CI^*}_{k,X,\bar{Y}}\!\!\in\mathfrak{D}^h_{n\text{-var}}$, there exists an $\CL^*$-embedding $f:\CI^*\to\CI^*$ such that 
\begin{itemize}
\item[(x)] $f(\CI^*)\cap \bar{\nu}\bar{\bar{\eta}}=\emptyset$,
\item[(xi)] $\bar{\xi}\sim_{\CL^*}\!\bar{\xi}'\;\Rightarrow\; \bar{\xi}\bar{\nu}\bar{\bar{\eta}}\sim_{\CL^*}\!\bar{\xi}'\bar{\nu}\bar{\bar{\eta}}$ for all $\bar{\xi},\bar{\xi}'\in f(\CI^*)$,
\item[(xii)] $(f(\CI^*)^k\cap X)\sim_{\CL^*}^\text{fin}X$.
\end{itemize} 
Let $a'_\xi:=a_{f(\xi)}\bar{a}_{\bar{\nu}}\bar{\bar{a}}_{\bar{\bar{\eta}}}$ for each $\xi\in\CI^*$.
Then $(a'_\xi)_{\xi\in\CI^*}$ is $\CI^*_{\CL^*}\!$-indiscernible by (xi). For each $i<k$, let $\bar{y}'''_i:=(y_i,\bar{y}'_i,\bar{\bar{y}}''_i)$ be a tuple of variables such that $|\bar{y}'_i|=|\bar{y}'|$ and $|\bar{\bar{y}}''_i|=|\bar{\bar{y}}''|$. Put 

\smallskip

\begin{itemize}
\item[(xiii)] $\varphi'''(\check{x}^{i_0};\bar{y}'''_0\!,...,\bar{y}'''_{k-1})$

\vspace{5pt}

$\hspace{-12pt}:=\varphi'''(\check{x}^{i_0};y_0,\bar{y}'_0,\bar{\bar{y}}''_0,\;...\;,y_{k-1},\bar{y}'_{k-1},\bar{\bar{y}}''_{k-1})$

\vspace{-6pt}

\[\hspace{-128pt}:=\varphi''(\check{x}^{i_0};y_0,...,y_{k-1},\bar{y}'_0,\bar{\bar{y}}''_0)\wedge\!\!\!\bigwedge_{i<j<k}\!\!\!\bar{y}'_i=\bar{y}'_j\;\wedge\!\!\!\bigwedge_{i<j<k}\!\!\!\bar{\bar{y}}''_i=\bar{\bar{y}}''_j.\]
\end{itemize}

\smallskip

We prove that $\varphi'''(\check{x}^{i_0};\bar{y}'''_0\!,...,\bar{y}'''_{k-1})$ witnesses $\neg D^{\CL^*\!\!\!,\,\CI^*}_{k, X,\bar{Y}}$ with $(a'_\xi)_{\xi\in\CI^*}$. Choose any finite subset $X_1$ of $X$. By (x), (xi), and (xii), there exists $X'_1\subset X\setminus X_0$ such that $X'_1\bar{\nu}\bar{\bar{\eta}}\sim_{\CL^*}\!\!f(X_1)\bar{\nu}\bar{\bar{\eta}}$. Thus $\{\varphi'''(\check{x}^{i_0};\bar{a}'_{\bar{\xi}})\}_{\bar{\xi}\in X_1}$ is consistent by (ix). Thus $\{\varphi'''(\check{x}^{i_0};\bar{a}'_{\bar{\xi}})\}_{\bar{\xi}\in X}$ is consistent by compactness.

Now we show that $\{\varphi'''(\check{x}^{i_0};\bar a'_{\bar{\xi}})\}_{\bar{\xi}\in Y}$ is inconsistent for each $Y\in\bar{Y}$. Suppose $\{\varphi'''(\check{x}^{i_0};\bar a'_{\bar{\xi}})\}_{\bar{\xi}\in Y}$ is realized by $\check{c}^{i_0}:=(c_0,...,c_{i_0-1},c_{i_0+1},...,c_{n-1})$. Let $\bar{c}^{i_0/x}:=(c_0,...,c_{i_0-1},x,c_{i_0+1},...,c_{n-1})$ for a variable $x$.
Then for each $\bar{\xi}\in Y$, we have
\[\hspace{-13pt}\textrm{(xiv)}\hspace{8pt}\exists x\!\! \bigwedge_{i<m}\!\! \varphi'(\bar{c}^{i_0/x}\!,\bar{a}_{\bar{\eta}_i},\bar{a}_{\bar{\nu}})\wedge\forall x \!\left( \!\!\left( \bigwedge_{i<m}\!\varphi'(\bar{c}^{i_0/x}\!,\bar{a}_{\bar{\eta}_i},\bar{a}_{\bar{\nu}})\wedge \varphi'(\bar{c}^{i_0/x}\!,\bar{a}_{f(\bar{\xi})},\bar{a}_{\bar{\nu}})\!\! \right)\!\leftrightarrow\! \bigwedge_{i<m}\!\varphi'(\bar{c}^{i_0/x}\!,\bar{a}_{\bar{\eta}_i},\bar{a}_{\bar{\nu}})\!\! \right).
\]
Thus there exists $c_{i_0}$ such that $\models\varphi'(\bar{c}^{i_0/c_{i_0}}\!,\bar{a}_{f(\bar{\xi})},\bar{a}_{\bar{\nu}})$ for each $\bar{\xi}\in Y$, where $\bar{c}^{i_0/c_{i_0}}=(c_0,...,c_{n-1})$. By the choice of $\varphi'$, $\{\varphi(\bar{x},\bar{a}_{\bar{\xi}})\}_{\bar{\xi}\in f(Y)}$ is consistent. It is a contradiction since $f(Y)\sim_{\CL^*}Y$. Thus $\varphi'''$ witnesses $\neg D^{\CL^*\!\!\!,\,\CI^*}_{k,X,\bar{Y}}$ with $(a'_\xi)_{\xi\in\CI^*}$.

Note that the length of the free variable part of the new witness of $\neg D^{\CL^*\!\!\!,\,\CI^*}_{k,X,\bar{Y}}$ is $n-1$. So we can repeat this process until we get a witness of $\neg D^{\CL^*\!\!\!,\,\CI^*}_{k,X,\bar{Y}}$ satisfying (ii), by \Cref{lem: 1-var lemma higher arity}.
\end{proof}
\end{theorem}

\subsection{Preservation results for $\mathfrak{D}^h_{n\text{-var}}$}

\begin{proposition}\label{prop: n-var theorem exchange higher}
Let $D^{\CL^*\!\!\!,\,\CI^*}_{k,X,\bar{Y}}\!\!\in\mathfrak{D}^h_{n\text{-var}}$. Let $\CL$ be a language, $T$ a complete $\CL$-theory, and $\mathbb{M}\models T$ a sufficiently saturated model. Suppose that $\acl$ in $T$ satisfies the exchange property. Then the following are equivalent.
\begin{itemize}
\item[(i)] $T\notin D^{\CL^*\!\!\!,\,\CI^*}_{k,X,\bar{Y}}\!$.
\item[(ii)] there exist $\varphi(\bar{x},y_0,...,y_{k-1})\in \CL$ and $\CI^*_{\CL^*}\!$-indiscernible $(a_\eta)_{\eta\in\CI^*}$ such that 
\begin{itemize}
\item[$\ast$] $\dim(\{\varphi(\bar{x},\bar{a}_{\bar{\eta}})\}_{\bar{\eta}\in X})=|\bar{x}|$,
\item[$\ast$] $\dim(\{\varphi(\bar{x},\bar{a}_{\bar{\nu}})\}_{\bar{\nu}\in Y})=0$ for all $Y\in\bar{Y}$. 
\end{itemize}
\item[(iii)] there exist $\varphi(\bar{x},y_0,...,y_{k-1})\in \CL$ and $\CI^*_{\CL^*}\!$-indiscernible $(a_\eta)_{\eta\in\CI^*}$ such that 
\begin{itemize}
\item[$\ast$] $\dim(\{\varphi(\bar{x},\bar{a}_{\bar{\eta}})\}_{\bar{\eta}\in X})=|\bar{x}|$,
\item[$\ast$] $\dim(\{\varphi(\bar{x},\bar{a}_{\bar{\nu}})\}_{\bar{\nu}\in Y})<|\bar{x}|$ for all $Y\in\bar{Y}$. 
\end{itemize}
\end{itemize}
\begin{proof}
The implication from (i) to (ii) follows from \Cref{thm: n-var theorem higher arity}. From (ii) to (iii) is trivial.  Suppose (iii). Let $n:=|\bar{x}|$. Then $\varphi$ can be written of the form $\varphi(x_0,...,x_{n-1},\bar{y})$ where $|x_i|=1$ for all $i<n$. Consider 
\[
\Phi((z_\xi)_{\xi\in\omega^{\le n}},\bar{y}):=\{\varphi(z_{\xi_1},...,z_{\xi_{n}},\bar{y}):\emptyset\lhd\xi_1\lhd\cdots\lhd\xi_n\}\cup\{z_\xi\neq z_{\xi'}\}_{\xi\neq\xi'}.
\] 
Then $\bigcup_{\bar{\eta}\in X}\Phi( (z_\xi)_{\xi\in\omega^{\le n}}, \bar{a}_{\bar{\eta}})$ is consistent. Fix $Y\in\bar{Y}$ and suppose $\bigcup_{\bar{\nu}\in Y}\Phi( (z_\xi)_{\xi\in\omega^{\le n}}, \bar{a}_{\bar{\nu}})$ is realized by $(b_\xi)_{\xi\in\omega^{\le n}}$. By applying the s-modeling property, we may assume that $(b_\xi)_{\xi\in\omega^{\le n}}$ is s-indiscernible over $(\bar{a}_{\bar{\nu}})_{\bar{\nu}\in Y}$. Then for any $\emptyset\lhd\xi_1\lhd\cdots\lhd\xi_n$, $b_{\xi_i}\notin\acl((\bar{a}_{\bar{\nu}})_{\bar{\nu}\in Y} b_{\xi_{<i}})$ for all $0<i\le n$. By the exchange property, we have $\dim(b_{\xi_1}...b_{\xi_n}/(\bar{a}_{\bar{\nu}})_{\bar{\nu}\in Y})=n=|\bar{x}|$, which yields a contradiction. Thus $\bigcup_{\bar{\nu}\in Y}\Phi( (z_\xi)_{\xi\in\omega^{\le n}}, \bar{a}_{\bar{\nu}})$ is inconsistent for all $Y\in \bar{Y}$. Since $|\bar{Y}|$ is finite, there exists a finite conjunction of formulas in $ \Phi((z_\xi)_{\xi\in\omega^{\le n}},\bar{y})$ witnessing $\neg D^{\CL^*\!\!\!,\,\CI^*}_{k,X,\bar{Y}}\!$ with $(a_\eta)_{\eta\in\CI^*}$. Thus $T\notin D^{\CL^*\!\!\!,\,\CI^*}_{k,X,\bar{Y}}$.
\end{proof}
\end{proposition}

\begin{theorem}\label{thm: T^gt preservation higher}
Let $\CL$ be a language, $T$ be a geometric $\CL$-theory, and $D^{\CL^*\!\!\!,\,\CI^*}_{k,X,\bar{Y}}\!\!\in\mathfrak{D}^h_{n\text{-var}}$. Suppose $\acl_\CL=\dcl_\CL$. Then $T\in D^{\CL^*\!\!\!,\,\CI^*}_{k,X,\bar{Y}}$ if and only if $T^{gt}\in\! D^{\CL^*\!\!\!,\,\CI^*}_{k,X,\bar{Y}}$.
\begin{proof}
From left to right follows from \Cref{prop: n-var theorem exchange higher} and the argument in \Cref{thm: T^gt preservation}. 

For the converse, suppose $T\notin D^{\CL^*\!,\CI^*}_{k,X,\bar Y}$. Choose a sufficiently saturated model $(\mathbb{M},H(\mathbb{M}))\models T^{ind}$. Since $\mathbb{M}\models T$, there exist $\varphi(x,y_0,...,y_{k-1})\in \CL$ and $\CL^*$-indiscernible $(a_\eta)_{\eta\in\CI^*}$ in $\CL$ such that 
\begin{itemize}
\item[(i)] $\dim_{\CL}(\{\varphi(x,\bar a_{\bar\eta})\}_{\bar\eta\in X})=|x|$,
\item[(ii)] $\{\varphi(x,\bar a_{\bar\nu})\}_{\bar\nu\in Y}$ is inconsistent for each $Y\in\bar Y$.
\end{itemize}
We may assume there exists $m<\omega$ such that $|y_0|=\cdots=|y_{k-1}|=m$. Since $\CI^*$ is $k$-monochromatically extendable over $\mathbbof{X}$, we have $\CI^\dagger$ with $\age(\CI^*)=\age(\CI^\dagger)$ such that for any coloring $c:(\CI^\dagger)^k\to |T|$, there exists an $\CL^*$-embedding $f:\CI^*\to\CI^\dagger$ such that $c(f(\bar\eta))=c(f(\bar\nu))$ for any $\bar\eta,\bar\nu\in\mathbbof{X}$ with $\la\bar\eta\ra\sim\la\bar\nu\ra$. By the modeling property, we can find $\CL^*$-indiscernible $( a'_\eta)_{\eta\in\CI^\dagger}$ in $\CL$ such that 
\begin{itemize}
\item[(iii)] $\dim_\CL(\{\varphi(x,\bar{a}'_{\bar\eta})\}_{\bar\eta\in X'})=|x|$ for any finite subset  $X'$ of $(\CI^\dagger)^k$ such that  $X'\sim X''$ for some $X''\subseteq X$,
\item[(iv)] for each $Y\in\bar Y$, there exists a finite subset $Y'$ of $Y$ such that $\{\varphi(x,\bar{a}'_{\bar\nu})\}_{\bar\nu\in Y''}$ is inconsistent for any $Y''\subseteq(\CI^\dagger)^k$ with $Y''\sim Y'$.
\end{itemize}
By the density property, we may assume $( a'_\eta)_{\eta\in\CI^\dagger}\subseteq\acl_\CL(H(\mathbb{M}))=\dcl_\CL(H(\mathbb{M}))$. Note that $a'_\eta$ is of the form $( a'_{\eta,0},...,a'_{\eta,m-1})$ with $|a'_{\eta,0}|=\cdots=|a'_{\eta,m-1}|=1$. Let $y'$ be a variable with $|y'|=1$. Then for each $\eta\in\CI^\dagger$ and $j<m$, there exist $\psi_{\eta,j}\in\CL$ and $b_{\eta,j}\in H(\mathbb M)$ such that $\models\exists^!y'\psi_{\eta,j}(y',b_{\eta,j})\wedge\psi_{\eta,j}(a'_{\eta,j},b_{\eta,j})$.

 For each $\eta\in\CI^\dagger$, let $\bar\psi_\eta:=(\psi_{\eta,0},...,\psi_{\eta,m-1})$. We define a $|T|$-coloring $c$ on $(\CI^\dagger)^k$ by $c(\eta_0,...,\eta_{k-1})=(\bar\psi_{\eta_0},...,\bar\psi_{\eta_{k-1}})$ for each $\bar\eta\in(\CI^\dagger)^k$. There exists an $\CL^*$-embedding $f:\CI^*\to\CI^\dagger$ such that $c(f(\bar\eta))=c(f(\bar\nu))$ for all $\bar\eta,\bar\nu\in\mathbbof X$ with $\la\bar\eta\ra\sim\la\bar\nu\ra$. By applying the modeling property, we obtain an $\CL^*$-indiscernible $(a''_\eta)_{\eta\in\CI^*}$ in $\CL_H$ which is $\CL^*$-locally based on $(a'_{f(\eta)})_{\eta\in\CI^*}$ in $\CL_H$. Note that $a''_\eta$ is also of the form $(a''_{\eta,0},...,a''_{\eta,m-1})$ such that $|a''_{\eta,j}|=1$ for all $j<m$.

We claim that $\bar a''_{\bar\eta}\in \dcl_\CL(H(\mathbb{M}))$ for each $\bar\eta\in\mathbbof{X}$. Choose any $\bar\eta:=(\eta_0,...,\eta_{k-1})\in \mathbbof{X}$. Then $\bar a'_{f(\bar\eta)}$ is defined by $\bar\psi_{f(\eta_0)},...,\bar\psi_{f(\eta_{k-1})}$ with parameters in $H(\mathbb{M})$. If $\bar a''_{\bar\eta}$ is not defined by $\bar\psi_{f(\eta_0)},...,\bar\psi_{f(\eta_{k-1})}$ with parameters in $H(\mathbb M)$, then there exist $j<m$ and $k'<k$ such that
\begin{itemize}
\item[(v)] \makebox[\linewidth][c]{$\displaystyle 
\models \neg\exists z(H(z)\wedge\exists^! y'\psi_{f(\eta_{k'}),j}(y',z)\wedge \psi_{f(\eta_{k'}),j}(a''_{\eta_{k'},j},z)).$
}
\end{itemize}
Since $(a''_\eta)_{\eta\in\CI^*}$ is locally based on $(a'_{f(\eta)})_{\eta\in\CI^*}$, there exists $\la\bar\nu\ra\sim\la\bar\eta\ra$ such that $\bar a'_{f(\bar\nu)}$ satisfies (v). Note that $\bar\nu\in\mathbbof{X}$. Thus we have $c(f(\bar\nu))=c(f(\bar\eta))$, and hence $\bar a'_{f(\bar\nu)}$ is defined by $\bar\psi_{f(\eta_0)},...,\bar\psi_{f(\eta_{k-1})}$. This yields a contradiction with (v). Thus $\bar a''_{\bar\eta}\in \dcl_\CL(H(\mathbb{M}))$ for all $\bar\eta\in\mathbbof{X}$. If $\eta\in\CI^*$ and $\bar\eta\notin\mathbbof{X}$ for any $\bar\eta\in(\CI^*)^k$ containing $\eta$, then we replace $a''_\eta$ with any element in $H(\mathbb{M})$. Then $(a''_\eta)_{\eta\in\CI^*}\subseteq \dcl_\CL(H(\mathbb{M}))$.

Since $\CI^*$ is $k$-weakly monochromatic, it is $1$-weakly monochromatic. Thus there exist $d<\omega$ and $\{\chi_{e,j}(y',z)\}_{e<d,j<m}\in\CL$, $(c_{\eta,j})_{\eta\in\CI^*,j<m}$, and a surjection $g:\CI^*\to d$ such that
\begin{itemize}
\item[(vi)] $\models\exists^!y'\chi_{g(\eta),j}(y',c_{\eta,j})\wedge \chi_{g(\eta),j}(a''_{\eta,j},c_{\eta,j})$ for each $\eta\in\CI^*$ and $j<m$,
\end{itemize}
For each $j<m$, let 
\[
\chi_j(y';z_0,...,z_{d-1},s_0,...,s_{d-1},t_0,...,t_{d-1})
:=\bigvee_{e<d}\Big( s_e=t_e \wedge \chi_{e,j}(y',z_e)\Big)
\]
Then for each $\eta\in\CI^*$ and $j<m$, $c_{\eta,j}$ can be extended to $\bar c_{\eta,j}\in H(\mathbb{M})$ such that 
\[
\models \exists^!y' \chi_j(y';\bar c_{\eta,j})\wedge   \chi_j(a''_{\eta,j};\bar c_{\eta,j})
\]
Let $y'_0,...,y'_{m-1}$ be variables with $|y'_0|=\cdots=|y'_{m-1}|=1$. Let $u:=(z_0,...,z_{d-1},s_0,...,s_{d-1},t_0,...,t_{d-1})$ and
\[
\chi(y'_0,...,y'_{m-1};u_0,...,u_{m-1}):=\bigwedge_{j<m}\chi_j(y'_j,u_j).
\]
Then 
\[
\models\exists^! y'_0...y'_{m-1} \chi(y'_0,...,y'_{m-1};\bar{c}_{\eta,0},...,\bar c_{\eta,m-1})\wedge\chi(a''_\eta;\bar{c}_{\eta,0},...,\bar c_{\eta,m-1}).
\]
Recall that $|y_0|=\cdots=|y_{k-1}|=m$. Let $\bar c_\eta:=(\bar c_{\eta,0},...,\bar c_{\eta,m-1})$, $w:=(u_0,...,u_{m-1})$, and 
\[
\varphi'(x,w_0,...,w_{k-1}):=\exists y_0,...,y_{k-1}\Big(\varphi(x,y_0,...,y_{k-1})\wedge\bigwedge_{k'<k}\chi(y_{k'},w_{k'})\Big)
\]
Then we have
\begin{itemize}
\item[(vii)] $\dim_\CL( \{\varphi'(x,\bar{c}_{\eta_0},...,\bar{c}_{\eta_{k-1}})\}_{\bar{\eta}\in X'})=|x|$ for any finite $X'$ such that $X'\sim X''$ for some $X''\subseteq X$,
\item[(viii)] for each $Y\in \bar{Y}$, there exists a finite subset $Y'$ of $Y$ such that $\{\varphi'(x,\bar{c}_{\nu_0},...,\bar{c}_{\nu_{k-1}})\}_{\bar\nu\in Y''}$ is inconsistent for any $Y''$ with $Y''\sim Y'$.
\end{itemize}
By the modeling property, we may assume $(\bar c_\eta)_{\eta\in\CI^*}$ is $\CL^*$-indiscernible in $\CL_H$.
By the density property, $\{\varphi'(x,\bar{c}_{\eta_0},...,\bar{c}_{\eta_{k-1}})\}_{\bar\eta\in X}$ has a realization in $H(\mathbb{M})$. Since $(\bar c_\eta)_{\eta\in\CI^*}\subseteq H(\mathbb{M})$ and $\varphi'\in \CL$, we can conclude that $T^{gt}\notin D^{\CL^*\!,\CI^*}_{k,X,\bar Y}$.
\end{proof}
\end{theorem}

\begin{proposition}\label{prop: n-var theorem 1-predicate higher}
Let $D^{\CL^*\!\!\!,\,\CI^*}_{k,X,\bar{Y}}\!\!\in\mathfrak{D}^h_{n\text{-var}}$. Let $\CL$ be a language, $T$ a complete $\CL$-theory, and $\mathbb{M}\models T$ a sufficiently saturated model. Assume that $\CL$ has a unary predicate symbol $P$. If $T\notin D^{\CL^*\!\!\!,\,\CI^*}_{k,X,\bar{Y}}$, then there exist $\varphi(\bar{x},\bar{x}',y_0,...,y_{k-1})\in\CL$ and $\CI^*_{\CL^*}$-indiscernible $(a_\eta)_{\eta\in \CI^*}\!$ such that
\begin{itemize}
\item[(i)] $\varphi'(\bar{x},\bar{x}',\bar y):=\varphi(\bar{x},\bar{x}',\bar y)\wedge \neg P(\bar{x})\wedge P(\bar{x}')$ witnesses $\neg D^{\CL^*\!\!\!,\,\CI^*}_{k,X,\bar{Y}}$ with $(a_\eta)_{\eta\in \CI^*}$,
\item[(ii)] there exists $\bar{b}\bar{b}'\models\{\varphi'(\bar{x},\bar{x}',\bar a_{\bar\eta})\}_{\bar\eta\in X}$ such that 
\begin{itemize}
\item[$\ast$] $\dim_\CL(\bar{b}\bar{b}'/(a_\eta)_{\eta\in\CI^*}\!)=|\bar{x}|+|\bar{x}'|$,
\item[$\ast$] $\dim_\CL(\bar{b}/P(\mathbb{M})(a_\eta)_{\eta\in\CI^*}\!)=|\bar{x}|$.
\end{itemize}
\end{itemize}
\begin{proof}
By the same argument of \Cref{prop: n-var theorem 1-predicate}.
\end{proof}
\end{proposition}

\begin{theorem}\label{thm: T_P T^ind preservation higher}
Let $T$ be a geometric theory and $D^{\CL^*\!\!\!,\,\CI^*}_{k,X,\bar{Y}}\!\!\in\mathfrak{D}^h_{n\text{-var}}$. If $T^*$ is $T_P$ or $T^{ind}$, then $T\in D^{\CL^*\!\!\!,\,\CI^*}_{k,X,\bar{Y}}$ if and only if $T^*\in D^{\CL^*\!\!\!,\,\CI^*}_{k,X,\bar{Y}}$. 
\begin{proof}
Right to left is clear.

 Suppose $T^*\notin D^{\CL^*\!\!\!,\,\CI^*}_{k,X,\bar{Y}}$. By the $k$-monochromatic extendability, there exists $\CI^\dagger$ with $\age(\CI^\dagger)=\age(\CI^*)$ such that for any coloring $c:(\CI^\dagger)^k\to 2^{|T|}$, there exists an $\CL^*$-embedding $f:\CI^*\to \CI^\dagger$ such that $c(f(\bar\eta))=c(f(\bar\nu))$ for any $\bar\eta,\bar\nu\in\mathbbof{X}$ with $\la\bar\eta\ra\sim\la\bar\nu\ra$.
  By \Cref{prop: n-var theorem 1-predicate higher}, the exchange property of $T$, and the modeling property, there exist $\CL^*$-indiscernible $(a_\eta)_{\eta\in\CI^\dagger}$ in $\CL_H$, $\varphi(\bar{x},\bar{x}',y_0,...,y_{k-1})\in\CL_H$ with $\bar{x}:=(x_0,...,x_{n-1})$, $\bar{x}':=(x'_0,...,x'_{m-1})$, and $|x_i|=|x'_j|=1$ for all $i<n$ and $j<m$, such that
\begin{itemize}
\item[(i)] $\varphi'(\bar{x},\bar{x}',\bar y):=\varphi(\bar{x},\bar{x}',\bar y)\wedge \neg H(\bar{x})\wedge H(\bar{x}')$ witnesses $\neg D^{\CL^*\!\!\!,\,\CI^*}_{k,X,\bar{Y}}$ with $(a_\eta)_{\eta\in \CI^\dagger}$, that is
\begin{itemize}
\item[$\ast$] $\{\varphi'(\bar x,\bar x',\bar a_{\bar\eta})\}_{\bar\eta\in X'}$ is consistent if there exists $X''\subseteq X\subseteq (\CI^*)^k$ such that $X''\sim X'$,
\item[$\ast$] For each $Y\in\bar Y$, there exists a finite subset $Y''$ of $Y$ such that $\{\varphi'(\bar x,\bar x',\bar a_{\bar\nu})\}_{\bar\nu\in Y'}$ is inconsistent if $Y'\sim Y''$,
\end{itemize}
\item[(ii)] if $X'\subseteq (\CI^\dagger)^k$ and $X''\sim X'$ for some $X''\subseteq X$, then there exists $\bar{b}\bar{b}'\models\{\varphi'(\bar{x},\bar{x}',\bar a_{\bar\eta})\}_{\bar\eta\in X'}$ such that 
\begin{itemize}
\item[$\ast$] $\dim_{\CL}(\bar{b}\bar{b}'/(a_\eta)_{\eta\in\CI^\dagger}\!)=|\bar{x}|+|\bar{x}'|$,
\item[$\ast$] $\dim_{\CL}(\bar{b}/H(\mathbb{M})(a_\eta)_{\eta\in\CI^\dagger}\!)=|\bar{x}|$.
\end{itemize}
\end{itemize}
Recall that $\mathbbof{X}:=\{\bar\eta\in (\CI^*)^k:\bar\eta\sim\bar\nu\text{ for some }\bar\nu\in X\}$. Let
\[
\mathbbof{X}_1:=\{\eta\in\CI^*:\eta\in\bar\eta\text{ for some }\bar\eta\in\mathbbof{X}\}.
\]
Choose any $C\subseteq\mathbb{M}\!\setminus\!\! H(\mathbb{M})$ with $|C|\ge \omega$ such that
\begin{itemize}
\item[(iii)] $c\notin\acl_{\CL}(H(\mathbb{M})C\!\setminus\!\{c\})$ for each $c\in C$,
\item[(iv)] $(a_\eta)_{\eta\in\CI^\dagger}\subseteq\acl_\CL(H(\mathbb{M})C)$.
\end{itemize}
We may assume there exists $l<\omega$ such that  $|a_\eta|=|a_\nu|=l$ for all $\eta,\nu\in \CI^\dagger$. So $a_\eta$ is of the form $(a_{\eta,0},...,a_{\eta,l-1})$.
For any $\eta\in\CI^\dagger$, choose any finite tuples $\bar c_\eta\in C$ and $\bar h_\eta\in H(\mathbb{M})$ such that
\begin{itemize}
\item[(v)] $a_\eta\subseteq \acl_\CL(\bar c_\eta \bar h_\eta)$,
\item[(vi)] if $a_{\eta,l'}\in\acl_\CL(H(\mathbb{M}))$, then $a_{\eta,l'}\in \acl_\CL(\bar h_\eta)$ for each $l'<l$,
\item[(vii)] if $a_{\eta,l'}\in H(\mathbb{M})$, then $a_{\eta,l'}=h$ for some $h\in \bar{h}_\eta$, for each $l'<l$.
\end{itemize}
Note that $|\bar c_\eta\bar h_\eta|$'s may differ and not be bounded. Also note that 
\begin{itemize}
\item[(viii)] $a_{\eta_0}\bar c_{\eta_0} \bar h_{\eta_0}...a_{\eta_{d-1}}\bar c_{\eta_{d-1}} \bar h_{\eta_{d-1}}$ is $H$-independent for each $\eta_0,...,\eta_{d-1}\in\CI^\dagger$,
\end{itemize}
by (v), (vi), and (vii).
Define a coloring $c$ on $(\CI^\dagger)^k$ such that 
\[
c(\eta_0,...,\eta_{k-1})=\tp_{\CL_H}(a_{\eta_0}\bar c_{\eta_0}\bar h_{\eta_0},...,a_{\eta_{k-1}}\bar c_{\eta_{k-1}}\bar h_{\eta_{k-1}}).
\] 
Since $\CI^\dagger$ is a $k$-monochromatic extension of $\CI^*$ over $\mathbbof{X}$, there exists an $\CL^*$-embedding $f:\CI^*\to \CI^\dagger$ such that $c(f(\bar\eta))=c(f(\bar\nu))$ for all $\bar\eta,\bar\nu\in\mathbbof{X}$ with $\la\bar\eta\ra\sim\la\bar\nu\ra$.
For each $\eta\in\CI^*$, let $a'_\eta:=a_{f(\eta)}$, $\bar c'_\eta:=\bar c_{f(\eta)}$, and $\bar h'_\eta:=\bar h_{f(\eta)}$. Since $\CI^*$ is $k$-weakly monochromatic, there exists $N<\omega$ such that $|\bar c'_{\eta}\bar h'_{\eta}|<N$ for all $\eta\in\mathbbof{X}_1$. Thus we may assume $|\bar c'_\eta|=|\bar c'_\nu|$ and $|\bar h'_\eta|=|\bar h'_\nu|$ for any $\eta,\nu\in\mathbbof{X}_1$. For $\eta\in\CI^*\!\setminus\!\mathbbof{X}_1$, just replace $\bar c'_\eta$ and $\bar h'_\eta$ with any tuples in $C$ and $H(\mathbb{M})$, respectively, having the appropriate lengths. Thus we have $|\bar c'_\eta|=|\bar c'_\nu|$ and $|\bar h'_\eta|=|\bar h'_\nu|$ for any $\eta,\nu\in\CI^*$. For each $\eta\in\CI^*$, let $d'_\eta:=a'_\eta \bar c'_\eta \bar h'_\eta$. For each $i<k$, let $u_i$ and $v_i$ be variables with $|u_i|=|\bar c'_\eta|$ and $|v_i|=|\bar h'_\eta|$, let  $y'_i:=(y_i,u_i,v_i)$, and 
\[
\varphi''(\bar x,\bar x',y'_0,...,y'_{k-1}):=\varphi'(\bar x,\bar x',y_0,...,y_{k-1})\wedge \bigwedge_{i<k}(u_i=u_i\wedge v_i=v_i).
\]
By the modeling property, the argument using the s-modeling property, and the exchange property of $T$, we may assume
\begin{itemize}
\item[(ix)] $(d'_\eta)_{\eta\in\CI^*}$ is $\CL^*$-indiscernible in $\CL_H$,
\item[(x)] $d'_{\eta_0},...,d'_{\eta_{k-1}}$ is $H$-independent for each $(\eta_0,...,\eta_{k-1})\in\mathbbof{X}$,
\item[(xi)] there exists $\bar e\bar e'\models\{\varphi''(\bar x,\bar x',\bar d'_{\bar\eta})\}_{\bar\eta\in X}$ such that
\begin{itemize}
\item[$\ast$] $\dim_\CL(\bar e\bar e'/(d'_\eta)_{\eta\in\CI^*} )=|\bar x|+|\bar x'|$,
\item[$\ast$] $\dim_\CL(\bar e/H(\mathbb{M})(d'_\eta)_{\eta\in\CI^*})=|\bar x|$,
\end{itemize}
\item[(xii)] $\{\varphi''(\bar x,\bar x',\bar d'_{\bar\nu})\}_{\bar\nu\in Y}$ is inconsistent for each $Y\in\bar Y$.
\end{itemize}

\medskip

\noindent\underline{Claim 1.} For $\bar\eta,\bar\nu\in\mathbbof{X}$, if $\bar d'_{\bar\eta}\equiv_{\CL}\bar d'_{\bar\nu}$, then $\bar d'_{\bar\eta}\equiv_{\CL_H}\bar d'_{\bar\nu}$.

\smallskip

\noindent{\it Proof of Claim 1.} By (x), it is enough to show that $\qftp_{\{H\}}(\bar d'_{\bar\eta})=\qftp_{\{H\}}(\bar d'_{\bar\nu})$. By the choice of $d'_\eta$, it is enough to show that $\qftp_{\{H\}}(\bar a'_{\bar\eta})=\qftp_{\{H\}}(\bar a'_{\bar\nu})$. Note that for each $\eta\in\CI^*$, $a'_\eta$ is of the form $(a'_{\eta,0},...,a'_{\eta,l-1})$ such that $|a'_{\eta,l'}|=1$ for each $l'<l$. Suppose $a'_{\eta_i,l'}\in H(\mathbb{M})$ for $i<k$ and $l'<l$. Then by (vii),  we have $a'_{\nu_i,l'}\in H(\mathbb{M})$ since $d'_{\eta_i}\equiv_\CL d'_{\nu_i}$. $\dashv$

\medskip

By Claim 1 and the $k$-weak monochromaticity, there exist $r<\omega$ and $\chi_0,...,\chi_{r-1}\in\CL$ such that
\begin{itemize}
\item[(xiii)] $\models \neg\exists y'_0...y'_{k-1}(\chi_i(y'_0,...,y'_{k-1})\wedge\chi_j(y'_0,...,y'_{k-1}))$ for $i<j<r$,
\item[(xiv)] for $\bar\eta\in\mathbbof{X}$, there exists $i<r$ such that $\models\chi_i(\bar d'_{\bar\eta})$,
\item[(xv)] for $\bar\eta,\bar\nu\in\mathbbof{X}$, $\models\bigwedge_{i<r}(\chi_i(\bar d'_{\bar\eta})\leftrightarrow\chi_i(\bar d'_{\bar\nu}))$ if and only if $\bar d'_{\bar\eta}\equiv_{\CL_H}\bar d'_{\bar\nu}$.
\end{itemize}
For each $\bar\eta\in(\CI^*)^k$, let 
\[
Z^0_{\bar\eta}:=\{\bar{e}\bar{e}'\in\mathbb{M}:\models\varphi''(\bar{e},\bar{e}',\bar d'_{\bar\eta})\wedge \neg H(\bar{e})\wedge H(\bar{e}'),\dim_{\CL}(\bar{e}\bar{e}'/\bar d'_{\bar\eta})=n+m,\; \dim_{\CL}(\bar{e}/ H(\mathbb{M})\bar d'_{\bar\eta})=n\}
\]
and
\[
Z^1_{\bar\eta}:=\{\bar{e}\bar{e}'\in\mathbb{M}:\models \neg\varphi''(\bar{e},\bar{e}',\bar d'_{\bar\eta})\wedge \neg H(\bar{e})\wedge H(\bar{e}'),\dim_{\CL}(\bar{e}\bar{e}'/\bar d'_{\bar\eta})=n+m,\; \dim_{\CL}(\bar{e}/ H(\mathbb{M})\bar d'_{\bar\eta})=n\}.
\]
 By the same argument of \Cref{thm: T_P T^ind preservation}, there exist $\psi_0(\bar{x},\bar{x}'\!,y'_0,...,y'_{k-1}),...,\psi_{r-1}(\bar{x},\bar{x}'\!,y'_0,...,y'_{k-1})\in\CL$ such that for each $\bar\eta\in\mathbbof{X}$, there exists $i<r$ such that $\models \chi_i(\bar d'_{\bar\eta})$ and
\begin{itemize}
\item[(xvi)] $Z^0_{\bar\eta}\subseteq\psi_i(\mathbb{M}^{n+m}\!,\bar d'_{\bar\eta})$,
\item[(xvii)] $Z^1_{\bar\eta}\cap \psi_i(\mathbb{M}^{n+m}\!,\bar d'_{\bar\eta})=\emptyset$.
\end{itemize}

Let 
\[
\psi(\bar x,\bar x',y'_0,...,y'_{k-1}):=\bigvee_{i<r}\Big(\chi_i(y'_0,...,y'_{k-1})\wedge\psi_i(\bar x,\bar x',y'_0,...,y'_{k-1})\Big).
\]
Then by the same argument of \Cref{thm: T_P T^ind preservation}, we can show that 
\begin{itemize}
\item[(xviii)] $\dim_\CL( \{\psi(\bar x,\bar x',\bar d'_{\bar\eta})\}_{\bar\eta\in X}) =|\bar x|+|\bar x'|$
\item[(xix)] $\dim_\CL( \{\psi(\bar x,\bar x',\bar d'_{\bar\nu})\}_{\bar\nu\in Y}) <|\bar x|+|\bar x'|$ for each $Y\in \bar Y$.
\end{itemize}
Thus $T\notin D^{\CL^*\!\!\!,\,\CI^*}_{k,X,\bar{Y}}$ by \Cref{prop: n-var theorem exchange higher}.
\end{proof}
\end{theorem}

The following can be proved by using techniques in \Cref{thm: T_P T^ind preservation higher} and the argument of \Cref{thm: T^G_R preservation}.

\begin{theorem}\label{thm: T^G_R preservation higher}
Let $\mathbb{F}$ be a field of characteristic zero,   $\CL\supseteq\CL_0:=\{+,0,\{\lambda\cdot\}_{\lambda\in\mathbb F}\}$, and $T$ a complete $\CL$-theory expanding the theory of vector spaces over $\mathbb{F}$ which has quantifier elimination in $\CL$ for which $\dcl_\CL=\acl_\CL=\span_\mathbb{F}$ and such that it eliminates the quantifier $\exists^\infty$.
Let $D\in\mathfrak{D}^h_{n\text{-var}}$ and assume that $R$ is a subfield of $\mathbb{F}$. Then $T\in D$ if and only if $T^G_R\in D$. 
\end{theorem}

The following can be proved by the same argument of \Cref{thm: T^delta_g preservation several nc}.

\begin{theorem}\label{thm: T^delta_g preservation higher}
Let $T$ be a complete algebraically bounded theory of fields of characteristic $0$, possibly with additional structure. If $T\in D$ for some $D\in\mathfrak{D}^h_{n\text{-var}}$, then every (some) completion of $T^{\bar\delta,nc}_g$ belongs to $D$.
\end{theorem}

\begin{corollary}\label{cor: all preservation results higher}
Let $T$ be a complete geometric theory with $\acl=\dcl$, and $1<k<\omega$.
\begin{itemize}
\item[(i)] $T$ is NOP$_k$ if and only if $T^{gt}$ is NOP$_k$.
\item[(ii)] $T$ is NFOP$_k$ if and only if $T^{gt}$ is NFOP$_k$.
\item[(iii)] $T$ is NIP$_k$ if and only if $T^{gt}$ is NIP$_k$.
\end{itemize}
Let $T$ be a complete geometric theory and $1<k<\omega$.
\begin{itemize}
\item[(iv)]  $T$ is NOP$_k$ if and only if $T_P$ is NOP$_k$ if and only if $T^{ind}$ is NOP$_k$.
\item[(v)]  $T$ is NFOP$_k$ if and only if $T_P$ is NFOP$_k$ if and only if $T^{ind}$ is NFOP$_k$.
\item[(vi)]  $T$ is NIP$_k$ if and only if $T_P$ is NIP$_k$ if and only if $T^{ind}$ is NIP$_k$.
\end{itemize}
Let $\mathbb{F}$ be a field of characteristic $0$. Let $T$ be a complete $\CL$-theory expanding the theory of vector spaces over $\mathbb{F}$ which has quantifier elimination in $\CL$ for which $\dcl_\CL=\acl_\CL=\span_\mathbb{F}$ and such that it eliminates the quantifier $\exists^\infty$. Let $R$ be a subfield of $\mathbb{F}$. For each $1<k<\omega$,
\begin{itemize}
\item[(vii)] $T$ is NOP$_k$ if and only if $T^G_R$ is NOP$_k$,
\item[(viii)] $T$ is NFOP$_k$ if and only if $T^G_R$ is NFOP$_k$,
\item[(ix)] $T$ is NIP$_k$ if and only if $T^G_R$ is NIP$_k$.
\end{itemize}
Let $T$ be a complete algebraically bounded theory of fields of characteristic $0$, possibly with additional structure. $1<k<\omega$.
\begin{itemize}
\item[(x)] $T$ is NOP$_k$ if and only if every (some) completion of $T^{\bar\delta,nc}_g$ is NOP$_k$.
\item[(xi)] $T$ is NFOP$_k$ if and only if every (some) completion of $T^{\bar\delta,nc}_g$ is NFOP$_k$.
\item[(xii)] $T$ is NIP$_k$ if and only if every (some) completion of $T^{\bar\delta,nc}_g$ is NIP$_k$.
\end{itemize}
\end{corollary}

\begin{remark}
\noindent\Cref{cor: all preservation results higher} (ix) gives a partial answer to the question posed in \cite[p.~25]{BdV22}.
\end{remark}

\appendix

\section{Higher arity TP$_2$}

While constructing a configuration for NIP$_k$ in \Cref{sec: higher arity}, we found a candidate for a higher-arity version of TP$_2$. We will temporarily call it TP$^{*k}_2$, and in this appendix, we briefly introduce its basic properties.

\begin{definition}\label{def: higher arity TP2}
Let $0<k<\omega$ and $1<d<\omega$. We say $\varphi(x,y_0,...,y_{k-1})$ has {\it $d$-TP$_2^{\,*k}$} if there exist $(b^0_{n,m})_{n,m<\omega}$,..., $(b^{k-1}_{n,m})_{n,m<\omega}$ such that 
\begin{itemize}
\item[(i)] $\{\varphi(x,b^0_{n_0,m_{0,e}},...,b^{k-1}_{n_{k-1},m_{k-1,e}})\}_{e<d}$ is inconsistent if $m_{i,0}<\cdots<m_{i,d-1}$ for all $i<k$.
\item[(ii)] $\{\varphi(x,b^0_{n_0, f_0(n_0,...,n_{k-1})},...,b^{k-1}_{n_{k-1},f_{k-1}(n_0,...,n_{k-1})})\}_{n_0,...,n_{k-1}<\omega}$ is consistent for all $f_0:\omega^k\to\omega$, $...$, $f_{k-1}:\omega^k\to \omega$.
\end{itemize}
We say $T$ has $d$-TP$^{*k}_2$ if $\varphi$ has it for some $\varphi$. We say $T$ has {\it TP$^{\,*k}_2$} if it has $d$-TP$^{*k}_2$ for some $1<d<\omega$. We say $T$ is {\it NTP$^{\,*k}_2$} if it does not have TP$^{*k}_2$. 
\end{definition}

\begin{remark}
It is easy to check that TP$^{*1}_2$ is equivalent to TP$_2$ and that TP$^{*k+1}_2$ implies TP$^{*k}_2$ for each $0<k<\omega$.
\end{remark}

\begin{notation}
For $k<\omega$, let $\CL_{\text{TP}^{*k}_2}:=\{<,E,P_0,...,P_{k-1}\}$. Let $\CI_{\text{TP}^{*k}_2}$ be an $\CL_{\text{TP}^{*k}_2}$-structure on $k\times \mathbb{Q}\times\mathbb{Q}$ such that 
\begin{itemize}
\item[(i)] $(i,q,r)<(i',q',r')$ if and only if one of the following holds
\begin{itemize}
\item[$\ast$] $i<i'$,
\item[$\ast$] $i=i'$ and $q<q'$,
\item[$\ast$] $i=i'$, $q=q'$, and $r<r'$,
\end{itemize}
\item[(ii)] $E((i,q,r),(i',q',r'))$ if and only if $i=i'$ and $q=q'$,
\item[(iii)] $P_i((j,q,r))$ if and only if $i=j$ for each $i,j<k$.
\end{itemize}
\end{notation}

\begin{remark}\label{rmk: modeling property higher TP2}
By compactness, if $\varphi$ witnesses $d$-TP$^{*k}_2$, then we can find $(b^i_{q,r})_{(i,q,r)\in k\times\mathbb{Q}\times\mathbb{Q}}$ that satisfies conditions (i) and (ii) in \Cref{def: higher arity TP2} with $\varphi$.
It is easy to check that $\CI_{\text{TP}^{*k}_2}$ has the modeling property. Thus if $\varphi$ witnesses TP$^{*k}_2$ with $(b^i_{n,m})_{(i,n,m)\in k\times\omega\times\omega}$, then by compactness and the modeling property, we can find $\CI_{\text{TP}^{*k}_2}$-indiscernible $(c^i_{q,r})_{(i,q,r)\in k\times \mathbb{Q}\times\mathbb{Q}}$ such that $\varphi$ witnesses TP$^{*k}_2$ with $(c^i_{q,r})_{(i,q,r)\in k\times \mathbb{Q}\times\mathbb{Q}}$.
\end{remark}

\begin{proposition}\label{prop: d-TPk2 => 2-TPk2}
If $T$ has TP$^{*k}_2$, then it has $2$-TP$^{*k}_2$.
\begin{proof}
Suppose $\varphi(x,y_0,...,y_{k-1})$ witnesses $(d+1)$-TP$^{*k}_2$ with $(b^0_{q,r})_{q,r\in\mathbb{Q}},...,(b^{k-1}_{q,r})_{q,r\in\mathbb{Q}}$, for some $d>1$. We may assume $(b^i_{q,r})_{(i,q,r)\in k\times \mathbb{Q}\times\mathbb{Q}}$ is $\CI_{\text{TP}^{*k}_2}$-indiscernible. Let 
\[\psi(x,y_{0,0},...,y_{0,d-1},...,y_{k-1,0},
...,y_{k-1,d-1}):=\bigwedge_{e<d}\varphi(x,y_{0,e},...,y_{k-1,e}).
\]
For each $(i,n,m)\in k\times\omega\times\omega$, let $c^i_{n,m}:= (b^i_{n,md},...,b^i_{n,md+d-1})$.

\medskip

\noindent\underline{Case 1.} For each $(f_i:\omega^k\to\omega)_{i<k}$, $\{\psi(x,c^0_{n_0,f_0(n_0,...,n_{k-1})},...,c^{k-1}_{n_{k-1},f_{k-1}(n_0,...,n_{k-1})}):n_0,...,n_{k-1}<\omega\}$ is consistent.

\smallskip

\noindent In this case, $\psi$ witnesses $2$-TP$^{*k}_2$ with $(c^i_{n,m})_{(i,n,m)\in k\times\omega\times\omega}$.

\medskip

\noindent\underline{Case 2.} For some $(f_i:\omega^k\to\omega)_{i<k}$, $\{\psi(x,c^0_{n_0,f_0(n_0,...,n_{k-1})},...,c^{k-1}_{n_{k-1},f_{k-1}(n_0,...,n_{k-1})}):n_0,...,n_{k-1}<\omega\}$ is inconsistent.

\smallskip

\noindent In this case, there exists $N<\omega$ such that 
\[
\{\psi(x,c^0_{n_0,f_0(n_0,...,n_{k-1})},...,c^{k-1}_{n_{k-1},f_{k-1}(n_0,...,n_{k-1})}):n_0,...,n_{k-1}<N\}
\]
is inconsistent. For each $i<k$, let $\bar y_i:=(y^i_{\bar n})_{\bar n\in N^k}$ such that $|y^i_{\bar n}|=|y_i|$. Let 
\[
\chi(x,\bar y_0,...,\bar y_{k-1}):=\bigwedge_{\bar n\in N^k}\varphi(x,y^0_{\bar n},...,y^{k-1}_{\bar n}).
\]
For each $(i,n,m)\in k\times \omega\times\omega$, let $d^i_{n,m}:= (b^i_{nN+n_i,f_i(n_0,...,n_{k-1})+1-{1\over {m+1} }})_{(n_0,...,n_{k-1})\in N^k}$.
Then by the indiscernibility, $\chi$ witnesses $d$-TP$^{*k}_2$ with $(d^i_{n,m})_{(i,n,m)\in k\times\omega\times\omega}$.
 
\medskip

We can repeat this process until we obtain $2$-TP$^{*k}_2$.
\end{proof}
\end{proposition}

\begin{remark} 
From \Cref{prop: d-TPk2 => 2-TPk2}, it is clear that NIP$_k$ $\Rightarrow$ NTP$^{*k}_2$ for each $1<k<\omega$.
\end{remark}
So we have the following diagram.
\vspace{-15pt}
\begin{figure}[H]
\[
\scalebox{1}
{
\begin{tikzpicture}[
    >={stealth[length=1pt,width=6pt]},
    x={(1.6cm,0cm)},
    y={(0cm,0.82cm)}
]

\footnotesize

\node (NIP)  at (0,0) {NIP};
\node (NTP2) at (1,0) {NTP$_2$};

\node (NIP2)  at (0,1.3) {NIP$_2$};
\node (NTP22) at (1,1.3) {NTP$_2^{*2}$};

\node (NIP3)  at (0,2.6) {NIP$_3$};
\node (NTP23) at (1,2.6) {NTP$_2^{*3}$};

\node (NIP4)  at (0,3.9) {NIP$_4$};
\node (NTP24) at (1,3.9) {NTP$_2^{*4}$};

\node (vdots1) at (0,5.2) {$\vdots$};
\node (vdots2) at (1,5.2) {$\vdots$};

\draw[->] (NIP) -- (NIP2);
\draw[->] (NIP2) -- (NIP3);
\draw[->] (NIP3) -- (NIP4);
\draw[->] (NIP4) -- (vdots1);

\draw[->] (NTP2) -- (NTP22);
\draw[->] (NTP22) -- (NTP23);
\draw[->] (NTP23) -- (NTP24);
\draw[->] (NTP24) -- (vdots2);

\draw[->] (NIP) -- (NTP2);
\draw[->] (NIP2) -- (NTP22);
\draw[->] (NIP3) -- (NTP23);
\draw[->] (NIP4) -- (NTP24);

\end{tikzpicture}
}
\]\vspace{-15pt}
\caption{}
\label{fig: 2}
\end{figure}

\begin{notation}\label{notation: universal TPk2 subset}
Let $(X_i)_{i<\omega}$ be a sequence of finite subsets of $(\CI_{\text{TP}^{*k}_2})^k$ such that
\begin{itemize}
\item[(i)] each $\bar\eta\in X_i$ is of the form $\big((0,q_0,r_0),...,(k-1,q_{k-1},r_{k-1})\big)$ such that $0<q_j<1$ for all $j<k$,
\item[(ii)] If $((i,q_i,r_i))_{i<k},((i,q'_i,r'_i))_{i<k}\in X_i$ are distinct, then $(q_0,...,q_{k-1})\neq (q'_0,...,q'_{k-1})$.
\item[(iii)] If $N<\omega$ and $X$ is a finite subset of $(\CI_{\text{TP}^{*k}_2})^k$ satisfying (i) and (ii), then there exists $i>N$ such that $X_i\sim X$.
\end{itemize}
Let 
\[
X_{\text{TP}^{*k}_2}\!:=\!\bigcup_{i<\omega}\!\bigg\lbrace\!\Big(\!(0,q_0-i,r_0),(1,q_1+i,r_1),...,(k-1,q_{k-1}+i,r_{k-1})\!\Big)\!:\!\Big(\!(0,q_0,r_0),...,(k-1,q_{k-1},r_{k-1})\!\Big)\!\in\! X_i\!\bigg\rbrace
\]
and
\[
Y_{\text{TP}^{*k}_2}:=
\bigg\lbrace
\Big((0,0,0),...,(k-1,0,0)\Big),
\Big((0,0,{1\over 2}),...,(k-1,0,{1\over 2})\Big)
\bigg\rbrace.
\]
\end{notation}

\begin{proposition}
Let $D_{\text{NTP}^{*k}_2}$ be the class of complete NTP$^{*k}_2$ theories. Then $D_{\text{NTP}^{*k}_2}\in \mathfrak{D}^h_{n\text{-var}}$. Thus the preservation results in \Cref{cor: all preservation results higher} also hold for NTP$^{*k}_2$.
\begin{proof}
By the same argument of \ref{lem: k-monochromatic extendable IxkNPk}, we can show that $\CI_{\text{TP}^{*k}_2}$ is $k$-monochromatically extendable over $\mathbbof{X}_{\text{TP}^{*k}_2}$. It is also easy to check that $X_{\text{TP}^{*k}_2}$ is cofinal. 
Clearly $D_{\text{NTP}^{*k}_2}=D^{\CL_{\text{TP}^{*k}_2}\!,\CI_{\text{TP}^{*k}_2}}_{k,X_{\text{TP}^{*k}_2},Y_{\text{TP}^{*k}_2}}\!$. 
\end{proof}
\end{proposition}

\end{document}